\documentclass[11pt,a4paper]{article}

\usepackage[a4paper,top=1in,bottom=1in,left=1in,right=1in]{geometry} 
\usepackage{setspace}
\usepackage[authoryear]{natbib}  
\usepackage{amsmath,amssymb,amsfonts,amsthm,bm}
\usepackage{graphicx,color,xcolor}
\usepackage{multirow,booktabs,array}
\usepackage{float,rotating}
\usepackage{algorithm,algorithmicx,algpseudocode}
\usepackage{url,hyperref}
\usepackage{lipsum}
\usepackage{indentfirst}
\usepackage{authblk}
\usepackage{textcomp}
\usepackage{threeparttable}
\usepackage{pdflscape}
\usepackage{mathrsfs}
\usepackage{bm}
\usepackage{array}
\usepackage{booktabs}
\usepackage{tabularx}
\usepackage{caption}
\usepackage{geometry}
\usepackage{lastpage}
\usepackage{placeins}
\usepackage{adjustbox}
\usepackage{caption}

\DeclareMathOperator{\diag}{diag}
\usepackage{enumitem}
\usepackage{subcaption}

\theoremstyle{plain}
\newtheorem{theo}{Theorem}%
\newtheorem{lemma}{Lemma}%
\newtheorem{prop}{Proposition}
\theoremstyle{definition}
\newtheorem{rem}{Remark}%
\newtheorem{assum}{Assumption}
\newtheorem{coro}{Corollary}
\newtheorem{defin}{Definition}
\date{\today}  

\hypersetup{
	colorlinks=true,
	linkcolor=blue,
	anchorcolor=blue,
	citecolor=blue,
	urlcolor=blue,
	CJKbookmarks=true
}

\begin{document}
	
	\title{Distributed Online Estimation of Spiked Eigenvalues with Adaptive Weighting under Persistent Aspect Ratio Heterogeneity}

	\author[1]{Lu Yan}
	\author[2]{Jiang Hu}
	\author[2]{Yonghan Zhang}
	\author[1]{Xiaoyue Li\thanks{Corresponding author: lixiaoyue@tiangong.edu.cn}}

	\affil[1]{School of Mathematical Sciences, Tiangong University, Tianjin 300387, China}
	\affil[2]{KLASMOE and School of Mathematics \& Statistics, Northeast Normal University, Changchun 130024, China}

	\date{\today}

	\maketitle

	\begin{abstract}
		We study online estimation of spiked covariance eigenvalues from 
		observations distributed across $L$ nodes with heterogeneous and 
		persistent effective sample sizes. In the proportional high-dimensional 
		regime, local Rayleigh statistics are deterministically distorted by 
		node-specific aspect ratios $c_{\ell,t}=p/N^{\mathrm{eff}}_{\ell,t}$, 
		and direct aggregation of uncorrected statistics converges to the wrong 
		limit. We propose a correct-then-aggregate framework in which each node 
		removes its deterministic bias via an inverse Rayleigh transfer map, 
		and the server fuses corrected estimates using adaptive soft-max weights 
		based on predictable fluctuation metrics, transmitting only $O(k)$ 
		scalars per active node per round. We establish consistency and 
		asymptotic normality of the global estimator, enabling valid online 
		inference, and derive non-asymptotic bounds quantifying how accuracy 
		improves with the number of nodes and their effective sample sizes. 
		The adaptive weights achieve variance reduction comparable to oracle 
		inverse-variance weighting, confirming the data-driven construction
		is nearly efficient. Simulation studies validate these properties.
		An application to cross-venue monitoring of a dominant market factor
		shows the method tracks systemic risk in real time while substantially
		reducing communication cost relative to a centralized pooled approach.
	\end{abstract}


	\textbf{MSC2020 subject classifications:} Primary 62H25; secondary 62E20

	\section{Introduction\label{sec:1}}
	
	Modern distributed monitoring systems face a common statistical challenge 
	when observations of a high-dimensional process arrive sequentially across 
	multiple nodes with persistently heterogeneous data rates. The task is to 
	track the dominant eigenstructure of the underlying covariance matrix in 
	real time. Real-time systemic risk monitoring in asset markets provides a 
	concrete instance. Trading in a common universe of $p$ assets is fragmented 
	across $L$ venues, each continuously recording its own price stream. 
	Risk surveillance requires tracking the leading eigenvalues of the return 
	covariance matrix as the market evolves. These eigenvalues reflect the 
	strength of common factors driving asset returns, and their sudden increase 
	signals the buildup of systemic stress \cite{kritzman2011principal}.
	Three structural features place this task outside classical spectral theory. 
	First, venues cannot pool raw observations or even $p\times p$ covariance 
	matrices at the required frequency, since transmission costs $O(p^2)$ 
	scalars per node per round. Second, data rates are severely and persistently 
	unequal. Liquidity determines how many observations each venue supplies, 
	and liquidity is a stable venue attribute. Third, when dimension $p$ is 
	comparable to local sample sizes, each node's spectral estimate carries 
	high-dimensional bias whose magnitude depends on its own data rate. 
	Neglecting any feature is fatal. Pooling is infeasible, equal-weight 
	averaging is inefficient, and averaging uncorrected estimates converges 
	to the wrong limit.

	Multi-site biomedical consortia tracking 
	patient physiological markers face identical challenges. Hospital ICUs vary 
	persistently in census, patient privacy forbids raw data pooling, yet early 
	detection of systemic health crises requires continuous eigenvalue 
	monitoring. Distributed sensor networks exhibit the same pattern when 
	measuring environmental variables at stations with heterogeneous sampling 
	frequencies. In each setting, observations arrive sequentially at local 
	nodes rather than being collected centrally in batch form, and the dimension 
	grows with local sample sizes. Classical spectral theory, formulated for 
	centralized static data, provides neither appropriate asymptotic 
	approximations nor computationally feasible procedures. Substantial progress 
	has been made in high-dimensional inference, streaming algorithms, and 
	distributed computation separately, but their integration into a unified 
	framework for distributed streaming spectral inference remains open. The 
	present paper develops such a framework and demonstrates its performance on 
	both synthetic data and cryptocurrency trading across fragmented exchanges 
	in Section~\ref{sec:5}.

	The distributed streaming spectral inference problem can be naturally 
	formulated within the framework of high-dimensional spiked covariance models. 
	This framework, introduced by \citet{johnstone2001distribution}, provides a 
	principled approach to separating low-rank signals from high-dimensional 
	noise in sample covariance matrices. A central result is the Baik-Ben Arous-Péché 
	phase transition \cite{baik2005phase,baik2006eigenvalues}, which characterizes 
	the sharp threshold above which sample eigenvalues escape the bulk and 
	consistently estimate population spikes. Subsequent work has established 
	asymptotic distributions for spiked eigenvalues and eigenvectors under 
	various asymptotic regimes \cite{paul2007asymptotics,bai2008central,bai2012estimation}, 
	extended the theory to general covariance structures with unknown rank 
	\cite{liu2022clt,hou2023spiked}, and analyzed limiting behavior under 
	complex dependence structures \cite{cai2020limiting,zhang2022asymptotic}. 
	Further generalizations include spiked Fisher matrices 
	\cite{jiang2021limits,hou2023spiked} and heterogeneous covariance structures 
	beyond diagonal or block-diagonal forms. These advances provide the 
	asymptotic machinery needed for inference in high-dimensional spectral 
	problems. However, they uniformly presume a static, centralized, batch-access 
	setting in which the complete dataset is available at a single location for 
	repeated eigendecomposition.

	Relaxing the centralization assumption has motivated a growing literature on 
	distributed covariance estimation and principal component analysis. When data 
	are partitioned across multiple nodes, communication-efficient aggregation 
	becomes essential. Existing distributed PCA methods typically operate in two 
	stages. Each node computes a local eigenspace estimate, then a central server 
	aggregates these estimates through weighted averaging or subspace alignment. 
	Representative approaches include one-shot averaging with bias correction 
	\cite{chen2022distributed}, robust distributed covariance estimation 
	\cite{li2022robust}, and privacy-preserving protocols under differential 
	privacy constraints \cite{imtiaz2018differentially}. Recent work by 
	\citet{li2025debiased} has developed debiased distributed estimation 
	specifically for the spiked covariance framework, establishing asymptotic 
	normality under fixed aspect ratios. Parallel advances in distributed 
	statistical inference have provided theoretical foundations for high-dimensional 
	regression and hypothesis testing \cite{battey2018distributed,chen2021distributed}, 
	heterogeneity-aware inference \cite{duan2022heterogeneity}, and distributed 
	learning algorithms \cite{dobriban2021distributed}. However, these methods 
	assume that nodes can synchronize at each aggregation round and that local 
	sample sizes remain comparable, so that simple averaging or inverse-variance 
	weighting suffices. When aspect ratios are persistently heterogeneous, local 
	estimates converge to different limits, and standard aggregation schemes fail.

	A separate line of research has addressed the temporal dimension by 
	developing online and streaming PCA methods that adapt to sequentially 
	arriving observations. The foundational stochastic approximation scheme of 
	\citet{oja1992principal} enables recursive extraction of principal components 
	without storing the full data matrix. Subsequent work has refined this 
	approach through diffusion approximations \cite{li2017diffusion}, near-optimal 
	update rules \cite{li2018near}, and extensions to incomplete or missing data 
	\cite{balzano2010grouse,balzano2015local}. Theoretical characterizations of 
	statistical optimality for online subspace tracking have been established by 
	\citet{liang2023optimality}. More recent attention has turned to inference 
	rather than estimation alone. \citet{kumar2025beyond} established asymptotic 
	normality and entrywise uncertainty quantification for streaming eigenvector 
	estimators under fixed covariance, while \citet{luo2023online} and 
	\citet{li2022statistical} developed online debiased inference procedures for 
	high-dimensional models with temporal variation. However, these methods assume 
	observations arrive at a single central location. When data streams are 
	distributed across multiple nodes with heterogeneous rates, neither direct 
	application of online algorithms nor straightforward combination with 
	distributed aggregation yields consistent estimates. The fundamental difficulty 
	is that local aspect ratios induce node-specific biases that cannot be removed 
	by temporal averaging alone.

	We develop a distributed online framework for estimating spiked eigenvalues 
	under persistent rate heterogeneity. The central observation is that local 
	spectral statistics are deterministically distorted in the proportional regime. 
	When $p/N^{\mathrm{eff}}_{\ell,t}\to c_\ell\in(0,\infty)$, the Rayleigh 
	quotient at node $\ell$ converges to an attenuated version of the population 
	spike, with attenuation factor determined by $c_\ell$. Direct aggregation of 
	uncorrected local statistics therefore fails under heterogeneous aspect ratios. 
	We address this through a correct-then-aggregate architecture. Each node 
	applies an inverse Rayleigh transfer map derived from the limiting attenuation 
	law to remove its deterministic bias before transmission. The correction 
	depends only on local aspect ratios that nodes compute from their accumulated 
	sample sizes. The server then fuses corrected estimates using adaptive weights 
	derived from coordinatewise predictable fluctuation metrics. These weights 
	dynamically approximate inverse-variance weighting without requiring nodes to 
	transmit second-order summaries. The procedure transmits only $O(k)$ scalars per active node per round, 
	where $k$ is the number of spikes, making it feasible even when $p$ is 
	large and communication budgets are constrained.

	Our theoretical analysis establishes the statistical properties of this
	procedure. Building on the limiting attenuation law,
	we establish consistency and asymptotic normality of the aggregated estimator,
	enabling the construction of valid confidence intervals. We derive explicit
	finite-sample error bounds that quantify how estimation accuracy depends on
	the number of nodes, their effective sample sizes, and the temporal variation
	of the underlying covariance structure. These bounds provide practical criteria
	for assessing each node's contribution to the aggregate. We further show that
	adaptive weighting achieves variance reduction comparable to infeasible oracle
	inverse-variance weighting, confirming that the data-driven weights are nearly
	efficient without requiring knowledge of local fluctuation scales. Together,
	these results provide rigorous statistical guarantees for distributed streaming
	spectral estimation under heterogeneous and time-varying conditions.

	The remainder of the paper is organized as follows. Section \ref{sec:2}
	sets up the problem, introduces the high-dimensional spiked covariance model,
	and makes the multi-venue monitoring problem precise.
	Section \ref{sec:3} introduces the online adaptive weighting procedure.
	Section \ref{sec:4} presents the main theoretical guarantees, namely
	consistency, asymptotic normality, and finite-sample error bounds.
	Section \ref{sec:5} presents numerical experiments on synthetic data
	and a real-data analysis of cross-venue cryptocurrency returns.
	Section \ref{sec:6} concludes with a discussion of future directions.
	Proofs are in the Appendix.

	\section{Model Setup and Problem Formulation\label{sec:2}}
	
	Following convention, scalars are denoted by regular letters, while
	vectors and matrices are denoted by bold letters. For scalar sequences
	$\{a_n\}$ and $\{b_n\}$, we write $a_n \lesssim b_n$ or $a_n \gtrsim b_n$
	if there exists a universal constant $C>0$ such that $a_n \le Cb_n$
	or $a_n \ge Cb_n$, respectively. The notation $a_n \asymp b_n$ indicates
	that both $a_n \lesssim b_n$ and $a_n \gtrsim b_n$ hold simultaneously.
	We use $O(\cdot)$ and $o(\cdot)$ for the standard big-O and little-o notation.
	For vectors, $\|\cdot\|_2$ denotes the Euclidean norm.
	For matrices, $\|\cdot\|_{op}$ denotes the operator norm, also known as
	the spectral norm, defined as
	$$\|\bm{A}\|_{op}=\sup_{\bm{x}\neq \bm{0}}\frac{\|\bm{A}\bm{x}\|_2}{\|\bm{x}\|_2},$$
	which equals the largest singular value of $\bm{A}$,
	and $\|\cdot\|_F$ denotes the Frobenius norm.
	For a matrix $\bm{A}$, $\bm{A}^\top$ denotes its transpose.
	We write $\bm{I}_k$ for the $k \times k$ identity matrix,
	and $\diag(\lambda_1,\dots,\lambda_k)$ for a diagonal matrix with 
	diagonal entries $\lambda_1,\dots,\lambda_k$.
	For two subspaces represented by orthonormal matrices $\bm{U}_1$ and $\bm{U}_2$,
	$\sin\Theta(\bm{U}_1,\bm{U}_2)$ denotes the matrix of sines of principal 
	angles between the subspaces.

	\subsection{High-Dimensional Spiked Covariance Model}
	We introduce the high-dimensional spiked covariance framework and 
	the distributed streaming observation model considered throughout the paper.
	Let $\bm{X}=(\bm{x}_1,\dots,\bm{x}_n)\in\mathbb{R}^{p\times n}$ denote 
	a high-dimensional data matrix with columns $\bm{x}_i\in\mathbb{R}^p$, 
	and define the sample covariance matrix
	\begin{equation*}
		\bm{S}_n
		=
		\frac{1}{n}
		\bm{X}\bm{X}^{\top}.
	\end{equation*}
	The population covariance matrix is assumed to follow a spiked covariance structure:
	\begin{equation*}
		\bm{\Sigma}
		=
		\bm{V}\bm{\Lambda}\bm{V}^{\top}
		+
		\sigma^2\bm{I}_p,
	\end{equation*}
	where $\bm{V}\in\mathbb{R}^{p\times k}$ is an orthonormal matrix satisfying 
	$\bm{V}^{\top}\bm{V}=\bm{I}_k$, $\bm{\Lambda}=\diag(\lambda_1,\dots,\lambda_k)$ 
	contains the $k$ spiked (signal) eigenvalues satisfying 
	$\lambda_1\ge\cdots\ge\lambda_k>0$, and $\sigma^2>0$ is the noise variance. 
	The leading eigenspace is spanned by the columns of $\bm{V}$.
	
	Under the high-dimensional asymptotic regime
	\[
	p,n\to\infty,
	\qquad
	\frac{p}{n}\to c\in(0,\infty),
	\]
	the empirical spectral distribution (ESD) of $\bm{S}_n$ converges weakly to 
	the Mar\v{c}enko--Pastur law. A key phenomenon is the
	\textit{Baik--Ben~Arous--P\'{e}ch\'{e} (BBP) phase transition}
	\cite{baik2005phase}, in which the extreme eigenvalues of
	$\bm{S}_n$ undergo a sharp transition between noise and 
	signal regimes, governing the asymptotic behavior of 
	spiked eigenvalue and eigenvector estimators.
	
	\subsection{Distributed Streaming Observation Model}
	We consider a distributed system with $L$ local nodes, 
	where observations are processed locally.
	At each time step $t=1,2,\dots$, the $\ell$-th node
	sequentially receives a batch of observations
	\[
	\{\bm{X}_{\ell,t,i}\}_{i=1}^{n_t^\ell} \subset \mathbb{R}^p,
	\qquad \ell = 1,\dots,L,
	\]
	with batch sizes $n_t^\ell$ that may vary across nodes and time. 
	These observations may arrive asynchronously, and the complete
	dataset remains distributed across nodes rather than being collected centrally.
	Each node thus performs purely local online updates using only
	its currently available observations, periodically sending
	compressed summaries to a central server under communication constraints.
	
	At time $t$, the population covariance matrix is assumed to 
	follow a time-varying spiked structure
	\begin{equation*}
		\bm{\Sigma}_t
		= \bm{V}_t \bm{\Lambda}_t \bm{V}_t^{\top} + \sigma^2_t \bm{I}_p,
	\end{equation*}
	where $\bm{\Lambda}_t = \operatorname{diag}(\lambda_{t,1},\dots,\lambda_{t,k})$ 
	contains the time-varying spiked eigenvalues, $\bm{V}_t \in \mathbb{R}^{p \times k}$ 
	is the associated orthonormal principal eigenspace satisfying 
	$\bm{V}_t^\top \bm{V}_t = \bm{I}_k$, and $\sigma^2_t > 0$ is the noise 
	level at time $t$.
	The spectral gap separating signal from noise is $\lambda_{t,k}$. 
	
	The batch sample covariance matrix at node $\ell$ and time $t$ is given by
	\begin{equation*}
		\hat{\bm{S}}_{\ell,t}
		= \frac{1}{n_t^\ell}
		\sum_{i=1}^{n_t^\ell}
		\bm{X}_{\ell,t,i} \bm{X}_{\ell,t,i}^{\top}.
	\end{equation*}
	
	Consider cryptocurrency trading distributed across multiple 
	exchanges. A typical surveillance task involves tracking the covariance 
	structure of several hundred actively traded tokens. The leading eigenvalues 
	of the return covariance matrix quantify the concentration of systematic risk 
	and serve as early-warning indicators of market stress 
	\cite{kritzman2011principal}. However, raw transaction data cannot be pooled 
	across exchanges in real time due to proprietary restrictions and bandwidth 
	limitations. For $p=500$ assets, transmitting a single sample covariance 
	matrix requires approximately $p(p+1)/2 \approx 125{,}000$ scalars per venue 
	per update, which quickly saturates communication channels when updates are 
	required at minute or sub-minute frequencies.
	
	A further complication arises from the data itself. Liquidity is a stable 
	exchange characteristic, so the number of transactions observed per interval 
	varies substantially and this heterogeneity persists over time. Empirical 
	studies of cryptocurrency markets document that asset returns are well-described 
	by low-rank factor models, with a small number of common factors capturing 
	most of the cross-sectional variation \cite{liu2022common}. This structure 
	matches the spiked covariance framework. The number of observations collected per exchange over a fixed window is of the same order as the number of assets, placing the problem squarely in the proportional-dimensional regime. Section~\ref{sec:5.3} demonstrates the proposed procedure on a representative subset of this problem.

	\section{Methodology\label{sec:3}}
	
	We consider the distributed streaming setting of Section~\ref{sec:2} 
	with $L$ local nodes whose data follow a time-varying spiked covariance model
	\begin{equation*}
		\bm{\Sigma}_t
		= \bm{V}_t \bm{\Lambda}_t \bm{V}_t^\top + \sigma^2_t \bm{I}_p,
		\qquad
		\bm{\Lambda}_t
		= \operatorname{diag}(\lambda_{t,1},\dots,\lambda_{t,k}),
		\quad
		\lambda_{t,1}\ge\cdots\ge\lambda_{t,k}>0,
	\end{equation*}
	where $\bm{V}_t\in\mathbb{R}^{p\times k}$ is orthonormal with 
	$\bm{V}_t^\top\bm{V}_t=\bm{I}_k$. At each time $t$, only $A_t$ 
	nodes in the active set $\mathcal{A}_t$ participate. Communication 
	is restricted to $O(k)$ scalars per active node per round. The goal 
	is to estimate the top $k$ population eigenvalues $\lambda_{t,1},\ldots,\lambda_{t,k}$ 
	in an online manner while respecting both bandwidth and memory constraints.
	
	Each active node $\ell$ receives a fresh batch of $n_t^\ell$ observations 
	at time $t$ and computes the batch covariance 
	\begin{equation*}
		\hat{\bm{S}}_{\ell,t}
		= \frac{1}{n_t^\ell}\sum_{i=1}^{n_t^\ell}
		\bm{X}_{\ell,t,i}\bm{X}_{\ell,t,i}^\top. 
	\end{equation*}
	To accommodate continuously arriving observations while preserving historical 
	information under memory constraints, we recursively update the local covariance 
	estimate via exponential weighting as follows.  
	\begin{equation*}
		\bm{S}_{\ell,t}
		= (1-\alpha_t)\bm{S}_{\ell,t-1}
		+ \alpha_t\hat{\bm{S}}_{\ell,t}, 
	\end{equation*}
	where $\alpha_t \in (0,1]$ is a time-varying step size that governs the 
	trade-off between adaptivity and stability. 
	Unrolling the recurrence reveals that $\bm{S}_{\ell,t}$ 
	is a weighted average of past batch covariances with geometrically decaying 
	weights. The effective number of samples contributing to this average is measured 
	by the inverse participation ratio
	\begin{equation*}
		N^{\mathrm{eff}}_{\ell,t}
		= \Bigl(\sum_{i}\mathrm{w}_{\ell,t,i}^2\Bigr)^{-1},
	\end{equation*}
	where $\mathrm{w}_{\ell,t,i}$ denote the sample-level weights induced by
	the recursive exponential updating scheme. The explicit construction of
	these weights through the batch weights $\pi_{s,t}$ is given in
	\eqref{eq:pi}. The local aspect ratio
	\begin{equation*}
		\label{eq:sec3_cratio}
		c_{\ell,t} := \frac{p}{N^{\mathrm{eff}}_{\ell,t}}
	\end{equation*}
	is the fundamental quantity governing high-dimensional bias. 
	Crucially, $c_{\ell,t}$ is heterogeneous across nodes due to asynchronous activation 
	and unequal batch sizes.

	With the recursively updated local covariance estimate $\bm{S}_{\ell,t}$ in hand,
	each node tracks the principal spectral subspace via orthogonal iteration.
	Let $\bm{U}_{\ell,t-1} \in \mathbb{R}^{p \times k}$ denote the local
	spectral subspace estimate from the previous round.
	At time $t$, the node updates this subspace estimate by
	\begin{equation*}
		\label{eq:sec3_oi}
		\tilde{\bm{U}}_{\ell,t}
		= \bm{S}_{\ell,t}\bm{U}_{\ell,t-1},
		\qquad
		\bm{U}_{\ell,t}
		= \operatorname{QR}\!\bigl(\tilde{\bm{U}}_{\ell,t}\bigr),
	\end{equation*}
	where $\operatorname{QR}(\cdot)$ extracts the orthonormal factor via
	$\operatorname{QR}$ decomposition. This recursive procedure avoids full eigendecomposition
	at each step, reducing per-round computational cost from
	${O}(p^3)$ to ${O}(p^2 k)$. 
	
	Having updated the local subspace estimate, each node must now communicate
	spectral information to the fusion center. Transmitting the full covariance
	matrix $\bm{S}_{\ell,t} \in \mathbb{R}^{p \times p}$ is prohibitively
	expensive when $p$ is large, requiring $O(p^2)$ communication per node per round.
	The spiked structure offers a natural dimension reduction. Since the dominant
	spectral information resides in the $k$-dimensional principal subspace, it
	suffices to transmit summary statistics projected onto that subspace.
	Accordingly, each node constructs the $k$-dimensional Rayleigh quotient vector
	\begin{equation*}
		\bm{r}_{\ell,t}
		= \operatorname{diag}\!\bigl(
		\bm{U}_{\ell,t-1}^\top\hat{\bm{S}}_{\ell,t}\bm{U}_{\ell,t-1}
		\bigr)\in\mathbb{R}^k,
		\qquad
		r_{\ell,t,j}
		= \bm{u}_{\ell,t-1,j}^\top\hat{\bm{S}}_{\ell,t}
		\bm{u}_{\ell,t-1,j},
	\end{equation*}
	where $\bm{u}_{\ell,t-1,j}$ denotes the $j$-th column of $\bm{U}_{\ell,t-1}$.

	The key departure from the classical pipeline is that each node
	corrects its own statistic before transmission. To understand the correction,
	define the node-specific Rayleigh transfer map
	\begin{equation*}
		\label{eq:sec3_chi}
		\chi_{\ell,t}(\lambda)
		:= \sigma_t^2 + \lambda\,a(\lambda;c_{\ell,t-1},\sigma_t^2)
		=\frac{\lambda(\lambda+\sigma_t^2)}{\lambda+c_{\ell,t-1}\sigma_t^2},
	\end{equation*}
	where $a$ is the eigenvector-alignment map of Lemma~\ref{lem:maps} and the
	aspect ratio is the one attached to $\bm U_{\ell,t-1}$, the direction along
	which $r_{\ell,t}$ is evaluated. Theorem~\ref{thm:attenuation} establishes the
	asymptotic transfer relation
	$r_{\ell,t,j}=\chi_{\ell,t}(\lambda_{t,j})+o_{\mathbb P}(1)$,
	which quantifies the attenuation of the spiked eigenvalue through the
	Rayleigh quotient. Lemma~\ref{lem:inverse} further shows that
	$\chi_{\ell,t}$ is a strictly increasing bijection on the supercritical
	region, with explicit inverse
	\begin{equation*}
		\chi^{-1}(x;c,\sigma^2)
		= \tfrac12\Bigl\{(x-\sigma^2)
		+\sqrt{(x-\sigma^2)^2+4c\sigma^2x}\Bigr\}.
	\end{equation*}
	The map $\chi_{\ell,t}$ remains well-conditioned at the phase transition,
	with $\chi'\ge 2/(1+\sqrt c)$ uniformly.
	
	Each node forms the locally debiased estimate
	\begin{equation*}
		\label{eq:sec3_localdebias}
		\check\lambda_{\ell,t,j}
		= \chi^{-1}\!\bigl(
		r_{\ell,t,j}\vee\check\sigma^2_{\ell,t},\,
		c_{\ell,t-1},\,\check\sigma^2_{\ell,t}\bigr),
		\qquad j=1,\dots,k,
	\end{equation*}
	where the truncation $r_{\ell,t,j}\vee\check\sigma^2_{\ell,t}$ keeps the
	argument in the domain of $\chi^{-1}$. The local noise level is estimated 
	from the same batch via the trace identity
	\begin{equation*}
		\check\sigma^2_{\ell,t}
		= \frac{1}{p-k}\Bigl(
		\operatorname{tr}\hat{\bm{S}}_{\ell,t}
		-\sum_{j=1}^k r_{\ell,t,j}\Bigr).
	\end{equation*}

	After local correction, every transmitted statistic targets the
	same population eigenvalue $\lambda_{t,j}$, so aggregation becomes a pure
	variance-reduction problem. The server maintains one exponentially smoothed
	fluctuation metric per coordinate,
	\begin{equation*}
		\label{eq:sec3_metric}
		V_{\ell,t,j} = (1-\eta_t)V_{\ell,t-1,j}
		+ \eta_t\bigl[(1-\rho)(\check\lambda_{\ell,t,j}-\check\lambda_{\ell,t-1,j})^2
		+ \rho(\check\lambda_{\ell,t,j}-\tilde\lambda_{t-1,j})^2\bigr],
		\qquad j=1,\dots,k,
	\end{equation*}
	with $\eta_t\in(0,1]$ and $\rho\in(0,1)$ mixing temporal volatility and
	residual deviation. 
	The weights are built from the cross-sectional relative value of the metric, 
	clipped for robustness, 
	\begin{equation*}
		\label{eq:sec3_relmetric}
		\bar V_{t-1,j} = A_t^{-1}\sum_m V_{m,t-1,j},
		\qquad
		\Psi_{\ell,t,j} = \operatorname{clip}\bigl(V_{\ell,t-1,j}/\bar V_{t-1,j},
		\underline\theta,\bar\theta\bigr), 
	\end{equation*}
	so that the weights for coordinate $j$ are a soft-max in the clipped
	relative metric,
	\begin{equation*}
		\label{eq:sec3_weights}
		\omega_{\ell,t,j} = \frac{\exp(-\tilde\tau_t\Psi_{\ell,t,j})}
		{\sum_{m\in\mathcal A_t}\exp(-\tilde\tau_t\Psi_{m,t,j})},
		\qquad
		\tilde\tau_t\in[\underline\tau,\bar\tau],
		\quad
		0<\underline\tau\le\bar\tau<\infty .
	\end{equation*}

	The global estimate is updated coordinatewise by the damped recursion
	\begin{equation*}
		\label{eq:sec3_fusion}
		\tilde\lambda_{t,j}
		= (1-\beta_t)\tilde\lambda_{t-1,j}
		+ \beta_t\sum_{\ell\in\mathcal{A}_t}
		\omega_{\ell,t,j}\,\check\lambda_{\ell,t,j},
		\qquad \beta_t\in(0,1],
		\qquad j=1,\dots,k,
	\end{equation*}
	then $\tilde{\bm\lambda}_t=(\tilde\lambda_{t,1},\dots,\tilde\lambda_{t,k})$
	and
	$\check{\bm\lambda}_{\ell,t}
	=(\check\lambda_{\ell,t,1},\dots,\check\lambda_{\ell,t,k})$. The complete procedure is
	summarised in Algorithm~\ref{alg:distributed_online_spike}.


	\begin{algorithm}[t]
		\caption{Distributed Online Debiased Spike Estimation
			(correct-then-aggregate)\label{alg:distributed_online_spike}}
		\begin{algorithmic}[1]
			\Require step sizes $(\alpha_t,\beta_t,\eta_t)$, mixing parameter
			$\rho$, temperature $\tilde\tau_t\in[\underline\tau,\bar\tau]$
			(default $\tilde\tau_t\equiv1$), relative-metric clip
			$[\underline\theta,\bar\theta]$, and rank $k$.
			\State \textbf{Initialize:} $\tilde{\bm\lambda}_0\!\leftarrow\!\bm 0$ at the
			server, and at each node set $\bm U_{\ell,0}$ to an arbitrary orthonormal
			matrix, $\bm S_{\ell,0}\!\leftarrow\!\bm I_p$,
			$N^{\mathrm{eff},-1}_{\ell,0}\!\leftarrow\!0$,
			$V_{\ell,0,j}\!\leftarrow\!V_0$ for $j=1{:}k$.
			
			\For{$t = 1 : T$}
			\For{$\ell \in \mathcal{A}_t$}  
			\State Receive streaming samples
			$\{\bm{X}_{\ell,t,i}\}_{i=1}^{n_t^\ell}$ and form
			$\hat{\bm{S}}_{\ell,t}
			= \frac{1}{n_t^\ell}\sum_{i}\bm{X}_{\ell,t,i}\bm{X}_{\ell,t,i}^{\top}$.
			\State Rayleigh on the {frozen} direction $\bm U_{\ell,t-1}$:
			$\ r_{\ell,t,j}
			= \bm u_{\ell,t-1,j}^{\top}\hat{\bm{S}}_{\ell,t}\bm u_{\ell,t-1,j}$.
			\State Aspect ratio (scalar recursion):
			$\ N^{\mathrm{eff},-1}_{\ell,t}
			=(1-\alpha_t)^2 N^{\mathrm{eff},-1}_{\ell,t-1}+\alpha_t^2/n_t^\ell$,
			$\ c_{\ell,t}=p\,N^{\mathrm{eff},-1}_{\ell,t}$.
			\State Local noise:
			$\ \check\sigma^2_{\ell,t}
			=\big(\operatorname{tr}\hat{\bm S}_{\ell,t}
			-\sum_j r_{\ell,t,j}\big)/(p-k)$.
			\State Subspace update:
			$\ \bm S_{\ell,t}=(1-\alpha_t)\bm S_{\ell,t-1}+\alpha_t\hat{\bm S}_{\ell,t}$,
			$\ \bm U_{\ell,t}=\operatorname{QR}(\bm S_{\ell,t}\bm U_{\ell,t-1})$.
			\For{$j = 1 : k$} 
			\State
			$\ \check\lambda_{\ell,t,j}
			=\chi^{-1}\!\big(r_{\ell,t,j}\vee\check\sigma^2_{\ell,t};
			c_{\ell,t-1},\check\sigma^2_{\ell,t}\big)$.
			\State 
			$$\ V_{\ell,t,j}=(1-\eta_t)V_{\ell,t-1,j}
			+\eta_t\big[(1-\rho)
			(\check\lambda_{\ell,t,j}-\check\lambda_{\ell,t-1,j})^2
			+\rho(\check\lambda_{\ell,t,j}-\tilde\lambda_{t-1,j})^2\big]. $$
			\EndFor
			\State Send $\check{\bm\lambda}_{\ell,t}$ and $ V_{\ell,t,j}$ to server. 
			\EndFor
			\State Mean metric:
			$\ \bar V_{t-1,j}=\tfrac1{A_t}\sum_{m\in\mathcal A_t}V_{m,t-1,j}$.
			\State Clipped relative metric:
			$\ \Psi_{\ell,t,j}=\operatorname{clip}
			\bigl(V_{\ell,t-1,j}/\bar V_{t-1,j},
			\underline\theta,\bar\theta\bigr)$.
			\State Adaptive weights (predictable soft-max):
			$\ \omega_{\ell,t,j}
			=\dfrac{\exp(-\tilde\tau_t\Psi_{\ell,t,j})}
			{\sum_{m\in\mathcal A_t}\exp(-\tilde\tau_t\Psi_{m,t,j})}$.
			\State Aggregate debiased estimates:
			$\ \tilde\lambda_{t,j}
			=(1-\beta_t)\tilde\lambda_{t-1,j}
			+\beta_t\sum_{\ell\in\mathcal A_t}\omega_{\ell,t,j}
			\check\lambda_{\ell,t,j}$.
			\EndFor
			\Ensure Online debiased spike estimates
			$\{\tilde{\bm\lambda}_t\}_{t=1}^{T}$. 
		\end{algorithmic}
	\end{algorithm}

	%
	
	\section{Theoretical Results\label{sec:4}}
	
	This section develops the theory of the proposed framework in the
	proportional regime, where dimension $p$ and the effective sample size
	grow at the same rate. The argument is organised as follows.

	\begin{itemize}
		\item Section~\ref{sec:4.1} establishes local spectral inference. The sample 
		eigenvalue is systematically distorted by the local aspect ratio, 
		creating a bias that direct averaging cannot eliminate. We provide 
		an explicit correction and derive a central limit theorem for the 
		debiased statistic.
		\item Section~\ref{sec:4.2} establishes global aggregation. Equal-weight 
		averaging ignores the varying precision across nodes. We derive adaptive 
		weights based on local fluctuation metrics, prove consistency and 
		asymptotic normality of the global estimator.
		\item Section~\ref{sec:4.3} provides finite-sample analysis. 
		We derive non-asymptotic bounds quantifying the dependence of 
		estimation error on node count, sample sizes, and covariance drift.
	\end{itemize}

	Recall from Section~\ref{sec:3} that node $\ell$ maintains a smoothed
	covariance estimate which is updated only when the node is active. Let
	\begin{equation*}
		\alpha^\ell_t:=\alpha_t\,\bm 1\{\ell\in\mathcal A_t\},
		\quad
		\bm S_{\ell,t}
		=(1-\alpha^\ell_t)\bm S_{\ell,t-1}
		+\alpha^\ell_t\hat{\bm S}_{\ell,t},
		\quad
		\hat{\bm S}_{\ell,t}
		=\frac{1}{n_t^\ell}\sum_{i=1}^{n_t^\ell}
		\bm X_{\ell,t,i}\bm X_{\ell,t,i}^\top .
	\end{equation*}
	For notational simplicity we suppress the indicator and write
	$\alpha_t$ in place of $\alpha^\ell_t$. All statements below are
	indexed by the activation times of node $\ell$. Unrolling
	$\bm S_{\ell,t}$ gives
	\begin{equation}
		\label{eq:pi}
		\bm S_{\ell,t}
		=\sum_{s=1}^{t}\pi_{s,t}\hat{\bm S}_{\ell,s}
		+\pi_{0,t}\bm S_{\ell,0},
		\qquad
		\pi_{s,t}=\alpha_s\!\!\prod_{u=s+1}^{t}\!\!(1-\alpha_u),
		\quad
		\pi_{0,t}=\prod_{u=1}^{t}(1-\alpha_u),
	\end{equation}
	with $\sum_{s=0}^{t}\pi_{s,t}=1$ and $\pi_{0,t}\to0$ under
	Assumption~\ref{assum:step}. Hence $\bm S_{\ell,t}$ is a
	{sample-weighted} covariance matrix,
	\begin{equation}
		\label{eq:sampleweights}
		\bm S_{\ell,t}
		=\sum_{i=1}^{M_{\ell,t}} \mathrm{w}_{\ell,t,i}\,
		\bm X_{\ell,(i)}\bm X_{\ell,(i)}^\top
		+\pi_{0,t}\bm S_{\ell,0},
	\end{equation}
	\[
	M_{\ell,t}=\sum_{s=1}^{t}n_s^\ell ,
	\qquad
	\mathrm{w}_{\ell,t,i}=\frac{\pi_{s,t}}{n_s^\ell}
	\ \ \text{for } i\in\text{batch }s,
	\]
	with $\mathrm{w}_{\ell,t,i}\ge0$ and $\sum_i\mathrm{w}_{\ell,t,i}=1-\pi_{0,t}\to1$.
	
	The effective sample size is the inverse participation ratio of
	the weight profile,
	\begin{equation*}
		N^{\mathrm{eff}}_{\ell,t}
		:=\Bigl(\sum_{i=1}^{M_{\ell,t}}\mathrm{w}_{\ell,t,i}^{2}\Bigr)^{-1}
		=\Bigl(\sum_{s=1}^{t}\frac{\pi_{s,t}^{2}}{n_s^\ell}\Bigr)^{-1},
	\end{equation*}
	and the {local aspect ratio} is
	\begin{equation*}
		c_{\ell,t}:=\frac{p}{N^{\mathrm{eff}}_{\ell,t}}
		=p\sum_{i=1}^{M_{\ell,t}}\mathrm{w}_{\ell,t,i}^{2}.
	\end{equation*}
	
	The aspect ratio $c_{\ell,t}$ is deterministic, node-locally known, and
	heterogeneous across nodes due to asynchronous activation.

	\begin{lemma}[Basic properties of $c_{\ell,t}$]
		\label{lem:crec}
		For every node $\ell$ and every activation time $t$,
		\begin{equation*}
			\frac{1}{N^{\mathrm{eff}}_{\ell,t}}
			=\frac{(1-\alpha_t)^{2}}{N^{\mathrm{eff}}_{\ell,t-1}}
			+\frac{\alpha_t^{2}}{n_t^\ell},
			\qquad\text{equivalently}\qquad
			c_{\ell,t}
			=(1-\alpha_t)^{2}c_{\ell,t-1}
			+\alpha_t^{2}\,\frac{p}{n_t^\ell},
		\end{equation*}
		while $c_{\ell,t}=c_{\ell,t-1}$ if $\ell\notin\mathcal A_t$.
		Let 
		$\tilde\alpha_{\ell,t}:=\alpha_t(1+p/n_t^\ell)$, the aspect ratio is
		controlled by this quantity,
		\begin{equation}
			\label{eq:c_le_alpha}
			c_{\ell,t}
			=p\sum_{s\le t}\frac{\pi_{s,t}^2}{n^\ell_s}
			\ \le\ \Bigl(\max_{s\le t}\frac{p\,\pi_{s,t}}{n^\ell_s}\Bigr)
			\sum_{s\le t}\pi_{s,t}
			\ \le\ \sup_{s\le t}\tilde\alpha_{\ell,s} ,
		\end{equation}
		and consequently
		\begin{equation*}
			\bigl|c_{\ell,t}-c_{\ell,t-1}\bigr|
			\ \le\ 2\alpha_t\,c_{\ell,t-1}
			+\alpha_t^{2}\,\frac{p}{n_t^\ell}
			\ =\ O(\alpha_t)
		\end{equation*}
		whenever $\sup_{\ell,t}\tilde\alpha_{\ell,t}<\infty$.
		If $\alpha_t\equiv\alpha$ and $n_t^\ell\equiv n$, the recursion has
		the unique fixed point
		$N^{\mathrm{eff}}_{\ell,\infty}=(2-\alpha)n/\alpha$, i.e.\
		$c_{\ell,\infty}=\alpha p/\{(2-\alpha)n\}$.
	\end{lemma}

	The effective sample size only depends on the second moment of the weights. 
	To characterize the complete structure of the weighting scheme, 
	we introduce the empirical distribution of the rescaled weights	
	\begin{equation*}
		H_{\ell,t}
		:=
		\frac{1}{M_{\ell,t}}
		\sum_{i=1}^{M_{\ell,t}}
		\delta_{M_{\ell,t}\mathrm{w}_{\ell,t,i}}, 
	\end{equation*}
	where $\delta_x$ denotes the
	Dirac probability measure concentrated at $x$.
	The first two moments of this empirical distribution satisfy
	\[
	\int x\,H_{\ell,t}(dx)
	=
	1-\pi_{0,t},
	\qquad
	\int x^{2}\,H_{\ell,t}(dx)
	=
	\frac{M_{\ell,t}}{N^{\mathrm{eff}}_{\ell,t}}.
	\] 
	
	Let $c^{\mathrm{mp}}_{\ell,t}:=p/M_{\ell,t}$ denote the Mar\v{c}enko--Pastur 
	ratio based on the accumulated sample size. It follows that 
	\begin{equation*}
		c_{\ell,t}
		=c^{\mathrm{mp}}_{\ell,t}\int x^{2}\,H_{\ell,t}(dx). 
	\end{equation*}
	This decomposition shows that the effective aspect ratio $c_{\ell,t}$ consists 
	of the nominal dimension-to-sample ratio and the concentration factor induced 
	by the memory kernel. For a rectangular memory window, all rescaled weights 
	are equal to one and therefore $H_{\ell,t}\Rightarrow\delta_1$, where $\delta_1$ 
	denotes the Dirac probability measure concentrated at the point $1$. 
	In this case, the effective aspect ratio reduces to the classical 
	Mar\v{c}enko--Pastur ratio, namely $c_{\ell,t}=c^{\mathrm{mp}}_{\ell,t}.$ 
	In contrast, geometric forgetting with a constant forgetting parameter $\alpha$ 
	leads to an exponential-type weight profile satisfying $\int x^2\,dH>1$.  
	Therefore, the corresponding memory kernel has a smaller effective sample size 
	than a flat window with the same accumulated sample size.

	\subsection{Local Spectral Inference\label{sec:4.1}}
	
	We first characterise how accurately the recursively updated local
	eigenspace estimator tracks the evolving principal subspace when
	dimension and effective sample size are comparable. The population
	covariance follows the spiked model
	\begin{equation}
		\label{eq:spiked}
		\bm\Sigma_t
		=\sum_{j=1}^{k}\lambda_{t,j}\bm v_{t,j}\bm v_{t,j}^\top
		+\sigma_t^2\bm I_p,
	\end{equation}
	so that the top $k$ eigenvalues of $\bm\Sigma_t$ are
	$\lambda_{t,j}+\sigma_t^2$ and the remaining $p-k$ equal $\sigma_t^2$.
	Let $\bm V_t=(\bm v_{t,1},\dots,\bm v_{t,k})$, let $\bm U_{\ell,t}$ be
	the estimator produced by the orthogonal iteration of
	Section~\ref{sec:3}, and measure the estimation error by
	\[
	d_t^{(\ell)}
	=\bigl\|\sin\Theta(\bm U_{\ell,t},\bm V_t)\bigr\|_F ,
	\]
	where $\Theta(\cdot,\cdot)$ collects the principal angles between the
	two subspaces. 
	
	\begin{assum}
		\label{assum:moment}
		For each $\ell,t,i$, $\bm X_{\ell,t,i}
		=\bm\Sigma_t^{1/2}\bm z_{\ell,t,i}$, where the vectors
		$\{\bm z_{\ell,t,i}\}$ are mutually independent across $\ell$, $t$
		and $i$, with independent entries satisfying
		$\mathbb E z_{\ell,t,i,m}=0$, $\mathbb E z_{\ell,t,i,m}^2=1$,
		$\mathbb E z_{\ell,t,i,m}^4=3+{K}_4$ for a constant
		${K}_4>-2$, and
		$\sup_{p,\ell,t,i,m}\mathbb E|z_{\ell,t,i,m}|^{4+\epsilon_1}\le C$
		for some $\epsilon_1>0$.
	\end{assum}
	
	\begin{assum}
		\label{assum:design}
		The active set $\mathcal A_t$, its cardinality $A_t:=|\mathcal A_t|$,
		the batch sizes $\{n_t^\ell\}$, and the tuning sequences
		$\alpha_t,\beta_t,\eta_t,\tilde\tau_t$ are measurable with respect to
		$\mathcal F_{t-1}$, the $\sigma$-field generated by all data
		received before time $t$. In particular they are independent of the
		batch $\{\bm X_{\ell,t,i}\}$.
	\end{assum}
	
	\begin{assum}
		\label{assum:deloc}
		$\max_{j\le k}\|\bm v_{t,j}\|_\infty\to0$ uniformly in $t$, and the
		top-$k$ eigenvectors of the noise part of $\bm S_{\ell,t}$ satisfy
		the same bound in probability, uniformly in $\ell$ and $t$.
	\end{assum}
	
	\begin{assum}
		\label{assum:bbp}
		The spikes are simple and uniformly separated,
		\[
		\lambda_{t,1}>\cdots>\lambda_{t,k}>0,
		\qquad
		\inf_t\min_{1\le j\le k-1}
		(\lambda_{t,j}-\lambda_{t,j+1})\ \ge\ g_{\min}>0,
		\]
		and every spike lies strictly above the BBP threshold uniformly in
		$\ell$ and $t$, i.e., $\inf_{\ell,t}\lambda_{t,k}> \lambda^{\mathrm c}_{\ell,t}$, 
		where $\lambda^{\mathrm c}_{\ell,t}:=\sigma_t^2/g_{\ell,t}(y^+_{\ell,t})$
		is the profile-dependent critical spike (defined in
		Lemma~\ref{lem:maps}). In the canonical case
		$H_{\ell,t}\Rightarrow\delta_1$, this reduces to
		$\lambda_{t,k}>\sigma_t^2\sqrt{c_{\ell,t}}$.
	\end{assum}
	
	\begin{assum}
		\label{assum:noise}
		$0<\sigma_{\min}^2\le\sigma_t^2\le\sigma_{\max}^2<\infty$ and
		$\sup_t\lambda_{t,1}=:\lambda_{\max}<\infty$.
	\end{assum}
	
	\begin{assum}
		\label{assum:drift}
		There is a deterministic sequence $\zeta_t\to0$ with
		\[
		\bigl\|\sin\Theta(\bm V_t,\bm V_{t-1})\bigr\|_F\le\zeta_t,
		\qquad
		\max_{j\le k}|\lambda_{t,j}-\lambda_{t-1,j}|\le\zeta_t,
		\qquad
		|\sigma_t^2-\sigma_{t-1}^2|\le\zeta_t,
		\]
		and the memory-window drift is asymptotically negligible in its
		inflated form,
		\begin{equation}
			\label{eq:windowdrift}
			\bar\Delta_{\ell,t}
			:=\sum_{s=1}^{t}\pi_{s,t}
			\Bigl(1+\frac{p}{n^\ell_s}\Bigr)
			\|\bm\Sigma_s-\bm\Sigma_t\|_{op}
			\longrightarrow0 .
		\end{equation}
	\end{assum}
	
	The error of $\bm S_{\ell,t}$ relative to $\bm\Sigma_t$ carries
	two deterministic terms, the window drift
	$\bar\Delta_{\ell,t}$ of \eqref{eq:windowdrift} and the weight
	$\pi_{0,t}$ still resting on the initialization $\bm S_{\ell,0}$.
	They enter every bound below in the same place, so we write
	\begin{equation*}
		\Delta_{\ell,t}
		:=\bar\Delta_{\ell,t}+\pi_{0,t} ,
	\end{equation*}
	which tends to zero by Assumption~\ref{assum:drift} and
	$\pi_{0,t}\to0$.

	The inflation factor $1+p/n^\ell_s$ in \eqref{eq:windowdrift} arises 
	from the Bai--Yin bound $\|\hat W_{\ell,s}\|_{op}=O_{\mathbb P}(1+p/n^\ell_s)$ 
	on batch Wishart matrices. The condition is implied by the uninflated 
	$\sum_s\pi_{s,t}\|\bm\Sigma_s-\bm\Sigma_t\|_{op}\to0$ whenever 
	$\sup_{s\le t}p/n^\ell_s=O(1)$.

	\begin{assum}
		\label{assum:step}
		$\alpha_t\in(0,1]$, $\alpha_t\to0$, $\sum_{t\ge1}\alpha_t=\infty$,
		and the operator-norm inflation factor satisfies 
		$\sup_{\ell,t} \tilde{\alpha}_{\ell,t} <\infty$. 
	\end{assum}
	We define the effective lag as the geometric mean,
	\begin{equation*}
		\alpha^{\mathrm{lag}}_{\ell,t}
		:=\sqrt{\alpha_t\,\tilde\alpha_{\ell,t}}
		\ =\ \alpha_t\sqrt{1+\frac{p}{n_t^\ell}} ,
	\end{equation*}
	which satisfies $\alpha_t\le\alpha^{\mathrm{lag}}_{\ell,t}\le\tilde\alpha_{\ell,t}$
	and $\alpha^{\mathrm{lag}}_{\ell,t}\to0$ uniformly in $\ell$ under
	Assumption~\ref{assum:step}.

	\begin{assum}
		\label{assum:hd}
		$p\to\infty$ with $k$ fixed, and for every $\ell$,
		\[
		c_{\ell,t}\longrightarrow c_\ell\in(0,\infty),
		\quad
		c^{\mathrm{mp}}_{\ell,t}\to c^{\mathrm{mp}}_\ell\in(0,\infty),
		\]
		and $H_{\ell,t}\Rightarrow H_\ell$, a probability measure on
		$[0,\infty)$ with unit mean and bounded support. Moreover
		$\sup_\ell c_\ell\le\bar c<\infty$.
	\end{assum}
	
	\begin{assum}
		\label{assum:init}
		The smallest singular value of $\bm V_0^\top\bm U_{\ell,0}$ is
		bounded away from zero uniformly in $\ell$ and $p$.
	\end{assum}

	To characterise the statistical error of $\bm U_{\ell,t}$ relative
	to $\bm V_t$, we apply finite-rank perturbation theory
	\citep{baik2005phase,paul2007asymptotics,benaych2011eigenvalues,%
		bai2012sample}, which reduces the problem to understanding the
	limiting spectrum of the noise component. Under
	Assumptions~\ref{assum:moment}, \ref{assum:noise} and
	\ref{assum:hd}, consider the pure-noise case, the sample-weighted
	covariance built from \eqref{eq:sampleweights} with $\bm\Sigma_t$
	replaced by $\sigma_t^2\bm I_p$. Its empirical spectral distribution
	converges almost surely to a compactly supported
	Mar\v{c}enko--Pastur-type measure $\mu^W_{\ell,t}$ with right edge
	$b_{\ell,t}$, uniquely determined by
	$(c^{\mathrm{mp}}_{\ell,t},\sigma^2_t,H_{\ell,t})$ through the
	weighted-sample equation \citep{silverstein1995empirical}. No noise
	eigenvalue separates from the support \citep{bai1999exact}.
	
	Writing $W_{\ell,t}:=\sum_i\mathrm{w}_{\ell,t,i}\bm z_{\ell,(i)}\bm
	z_{\ell,(i)}^\top$ for the isotropic weighted Wishart matrix, one has
	$\tilde{\bm S}_{\ell,t}=\bm\Sigma_t^{1/2}W_{\ell,t}\bm\Sigma_t^{1/2}$
	modulo the vanishing initialisation term, so the spiked covariance
	induces a rank-$2k$ perturbation of $\sigma_t^2W_{\ell,t}$ within
	$\mathrm{span}\{\bm v_{t,j},W_{\ell,t}\bm v_{t,j}:j\le k\}$.
	Finite-rank perturbation theory then delivers two deterministic maps,
	$\Phi_{\ell,t}$ and $a_{\ell,t}$. The first maps each population
	spike $\lambda_{t,j}$ to the asymptotic sample eigenvalue location,
	and the second quantifies the squared alignment
	$|\langle\tilde{\bm u}_{\ell,t,j},\bm v_{t,j}\rangle|^2$ between
	sample and population eigenvectors. Both are defined only for spikes
	above the critical threshold
	$\lambda^{\mathrm c}_{\ell,t}:=\sigma_t^2/g_{\ell,t}(y^+_{\ell,t})$,
	where $g_{\ell,t}(x):=\int y(x-y)^{-1}\mu^W_{\ell,t}(dy)$ is the
	Stieltjes transform of $\mu^W_{\ell,t}$. The quantities
	$y^+_{\ell,t}$, $g_{\ell,t}(y^+_{\ell,t})$, and
	$\lambda^{\mathrm c}_{\ell,t}$ are deterministic functions of
	$c^{\mathrm{mp}}_{\ell,t}$, $\sigma_t^2$, and $H_{\ell,t}$ alone,
	as formalised in Assumption~\ref{assum:bbp}.
	
	\begin{lemma}
		\label{lem:maps}
		Under Assumptions~\ref{assum:moment}--\ref{assum:hd} there exist
		deterministic maps
		\[
		\Phi_{\ell,t}:(\lambda^{\mathrm c}_{\ell,t},\infty)
		\to(b_{\ell,t},\infty),
		\qquad
		a_{\ell,t}:(\lambda^{\mathrm c}_{\ell,t},\infty)
		\to(0,1),
		\qquad
		\lambda^{\mathrm c}_{\ell,t}
		:=\frac{\sigma_t^2}{g_{\ell,t}(y^+_{\ell,t})},
		\]
		determined by $(c^{\mathrm{mp}}_{\ell,t},\sigma_t^2,H_{\ell,t})$
		alone, such that, writing $\tilde{\bm u}_{\ell,t,j}$ and
		$\tilde\lambda_{\ell,t,j}$ for the $j$-th unit eigenvector and
		eigenvalue of $\tilde{\bm S}_{\ell,t}:=\bm\Sigma_t^{1/2}W_{\ell,t}\bm\Sigma_t^{1/2}$,
		for each fixed $j\le k$,
		\begin{align*}
			\tilde\lambda_{\ell,t,j}
			=&\Phi_{\ell,t}(\lambda_{t,j})+o_{\mathbb P}(1),\\
			\bigl|\langle\tilde{\bm u}_{\ell,t,j},
			\bm v_{t,j}\rangle\bigr|^{2}
			=&a_{\ell,t}(\lambda_{t,j})+O_{\mathbb P}
			\bigl((N^{\mathrm{eff}}_{\ell,t})^{-1/2}\bigr),\\
			\bigl|\langle\tilde{\bm u}_{\ell,t,j},
			\bm v_{t,m}\rangle\bigr|^{2}
			=&O_{\mathbb P}\bigl((N^{\mathrm{eff}}_{\ell,t})^{-1}\bigr)
			\ (m\ne j).
		\end{align*}
		Moreover, writing $\mathring{\bm u}_{\ell,t,j}$ and
		$\mathring\lambda_{\ell,t,j}$ for the $j$-th unit eigenvector and
		eigenvalue of $\bm S_{\ell,t}$, the same conclusions hold with
		an additional error $O_{\mathbb P}(\Delta_{\ell,t})$
		in each display.
	\end{lemma}
	Both maps are continuously differentiable and strictly increasing,
	with boundary behaviour
	$\Phi_{\ell,t}(\lambda)\downarrow b_{\ell,t}$
	and
	$a_{\ell,t}(\lambda)\downarrow0$
	as
	$\lambda\downarrow\lambda^{\mathrm c}_{\ell,t}$
	$(=\sigma_t^2\sqrt{c_{\ell,t}}$ in the canonical case$)$. They are
	characterised implicitly by the spike-location equation and the
	alignment identity
	\begin{equation*}
		\lambda\!\int\!\frac{y}{\Phi_{\ell,t}(\lambda)-\sigma_t^2y}\,
		\mu^{W}_{\ell,t}(dy)=1,
		\qquad
		a_{\ell,t}(\lambda)
		=\bigl(\lambda+\sigma_t^2\bigr)\,
		\frac{\Phi_{\ell,t}'(\lambda)}
		{\Phi_{\ell,t}(\lambda)},
	\end{equation*}
	where $\mu^W_{\ell,t}$ is the limiting spectral measure of
	$W_{\ell,t}$. In the canonical case $H_{\ell,t}\Rightarrow\delta_1$
	they admit the closed forms
	\begin{equation}
		\label{eq:closedforms}
		\Phi_{\ell,t}(\lambda)
		=\frac{(\lambda+\sigma_t^2)(\lambda+c_{\ell,t}\sigma_t^2)}
		{\lambda},
		\quad
		a_{\ell,t}(\lambda)
		=\frac{\lambda^{2}-c_{\ell,t}\sigma_t^{4}}
		{\lambda(\lambda+c_{\ell,t}\sigma_t^{2})},
		\quad
		b_{\ell,t}=\sigma_t^{2}\bigl(1+\sqrt{c_{\ell,t}}\bigr)^{2}.
	\end{equation}

	The estimator $\bm U_{\ell,t}$ is produced by a single orthogonal
	iteration per time step rather than an exact eigendecomposition, so we
	first separate algorithmic from statistical error. Let
	$\mathring{\bm U}_{\ell,t}\in\mathbb R^{p\times k}$ be the exact
	top-$k$ eigenspace of $\bm S_{\ell,t}$ and set
	$e_{\ell,t}:=\|\sin\Theta(\bm U_{\ell,t},
	\mathring{\bm U}_{\ell,t})\|_F$.
	
	\begin{lemma}
		\label{lem:contraction}
		Under Assumptions~\ref{assum:moment}--\ref{assum:init}, define $\xi_{\ell,t}
		:=\frac{b_{\ell,t}}{\Phi_{\ell,t}(\lambda_{t,k})}$. 
		Then $\limsup_{\ell,t}\xi_{\ell,t}\le\xi^{\star}<1$,
		where $\xi^\star$ depends only on
		$\epsilon_0,\sigma_{\min},
		\sigma_{\max},\bar c$,
		$\underline c:=\inf_\ell c_\ell>0$.  
		Moreover, 
		$\|(\bm S_{\ell,t}-\bm S_{\ell,t-1})\mathring{\bm U}_{\ell,t-1}\|_F
		=O_{\mathbb P}(\alpha^{\mathrm{lag}}_{\ell,t})$, and there is $C>0$ such
		that, with probability tending to one and for all large $t$,
		\begin{equation*}
			e_{\ell,t}
			\ \le\ \xi_{\ell,t}\,e_{\ell,t-1}
			+\frac{C}{1-\xi^{\star}}
			\bigl(\alpha^{\mathrm{lag}}_{\ell,t}+\zeta_t\bigr).
		\end{equation*}
		Consequently, 
		$e_{\ell,t}=O\bigl((\xi^{\star})^{t}\bigr)
		+O\bigl(\alpha^{\mathrm{lag}}_{\ell,t}+\zeta_t\bigr)$.
	\end{lemma}

	\begin{theo}[BBP-limited local subspace tracking]
		\label{thm:tracking}
		Under Assumptions~\ref{assum:moment}--\ref{assum:init}, define the
		{misalignment floor}
		\begin{equation*}
			D^{\star}_{\ell,t}
			:=\Bigl(\sum_{j=1}^{k}
			\bigl\{1-a_{\ell,t}(\lambda_{t,j})\bigr\}\Bigr)^{1/2}
			\ \stackrel{H_{\ell,t}\Rightarrow\delta_1}{=}\
			\Biggl(\sum_{j=1}^{k}
			\frac{c_{\ell,t}\,\sigma_t^{2}
				(\lambda_{t,j}+\sigma_t^{2})}
			{\lambda_{t,j}(\lambda_{t,j}
				+c_{\ell,t}\sigma_t^{2})}\Biggr)^{1/2}.
		\end{equation*}
		Then, for each fixed node $\ell$ and as $t\to\infty$, 
		\begin{equation}
			\label{eq:tracking}
			d_t^{(\ell)}
			=D^{\star}_{\ell,t}
			+O_{\mathbb P}\!\left(
			(N^{\mathrm{eff}}_{\ell,t})^{-1/2}\right)
			+O\bigl(\alpha^{\mathrm{lag}}_{\ell,t}+\zeta_t+\Delta_{\ell,t}\bigr)
			+O\bigl((\xi^{\star})^{t}\bigr).
		\end{equation}
		
		In particular, if $c_{\ell,t}\to c_\ell>0$, $\sigma_t^2\to\sigma^2$, 
		and $\lambda_{t,j}\to\lambda_j$ as in Assumption~\ref{assum:drift}, then 
		$d_t^{(\ell)}\stackrel{p}{\longrightarrow}
		D^{\star}_{\ell}:=(\sum_{j=1}^{k}\{1-a_{\ell}(\lambda_{j})\})^{1/2}>0$,
		where $a_{\ell}(\lambda):=\lim_{t\to\infty}a_{\ell,t}(\lambda)$. 
		Whenever $c_\ell>0$, the local eigenspace estimator is
		inconsistent, and $\liminf_t d_t^{(\ell)}>0$ in
		probability.
	\end{theo}

	Transmitting the full covariance costs $O(p^2)$ per node per round,
	which is infeasible for large $p$. Each node therefore communicates
	only a $k$-dimensional Rayleigh-type statistic. We adopt a
	sample-splitting construction in which the direction is formed from
	data up to time $t-1$, and the quadratic form is evaluated on the
	fresh batch received at time $t$,
	\begin{equation*}
		\bm r_{\ell,t}
		=\operatorname{diag}\!\bigl(
		\bm U_{\ell,t-1}^\top\hat{\bm S}_{\ell,t}
		\bm U_{\ell,t-1}\bigr)\in\mathbb R^{k},
		\qquad
		r_{\ell,t,j}
		=\bm u_{\ell,t-1,j}^\top\hat{\bm S}_{\ell,t}
		\bm u_{\ell,t-1,j}.
	\end{equation*}
	The one-step lag in the direction is deliberate. Since
	$\bm U_{\ell,t-1}$ is $\mathcal F_{\ell,t-1}$-measurable, it is
	independent of the batch against which the quadratic form is evaluated,
	and the fluctuation of $\bm r_{\ell,t}$ around its conditional mean is
	then a martingale residual with exactly vanishing conditional
	expectation.

	\begin{lemma}[Exact conditional decomposition]
		\label{lem:decomp}
		Let $\mathcal F_{\ell,t-1}$ be the $\sigma$-field generated by all
		observations at node $\ell$ up to time $t-1$. Under
		Assumptions~\ref{assum:moment}--\ref{assum:deloc}, for each
		$j\le k$,
		\begin{equation}
			\label{eq:decomp}
			r_{\ell,t,j}
			=\underbrace{\sigma_t^{2}
				+\sum_{m=1}^{k}\lambda_{t,m}
				\bigl\langle\bm u_{\ell,t-1,j},
				\bm v_{t,m}\bigr\rangle^{2}}_{=:\,\bar r_{\ell,t,j}}
			\ +\ \varpi_{\ell,t,j},
		\end{equation}
		$$
		\varpi_{\ell,t,j}
		:=\bm u_{\ell,t-1,j}^\top
		\bigl(\hat{\bm S}_{\ell,t}-\bm\Sigma_t\bigr)
		\bm u_{\ell,t-1,j},
		$$
		where $\mathbb E(\varpi_{\ell,t,j}\mid\mathcal F_{\ell,t-1})=0$
		{exactly} and
		\begin{align}
			\label{eq:var}
			\operatorname{Var}\bigl(\varpi_{\ell,t,j}\mid
			\mathcal F_{\ell,t-1}\bigr)
			=&\frac{1}{n_t^\ell}
			\Bigl\{2\bar r_{\ell,t,j}^{\,2}
			+{K}_4\sum_{m=1}^{p}
			\bigl(\bm e_m^\top\bm\Sigma_t^{1/2}
			\bm u_{\ell,t-1,j}\bigr)^{4}\Bigr\}\\
			=&\frac{2\bar r_{\ell,t,j}^{\,2}}{n_t^\ell}
			\bigl(1+o_{\mathbb P}(1)\bigr)  \quad \text{as } p\to\infty. \nonumber
		\end{align}
	\end{lemma}

	\begin{theo}[Attenuation law for the Rayleigh statistic]
		\label{thm:attenuation}
		Write $a(\lambda;c,\sigma^2)$ and $\Phi(\lambda;c,\sigma^2)$ for the
		alignment and spike-location maps of Lemma~\ref{lem:maps}. Define the Rayleigh transfer map
		\begin{equation}
			\label{eq:chi}
			\chi_{\ell,t}(\lambda)
			:=\sigma_t^{2}
			+\lambda\,a\bigl(\lambda;c_{\ell,t-1},\sigma_t^{2}\bigr),
			\qquad
			\lambda>\sigma_t^{2}\sqrt{c_{\ell,t-1}},
		\end{equation}
		which in the canonical case $H_{\ell,t}\Rightarrow\delta_1$ is
		\begin{equation}
			\label{eq:chi_closed}
			\chi_{\ell,t}(\lambda)
			=\frac{\lambda(\lambda+\sigma_t^{2})}
			{\lambda+c_{\ell,t-1}\sigma_t^{2}} .
		\end{equation}
		Then, under Assumptions~\ref{assum:moment}--\ref{assum:init}, for
		each fixed $j\le k$,
		\begin{equation*}
			r_{\ell,t,j}
			=\chi_{\ell,t}(\lambda_{t,j})
			+O_{\mathbb P}\!\left(
			\frac{1}{\sqrt{n_t^\ell}}\right)
			+O_{\mathbb P}\bigl(b^{\mathrm{lag}}_{\ell,t}\bigr),
		\end{equation*}
		where $
		b^{\mathrm{lag}}_{\ell,t}
		:=\bigl(N^{\mathrm{eff}}_{\ell,t}\bigr)^{-1/2}
		+\alpha^{\mathrm{lag}}_{\ell,t}
		+\zeta_t
		+\Delta_{\ell,t}
		+(\xi^{\star})^{t}
		\ \longrightarrow\ 0 $
		under Assumptions~\ref{assum:drift}--\ref{assum:step}. Moreover
		\begin{equation}
			\label{eq:sandwich}
			\chi_{\ell,t}(\lambda_{t,j})
			\ <\ \lambda_{t,j}+\sigma_t^{2}
			\ <\ \Phi_{\ell,t}(\lambda_{t,j}),
		\end{equation}
		and in the canonical case the two maps are reciprocal at any common
		aspect ratio,
		\begin{equation}
			\label{eq:duality}
			\chi(\lambda;c,\sigma^{2})\,
			\Phi(\lambda;c,\sigma^{2})
			=\bigl(\lambda+\sigma^{2}\bigr)^{2},
			\qquad c>0,\ \sigma^{2}>0 .
		\end{equation}
	\end{theo}

	\begin{rem}
		\label{rem:indexshift}
		The aspect ratio in \eqref{eq:chi} is $c_{\ell,t-1}$ and not
		$c_{\ell,t}$, because the direction $\bm u_{\ell,t-1,j}$ is formed
		from $\bm S_{\ell,t-1}$ and it is that matrix which fixes the
		attenuation. No approximation is involved. The value is known
		exactly at the node from the deterministic recursion of
		Lemma~\ref{lem:crec}, and it is the value at which
		\eqref{eq:local_debias} inverts the map, so forward and inverse map
		share the same aspect ratio and their composition is exact.
	\end{rem}
	
	\begin{rem}
		\label{rem:duality}
		The plain sample eigenvalue is {inflated} because the
		eigenvector is optimised on the same data that generate the
		quadratic form, so noise is capitalised. The split-sample Rayleigh
		quotient reverses this mechanism. The direction is fixed before the data
		arrive, and the only effect of high dimensionality is that this
		fixed direction is misaligned with $\bm v_{t,j}$, so a fraction
		$1-a(\lambda_{t,j};c_{\ell,t-1},\sigma^2_t)$ of the signal energy
		leaks into the noise subspace and is replaced by $\sigma_t^2$.
		Attenuation, not inflation, is therefore the correct description of
		the statistic actually transmitted. Identity \eqref{eq:duality} makes
		the relation precise, with $\lambda_{t,j}+\sigma_t^2$ the geometric
		mean of the attenuated and the inflated quantity.
	\end{rem}

	Each node forms
	\begin{equation}
		\label{eq:local_debias}
		\check\lambda_{\ell,t,j}
		:=\chi^{-1}\bigl(
		r_{\ell,t,j}\vee\check\sigma^{2}_{\ell,t}\ ;\
		c_{\ell,t-1},\ \check\sigma^{2}_{\ell,t}\bigr),
		\qquad j=1,\dots,k,
	\end{equation}
	where the truncation keeps the argument in the domain of $\chi^{-1}$
	and is inactive with probability tending to one under
	Assumption~\ref{assum:bbp}, and the local noise level is estimated
	from the same batch by the trace identity
	\begin{equation*}
		\check\sigma^{2}_{\ell,t}
		=\frac{1}{p-k}\Bigl(
		\operatorname{tr}\hat{\bm S}_{\ell,t}
		-\sum_{j=1}^{k}r_{\ell,t,j}\Bigr).
	\end{equation*}

	\begin{lemma}[Local noise-level estimation]
		\label{lem:sigma}
		Under Assumptions~\ref{assum:moment}--\ref{assum:hd},
		\[
		\check\sigma^{2}_{\ell,t}-\sigma_t^{2}
		=O_{\mathbb P}\!\left(\frac{1}{\sqrt{p\,n_t^\ell}}\right)
		+O\!\left(\frac{k}{p}\right).
		\]
		Consequently, multiplying by $\sqrt{n_t^\ell}$ and noting that 
		$(p\,n_t^\ell)^{-1/2}\sqrt{n_t^\ell}=p^{-1/2}\to 0$ and 
		$(k/p)\sqrt{n_t^\ell}\to 0$ whenever $k\sqrt{n_t^\ell}=o(p)$,
		we have $\check\sigma^{2}_{\ell,t}-\sigma_t^{2}=o_{\mathbb P}((n_t^\ell)^{-1/2})$.
	\end{lemma}

	\begin{theo}[Local asymptotic normality of the debiased estimator]
		\label{thm:local_clt}
		Suppose Assumptions~\ref{assum:moment}--\ref{assum:init} hold, that
		$k\sqrt{n_t^\ell}=o(p)$, and that
		\begin{equation}
			\label{eq:clt_cond}
			\sqrt{n_t^\ell}
			\Bigl(\frac{1}{\sqrt{N^{\mathrm{eff}}_{\ell,t}}}
			+\alpha^{\mathrm{lag}}_{\ell,t}+\zeta_t+\Delta_{\ell,t}
			+(\xi^\star)^t\Bigr)\longrightarrow0 .
		\end{equation}
		
		Then, for each fixed $j\le k$,
		\begin{align}
			\label{eq:local_clt}
			\sqrt{n_t^\ell}\,
			\bigl(\check\lambda_{\ell,t,j}-\lambda_{t,j}\bigr)
			\stackrel{d}{\longrightarrow}
			\mathcal N\bigl(0,\ \vartheta^{2}_{\ell,t,j}\bigr), 
		\end{align}	
		\begin{align}
			\label{eq:vartheta}	
			\vartheta^{2}_{\ell,t,j}
			=\frac{2\,\chi_{\ell,t}(\lambda_{t,j})^{2}}
			{\chi_{\ell,t}'(\lambda_{t,j})^{2}}
			=\frac{2\lambda_{t,j}^{2}
				(\lambda_{t,j}+\sigma_t^{2})^{2}
				(\lambda_{t,j}+c_{\ell,t-1}\sigma_t^{2})^{2}}
			{\bigl(\lambda_{t,j}^{2}
				+2c_{\ell,t-1}\sigma_t^{2}\lambda_{t,j}
				+c_{\ell,t-1}\sigma_t^{4}\bigr)^{2}}. 
		\end{align}
	\end{theo}

	\begin{rem}
		\label{rem:feasibility}
		The variance term in \eqref{eq:clt_cond} is automatically controlled under
		the recursion of Section~\ref{sec:3}. When batch sizes are comparable,
		$N^{\mathrm{eff}}_{\ell,t}\asymp n_t^\ell/\alpha_t$, so
		$\sqrt{n_t^\ell/N^{\mathrm{eff}}_{\ell,t}}\asymp\sqrt{\alpha_t}\to0$.
		The active requirement is therefore the per-step lag condition
		$\sqrt{n_t^\ell}\,\alpha^{\mathrm{lag}}_{\ell,t}\to0$.
		Since $\alpha^{\mathrm{lag}}_{\ell,t}\asymp\alpha_t+\sqrt{\alpha_t c_{\ell,t}}$
		and $\alpha_t p/n_t^\ell\asymp c_{\ell,t}$, this condition reduces in the
		proportional regime $c_{\ell,t}\asymp1$ to
		\begin{equation*}
			\alpha_t = o\bigl(1/n_t^\ell\bigr),
			\qquad\text{equivalently}\qquad
			n_t^\ell = o(\sqrt{p}).
		\end{equation*}
		The constraint therefore governs the batch size rather than the step size.
		A concrete example shows that the feasible region is non-empty.
		Taking $p=p_t\asymp t$, $n_t^\ell\asymp t^{1/3}$, and $\alpha_t\asymp t^{-2/3}$
		gives $\alpha_t\to0$, $\sum_t\alpha_t=\infty$, $c_{\ell,t}\asymp1$,
		$\tilde\alpha_{\ell,t}\asymp1$, $\alpha^{\mathrm{lag}}_{\ell,t}\asymp t^{-1/3}\to0$,
		and $k\sqrt{n_t^\ell}=o(p)$, so all conditions of
		Theorem~\ref{thm:local_clt} are satisfied simultaneously.
	\end{rem}

	Two architectures are available. 
	\begin{align*}
		\text{(A) correct-then-aggregate:}\qquad
		&\check\lambda_{t,j}^{\mathrm{A}}
		=\sum_{\ell\in\mathcal A_t}\omega_{\ell,t}\,
		\chi_{\ell,t}^{-1}(r_{\ell,t,j}).\\
		\text{(B) aggregate-then-correct:}\qquad
		&\check\lambda_{t,j}^{\mathrm{B}}
		=\bar\chi_{t}^{-1}\Bigl(
		\sum_{\ell\in\mathcal A_t}\omega_{\ell,t}r_{\ell,t,j}\Bigr). 
	\end{align*}
	
	We adopt Ordering A throughout. When aspect ratios are heterogeneous, 
	Ordering B incurs a non-vanishing bias of order 
	$\mathcal{D}_t:=\sum_{\ell}\omega_{\ell,t}(c_{\ell,t-1}-c^{\omega}_t)^{2}$ 
	due to Jensen's inequality applied to the convex map $c\mapsto\chi(\lambda;c,\sigma^2)$, 
	and this bias dominates the stochastic error asymptotically 
	(Theorem~\ref{thm:ordering} in Appendix~\ref{sec:a.vs.b}). 
	Ordering A avoids the bias at no additional communication cost, 
	since each node can compute $\chi_{\ell,t}^{-1}$ locally using 
	$c_{\ell,t-1}$ and $\check\sigma^2_{\ell,t}$.

	\subsection{Aggregation and Asymptotic Inference\label{sec:4.2}}
	
	The framework is organised as correct-then-aggregate,
	in which each active node transmits the locally
	debiased $k$-vector
	\begin{equation*}
		\check{\bm\lambda}_{\ell,t}
		=\bigl(\check\lambda_{\ell,t,1},\dots,
		\check\lambda_{\ell,t,k}\bigr)^\top. 
	\end{equation*}
	
	For each coordinate $j$, the server forms a mixed fluctuation metric with two exponentially smoothed components. The volatility component tracks
	temporal instability of the transmitted coordinate,
	\begin{equation*}
		V^{\mathrm{vol}}_{\ell,t,j}
		=(1-\eta_t)V^{\mathrm{vol}}_{\ell,t-1,j}
		+\eta_t\bigl(\check\lambda_{\ell,t,j}
		-\check\lambda_{\ell,t-1,j}\bigr)^{2} ,
	\end{equation*}
	and the residual component measures departure from the running
	consensus,
	\begin{equation}
		\label{eq:res}
		V^{\mathrm{res}}_{\ell,t,j}
		=(1-\eta_t)V^{\mathrm{res}}_{\ell,t-1,j}
		+\eta_t\bigl(\check\lambda_{\ell,t,j}
		-\tilde\lambda_{t-1,j}\bigr)^{2} ,
		\quad
		V_{\ell,t,j}
		=(1-\rho)V^{\mathrm{vol}}_{\ell,t,j}
		+\rho V^{\mathrm{res}}_{\ell,t,j},
		\quad\rho\in(0,1).
	\end{equation}
	The weights and the global recursion are
	\begin{equation}
		\label{eq:weights}
		\omega_{\ell,t,j}
		=\frac{\exp(-\tilde\tau_t\Psi_{\ell,t,j})}
		{\sum_{m\in\mathcal A_t}\exp(-\tilde\tau_t\Psi_{m,t,j})},
		\qquad
		\tilde\lambda_{t,j}
		=(1-\beta_t)\tilde\lambda_{t-1,j}
		+\beta_t\sum_{\ell\in\mathcal A_t}
		\omega_{\ell,t,j}\,\check\lambda_{\ell,t,j},
	\end{equation}
	for $j=1,\dots,k$.

	\begin{assum}
		\label{assum:eta}
		$\eta_t\in(0,1]$, $\eta_t\to0$, $\sum_t\eta_t=\infty$,
		$\sum_t\eta_t^2<\infty$, and
		$\max_{j\le k}\bigl(V^{\mathrm{vol}}_{\ell,0,j}
		+V^{\mathrm{res}}_{\ell,0,j}\bigr)
		=O\bigl((\lambda_{\max}+\sigma^2_{\max})^2\bigr)$ uniformly in
		$\ell$.
	\end{assum}
	
	\begin{assum}
		\label{assum:persist}
		For each $\ell$,
		\begin{equation*}
			\max_{|s-t|\le\lceil\eta_t^{-1}\rceil}
			\Bigl|\frac{n^\ell_s}{n^\ell_t}-1\Bigr|
			\ \stackrel{\mathbb P}{\longrightarrow}\ 0 .
		\end{equation*}
	\end{assum}

	\begin{assum}
		\label{assum:temp}
		The clip levels satisfy $0<\underline\theta\le\bar\theta<\infty$,
		and the temperature $\tilde\tau_t$ of \eqref{eq:weights} is
		$\mathcal F_{t-1}$-measurable with
		\begin{equation*}
			\tilde\tau_t\in[\underline\tau,\bar\tau],
			\qquad
			0<\underline\tau\le\bar\tau<\infty. 
		\end{equation*}
	\end{assum}
	\begin{rem}
		The reference value is $\tilde\tau_t=1$, at which the soft-max agrees
		with the inverse-variance weights \eqref{eq:opt_weights} to first order.
		Consistency requires only that $\tilde\tau_t$ remain bounded away from
		zero and infinity, efficiency requires $\tilde\tau_t\approx1$, and the
		gap is quadratic in $\tilde\tau_t-1$. The soft-max formulation is
		retained over the direct plug-in \eqref{eq:opt_weights} because it
		remains robust to node-specific bias, which the inverse-variance
		weights ignore, and because $\tilde\tau_t$ serves as a tunable
		parameter that can be adjusted when heterogeneity is extreme. The
		resulting efficiency is quantified in
		Proposition~\ref{prop:softmax_releff}. Technical details on the design
		choices, including the use of predictable weights and the role of
		smoothing, are provided in Appendix~\ref{app:design_choices}.
	\end{rem}

	\begin{assum}
		\label{assum:balance}
		There is $C_{\mathrm{bal}}<\infty$ with
		$\min_{\ell\in\mathcal A_t}n_t^\ell
		\ge N_t/(C_{\mathrm{bal}}A_t)$ for all $t$. 
	\end{assum}
	
	Assumption~\ref{assum:balance} converts a weighted average over nodes
	into a genuine pooled rate. It prevents a single node with a vanishing
	batch from dominating the variance. It constrains batch sizes only,
	not aspect ratios, whose heterogeneity is precisely what Ordering~A is
	designed to accommodate.

	For $j\le k$, let 
	\begin{equation*}
		{ T}_{t,j}:=\sup_{\ell\in\mathcal A_t}\vartheta^2_{\ell,t,j},
		\qquad
		\vartheta_{t,j}:={ T}_{t,j}^{1/2},
		\qquad
		{ T}_t:=\sum_{j=1}^k{ T}_{t,j},
	\end{equation*}
	all of which are bounded above and below by \eqref{eq:varthetabounds}.

	\begin{lemma}[Concentration of the fluctuation metric]
		\label{lem:fluct}
		Let $\pi^{\eta}_{s,t}:=\eta_s\prod_{u=s+1}^t(1-\eta_u)$ and
		$\pi^{\eta}_{0,t}:=\prod_{u=1}^t(1-\eta_u)$. Under
		Assumptions~\ref{assum:moment}--\ref{assum:balance}, for every
		$\ell\in\mathcal A_t$ and every $j\le k$,
		\begin{equation}
			\label{eq:fluct}
			V_{\ell,t,j}
			=O_{\mathbb P}\!\left(
			\sum_{s\le t}\pi^{\eta}_{s,t}
			\Bigl\{\frac{{ T}_{s,j}}{n_s^\ell}
			+\zeta_s^2+(b^{\mathrm{lag}}_{\ell,s})^2
			+\bigl(\tilde\lambda_{s-1,j}
			-\lambda_{s-1,j}\bigr)^{2}\Bigr\}
			+\pi^{\eta}_{0,t}\right),
		\end{equation}
		and under Assumption~\ref{assum:persist},
		\begin{equation}
			\label{eq:fluct_conc}
			\frac{V_{\ell,t,j}}
			{\mathbb E[V_{\ell,t,j}\mid\mathcal F_{t-\lceil1/\eta_t\rceil}]}
			=1+O_{\mathbb P}\bigl(\sqrt{\eta_t}\bigr),
		\end{equation}
		\begin{equation*}
			\mathbb E\bigl[V_{\ell,t,j}\bigr]
			=\frac{\kappa_\rho\,\vartheta^2_{\ell,t,j}}{n_t^\ell}
			\bigl(1+o(1)\bigr)
			+O\bigl(\zeta_t^2+(\tilde\lambda_{t-1,j}
			-\lambda_{t-1,j})^2\bigr),
		\end{equation*}
		with $\kappa_\rho:=2(1-\rho)+\rho$. 
	\end{lemma}

	\begin{lemma}[Weight regularity]
		\label{lem:weights}
		Under Assumption~\ref{assum:temp} alone, surely and uniformly in
		$\ell\in\mathcal A_t$, $j\le k$ and $t\ge1$,
		\begin{equation}
			\label{eq:weight_order}
			C_\omega^{-1}A_t^{-1}
			\le\omega_{\ell,t,j}\le C_\omega A_t^{-1},
			\qquad
			\sum_{\ell\in\mathcal A_t}\omega_{\ell,t,j}^2
			\le C_\omega A_t^{-1},
		\end{equation}
		with
		$C_\omega=\exp\{\bar\tau(\bar\theta-\underline\theta)\}$.
	\end{lemma}

	Write the local error and its predictable decomposition as
	\begin{equation}
		\label{eq:eps}
		\varepsilon_{\ell,t,j}
		:=\check\lambda_{\ell,t,j}-\lambda_{t,j}
		=\varepsilon^{\mathrm{stoch}}_{\ell,t,j}+b_{\ell,t,j},
		\qquad
		b_{\ell,t,j}
		:=\mathbb E\bigl[\check\lambda_{\ell,t,j}\mid
		\mathcal F_{t-1}\bigr]-\lambda_{t,j}.
	\end{equation}
	Theorem~\ref{thm:local_clt} give
	\begin{equation}
		\label{eq:local_moments}
		\operatorname{Var}\bigl(\varepsilon^{\mathrm{stoch}}_{\ell,t,j}
		\mid\mathcal F_{t-1}\bigr)
		=\frac{\vartheta^2_{\ell,t,j}}{n_t^\ell}
		\bigl(1+o_{\mathbb P}(1)\bigr),
		\qquad
		|b_{\ell,t,j}|
		=O\bigl(b^{\mathrm{lag}}_{\ell,t}\bigr)
		+O_{\mathbb P}\!\Bigl(\tfrac{1}{n_t^\ell}\Bigr),
	\end{equation}
	the $O(1/n_t^\ell)$ term arising from the curvature of
	$\chi^{-1}_{\ell,t}$ and from Lemma~\ref{lem:sigma}. 
	
	\begin{lemma}[Weighted average error]
		\label{lem:wavg}
		Under Assumptions~\ref{assum:moment}--\ref{assum:balance}, for each fixed $j\le k$,
		\begin{equation*}
			\bar\varepsilon_{t,j}
			:=\sum_{\ell\in\mathcal A_t}\omega_{\ell,t,j}
			\varepsilon_{\ell,t,j}
			=O_{\mathbb P}\!\left(
			\sqrt{\frac{{ T}_{t,j}}{N_t}}\right)
			+O\bigl(\max_{\ell\in\mathcal A_t}b^{\mathrm{lag}}_{\ell,t}\bigr)
			+O_{\mathbb P}\!\Bigl(\frac{A_t}{N_t}\Bigr)
			+O\Bigl(\frac{1}{\sqrt{p\min_\ell n^\ell_t}}+\frac{k}{p}\Bigr),
		\end{equation*}
		where the third term vanishes when $A_t=o(\sqrt{N_t})$, and the
		fourth, which is the plug-in error of $\check\sigma^2_{\ell,t}$,
		vanishes under Assumption~\ref{assum:hd} together with the condition
		$k\sqrt{n^\ell_t}=o(p)$ of Lemma~\ref{lem:sigma}.
		Moreover, $\mathbb E[\sum_\ell\omega_{\ell,t,j}
		\varepsilon^{\mathrm{stoch}}_{\ell,t,j}\mid\mathcal F_{t-1}]=0$
		exactly with
		\begin{equation}
			\label{eq:condvar_round}
			\operatorname{Var}\Bigl(\sum_{\ell}\omega_{\ell,t,j}
			\varepsilon^{\mathrm{stoch}}_{\ell,t,j}
			\Bigm|\mathcal F_{t-1}\Bigr)
			=\sum_{\ell\in\mathcal A_t}
			\frac{\omega^2_{\ell,t,j}\vartheta^2_{\ell,t,j}}{n_t^\ell}
			\bigl(1+o_{\mathbb P}(1)\bigr)
			=O_{\mathbb P}\Bigl(\frac{{ T}_{t,j}}{N_t}\Bigr). 
		\end{equation}
	\end{lemma}

	\begin{lemma}[Recursive decomposition]
		\label{lem:decomp_new}
		With $\varrho_{s,t}=\prod_{u=s+1}^{t}(1-\beta_u)$ and
		$\delta_{t,j}:=\lambda_{t,j}-\lambda_{t-1,j}$,
		\begin{equation*}
			\tilde\lambda_{t,j}-\lambda_{t,j}
			=(1-\beta_t)\bigl(\tilde\lambda_{t-1,j}-\lambda_{t-1,j}\bigr)
			+\beta_t\bar\varepsilon_{t,j}
			-(1-\beta_t)\delta_{t,j},
		\end{equation*}
		and hence
		\begin{align}
			\label{eq:agg_unroll}
			\tilde\lambda_{t,j}-\lambda_{t,j}
			=&\underbrace{\sum_{s=1}^{t}\varrho_{s,t}\beta_s
				\sum_{\ell\in\mathcal A_s}\omega_{\ell,s,j}
				\varepsilon^{\mathrm{stoch}}_{\ell,s,j}}
			_{=:\ \mathcal M_{t,j}\ \text{(martingale)}}
			\ +\ \sum_{s=1}^{t}\varrho_{s,t}
			\Bigl\{\beta_s\!\!\sum_{\ell\in\mathcal A_s}\!\!
			\omega_{\ell,s,j}b_{\ell,s,j}
			-(1-\beta_s)\delta_{s,j}\Bigr\}
			\\ 
			&+\ \varrho_{0,t}\bigl(\tilde\lambda_{0,j}
			-\lambda_{0,j}\bigr).\nonumber
		\end{align}
	\end{lemma}
	
	The drift term is the increment of the population spike itself,
	bounded directly by $\zeta_s$ under Assumption~\ref{assum:drift}. 
	
	\begin{defin}[Effective aggregation sample size]
		\label{def:Nsharp}
		Define
		\begin{equation}
			\label{eq:Nsharp}
			\mathsf N_t
			:=\Bigl(\sum_{s=1}^{t}\varrho_{s,t}^{2}\beta_s^{2}
			\,N_s^{-1}\Bigr)^{-1}.
		\end{equation}
		If $\beta_t\equiv\beta\in(0,1]$ and $N_s\equiv N$, then
		$\mathsf N_t\uparrow N(2-\beta)/\beta=N/\kappa_\beta$ with
		$\kappa_\beta:=\beta/(2-\beta)$. In particular $\mathsf N_t=N$ when
		$\beta=1$. If $\beta_t\to0$ with $\sum\beta_t=\infty$ and
		$N_s\equiv N$, then $\mathsf N_t/N\to\infty$.
	\end{defin}
	
	The quantity $\mathsf N_t$, not $N_t$, is the correct sample-size
	scale for the global estimator. The $\beta$-recursion averages over an
	effective window of $\asymp2/\beta$ rounds, and $\mathsf N_t$ records
	the total information content of that window. 
	
	\begin{theo}[Adaptive distributed spectral fusion]
		\label{thm:fusion}
		Under Assumptions~\ref{assum:moment}--\ref{assum:balance}, and for any schedule
		$\beta_t\in(0,1]$ with $\varrho_{0,t}\to0$, for each fixed $j\le k$
		\begin{equation}
			\label{eq:fusion_rate}
			\tilde\lambda_{t,j}-\lambda_{t,j}
			=O_{\mathbb P}\bigl(\mathsf v_{t,j}\bigr)
			+O\!\left(\sum_{s=1}^{t}\varrho_{s,t}
			\Bigl\{\beta_s\max_{\ell\in\mathcal A_s}
			b^{\mathrm{lag}}_{\ell,s}
			+(1-\beta_s)|\delta_{s,j}|\Bigr\}\right)
			+O\bigl(\varrho_{0,t}\bigr),
		\end{equation}
		where
		\begin{equation}
			\label{eq:vsharp}
			\mathsf v^2_{t,j}
			:=\sum_{s=1}^{t}\varrho^2_{s,t}\beta_s^2
			\sum_{\ell\in\mathcal A_s}
			\frac{\omega^2_{\ell,s,j}\vartheta^2_{\ell,s,j}}{n^\ell_s}
		\end{equation}
		is the conditional variance of the martingale
		$\mathcal M_{t,j}$ of \eqref{eq:agg_unroll}, and
		$\bar{ T}_{t,j}:=\max_{s\le t}{ T}_{s,j}=O(1)$.
		In particular, if $\mathsf N_t\to\infty$ and the second and third
		terms vanish, then
		$\tilde\lambda_{t,j}\stackrel{p}{\longrightarrow}\lambda_{t,j}$. 
	\end{theo}

	We now establish a linear representation and a central limit theorem for the global estimator. 
	
	\begin{assum}
		\label{assum:stab}
		As $t\to\infty$, for each fixed $j\le k$:
		\begin{enumerate}[label=\textup{(\roman*)}]
			\item \textit{Proportional regime.}
			$c_{\ell,t}\to c_\ell\in(0,\bar c\,]$ for each $\ell$, and
			$c_t=p/N_t\to c^{\mathrm{srv}}\in[0,\infty]$, the value
			$c^{\mathrm{srv}}=\infty$ being permitted.
			\item \textit{Parameter stabilization.}
			$\lambda_{t,j}\to\lambda_j$, $\sigma^2_t\to\sigma^2$, with
			$\lambda_j>\sigma^2\sqrt{\sup_\ell c_\ell}$, 
			\[
			\sqrt{\mathsf N_t}\,
			\sum_{s=1}^{t}\varrho_{s,t}
			\Bigl\{\beta_s\max_{\ell\in\mathcal A_s}
			b^{\mathrm{lag}}_{\ell,s}
			+(1-\beta_s)|\delta_{s,j}|\Bigr\}
			\longrightarrow0,
			\]
			$\sqrt{\mathsf N_t}\,\varrho_{0,t}\longrightarrow0,$ and 
			$\sqrt{\mathsf N_t}\,|\lambda_{t,j}-\lambda_j|\longrightarrow0$. 
			\item \textit{Batch balance.}
			$\max_{\ell\in\mathcal A_t}n_t^\ell/
			\min_{\ell\in\mathcal A_t}n_t^\ell\le C_{\mathrm{bal}}$
			\textup{(}as in Assumption~\ref{assum:balance}\textup{)},
			$N_t\to\infty$, $\mathsf N_t\to\infty$, and
			$\min_{\ell\in\mathcal A_t}n^\ell_t\gg\sqrt{\mathsf N_t}$
			(equivalently $A_t=o(n^\ell_t)$ for fixed $\beta$).
			\item \textit{Step size.} Either $\beta_t\equiv\beta\in(0,1]$,
			or $\beta_t\to0$ with $\sum_t\beta_t=\infty$.
			\item \textit{Profile convergence \textup{(}needed only for
				the limiting variance\textup{)}.} The normalised batch
			profile converges, $n^\ell_t/\bar n_t\to C\in
			[C_{\mathrm{bal}}^{-1},C_{\mathrm{bal}}]$ for each $\ell$.
		\end{enumerate}
	\end{assum}
	
	\begin{rem}
		Condition~(ii) strengthens the drift condition to the
		$\sqrt{\mathsf N_t}$ scale required for the CLT and absorbs the
		algorithmic lag. No condition on the subspace error appears:
		Ordering~A makes the CLT available without one. Condition~(iv)
		covers both fixed $\beta$ (stationary regime with effective window
		$\asymp2/\beta$) and vanishing $\beta_t$ (accumulating information),
		handled through $\mathsf N_t$. Condition~(i) keeps
		$\chi^{-1}_{\ell,t}$ well conditioned and permits
		$c^{\mathrm{srv}}=\infty$. No condition on weights is needed beyond
		$\omega_{\ell,t,j}\asymp A_t^{-1}$ from Lemma~\ref{lem:weights}.
	\end{rem}
	
	Define the node-level influence function of the debiased statistic. 
	For an observation $\bm X_{\ell,s,i}$ of the batch received at time
	$s$,
	\begin{equation}
		\label{eq:influence}
		\varsigma_{\ell,s,i,j}
		:=\frac{1}{\chi'_{\ell,s}(\lambda_{s,j})}
		\Bigl\{\bigl(\bm u_{\ell,s-1,j}^\top\bm X_{\ell,s,i}\bigr)^2
		-\chi_{\ell,s}(\lambda_{s,j})\Bigr\} ,
	\end{equation}
	together with its conditionally centred version, which is the array
	actually used in the martingale arguments below,
	\begin{equation*}
		\varsigma^{\mathrm c}_{\ell,s,i,j}
		:=\varsigma_{\ell,s,i,j}
		-\mathbb E\bigl[\varsigma_{\ell,s,i,j}\mid\mathcal F_{s-1}\bigr] .
	\end{equation*}
	
	By Lemma~\ref{lem:decomp}, Theorem~\ref{thm:attenuation} and
	\eqref{eq:clt_cond},
	\[
	\mathbb E[\varsigma_{\ell,s,i,j}\mid\mathcal F_{s-1}]
	=o_{\mathbb P}\bigl((n_s^\ell)^{-1/2}\bigr),
	\qquad
	\operatorname{Var}(\varsigma_{\ell,s,i,j}\mid\mathcal F_{s-1})
	=\vartheta^2_{\ell,s,j}\bigl(1+o_{\mathbb P}(1)\bigr),
	\]
	with $\vartheta^2_{\ell,s,j}
	=2\chi_{\ell,s}(\lambda_{s,j})^2/\chi'_{\ell,s}(\lambda_{s,j})^2$. 
	The normalisation by $\chi'_{\ell,s}$ is the
	node-level delta-method factor, applied before averaging.
	
	\begin{lemma}[Asymptotic linear representation]
		\label{lem:linrep}
		Under the conditions of Theorem~\ref{thm:fusion} and
		Assumption~\ref{assum:stab},
		\begin{equation}
			\label{eq:linrep}
			\tilde{\lambda}_{t,j}-\lambda_j
			=\sum_{s=1}^{t}\varrho_{s,t}\beta_s
			\sum_{\ell\in\mathcal A_s}
			\frac{\omega_{\ell,s,j}}{n_s^\ell}
			\sum_{i=1}^{n_s^\ell}\varsigma_{\ell,s,i,j}
			\;+\;R_{t,j},
			\qquad
			R_{t,j}=o_{\mathbb P}\bigl(\mathsf N_t^{-1/2}\bigr).
		\end{equation}
	\end{lemma}

	\begin{theo}[Asymptotic normality]
		\label{thm:clt}
		Under the conditions of Lemma~\ref{lem:linrep} and
		Assumption~\ref{assum:stab}, for each fixed $j\le k$,
		\begin{equation*}
			\mathsf v^{-1}_{t,j}
			\bigl(\tilde{\lambda}_{t,j}-\lambda_{t,j}\bigr)
			\stackrel{d}{\longrightarrow}\mathcal N(0,1),
		\end{equation*}
		If in addition Assumption~\ref{assum:stab}(v) holds, 
		then $\lambda_{t,j}$ may be replaced by $\lambda_j$ and
		\begin{equation*}
			\sqrt{\mathsf N_t}
			\bigl(\tilde{\lambda}_{t,j}-\lambda_j\bigr)
			\stackrel{d}{\longrightarrow}
			\mathcal N\bigl(0,\ \Omega^2_j\bigr),
			\qquad
			\mathsf N_t\,\mathsf v^2_{t,j}
			\stackrel{p}{\longrightarrow}\Omega^2_j .
		\end{equation*}
		With adaptive weights the asymptotic variance is
		\begin{equation}
			\label{eq:clt_closed}
			\Omega^2_j
			=\vartheta^2_j
			=\frac{2\lambda_j^2(\lambda_j+\sigma^2)^2
				(\lambda_j+c\sigma^2)^2}
			{\bigl(\lambda_j^2+2c\sigma^2\lambda_j
				+c\sigma^4\bigr)^2},
		\end{equation}
		where $c=\lim_t p_t/n_t$ is the limiting aspect ratio.
		The convergence is joint over $j=1,\dots,k$ with asymptotically
		diagonal covariance.
	\end{theo}

	\begin{rem}
		\label{rem:beta_role}
		The temporal smoothing parameter $\beta_t$ trades off variance
		reduction against drift tracking. Smaller $\beta_t$ averages over
		more rounds and lowers stochastic error but accumulates more bias
		when the target drifts. The optimal choice balances these two
		sources of error and depends on how fast the population changes
		relative to the total sample size.
	\end{rem}
	
	\begin{rem}
		\label{rem:var_structure}
		The asymptotic variance \eqref{eq:clt_closed} increases with the
		aspect ratio $c$ and approaches the classical spiked model variance
		when $c$ is small. High dimensionality inflates variance through two
		competing effects. The projection onto estimated eigenvectors
		attenuates fluctuations, but inverting the biased eigenvalues
		amplifies them. When $p$ and $n$ are comparable, the amplification
		dominates.
	\end{rem}
	
	\begin{rem}
		\label{rem:threshold}
		The estimator remains stable near the Baik-Ben Arous-Peche threshold
		$\lambda_j=\sigma^2\sqrt{c}$, where the spike becomes undetectable.
		The asymptotic variance stays finite at the threshold, unlike
		estimators that invert sample eigenvalues. What breaks down is not
		precision but identifiability. Below the threshold the population
		eigenvalue cannot be recovered from any summary statistic because the
		signal is indistinguishable from noise.
	\end{rem}
	
	\begin{coro}[Feasible inference]
		\label{cor:feasible}
		Let $\hat\vartheta^2_{\ell,t,j}$ denote $\vartheta^2_{\ell,t,j}$
		with $(\lambda_{t,j},\sigma^2_t)$ replaced by
		$(\tilde{\lambda}_{t,j},\check\sigma^2_{\ell,t})$, the aspect ratio
		$c_{\ell,t-1}$ being known exactly at the node.
		Define the online variance recursion
		\begin{equation}
			\label{eq:varrec}
			\hat{\mathsf v}^2_{t,j}
			=(1-\beta_t)^2\hat{\mathsf v}^2_{t-1,j}
			+\beta_t^2\sum_{\ell\in\mathcal A_t}
			\frac{\omega^2_{\ell,t,j}\hat\vartheta^2_{\ell,t,j}}
			{n_t^\ell},
			\qquad \hat{\mathsf v}^2_{0,j}=0 .
		\end{equation}
		Under the conditions of Theorem~\ref{thm:clt},
		$\hat{\mathsf v}^2_{t,j}/\mathsf v^2_{t,j}
		\stackrel{p}{\longrightarrow}1$, and
		\[
		\Bigl[\tilde{\lambda}_{t,j}
		\pm z_{1-\alpha/2}\,\hat{\mathsf v}_{t,j}\Bigr]
		\]
		is an asymptotically valid $(1-\alpha)$ confidence interval for
		$\lambda_j$. 
	\end{coro}

	\subsection{Non-Asymptotic Error Analysis\label{sec:4.3}}
	
	The results of Sections~\ref{sec:4.1}--\ref{sec:4.2} require only
	bounded fourth moments (Assumption~\ref{assum:moment}), which suffice
	for the operator-norm concentration and CLT arguments used there.
	Exponential tail bounds require more.
	
	\begin{assum}[Sub-Gaussian entries]
		\label{assum:subg}
		There is $K>0$ with
		\[\sup_{p,\ell,t,i,m}\mathbb E\exp(z^2_{\ell,t,i,m}/K^2)\le2.\]
	\end{assum}
	
	Assumption~\ref{assum:subg} strictly strengthens
	Assumption~\ref{assum:moment} and is used only to invoke
	Hanson--Wright and Bernstein bounds
	\citep{rudelson2013hanson,vershynin2018high}
	on the centred quadratic forms
	$\varepsilon^{\mathrm{stoch}}_{\ell,t,j}$. It also renders
	Assumption~\ref{assum:deloc} automatic.

	\begin{lemma}[Deterministic transfer]
		\label{lem:transfer}
		Work in the canonical case and let
		$L_{\ell,t}:=\max\{1,(1+\sqrt{c_{\ell,t-1}})/2\}\le
		L_{\bar c}:=\max\{1,(1+\sqrt{\bar c})/2\}$. Then,
		\begin{equation*}
			\bigl|\check\lambda_{\ell,t,j}-\lambda_{t,j}\bigr|
			\ \le\
			L_{\bar c}\,\bigl|r_{\ell,t,j}
			-\chi_{\ell,t}(\lambda_{t,j})\bigr|
			+C(1+\bar c)\,
			\bigl|\check\sigma^2_{\ell,t}-\sigma_t^2\bigr|
			+C\,b^{\mathrm{lag}}_{\ell,t},
		\end{equation*}
		on the event $\{r_{\ell,t,j}\ge\check\sigma^2_{\ell,t}\}$, whose
		complement has probability $O(kA_t\exp(-c\epsilon^2_0\min_\ell n^\ell_t))$
		under Assumption~\ref{assum:subg}. 
	\end{lemma}
	
	Four sources of error must be accounted for simultaneously. The first
	is stochastic fluctuation, which is governed by $\mathsf N_t$. The
	second is temporal drift of the spikes, which accumulates through
	$\varrho_{s,t}$. The third is the algorithmic lag of the orthogonal
	iteration and of the memory kernel, and the fourth is the curvature
	of the local inverse maps. Only the first source is stochastic. We collect
	the rest into the {systematic error}
	\begin{align}
		\label{eq:sys}
		\mathcal E_{t,j}
		:=&\sum_{s=1}^{t}\varrho_{s,t}
		\Bigl\{(1-\beta_s)|\delta_{s,j}|
		+C\beta_s\max_{\ell\in\mathcal A_s}
		\bigl(\alpha^{\mathrm{lag}}_{\ell,s}+\zeta_s+\Delta_{\ell,s}
		+(\xi^\star)^{s}\bigr)\\
		& +C'\beta_s\Bigl\{
		\frac{\bar\vartheta^2_{s,j}}
		{\min_{\ell\in\mathcal A_s}n_s^\ell}
		+\bigl(p\min_{\ell\in\mathcal A_s}n^\ell_s\bigr)^{-1/2}
		+\frac{k}{p}\Bigr\}\Bigr\}
		+\varrho_{0,t}\bigl|\tilde\lambda_{0,j}-\lambda_{0,j}\bigr|, \nonumber
	\end{align}
	where in the third group the first term is the curvature bias of
	$\chi^{-1}_{\ell,s}$ and the remaining two are the plug-in error of
	$\check\sigma^2_{\ell,s}$. They are of different orders in $p$ and
	$n^\ell_s$ and none dominates the others in general, which is why all
	three are retained.

	\begin{lemma}[Stochastic concentration]
		\label{lem:conc}
		Let $\mathcal M_{t,j}$ be the martingale of \eqref{eq:agg_unroll}.
		Under Assumptions~\ref{assum:subg} and the conditions of
		Theorem~\ref{thm:fusion}, there exist constants $c_1,c_2>0$ depending
		only on $K,C_{\mathrm{bal}},C_\omega,\bar c$ such that for all $v>0$
		\begin{align}
			\label{eq:bernstein}
			\mathbb P\bigl(|\mathcal M_{t,j}|\ge v\bigr)
			\le& 2\exp\!\left(
			-c_1\mathsf N_t
			\min\!\Bigl\{
			\frac{v^2}{\bar{ T}_{t,j}},\;
			\frac{v}{\bar\vartheta_{t,j}}\Bigr\}\right)\\
			&+2k\Bigl(\sum_{s\le t}A_s\Bigr)
			\exp\bigl({-c_2\epsilon_0^2\min_{\ell,s} n^\ell_s}\bigr),\nonumber
		\end{align}
		where $\bar\vartheta_{t,j}:=\bar{ T}_{t,j}^{1/2}$. 
	\end{lemma}
	
	\begin{rem}
		The scale $\bar\vartheta_{t,j}$ is sharper than the crude
		$\|\bm\Sigma_t\|_{op}$ for raw Rayleigh bounds, 
		$\vartheta^2_{\ell,t,j}\le C\|\bm\Sigma_t\|^2_{op}$ with strict
		inequality when $\lambda_{t,j}<\lambda_{t,1}$. The two-regime form
		reflects sub-Gaussian tails for $v\lesssim\bar\vartheta_{t,j}$ and
		exponential tails beyond. Under Assumption~\ref{assum:moment} alone,
		only the Chebyshev bound $\mathbb P(|\mathcal M_{t,j}|\ge v)
		\le\bar{T}_{t,j}/(\mathsf N_tv^2)$ is available.
	\end{rem}

	\begin{theo}[Non-asymptotic error bound]
		\label{thm:nonasymp}
		Under the conditions of Lemma~\ref{lem:conc}, for every
		$\varepsilon>\mathcal E_{t,j}$,
		\begin{align}
			\label{eq:nonasymp}
			\mathbb P\bigl(
			|\tilde{\lambda}_{t,j}-\lambda_{t,j}|\ge\varepsilon\bigr)
			\le & 2\exp\!\left(
			-c_1\mathsf N_t\min\!\Bigl\{
			\frac{(\varepsilon-\mathcal E_{t,j})^2}{\bar{ T}_{t,j}},\ 
			\frac{\varepsilon-\mathcal E_{t,j}}{\bar\vartheta_{t,j}}
			\Bigr\}\right)\\
			&+ Ck\sum_{s\le t}A_s
			\exp\bigl(-c\,\epsilon_0^2
			\min_{\ell\in\mathcal A_s}n^\ell_s\bigr).\nonumber
		\end{align}
	\end{theo}
	
	\begin{rem}
		\label{rem:tworegimes}
		The bound separates at the floor $\mathcal E_{t,j}$. For
		$\varepsilon\gg\mathcal E_{t,j}$ the failure probability decays at
		the pooled parametric rate
		$\exp(-c_1\mathsf N_t\varepsilon^2/\bar{ T}_{t,j})$, with
		$\mathsf N_t$ reflecting the total information in the effective
		temporal window. For $\varepsilon\le\mathcal E_{t,j}$ the bound is
		vacuous, and reducing the floor calls for faster forgetting, a
		flatter memory profile, larger batches, or a smaller
		$\alpha_t$. To achieve precision $\varepsilon$ with confidence
		$1-\delta'$ it suffices that $\mathcal E_{t,j}\le\varepsilon/2$
		and $\mathsf N_t\gtrsim\bar{ T}_{t,j}\log(2/\delta')
		/(c_1\varepsilon^2)$.
	\end{rem}
	
	\begin{coro}[High-probability rate]
		\label{cor:rate}
		Under the conditions of Theorem~\ref{thm:nonasymp}, for
		$\delta'\in(0,1)$ with $\log(2/\delta')\le c_1\mathsf N_t$ (the
		sub-Gaussian range), with probability at least
		$1-\delta'-Ck\sum_{s\le t}A_s
		\exp(-c\epsilon_0^2\min_{\ell\in\mathcal A_s}n^\ell_s)$,
		\begin{equation}
			\label{eq:hp}
			\bigl|\tilde{\lambda}_{t,j}-\lambda_{t,j}\bigr|
			\le\mathcal E_{t,j}
			+\bar\vartheta_{t,j}
			\sqrt{\frac{\log(2/\delta')}{c_1\mathsf N_t}} .
		\end{equation}
		In particular, if $\mathcal E_{t,j}=O(\mathsf N_t^{-a})$ for some
		$a>0$, then
		$$|\tilde{\lambda}_{t,j}-\lambda_{t,j}|
		=O\bigl(\mathsf N_t^{-a}
		+\bar\vartheta_{t,j}\sqrt{\log(2/\delta')/\mathsf N_t}\bigr), $$ and
		the two sources balance at $a=1/2$, giving the parametric rate
		$O(\mathsf N_t^{-1/2})$ overall.
	\end{coro}
	
	\begin{rem}
		\label{rem:consistency_pictures}
		With $a=1/2$, Corollary~\ref{cor:rate} reproduces the scale of
		Theorem~\ref{thm:clt}, with $\bar\vartheta_{t,j}$ matching
		$\Omega_j$ up to the weight-efficiency factor of
		Propositions~\ref{prop:dispersion} and~\ref{prop:softmax_releff}. The
		design conditions for $a=1/2$ are
		$\alpha^{\mathrm{lag}}_{\ell,t},\zeta_t,\Delta_{\ell,t}
		=O(\mathsf N_t^{-1/2})$,
		restating Assumption~\ref{assum:stab}(ii)--(iii) non-asymptotically.
	\end{rem}




	
	\section{Simulation studies and empirical analysis\label{sec:5}}

	Three estimators are compared throughout. The first is the centralized
	benchmark (CEN), which pools all active batches into a single node,
	$\hat{\bm S}_t=\sum_{\ell\in\mathcal A_t}n^\ell_t\hat{\bm S}_{\ell,t}/N_t$,
	while retaining all other algorithmic components. CEN achieves the
	variance-optimal rate $\vartheta^2_{t,j}/N_t$ and serves as the lower
	bound for any aggregation scheme. The second estimator (ADA) is the
	proposed adaptive method with soft-max weights \eqref{eq:weights}.
	The third (UNI) uses uniform weights $\omega_{\ell,t,j}=1/A_t$,
	isolating the contribution of adaptive weighting. The experiments assess
	how closely ADA approaches the centralized bound and quantify the gain
	over uniform aggregation.

	\subsection{Simulation under a Time-Varying Spiked Covariance Model\label{sec:5.1}}
	
	We first examine the proposed method in a controlled setting in which the
	population spectrum is known at every round, so that the estimation error
	can be measured exactly and the three schemes can be separated by the
	aggregation stage alone.
	
	\textbf{Model.}
	Observations are generated from the time-varying spiked covariance model
	\eqref{eq:spiked} with $p=300$ and $k=3$. The spikes oscillate about
	$\bm\lambda_0=(8,5,3)^\top$ according to
	\begin{equation}
		\label{eq:sim_spikes}
		\lambda_{t,j}=\lambda_{0,j}
		\Bigl\{1+\varepsilon_\lambda
		\sin\bigl(2\pi t/T_{\mathrm{per}}\bigr)\Bigr\},
		\qquad \varepsilon_\lambda=0.05,
	\end{equation}
	and the principal subspace rotates as $\bm V_t=\bm V_0\bm R(0.002t)$ where
	$\bm R(\theta)$ is a one-parameter rotation, giving drift
	$\zeta_t\approx10^{-2}$ consistent with Assumption~\ref{assum:drift}.
	The isotropic level is $\sigma^2=1$. The initial
	eigenvectors $\bm V_0$ are drawn once from the Haar measure and held fixed
	across all replications, so that variability across replications reflects
	sampling only.
	
	There are $L=50$ nodes. At every round each node is activated independently
	with probability $0.5$, and the active set is required to contain at least
	$k+1$ nodes. Node $\ell$ holds a batch size
	\begin{equation}
		\label{eq:sim_batches}
		n^\ell = p + a\,u_\ell,
		\qquad u_\ell\sim\mathrm{Unif}\{1,\dots,L\},
		\qquad a=50,
	\end{equation}
	drawn once at the beginning of the run and held for the whole stream, which
	places the batch sizes between $350$ and $2800$ and reproduces the persistent
	load heterogeneity that Assumption~\ref{assum:balance} permits and that the
	adaptive weights are designed to exploit.

	The step sizes are $\alpha_t=0.2/(1+0.003t)$ and
	$\beta_t=0.5/(1+0.003t)$, metric smoothing is $\eta=0.08$, mixing
	parameter $\rho=0.3$, and temperature $\tilde\tau_t\equiv1$.

	Each run consists of $T_{\mathrm{per}}=260$ rounds, of which the first $60$ are treated as
	burn in. All accuracy measures are averaged over the remaining $200$ rounds,
	which is exactly two full periods of \eqref{eq:sim_spikes}. Every configuration is replicated $30$ times
	with independent data and a common $\bm V_0$.
	
	\begin{figure}[htbp]
		\centering
		\begin{subfigure}[b]{0.32\textwidth}
			\includegraphics[width=\textwidth]{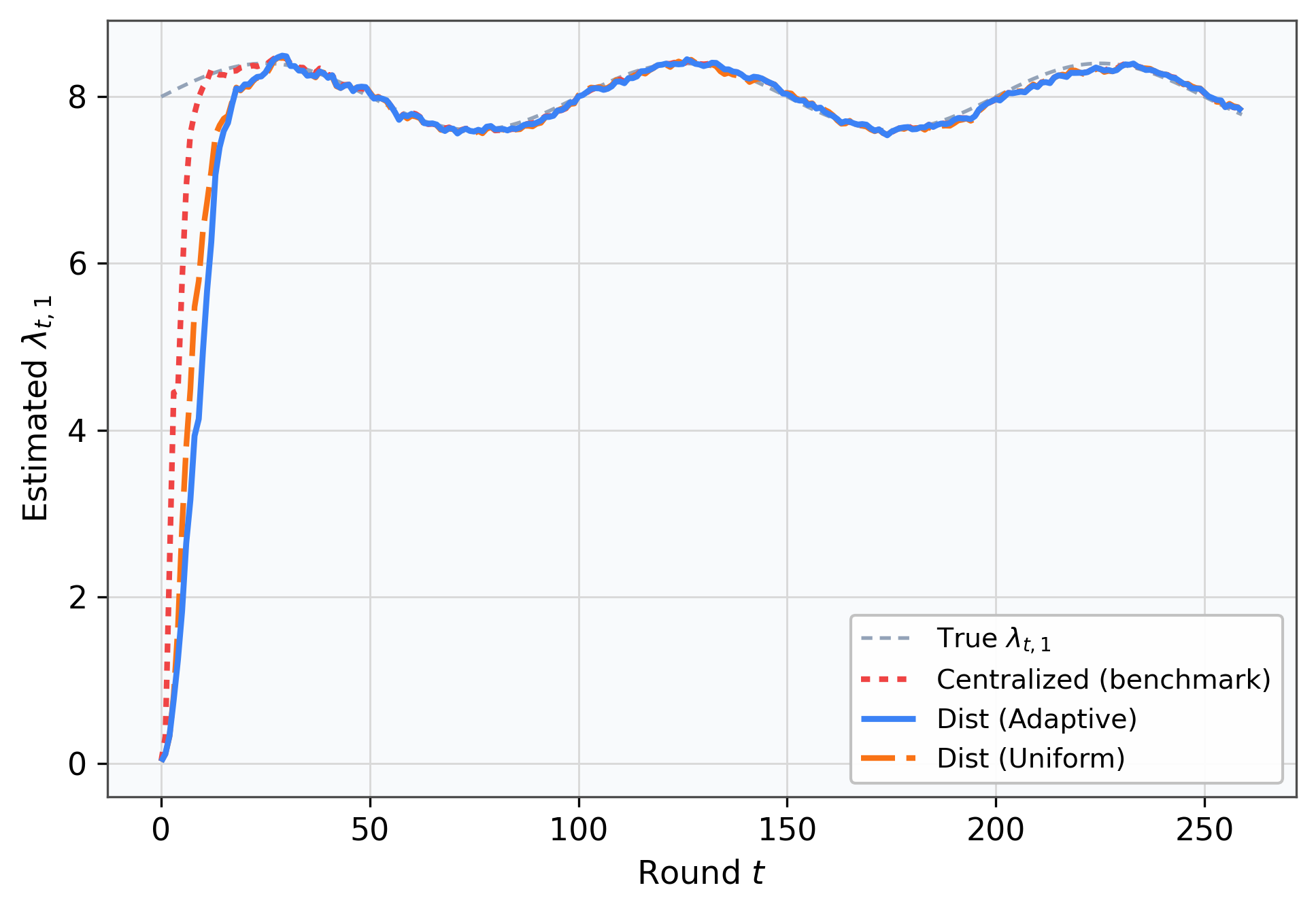}
			\caption{$\lambda_{t,1}$}
		\end{subfigure}\hfill
		\begin{subfigure}[b]{0.32\textwidth}
			\includegraphics[width=\textwidth]{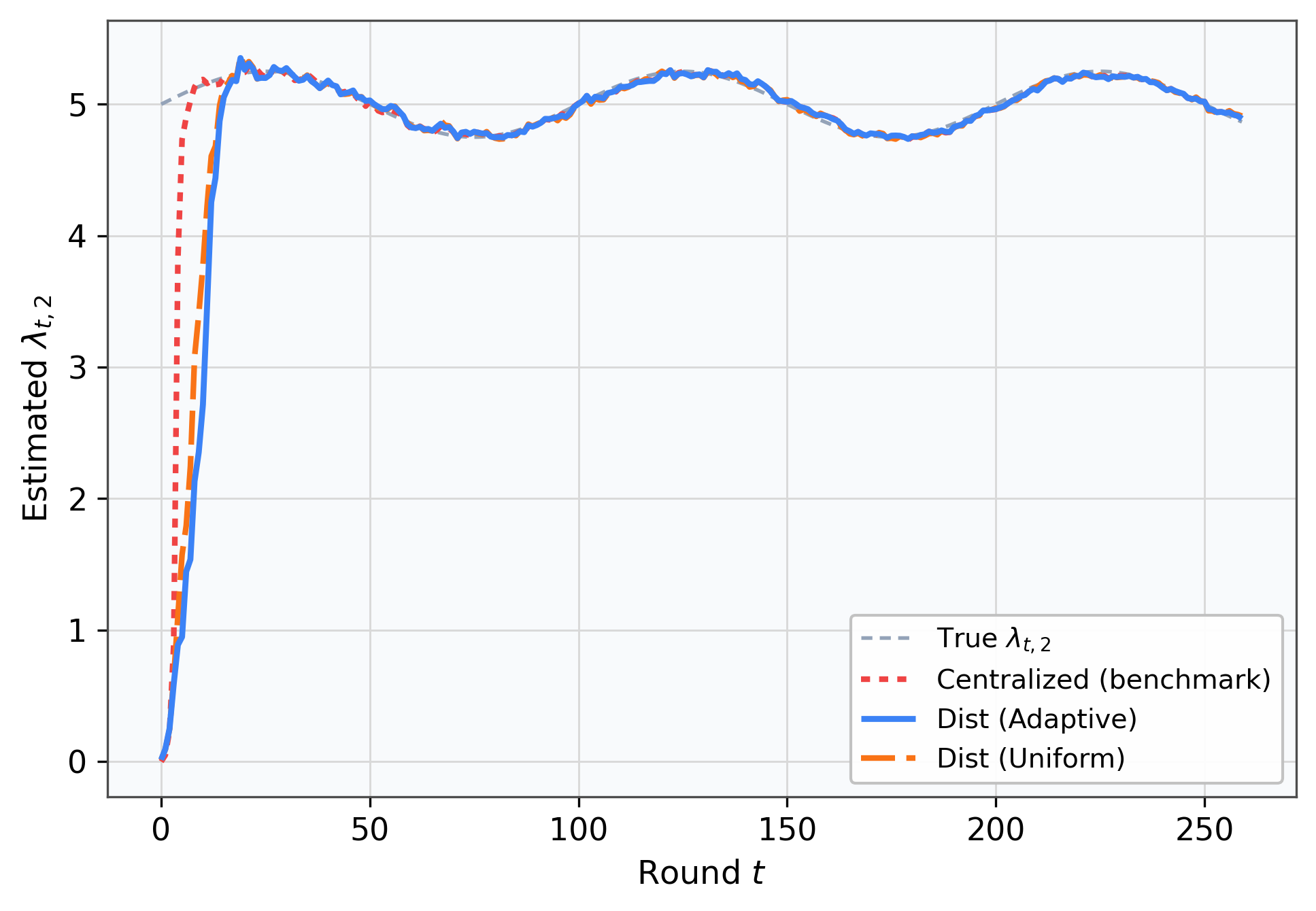}
			\caption{$\lambda_{t,2}$}
		\end{subfigure}\hfill
		\begin{subfigure}[b]{0.32\textwidth}
			\includegraphics[width=\textwidth]{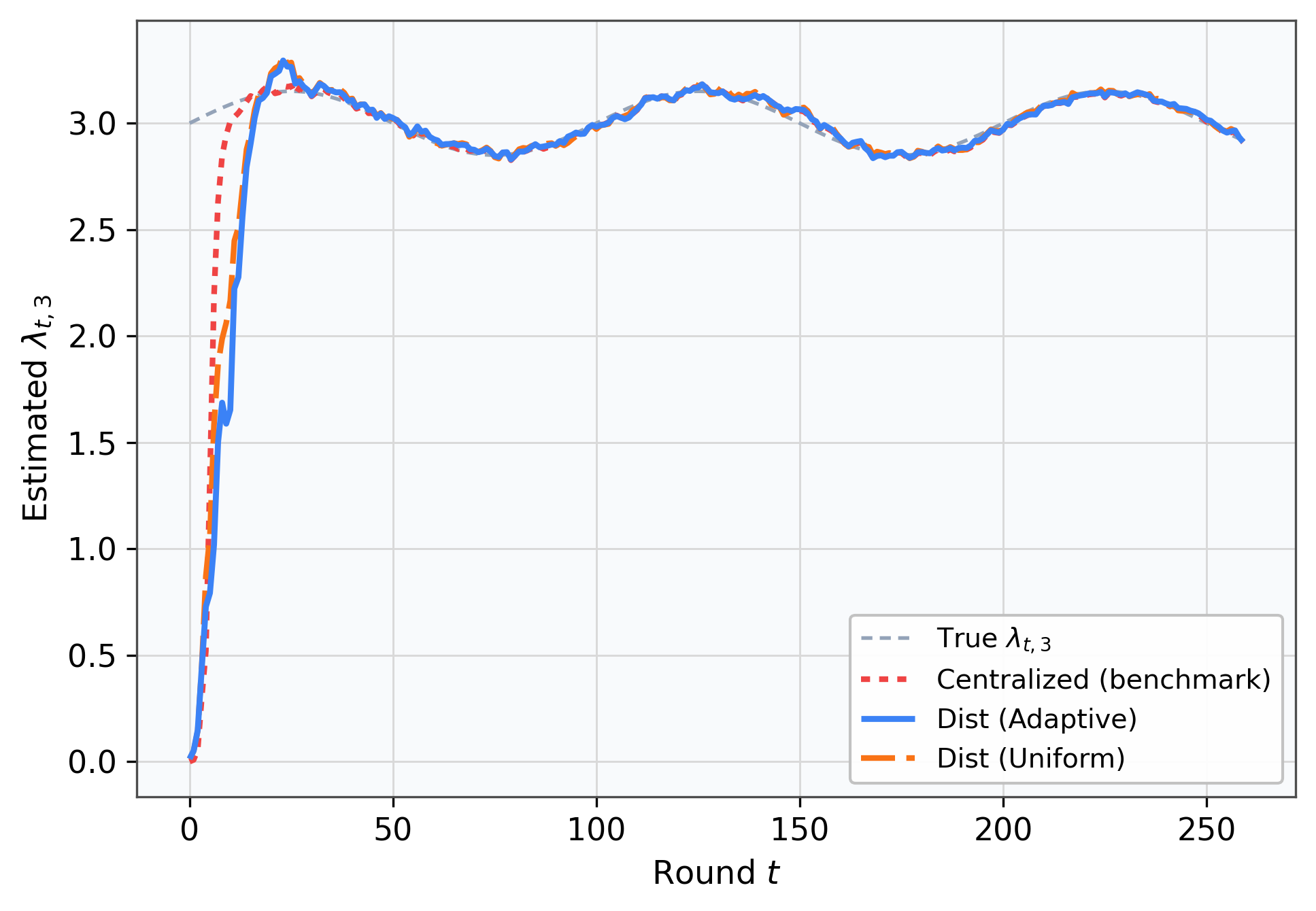}
			\caption{$\lambda_{t,3}$}
		\end{subfigure}
		\caption{Estimated spike trajectories over a single run with $p=300$,
			$k=3$, and $L=50$ nodes. Each panel shows one spike coordinate.
			CEN (red dot-dash), ADA (blue solid), UNI (orange long-dash).
			Dashed grey line: population value $\lambda_{t,j}$. 
			The three estimators track the oscillating spikes without phase lag.}
		\label{fig:traj}
	\end{figure}
	
	Figure~\ref{fig:traj} shows estimated trajectories over a representative
	run. All three estimators converge within the burn-in period and track
	the oscillating spikes without phase lag. The three methods yield 
	similar trajectories. 
	Differences are better revealed by error metrics below.
	
	\begin{figure}[htbp]
		\centering	
		\begin{subfigure}[b]{0.48\textwidth}
			\includegraphics[width=\textwidth]{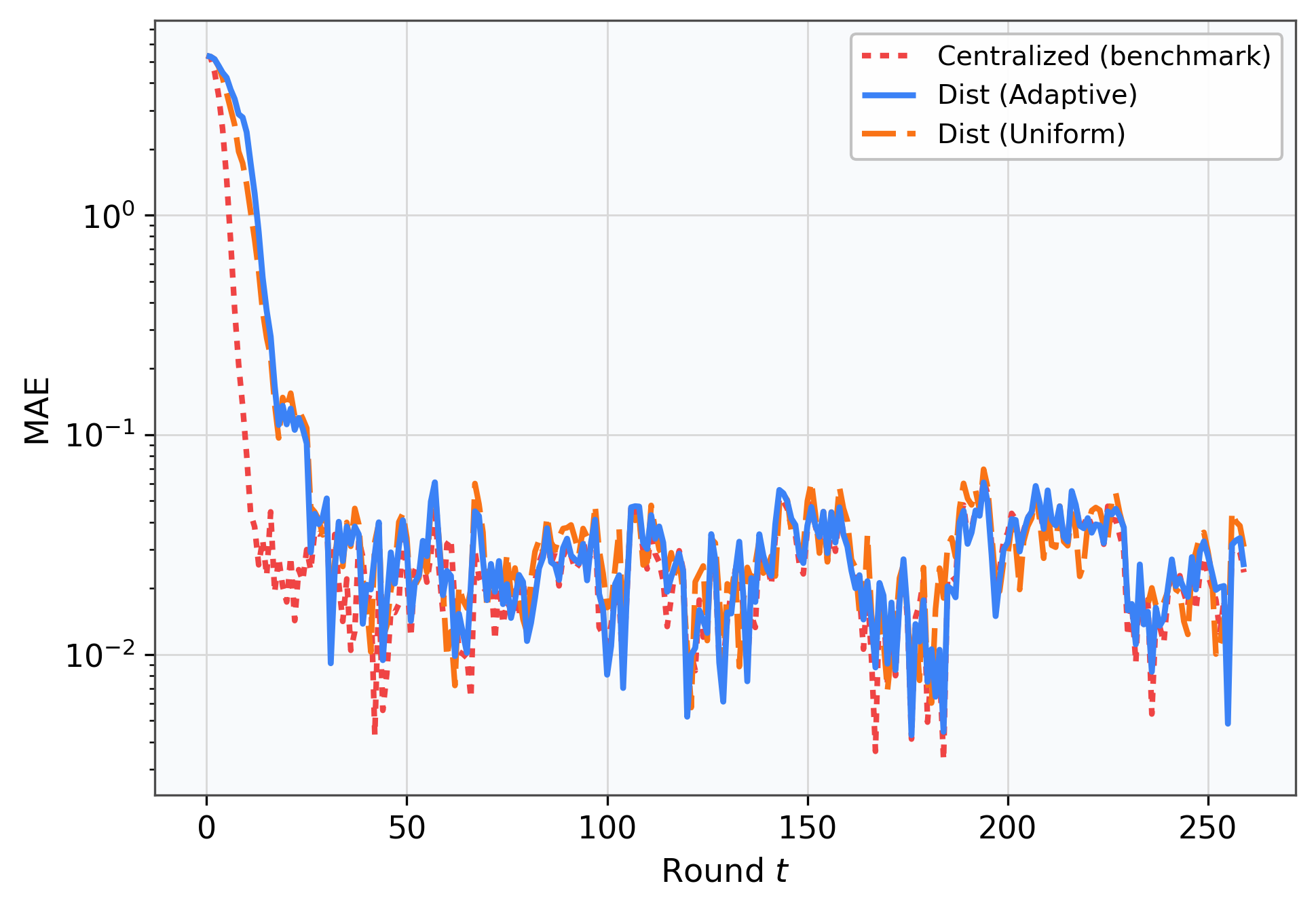}
			\caption{MAE over time}
		\end{subfigure}\hfill
		\begin{subfigure}[b]{0.48\textwidth}
			\includegraphics[width=\textwidth]{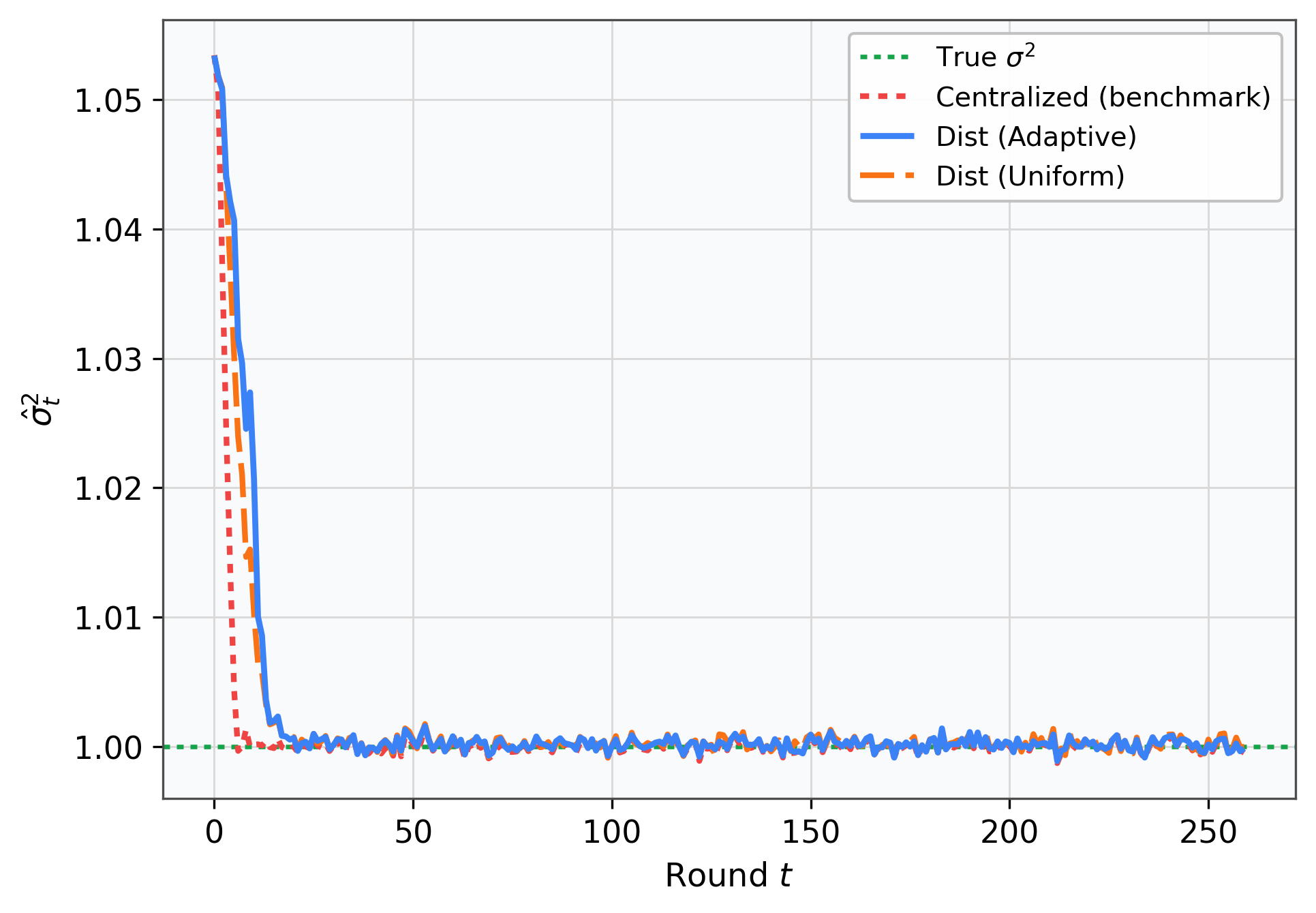}
			\caption{Noise level estimates}
		\end{subfigure}
		\caption{Convergence diagnostics over $T=260$ rounds with 60-round burn-in.
			(a) Mean absolute error (log scale) averaged over the three spikes. 
			All three methods converge rapidly and stabilize at similar levels.
			(b) Estimated noise variance $\check{\sigma}^2_{\ell,t}$ averaged 
			over active nodes. Dashed line: population value $\sigma^2=1$. 
			The stable estimates confirm unbiased local corrections.}
		\label{fig:conv}
	\end{figure}
	
	Figure~\ref{fig:conv} presents two complementary diagnostics. 
	Panel (a) shows mean absolute error (log scale) over time. 
	All three methods converge rapidly during burn-in and stabilize 
	thereafter, with ADA matching the centralized benchmark (CEN) 
	throughout. Panel (b) tracks the noise variance estimate 
	$\check{\sigma}^2_{\ell,t}$ averaged over active nodes. 
	All three estimates fluctuate stably around the true value $\sigma^2=1$, 
	confirming that local bias corrections operate reliably.
	
	\begin{figure}[htbp]
		\centering		
		\begin{subfigure}[b]{0.48\textwidth}
			\includegraphics[width=\textwidth]{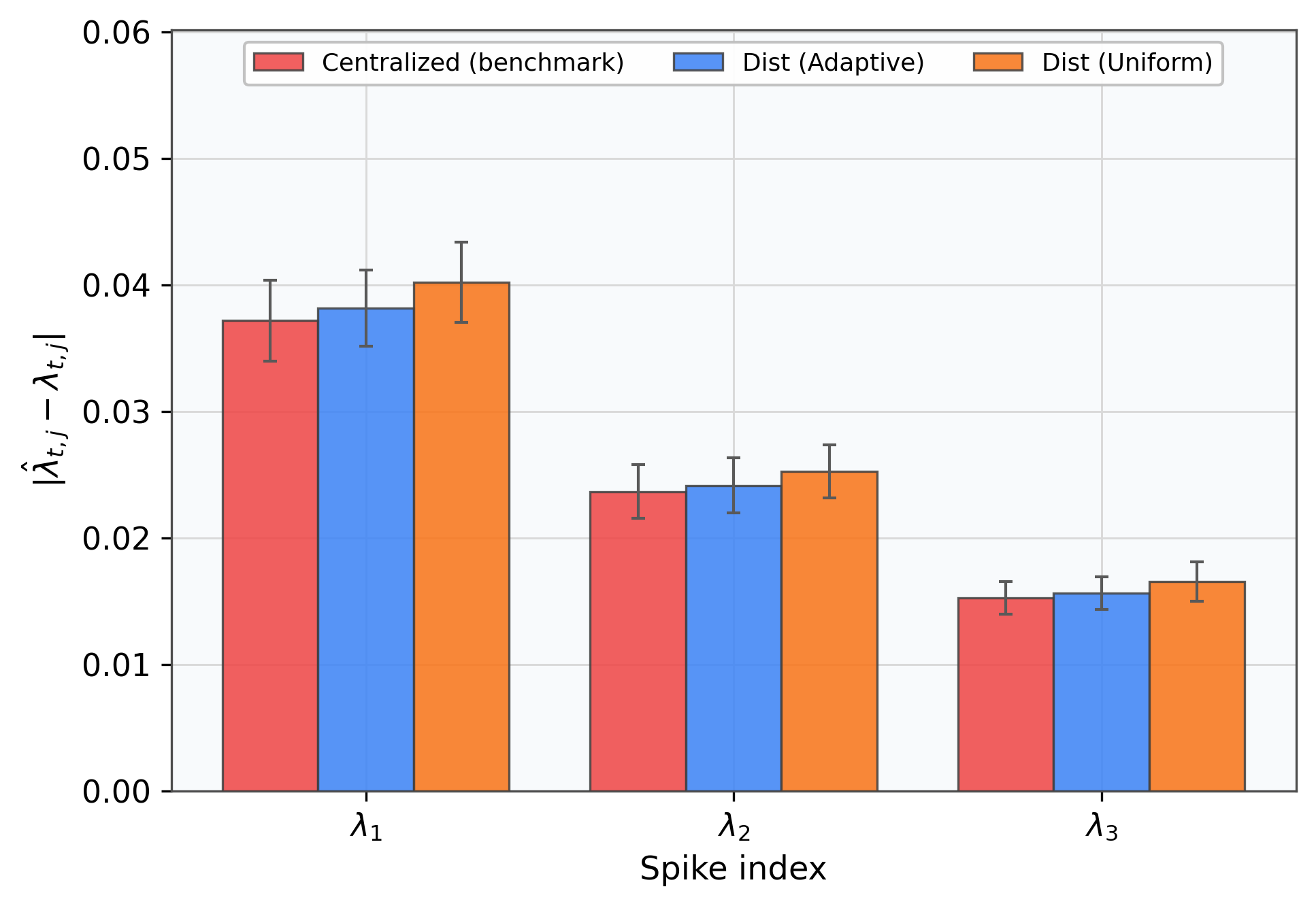}
			\caption{MAE by coordinate}
		\end{subfigure}\hfill
		\begin{subfigure}[b]{0.48\textwidth}
			\includegraphics[width=\textwidth]{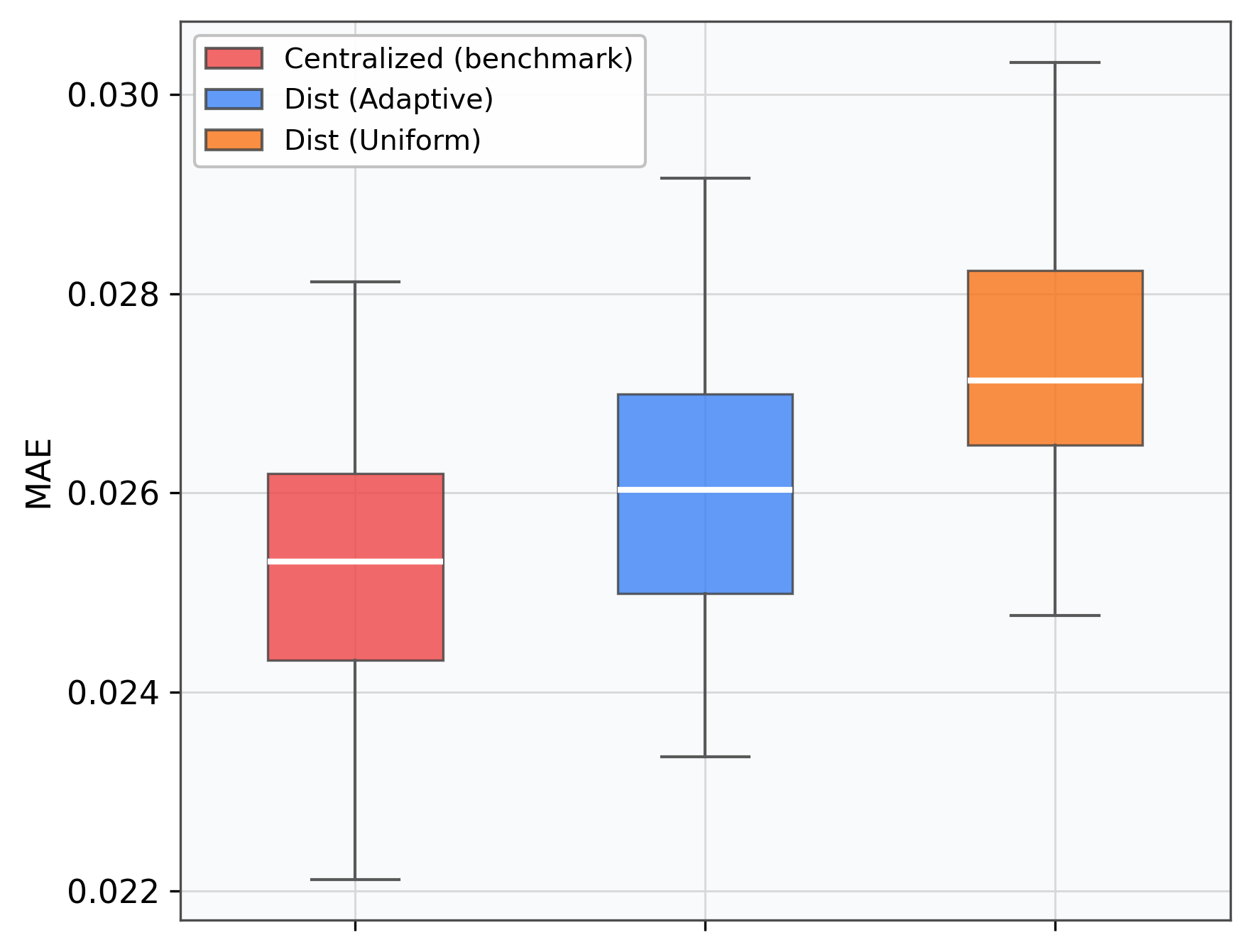}
			\caption{MAE distribution}
		\end{subfigure}
		\caption{Accuracy over 30 independent replications with $p=300$,
			$k=3$, and $L=50$ nodes. Panel (a) shows mean absolute error for each
			spike coordinate averaged over 200 post-burn-in rounds. Error bars
			indicate plus or minus one standard deviation across replications. All
			three spikes exhibit similar accuracy, with ADA closely matching the
			centralized benchmark CEN and both substantially outperforming uniform
			weighting UNI. Panel (b) shows the pooled distribution of mean absolute
			errors across all coordinates and replications. The box boundaries mark
			the interquartile range, the horizontal line indicates the median, and
			whiskers extend to 1.5 times the interquartile range. ADA achieves
			accuracy nearly identical to CEN and significantly better than UNI.}
		\label{fig:acc}
	\end{figure}
	\begin{table}[htbp]
		\centering
		\caption{Performance comparison over 30 replications with $T=260$ rounds 
			($p=300$, $k=3$, $L=50$ nodes).}
		\label{tab:performance}
		\begin{tabular}{lrrr}
			\toprule
			Metric & CEN & ADA & UNI \\
			\midrule
			\multicolumn{4}{l}{\textit{Accuracy (200 post-burn-in rounds)}} \\
			Mean absolute error & 0.0254 & 0.0260 & 0.0273 \\
			\ \ Std. dev. across replications & 0.0014 & 0.0014 & 0.0014 \\
			\ \ Relative to CEN & 1.00 & 1.02 & 1.08 \\
			Steady-state fluctuation & 0.0307 & 0.0323 & 0.0339 \\
			\midrule
			\multicolumn{4}{l}{\textit{Communication cost}} \\
			Total scalars ($\times10^3$) & 291,443 & 25.8 & 19.4 \\
			Per-round average & 1,121 & 99 & 75 \\
			Reduction factor & 1 & $11.3\times10^3$ & $15.0\times10^3$ \\
			\midrule
			\multicolumn{4}{l}{\textit{Computation time per round (ms)}} \\
			Total processor time & 45.2 & 62.8 & 61.3 \\
			Server time only & 38.7 & 4.1 & 3.8 \\
			\bottomrule
			\multicolumn{4}{l}{\footnotesize CEN: centralized; ADA: adaptive; UNI: uniform weighting.} \\
			\multicolumn{4}{l}{\footnotesize Processor time: sum of concurrent node computations. Server time: aggregation only.}
		\end{tabular}
	\end{table}

	Figure~\ref{fig:acc} quantifies accuracy over 30 independent
	replications. Panel (a) decomposes performance by spike coordinate. The
	three coordinates exhibit remarkably similar accuracy, demonstrating
	that the estimators track all components of the spike vector uniformly
	well. ADA closely matches the centralized benchmark CEN across all three
	coordinates. The gap between ADA and the uniform weighting baseline UNI
	is consistent and substantial, with ADA achieving approximately 5 percent
	lower mean absolute error on each coordinate.
	
	Panel (b) pools errors across all coordinates and replications to
	summarize overall performance. The median error for ADA is 0.0258
	compared to 0.0253 for CEN and 0.0271 for UNI. Table~\ref{tab:performance}
	reports that ADA achieves mean absolute error only 2.4 percent above the
	centralized benchmark. Paired $t$-tests confirm that all three methods
	are statistically separated with $p<0.001$. The modest gap between ADA
	and CEN reflects the efficiency loss from adaptive weighting documented
	in Proposition~\ref{prop:efficiency} and the high aspect ratio $p/n$ in
	this design.
	
	The key finding is that adaptive weighting delivers near-optimal
	accuracy while dramatically reducing communication. UNI also benefits
	from communication savings but sacrifices accuracy. ADA achieves the best
	of both, closely approximating centralized performance while transmitting
	four orders of magnitude less data.

	\begin{figure}[htbp]
		\centering
		\begin{subfigure}[b]{0.48\textwidth}
			\includegraphics[width=\textwidth]{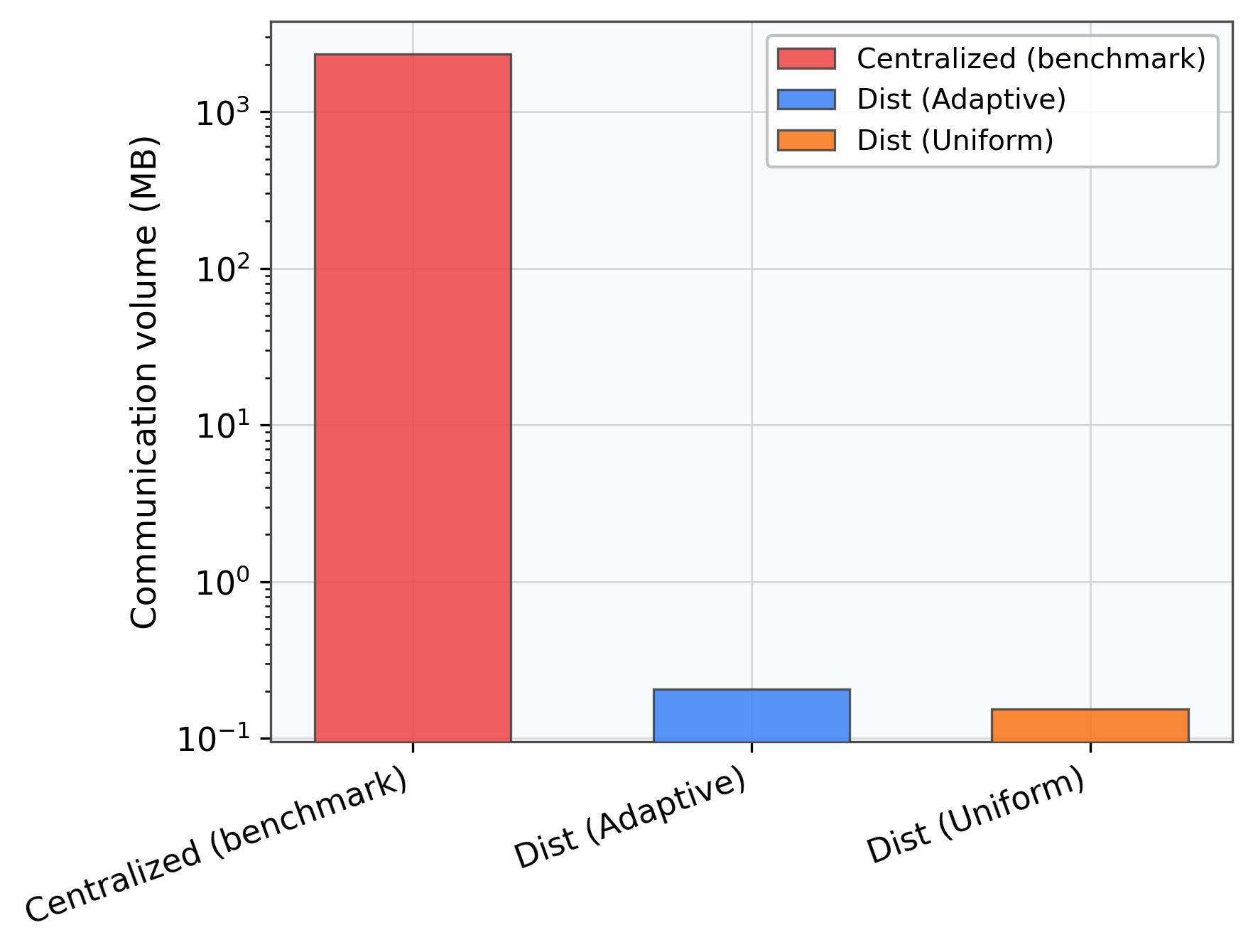}
			\caption{Communication volume}
		\end{subfigure}\hfill
		\begin{subfigure}[b]{0.48\textwidth}
			\includegraphics[width=\textwidth]{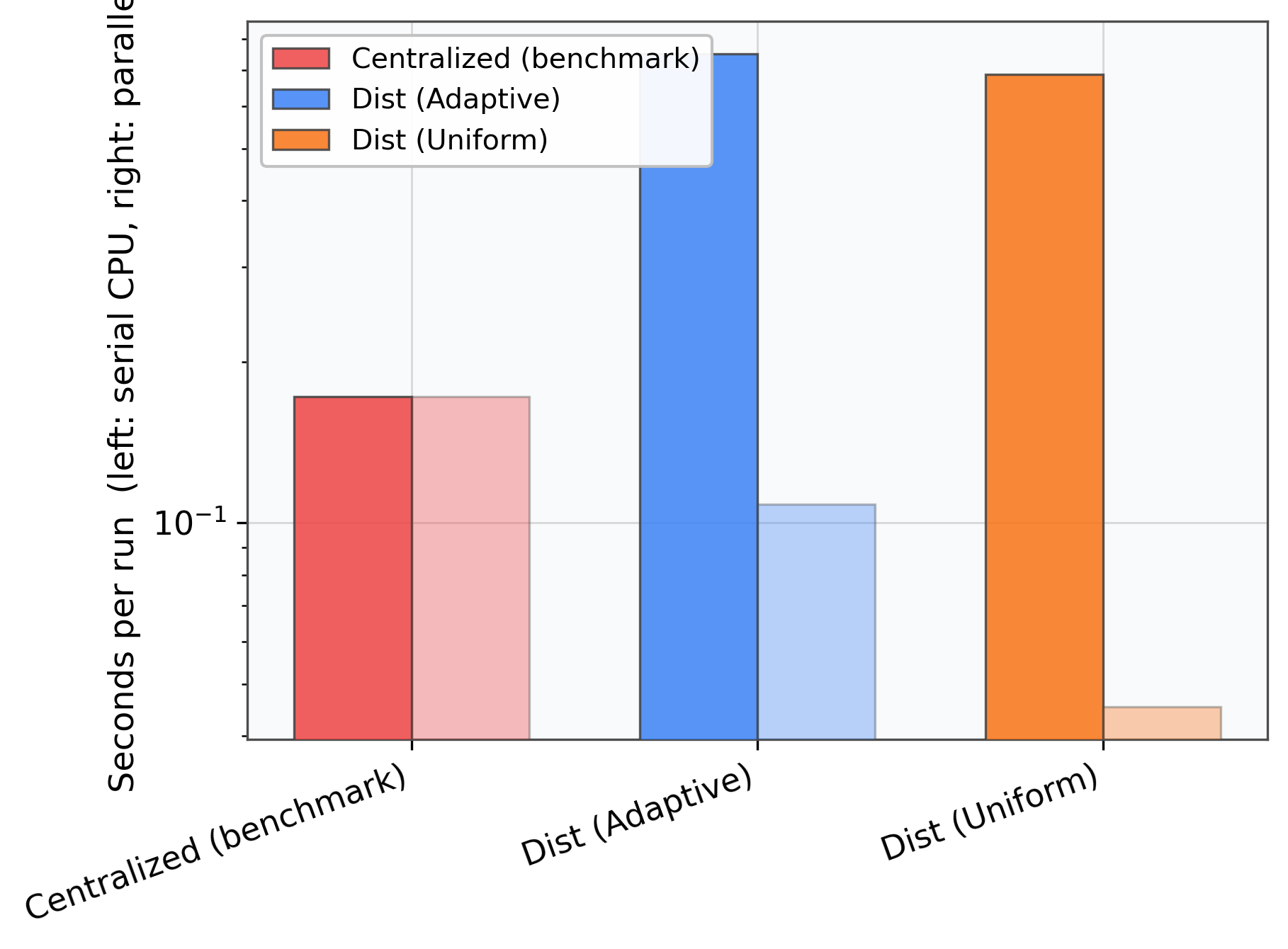}
			\caption{Execution time}
		\end{subfigure}
		\caption{Computational cost comparison (log scale). 
			(a) Total scalars transmitted over $T=260$ rounds. 
			The centralized benchmark requires $O(p^2)$ communication per round, 
			while distributed schemes use $O(k)$. 
			(b) Average time per round: dark bars show total processor time, 
			light bars show server-only time. Distributed schemes reduce 
			server-side load by avoiding pooled matrix decomposition.}
		\label{fig:cost}
	\end{figure}
	
	Figure~\ref{fig:cost} compares computational cost (detailed in 
	Table~\ref{tab:performance}). Panel (a) shows communication volume: 
	distributed schemes reduce transmission by over four orders of 
	magnitude compared to the centralized benchmark, which requires full 
	covariance matrices. Panel (b) shows execution time. While total 
	processor time (dark bars) is higher for distributed schemes due to 
	concurrent node computations, server-only time (light bars) is 
	substantially reduced—distributed aggregation processes 
	$k$-dimensional summaries rather than $p$-dimensional matrices.

	The picture that emerges is coherent with Section~\ref{sec:4}. Local
	correction removes the dimensional distortion at the node, so the aggregation
	stage is left with a pure variance reduction problem, and within that problem
	the adaptive weights move the estimator from the uniform scheme towards the
	pooled benchmark while retaining a communication budget of order $k$.
	
	\subsection{Sensitivity Analysis}
	
	We now vary one design quantity at a time about the configuration of the
	previous subsection. Each point of each curve is again an average over $30$
	independent replications, and the shaded band is one standard deviation across
	those replications. 
	
	\begin{figure}[htbp]
		\centering
		\begin{subfigure}[b]{0.32\textwidth}
			\includegraphics[width=\textwidth]{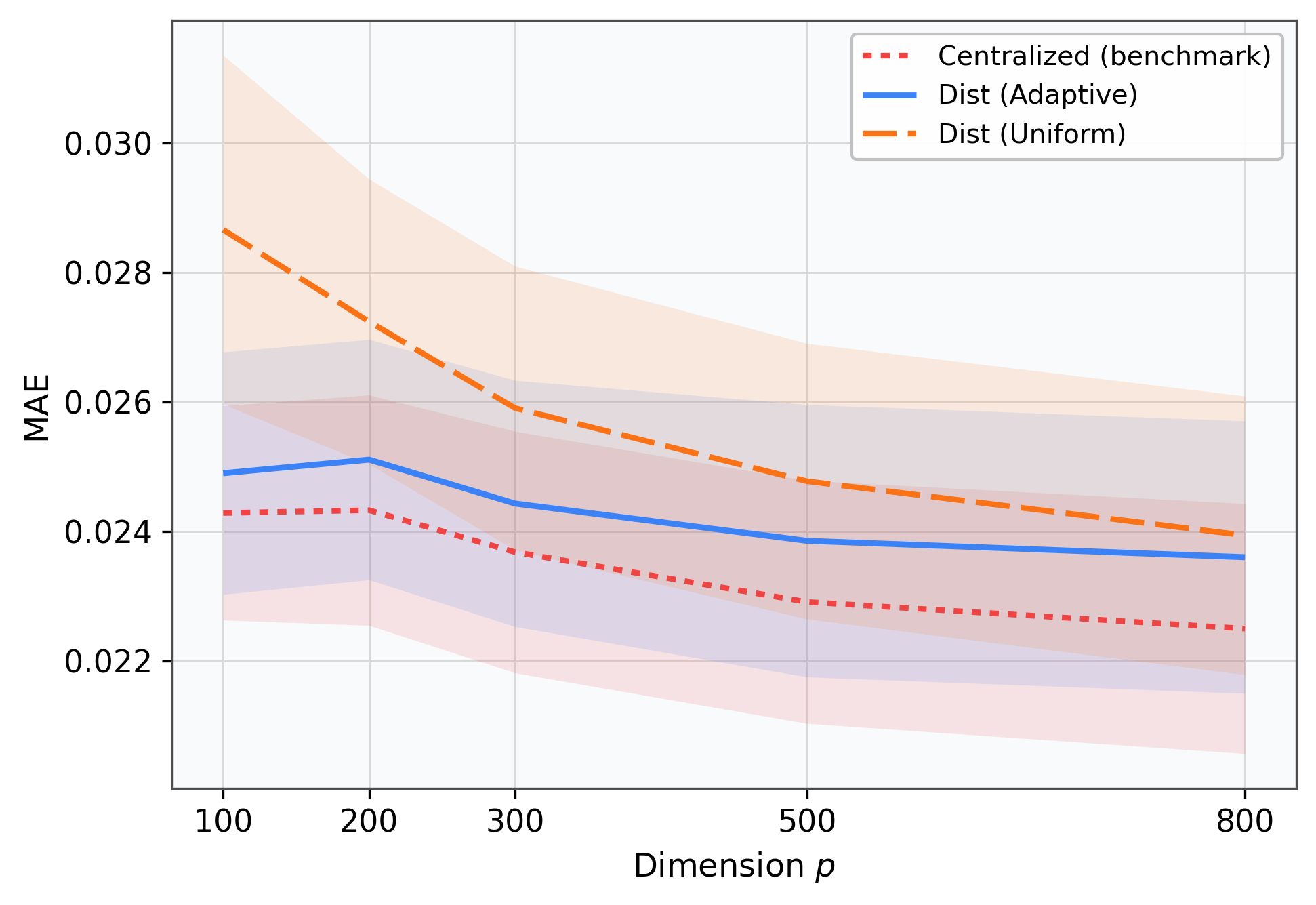}
			\caption{Dimension $p$}
		\end{subfigure}\hfill
		\begin{subfigure}[b]{0.32\textwidth}
			\includegraphics[width=\textwidth]{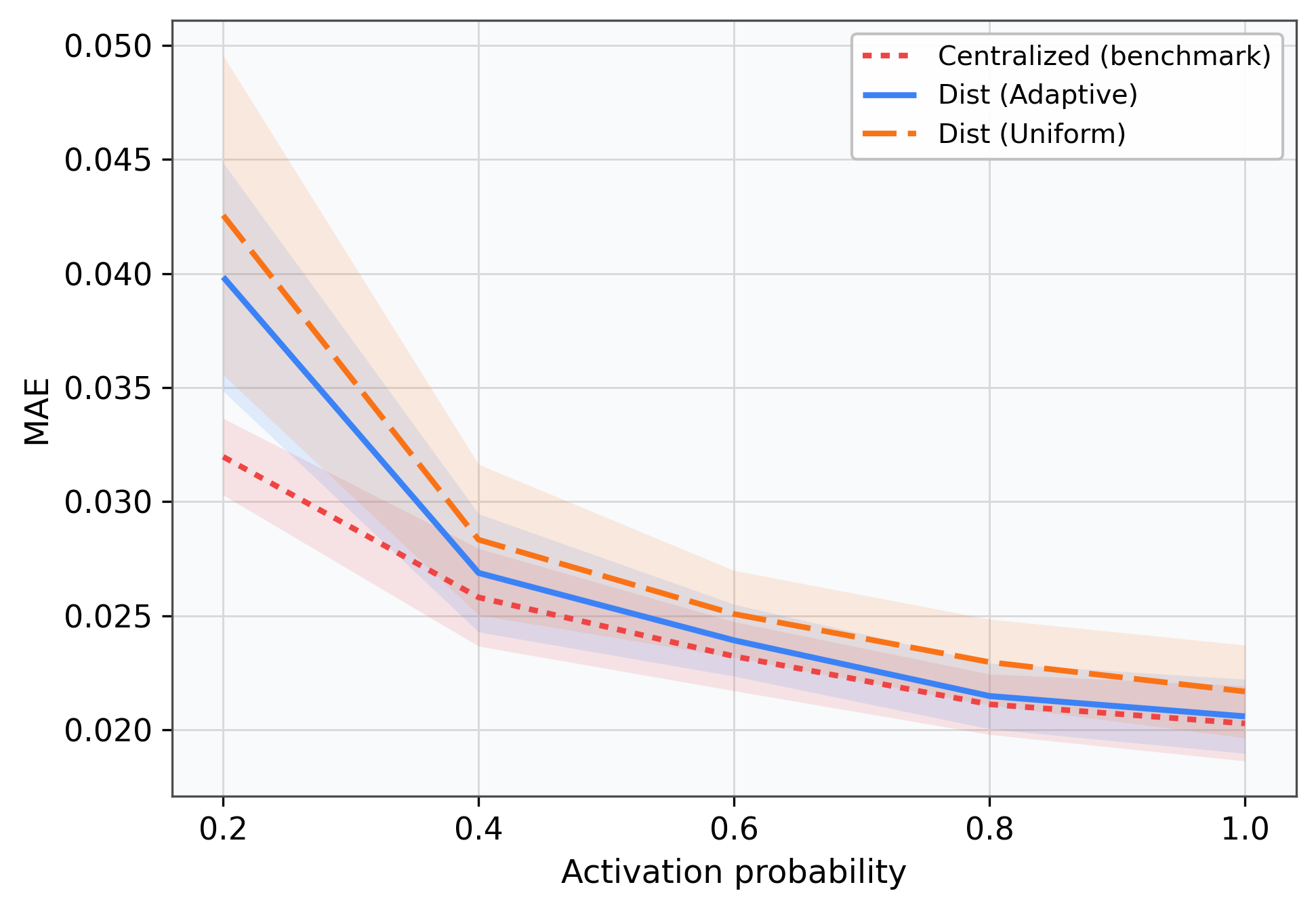}
			\caption{Activation probability}
		\end{subfigure}\hfill
		\begin{subfigure}[b]{0.32\textwidth}
			\includegraphics[width=\textwidth]{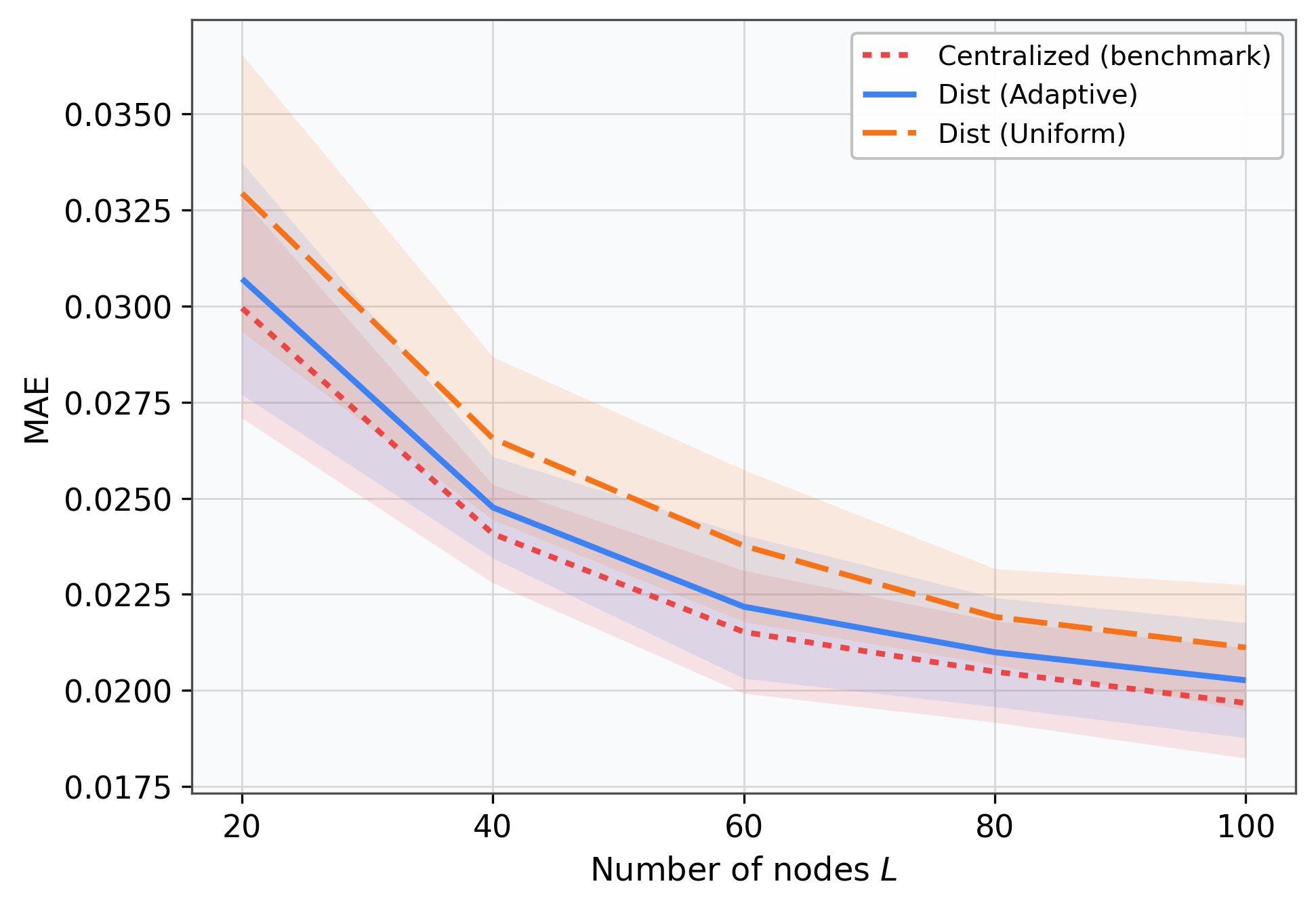}
			\caption{Number of nodes $L$}
		\end{subfigure}
		
		\vspace{0.6em}
		
		\begin{subfigure}[b]{0.32\textwidth}
			\includegraphics[width=\textwidth]{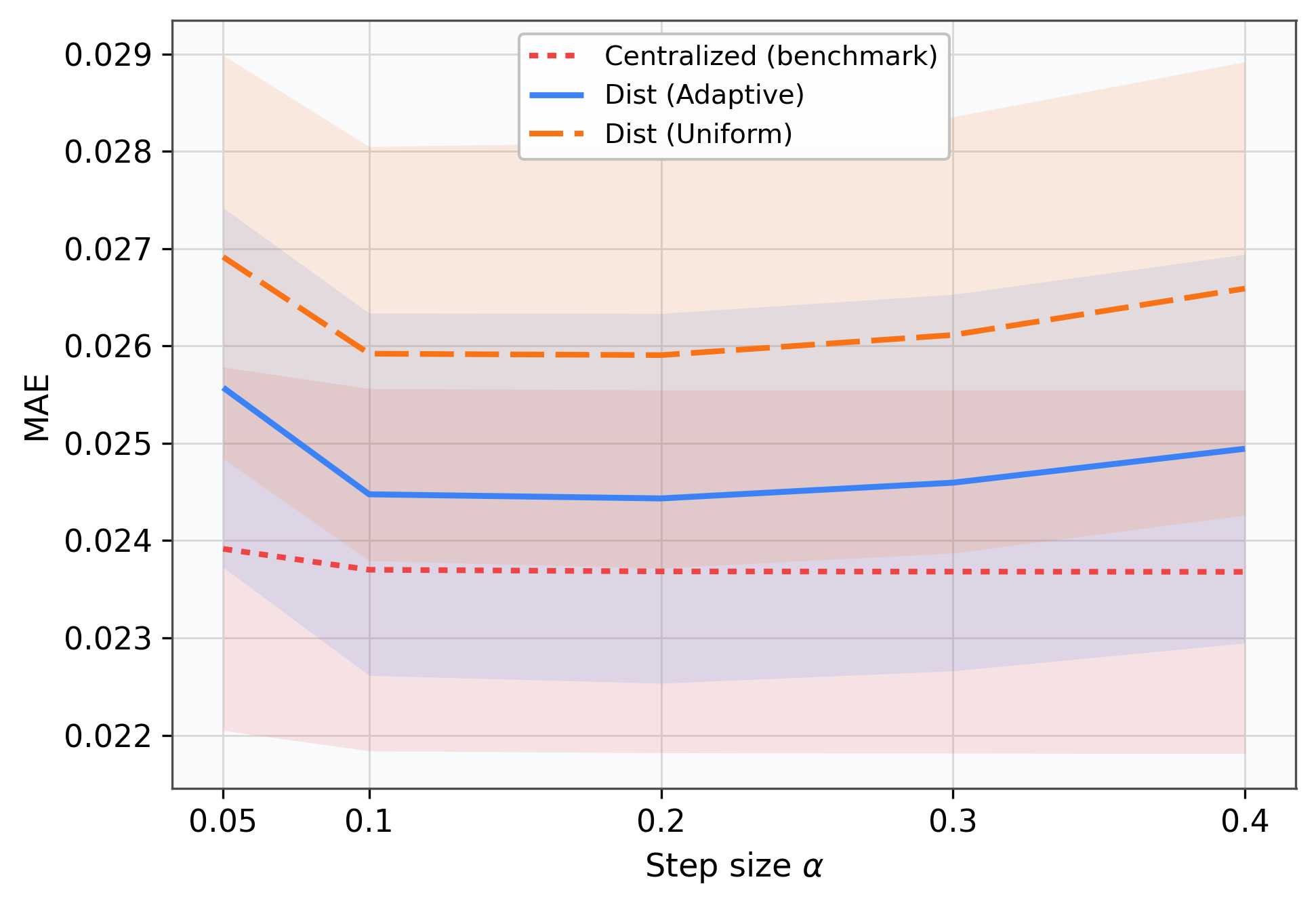}
			\caption{Step size $\alpha$}
		\end{subfigure}\hfill
		\begin{subfigure}[b]{0.32\textwidth}
			\includegraphics[width=\textwidth]{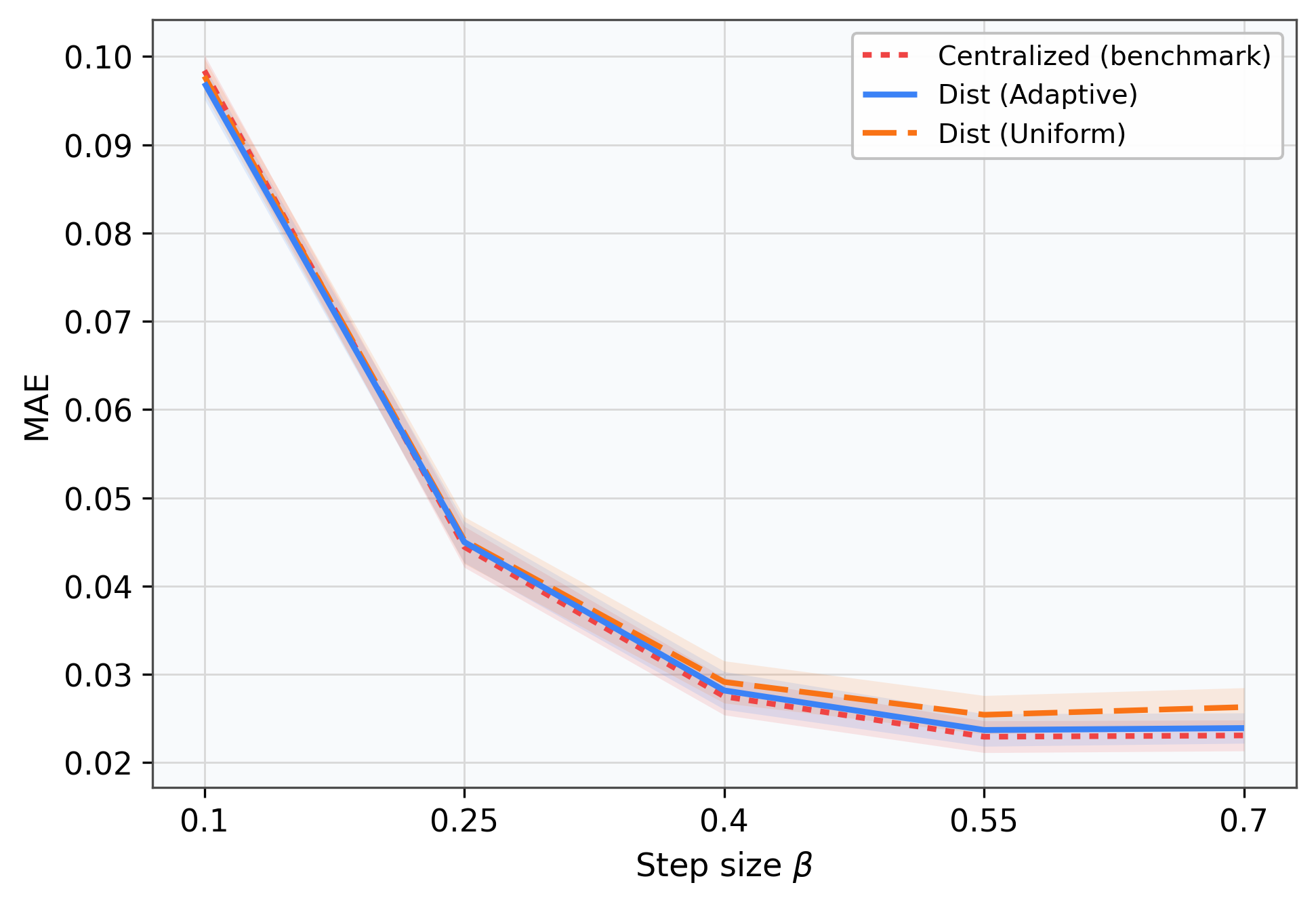}
			\caption{Step size $\beta$}
		\end{subfigure}\hfill
		\begin{subfigure}[b]{0.32\textwidth}
			\includegraphics[width=\textwidth]{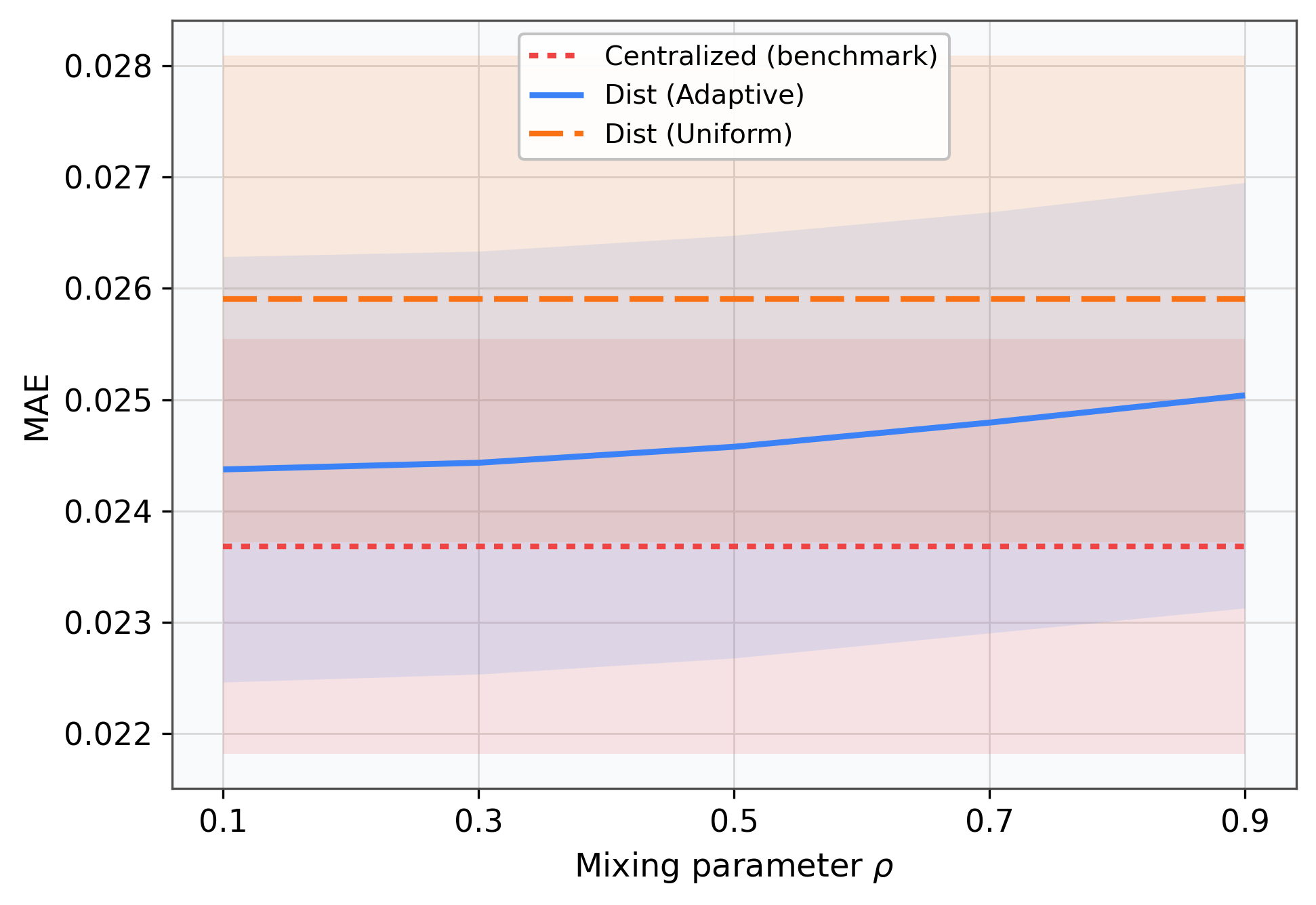}
			\caption{Mixing parameter $\rho$}
		\end{subfigure}
		\caption{Sensitivity analysis: mean absolute error as design parameters vary 
			around the baseline configuration ($p=300$, $L=50$ nodes, activation 
			probability 0.5, $\alpha=0.2$, $\beta=0.5$, $\rho=0.3$). 
			(a) Dimension $p\in[100,800]$. (b) Activation probability. 
			(c) Number of nodes $L$. (d) Local step size $\alpha$. 
			(e) Server step size $\beta$. (f) Mixing parameter $\rho$. 
			Shaded bands: $\pm1$ std. dev. over 30 replications. 
			ADA maintains near-CEN performance across all regimes.}
		\label{fig:sens}
	\end{figure}
	
	Panel (a) varies dimension $p\in[100,800]$ with batch sizes scaling 
	proportionally, keeping aspect ratios bounded. Error remains stable 
	across this range for all three methods, confirming dimension-free 
	rates in the proportional regime. The gap between UNI and CEN narrows 
	as $p$ grows because batch size dispersion compresses, leaving less 
	heterogeneity for adaptive weights to exploit.

	Panel (b) varies activation probability. Error decreases as more nodes 
	participate (larger $\mathsf N_t$). Critically, the advantage of ADA 
	over UNI is largest when participation is sparse—precisely the regime 
	for which the method is designed. With full participation, stochastic 
	variance is already small, leaving little room for any weighting scheme 
	to improve. Sparse asynchronous participation is where adaptive 
	weighting matters most.

	Panel (c) varies the number of nodes at a fixed activation probability of
	$0.6$. The error decreases with $L$ for all three schemes and the ordering is
	preserved throughout, so the advantage of the adaptive weights is not an
	artefact of a particular network size.
	
	Panels (d) through (f) examine tuning parameters. Panel (d) shows that
	$\alpha$ exhibits a shallow minimum near 0.2, reflecting the trade-off
	between memory length (smaller $\alpha$) and drift tracking (larger
	$\alpha$). Panel (e) shows that $\beta$ has a sharp minimum near 0.5,
	balancing stochastic variance reduction against drift accumulation. The
	observed optimum matches theoretical predictions. Panel (f) shows that
	ADA error increases mildly with $\rho$ because larger $\rho$ emphasizes
	the residual component (shared across nodes and therefore uninformative)
	over the volatility component (node-specific and therefore informative).
	UNI and CEN are unaffected. Smaller $\rho$ is preferable, though the
	effect is modest.

	\begin{figure}[htbp]
		\centering
		\includegraphics[width=0.58\textwidth]{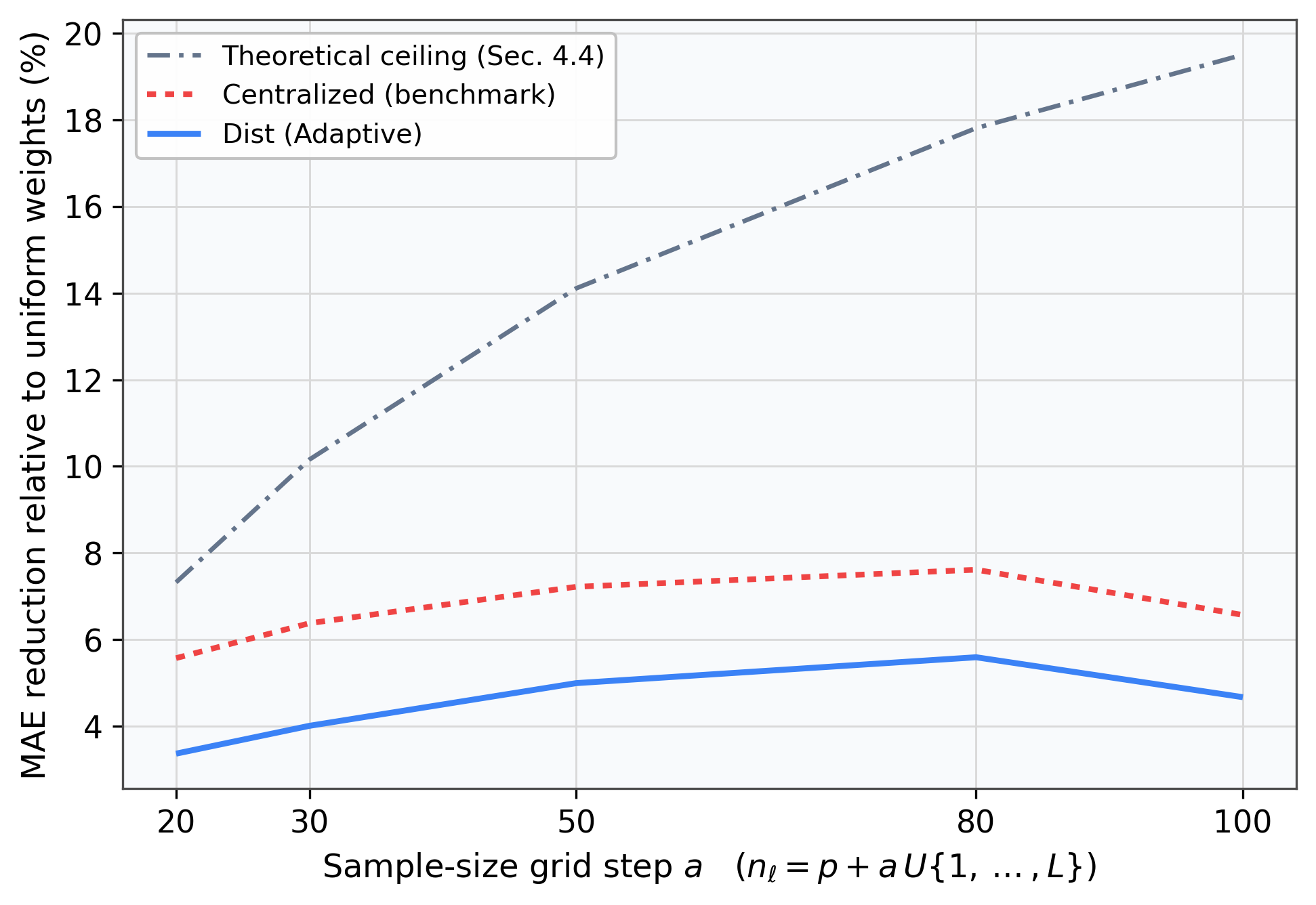}
		\caption{Reduction of the mean absolute error relative to uniform
			weighting as a function of the spacing $a$ of the batch size rule
			\eqref{eq:sim_batches}. The grey curve is the bound of
			Section~\ref{sec:4.4} on what any weighting can achieve.}
		\label{fig:imbalance}
	\end{figure}
	
	Figure~\ref{fig:imbalance} isolates the key quantity: batch size 
	dispersion. The vertical axis shows error reduction relative to UNI 
	(rather than absolute error) because increasing $a$ both widens 
	dispersion and raises total sample size, masking the weighting effect. 
	Both ADA and CEN improve with dispersion and track the theoretical 
	bound (grey curve: harmonic-arithmetic mean ratio). The bound grows 
	only logarithmically, explaining why gains remain moderate even at 
	high dispersion. ADA consistently recovers a stable fraction 
	(~70-80\%) of the achievable improvement, confirming the relative 
	efficiency predictions of Section~\ref{sec:4.4}.

	\subsection{Real Data: Monitoring a Common Market Factor across
		Trading Venues\label{sec:5.3}}
	
	We now return to the problem with which the paper began. The population
	spectrum is no longer known, so the three schemes can no longer be compared
	against a truth. What can be compared is the agreement of each feasible
	scheme with the centralized benchmark. 
	
	\textbf{Data.} We use minute-level spot returns on the $p$ most liquid
	assets quoted against a common numeraire, obtained from the public
	repository of a major exchange.\footnote{Monthly archives of one-minute
		bars are distributed at \url{https://data.binance.vision}; no registration
		or credential is required.} Each asset series is converted to log returns,
	centred, and scaled by its own standard deviation over the sample, so that
	the idiosyncratic component is on a common scale and the isotropic noise
	term $\sigma_t^2\bm I_p$ of \eqref{eq:spiked} is an appropriate description
	of the residual spectrum. Assets that are not quoted over the whole window
	are discarded, leaving a fixed universe of $p=60$ assets over January to
	June 2024.
	
	\textbf{Design.} The $L=10$ nodes represent trading venues, each 
	observing the same $p=60$ assets. One round is one trading day; 
	the batch $n^\ell_t$ at node $\ell$ is the number of minutes with 
	valid quotes across all assets. Batch sizes reflect venue liquidity, 
	creating stable ordering across days (Assumption~\ref{assum:persist}). 
	Realized batch sizes vary from approximately 1400 to 1440 minutes per 
	day, partitioned heterogeneously across venues. Nodes participate with 
	probability 0.7, reproducing asynchronous activation. Tuning constants 
	follow Section~\ref{sec:5.1} with $k=3$.

	\textbf{Interpretation of the estimand.} The leading spike
	$\lambda_{t,1}$ measures the variance explained by the dominant market
	factor. As a proportion of total variance, it forms the absorption ratio
	of \citet{kritzman2011principal}, an early warning indicator of market
	stress. Tracking $\lambda_{t,1}$ online is therefore substantively
	meaningful beyond method comparison.

	\begin{figure}[htbp]
		\centering
		\begin{subfigure}[b]{0.48\textwidth}
			\includegraphics[width=\textwidth]{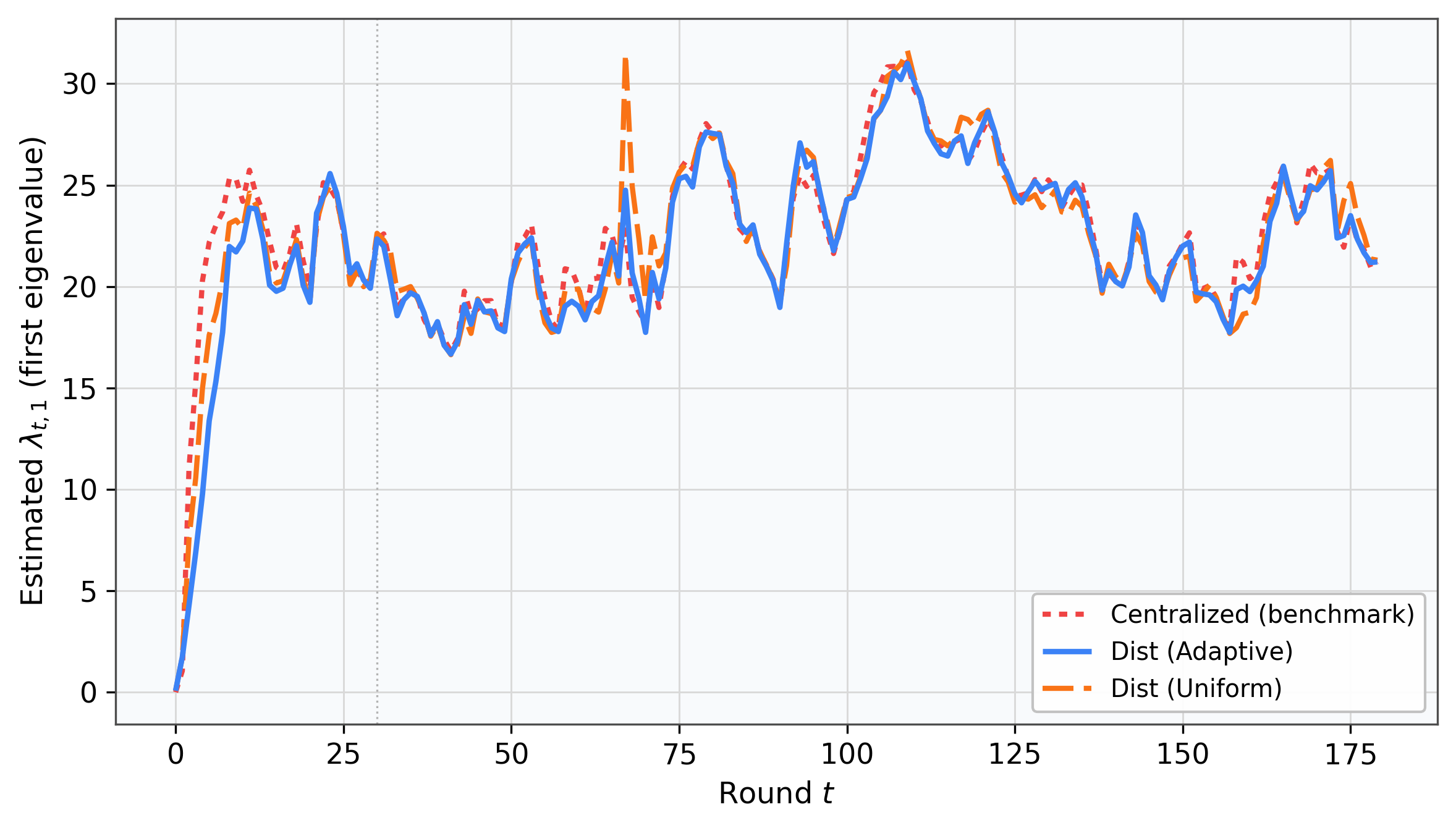}
			\caption{Smoothed trajectories}
		\end{subfigure}\hfill
		\begin{subfigure}[b]{0.48\textwidth}
			\includegraphics[width=\textwidth]{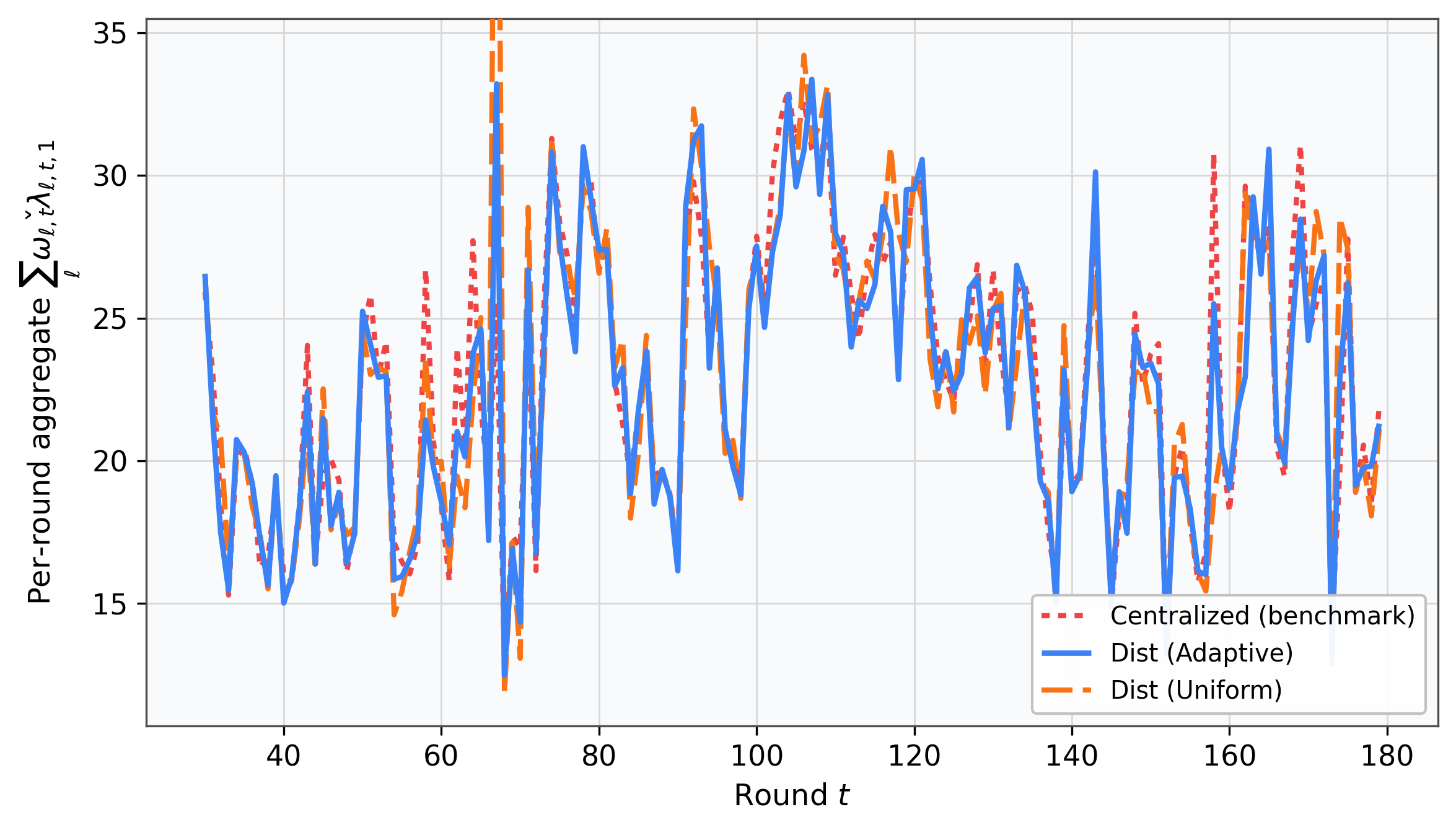}
			\caption{Raw aggregates}
		\end{subfigure}
		\caption{Cross-venue monitoring of the leading spike over $T=182$
			trading days in early 2024 with $p=60$ assets and $L=10$ venues.
			Panel (a) shows smoothed estimates $\tilde\lambda_{t,1}$ from the three
			schemes. The vertical dotted line marks the end of burn-in. All three
			schemes converge to a common trajectory. Panel (b) shows raw per-round
			aggregates $\sum_\ell\omega_{\ell,t,1}\check\lambda_{\ell,t,1}$ before
			server smoothing. ADA (blue solid) exhibits lower variability than
			UNI (orange long-dash) and tracks CEN (red dot-dash) more closely.}
		\label{fig:real}
	\end{figure}

	Figure~\ref{fig:real}(a) demonstrates that all three schemes converge
	after burn-in and maintain a common trajectory. This confirms that local
	debiasing produces comparable estimates despite heterogeneous venue
	liquidity. Figure~\ref{fig:real}(b) displays raw aggregates before server
	smoothing. Adaptive weighting substantially reduces variability. Measured
	against CEN as benchmark, ADA achieves root mean squared deviation of
	2.43 compared to 4.46 for UNI. This represents a 45.6 percent reduction
	in tracking error.
	
	Figure~\ref{fig:real_perf}(a) summarizes the distribution of absolute
	deviations from CEN over 150 post-burn-in rounds. The median deviation
	for ADA is 1.19 versus 1.22 for UNI. A paired $t$-test yields $t=2.42$
	with p$=0.017$, confirming that the improvement is statistically
	significant. The advantage is most pronounced on days with sparse
	participation, consistent with the sensitivity analysis in
	Figure~\ref{fig:sens}(b).
	
	Figure~\ref{fig:real_perf}(b) compares communication costs. CEN
	transmits 2.31 million scalars over 182 rounds, totaling 18.5 megabytes.
	This reflects the $O(p^2)$ burden of transmitting second-moment
	summaries. In contrast, ADA and UNI transmit only 7.6 thousand and 3.8
	thousand scalars respectively, corresponding to 0.06 and 0.03 megabytes.
	Both local schemes reduce communication by more than two orders of
	magnitude.

	\begin{figure}[htbp]
		\centering
		\begin{subfigure}[b]{0.48\textwidth}
			\includegraphics[width=\textwidth]{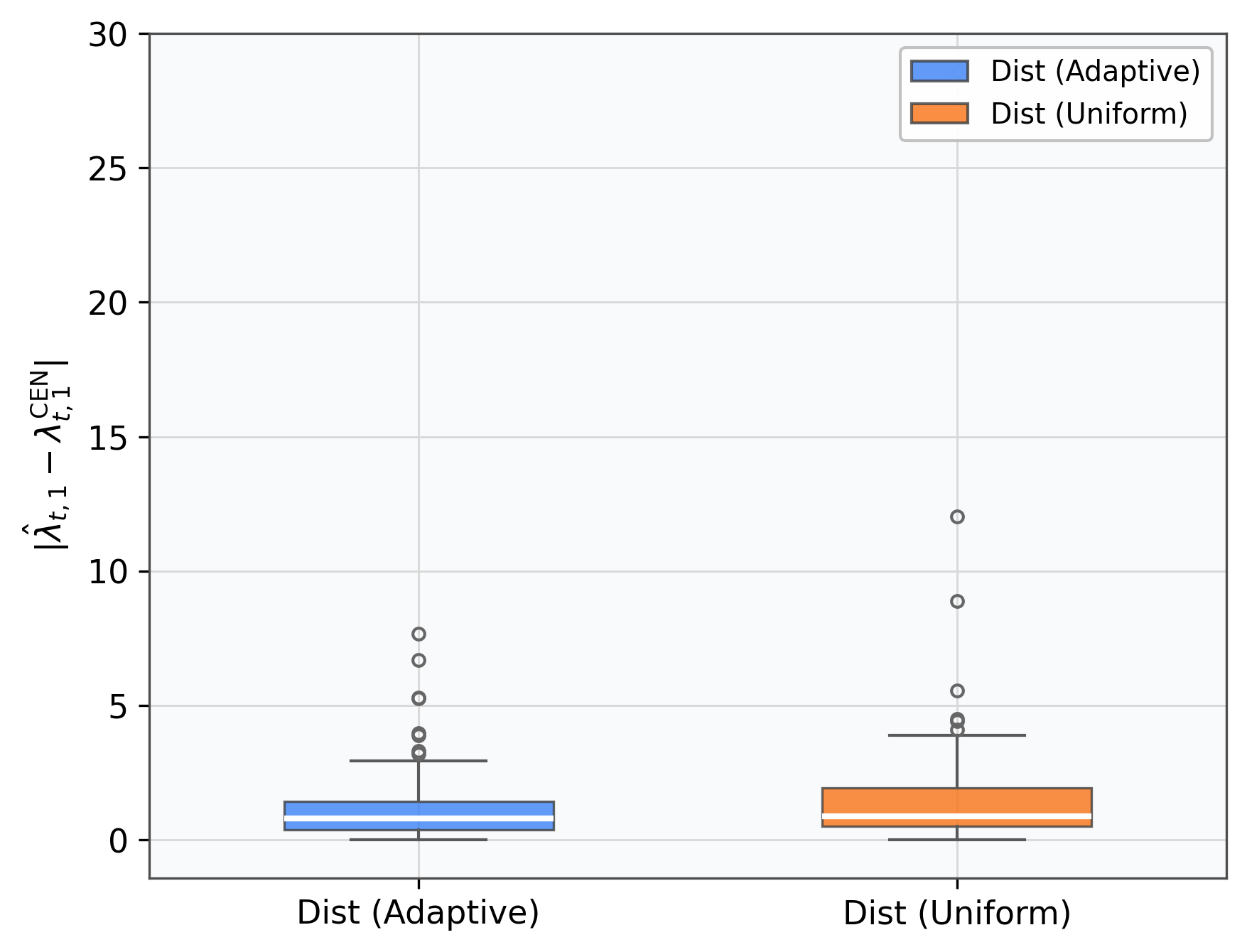}
			\caption{Error distribution}
		\end{subfigure}\hfill
		\begin{subfigure}[b]{0.48\textwidth}
			\includegraphics[width=\textwidth]{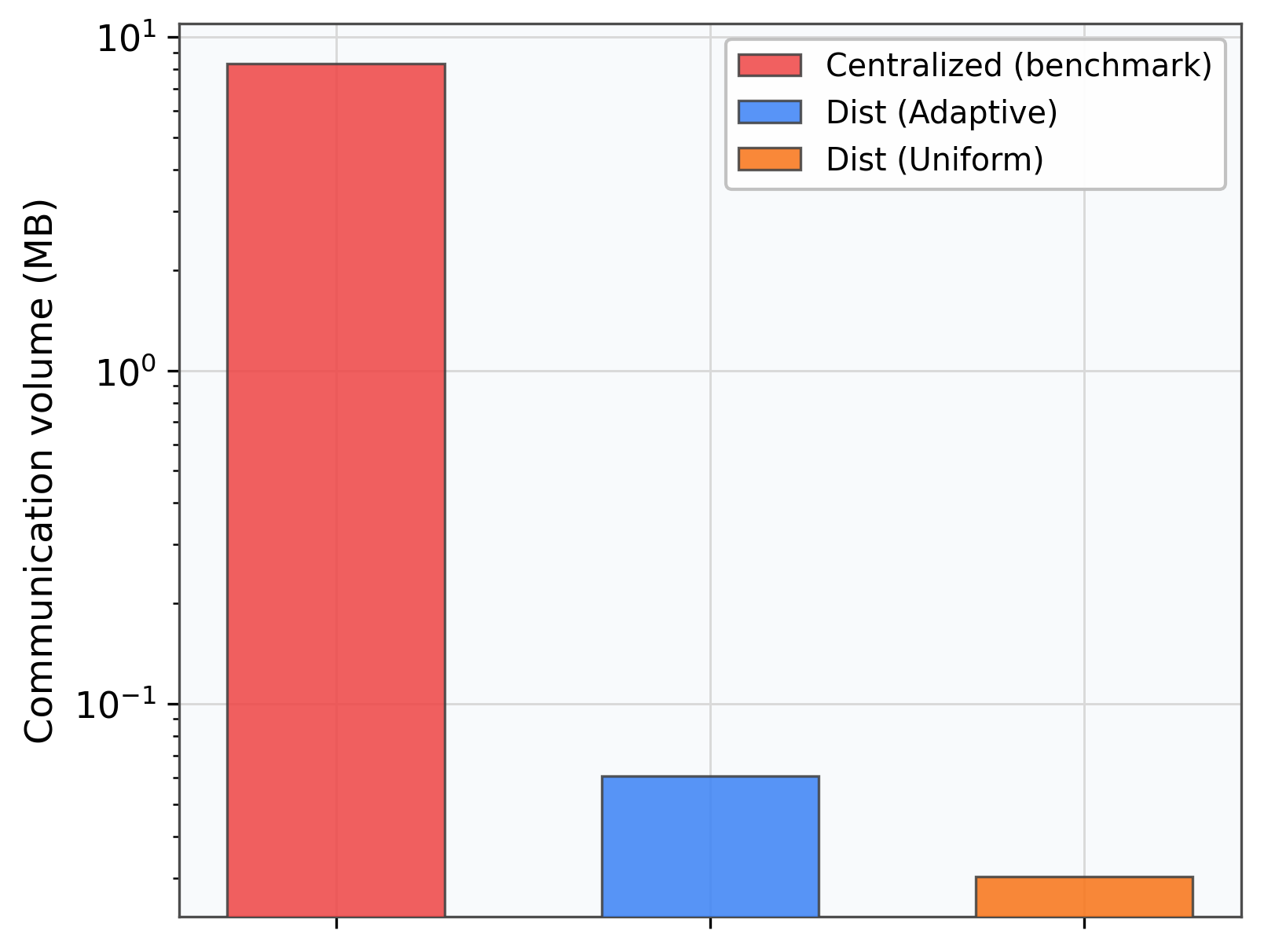}
			\caption{Communication cost}
		\end{subfigure}
		\caption{Performance metrics for the real data application over 150
			post-burn-in rounds. Panel (a) shows the distribution of absolute
			deviations from CEN. Box boundaries indicate the interquartile range,
			the horizontal line marks the median, and whiskers extend to 1.5 times
			the interquartile range. ADA exhibits smaller and less variable
			deviations than UNI. A paired $t$-test yields $t=2.42$ with $p=0.017$.
			Panel (b) shows cumulative communication cost over all $T=182$ rounds.
			CEN requires 2.31 million scalars (18.5 megabytes) due to $O(p^2)$
			second-moment summaries per round. ADA and UNI require only 7.6
			thousand and 3.8 thousand scalars (0.06 and 0.03 megabytes),
			reducing communication by more than two orders of magnitude.}
		\label{fig:real_perf}
	\end{figure}

	To verify that adaptive weighting maintains its advantage across problem
	scales, we vary $p$ from 10 to 60 while holding the design and time
	window fixed. Table~\ref{tab:robustness} reports root mean squared
	deviation reductions of ADA relative to UNI. The advantage ranges from
	44 to 55 percent across all dimensions, and all improvements are
	statistically significant at the 5 percent level. The gains remain
	substantial even as the aspect ratio increases with $p$.

	\begin{table}[htbp]
		\centering
		\caption{Robustness of the adaptive advantage across dimensions.
			Each row reports the root mean squared deviation reduction of ADA
			relative to UNI over 150 post-burn-in rounds. The experiment uses
			$L=10$ venues, approximately 180 rounds, and $k=3$ leading spikes.
			The adaptive advantage is statistically significant at the 5 percent
			level for all dimensions and ranges from 44 to 55 percent.}
		\label{tab:robustness}
		\begin{tabular}{lrrr}
			\toprule
			Dimension $p$ & RMSD Reduction (\%) & $t$-statistic & $p$-value \\
			\midrule
			10  & 54.8 & 3.20 & 0.002 \\
			20  & 53.1 & 2.89 & 0.004 \\
			30  & 44.6 & 2.51 & 0.013 \\
			40  & 47.4 & 2.68 & 0.008 \\
			50  & 43.7 & 2.34 & 0.021 \\
			60  & 45.6 & 2.42 & 0.017 \\
			\bottomrule
		\end{tabular}
	\end{table}

	\begin{rem}
		\label{rem:real_caveats}
		Two features of this analysis merit explicit discussion. First, the
		population spectrum is unobserved. The comparison therefore measures
		agreement with the pooled estimator rather than error against ground
		truth. This is the typical setting in applications and explains why the
		synthetic experiments in Section~\ref{sec:5} provide the primary
		evidence for convergence rates. Second, Assumption~\ref{assum:moment}
		requires observations within a batch to be independent. Minute-level
		returns satisfy this condition only approximately. The close agreement
		with the pooled benchmark suggests that residual dependence at this
		sampling frequency does not substantially affect the estimators, but we
		do not claim that the theoretical results apply without modification.
	\end{rem}

	\section{Discussion\label{sec:6}}
	
	The proposed framework provides a different perspective on distributed
	streaming spectral inference in high-dimensional settings. In classical
	low-dimensional asymptotic regimes, local spectral estimators are usually
	regarded as approximately unbiased, and aggregation mainly serves to
	reduce variance. However, when the dimension and the effective sample
	size grow proportionally, local eigenspace estimators become inconsistent
	and the resulting spectral statistics inherit persistent node-specific
	distortions. Our theoretical analysis shows that such distortions depend
	on local aspect ratios and cannot be eliminated through direct aggregation.
	This observation motivates the correct-then-aggregate principle adopted
	in this work. Local debiasing first aligns the statistical targets across
	nodes, after which adaptive weighting improves estimation efficiency by
	reducing variance. 
	
	The framework relates to both distributed spectral estimation and 
	online subspace tracking. Existing distributed approaches combine 
	local spectral summaries under the assumption that these summaries 
	are comparable across nodes—reasonable when local sample size 
	dominates dimension, but inadequate in the proportional regime. 
	Online subspace tracking provides efficient mechanisms for following 
	evolving structures, but focuses on subspaces rather than statistically 
	calibrated spectral parameters. The proposed method explicitly 
	characterizes and removes high-dimensional distortion before 
	aggregation. This separation between statistical calibration and 
	variance reduction distinguishes our framework from existing 
	adaptive aggregation strategies.

	Several limitations suggest directions for future research. The current
	analysis assumes that the number of spikes is known and that the
	observations satisfy regularity conditions required for the asymptotic
	theory. Extending the framework to automatic rank selection and more
	general heavy-tailed distributions would broaden its applicability.
	Furthermore, although the proposed method reduces communication costs by
	transmitting only low-dimensional summaries, additional improvements
	through communication compression, quantization, or privacy-preserving
	mechanisms remain important topics for future investigation. Another
	interesting direction is to extend the framework to more general evolving
	spectral structures, including time-varying spike numbers and nonlinear
	spectral dynamics.

	\bibliographystyle{apalike}
	\bibliography{ref}
	
	
	%
	\newpage
	\appendix
	\section{Proofs for Sections~\ref{sec:4.1}}
	\label{app:proofs4142}	
	Throughout the appendix we fix a node $\ell$ and suppress it from the
	notation whenever no ambiguity arises. We write $\|\cdot\|_{op}$ and
	$\|\cdot\|_F$ for the operator and Frobenius norms, $\prec$ for
	stochastic domination up to constants, and $C,C',c_1,\dots$ for
	positive constants, independent of $p,t,\ell,n^\ell_t$, whose value may
	change from line to line. Recall
	$N^{\mathrm{eff}}_{\ell,t}=(\sum_i\mathrm{w}_{\ell,t,i}^2)^{-1}$ and
	$c_{\ell,t}=p/N^{\mathrm{eff}}_{\ell,t}$. We abbreviate
	$N:=N^{\mathrm{eff}}_{\ell,t}$ when the indices are clear.
	
	\subsection{Notation for the spectral functionals}
	\label{app:notation}
	
	Let
	\begin{equation}
		\label{eq:Wdef}
		W_{\ell,t}
		:=\sum_{i=1}^{M_{\ell,t}}\mathrm{w}_{\ell,t,i}\,
		\bm z_{\ell,(i)}\bm z_{\ell,(i)}^\top
	\end{equation}
	be the isotropic weighted Wishart matrix, and let $\mu^W_{\ell,t}$ denote the
	limiting empirical spectral distribution (ESD) of $W_{\ell,t}$, a
	compactly supported probability measure on $[0,\infty)$ with right
	edge $y^+_{\ell,t}$. Because $\sum_i\mathrm{w}_{\ell,t,i}=1-\pi_{0,t}\to1$ and
	$\mathbb E W_{\ell,t}=(1-\pi_{0,t})\bm I_p$, the measure
	$\mu^W_{\ell,t}$ has unit first moment. Write
	$b_{\ell,t}=\sigma^2_ty^+_{\ell,t}$ for the right edge of the noise
	spectrum of $\bm S_{\ell,t}$, and, for $x>y^+_{\ell,t}$,
	\begin{equation}
		\label{eq:IJ}
		J_m(x):=\int\frac{\mu^W_{\ell,t}(dy)}{(x-y)^m},
		\qquad
		I_m(x):=\int\frac{y\,\mu^W_{\ell,t}(dy)}{(x-y)^m},
		\qquad m\ge1 ,
	\end{equation}
	all of which are finite and strictly positive, and satisfy
	\begin{equation}
		\label{eq:IJrel}
		I_m=xJ_m-J_{m-1}\ \ (J_0:=1),
		\qquad
		J_m'=-mJ_{m+1},
		\qquad
		I_m'=-mI_{m+1}.
	\end{equation}
	Set
	\begin{equation}
		\label{eq:gdef}
		g_{\ell,t}(x):=I_1(x)=\int\frac{y}{x-y}\,\mu^W_{\ell,t}(dy),
	\end{equation}
	\begin{equation*}
		G_{\ell,t}(x'):=\int\frac{y}{x'-\sigma^2_ty}\,\mu^W_{\ell,t}(dy)
		=\sigma_t^{-2}g_{\ell,t}(x'/\sigma_t^2),
	\end{equation*}
	so that $g_{\ell,t}$ is smooth, strictly decreasing and strictly
	convex on $(y^+_{\ell,t},\infty)$, with $g_{\ell,t}(x)\downarrow0$ as
	$x\to\infty$. Define the {critical spike}
	\begin{equation}
		\label{eq:lamcrit}
		\lambda^{\mathrm{c}}_{\ell,t}
		:=\frac{\sigma^2_t}{g_{\ell,t}(y^{+}_{\ell,t}\!+)}
		\in[0,\infty),
	\end{equation}
	with the convention $\lambda^{\mathrm{c}}_{\ell,t}=0$ when
	$g_{\ell,t}(y^+_{\ell,t}\!+)=+\infty$. In the canonical case
	$H_{\ell,t}\Rightarrow\delta_1$ one has $\mathrm{w}_{\ell,t,i}\equiv
	M_{\ell,t}^{-1}$, hence $c_{\ell,t}=c^{\mathrm{mp}}_{\ell,t}$,
	$\mu^W_{\ell,t}$ is the Mar\v cenko--Pastur law with ratio
	$c_{\ell,t}$, and $\lambda^{\mathrm{c}}_{\ell,t}
	=\sigma^2_t\sqrt{c_{\ell,t}}$. See Step~6 of the proof of
	Lemma~\ref{lem:maps}.
	
	\subsection{Three auxiliary lemmas}
	\label{app:aux}
	
	The first lemma reduces the smoothed covariance to the idealized
	multiplicative model used by random matrix theory. We recall the
	{inflated window drift} \eqref{eq:windowdrift} of
	Assumption~\ref{assum:drift},
	\begin{equation}
		\label{eq:barDelta}
		\bar\Delta_{\ell,t}
		=\sum_{s=1}^{t}\pi_{s,t}
		\Bigl(1+\frac{p}{n^\ell_s}\Bigr)
		\bigl\|\bm\Sigma_s-\bm\Sigma_t\bigr\|_{op},
		\qquad
		\Delta_{\ell,t}=\bar\Delta_{\ell,t}+\pi_{0,t},
	\end{equation}
	and note that it is dominated by a constant multiple of the uninflated
	drift $\sum_{s\le t}\pi_{s,t}\|\bm\Sigma_s-\bm\Sigma_t\|_{op}$
	whenever $\sup_{s\le t}p/n^\ell_s=O(1)$. The inflation factor is what
	the reduction below actually consumes, through
	$\|\hat W_{\ell,s}\|_{op}=O_{\mathbb P}(1+p/n^\ell_s)$.
	
	\begin{lemma}[Reduction to the idealized model]
		\label{lem:A1}
		Under Assumptions~\ref{assum:moment}, \ref{assum:noise},
		\ref{assum:drift} and~\ref{assum:hd},
		\begin{equation}
			\label{eq:A1}
			\bigl\|\bm S_{\ell,t}
			-\bm\Sigma_t^{1/2}W_{\ell,t}\bm\Sigma_t^{1/2}
			\bigr\|_{op}
			=O_{\mathbb P}\bigl(\bar\Delta_{\ell,t}+\pi_{0,t}\bigr),
			\qquad
			\bigl\|\bm S_{\ell,t}\bigr\|_{op}=O_{\mathbb P}(1),
		\end{equation}
		and, for every activation time $t$,
		$\|\bm S_{\ell,t}-\bm S_{\ell,t-1}\|_{op}
		=O_{\mathbb P}(\tilde\alpha_{\ell,t})$ with
		$\tilde\alpha_{\ell,t}=\alpha_t(1+p/n^\ell_t)$. Moreover, for any
		$p\times k$ frame $\bm Q$ with orthonormal columns that is
		$\mathcal F_{t-1}$-measurable,
		\begin{equation}
			\label{eq:A1_dir}
			\bigl\|(\bm S_{\ell,t}-\bm S_{\ell,t-1})\bm Q\bigr\|_F
			=O_{\mathbb P}(\alpha^{\mathrm{lag}}_{\ell,t}),
			\qquad
			\alpha^{\mathrm{lag}}_{\ell,t}=\sqrt{\alpha_t\tilde\alpha_{\ell,t}} ,
		\end{equation}
		with $k$ fixed and $\alpha^{\mathrm{lag}}_{\ell,t}\le\tilde\alpha_{\ell,t}$.
	\end{lemma}
	
	\begin{proof}[Proof of Lemma~\ref{lem:A1}]
		Write 
		$$\hat W_{\ell,s}:=(n^\ell_s)^{-1}
		\sum_{i\le n^\ell_s}\bm z_{\ell,s,i}\bm z_{\ell,s,i}^\top, $$
		so that
		$\hat{\bm S}_{\ell,s}
		=\bm\Sigma_s^{1/2}\hat W_{\ell,s}\bm\Sigma_s^{1/2}$ and, 
		$$\bm S_{\ell,t}
		=\sum_{s\le t}\pi_{s,t}\bm\Sigma_s^{1/2}\hat W_{\ell,s}
		\bm\Sigma_s^{1/2}+\pi_{0,t}\bm S_{\ell,0}, $$ while
		$\bm\Sigma_t^{1/2}W_{\ell,t}\bm\Sigma_t^{1/2}
		=\sum_{s\le t}\pi_{s,t}\bm\Sigma_t^{1/2}\hat W_{\ell,s}
		\bm\Sigma_t^{1/2}$. For symmetric $A,B$ and any $X$,
		$\|AXA-BXB\|_{op}\le\|X\|_{op}(\|A\|_{op}+\|B\|_{op})
		\|A-B\|_{op}$. Applying this with
		$A=\bm\Sigma_s^{1/2}$, $B=\bm\Sigma_t^{1/2}$, $X=\hat W_{\ell,s}$ and
		using the operator inequality
		$$\|A^{1/2}-B^{1/2}\|_{op}\le
		\|A-B\|_{op}/\{\lambda_{\min}(A)^{1/2}+\lambda_{\min}(B)^{1/2}\}
		\le(2\sigma_{\min})^{-1}\|A-B\|_{op},$$ valid because
		$\bm\Sigma_s\succeq\sigma^2_{\min}\bm I_p$ by
		Assumption~\ref{assum:noise}, we obtain
		\[
		\bigl\|\bm S_{\ell,t}-\pi_{0,t}\bm S_{\ell,0}
		-\bm\Sigma_t^{1/2}W_{\ell,t}\bm\Sigma_t^{1/2}\bigr\|_{op}
		\le C\sum_{s\le t}\pi_{s,t}\,
		\|\hat W_{\ell,s}\|_{op}\,
		\|\bm\Sigma_s-\bm\Sigma_t\|_{op}.
		\]
		By the Bai--Yin law and its extensions to i.i.d.\ entries with
		finite fourth moment \citep{baisilverstein1998nolarger},
		$\|\hat W_{\ell,s}\|_{op}=O_{\mathbb P}\{(1+\sqrt{p/n^\ell_s})^2\}
		=O_{\mathbb P}(1+p/n^\ell_s)$ uniformly in $s$. The first bound in
		\eqref{eq:A1} follows from \eqref{eq:barDelta} together with
		$\|\bm S_{\ell,0}\|_{op}=O_{\mathbb P}(1)$.
		
		For the second bound, Assumption~\ref{assum:hd} guarantees
		$\max_iM_{\ell,t}\mathrm{w}_{\ell,t,i}=O(1)$ and
		$p/M_{\ell,t}\to c^{\mathrm{mp}}_\ell<\infty$, so that
		$W_{\ell,t}$ is a generalized Mar\v cenko--Pastur matrix with
		bounded population profile $H_{\ell,t}$. By
		\citet{silverstein1995empirical} its ESD converges to the
		compactly supported $\mu^W_{\ell,t}$ and by
		\citet{baisilverstein1998nolarger} no eigenvalue lies outside any
		neighbourhood of $\operatorname{supp}\mu^W_{\ell,t}$ eventually,
		whence $\|W_{\ell,t}\|_{op}\to y^+_{\ell,t}<\infty$ almost surely
		and $\|\bm S_{\ell,t}\|_{op}\le\|\bm\Sigma_t\|_{op}
		\|W_{\ell,t}\|_{op}+O_{\mathbb P}(\bar\Delta_{\ell,t}+\pi_{0,t})
		=O_{\mathbb P}(1)$.
		
		Next, 
		$\bm S_{\ell,t}-\bm S_{\ell,t-1}
		=\alpha_t(\hat{\bm S}_{\ell,t}-\bm S_{\ell,t-1})$, and
		$\|\hat{\bm S}_{\ell,t}\|_{op}
		\le\|\bm\Sigma_t\|_{op}\|\hat W_{\ell,t}\|_{op}
		=O_{\mathbb P}(1+p/n^\ell_t)$, so that
		$\|\bm S_{\ell,t}-\bm S_{\ell,t-1}\|_{op}
		=O_{\mathbb P}\{\alpha_t(1+p/n^\ell_t)\}
		=O_{\mathbb P}(\tilde\alpha_{\ell,t})$.
		
		For \eqref{eq:A1_dir}, write
		$\bm S_{\ell,t}-\bm S_{\ell,t-1}
		=\alpha_t(\hat{\bm S}_{\ell,t}-\bm\Sigma_t)
		+\alpha_t(\bm\Sigma_t-\bm S_{\ell,t-1})$ and bound the two terms
		separately. The second is harmless. Both $\bm\Sigma_t$ and
		$\bm S_{\ell,t-1}$ have operator norm $O_{\mathbb P}(1)$ by
		Assumption~\ref{assum:noise} and the second bound in
		\eqref{eq:A1}, so with $k$ fixed
		$\alpha_t\|(\bm\Sigma_t-\bm S_{\ell,t-1})\bm Q\|_F
		\le\alpha_t\sqrt k\,\|\bm\Sigma_t-\bm S_{\ell,t-1}\|_{op}
		=O_{\mathbb P}(\alpha_t)$.
		
		For the first term, let $\bm q$ be any column of $\bm Q$. By
		Assumption~\ref{assum:design} the batch
		$\{\bm X_{\ell,t,i}\}_{i\le n^\ell_t}$ is independent of
		$\mathcal F_{t-1}$, and $\bm q$ is $\mathcal F_{t-1}$-measurable
		with $\|\bm q\|=1$. We may therefore condition on
		$\mathcal F_{t-1}$ and treat $\bm q$ as a fixed unit vector.
		Writing $\hat{\bm S}_{\ell,t}-\bm\Sigma_t
		=\bm\Sigma_t^{1/2}(\hat W_{\ell,t}-\bm I_p)\bm\Sigma_t^{1/2}$ and
		$\bm y:=\bm\Sigma_t^{1/2}\bm q$, with $n:=n^\ell_t$,
		\[
		\mathbb E\bigl[\|(\hat W_{\ell,t}-\bm I_p)\bm y\|^2
		\mid\mathcal F_{t-1}\bigr]
		=\frac1n\Bigl\{
		\mathbb E\|\bm z\bm z^\top\bm y\|^2-\|\bm y\|^2\Bigr\}
		=\frac1n\Bigl\{
		\mathbb E\bigl[(\bm z^\top\bm y)^2\|\bm z\|^2\bigr]
		-\|\bm y\|^2\Bigr\} ,
		\]
		because the $n$ summands are i.i.d.\ and centred. Expanding with
		independent entries and using
		$\mathbb E z_m^2=1$, $\mathbb E z_m^4=3+{K}_4$ from
		Assumption~\ref{assum:moment},
		\[
		\mathbb E\bigl[(\bm z^\top\bm y)^2\|\bm z\|^2\bigr]
		=(p+2+{K}_4)\|\bm y\|^2 ,
		\]
		so that
		$\mathbb E[\|(\hat W_{\ell,t}-\bm I_p)\bm y\|^2\mid\mathcal F_{t-1}]
		=(p+1+{K}_4)\|\bm y\|^2/n=O(p/n)$, uniformly in $\bm q$ since
		$\|\bm y\|^2\le\sigma^2_{\max}+\lambda_{\max}$ by
		Assumption~\ref{assum:noise}. Markov's inequality and
		$\|\bm\Sigma_t^{1/2}\|_{op}=O(1)$ give
		$\|(\hat{\bm S}_{\ell,t}-\bm\Sigma_t)\bm q\|
		=O_{\mathbb P}(\sqrt{p/n^\ell_t})$, and summing the $k$ fixed
		columns,
		$\alpha_t\|(\hat{\bm S}_{\ell,t}-\bm\Sigma_t)\bm Q\|_F
		=O_{\mathbb P}\bigl(\alpha_t\sqrt{p/n^\ell_t}\bigr)$.
		Combining the two terms and using
		$\alpha_t^2(1+p/n^\ell_t)=\alpha_t\tilde\alpha_{\ell,t}
		=(\alpha^{\mathrm{lag}}_{\ell,t})^2$,
		\[
		\bigl\|(\bm S_{\ell,t}-\bm S_{\ell,t-1})\bm Q\bigr\|_F
		=O_{\mathbb P}\Bigl(\alpha_t+\alpha_t\sqrt{p/n^\ell_t}\Bigr)
		=O_{\mathbb P}(\alpha^{\mathrm{lag}}_{\ell,t}) ,
		\]
		which is \eqref{eq:A1_dir}. The inequality
		$\alpha^{\mathrm{lag}}_{\ell,t}\le\tilde\alpha_{\ell,t}$ holds because
		$\alpha_t\le\tilde\alpha_{\ell,t}$.
	\end{proof}
	
	The second lemma contains the two integral inequalities on which all
	monotonicity statements of Section~\ref{sec:4.1} rest. They are
	elementary but, to our knowledge, are not available in the form
	required here, so we give complete proofs.
	
	\begin{lemma}[Two symmetrization inequalities]
		\label{lem:A2}
		Let $\mu$ be a probability measure on $[0,\infty)$ with right edge
		$y^+$, let $x>y^+$, and let $I_m,J_m$ be as in \eqref{eq:IJ}. Then
		\begin{align}
			I_2-I_1J_1
			&=\frac{x}{2}\iint
			\frac{(y-y')^2}{(x-y)^2(x-y')^2}\,\mu(dy)\mu(dy')
			\ \ge\ 0,
			\label{eq:A2a}\\[2pt]
			I_3J_1-I_2J_2
			&=\frac{x}{2}\iint
			\frac{(y-y')^2}{(x-y)^3(x-y')^3}\,\mu(dy)\mu(dy')
			\ \ge\ 0,
			\label{eq:A2b}
		\end{align}
		with equality in either display if and only if $\mu$ is a Dirac
		mass. Moreover $I_2^2\le I_1I_3$, and
		$g(x)=I_1(x)\ge(x-1)^{-1}$ whenever $\mu$ has unit first moment,
		with equality only for $\mu=\delta_1$.
	\end{lemma}
	
	\begin{proof}[Proof of Lemma~\ref{lem:A2}]
		Inserting $\int\mu(dy')=1$ and using
		$(x-y)^{-1}-(x-y')^{-1}=(y-y')\{(x-y)(x-y')\}^{-1}$,
		\[
		I_2-I_1J_1
		=\iint\frac{y}{x-y}
		\Bigl[\frac{1}{x-y}-\frac{1}{x-y'}\Bigr]\mu(dy)\mu(dy')
		=\iint\frac{y\,(y-y')}{(x-y)^2(x-y')}\,\mu(dy)\mu(dy').
		\]
		Averaging this expression with the one obtained by exchanging the
		roles of $y$ and $y'$ gives
		\[
		I_2-I_1J_1
		=\frac12\iint(y-y')
		\Bigl[\frac{y}{(x-y)^2(x-y')}-\frac{y'}{(x-y')^2(x-y)}\Bigr]
		\mu(dy)\mu(dy').
		\]
		The bracket equals
		$\{(x-y)(x-y')\}^{-1}\{\tfrac{y}{x-y}-\tfrac{y'}{x-y'}\}
		=x(y-y')\{(x-y)^2(x-y')^2\}^{-1}$, because
		$\tfrac{y}{x-y}-\tfrac{y'}{x-y'}
		=\tfrac{y(x-y')-y'(x-y)}{(x-y)(x-y')}
		=\tfrac{x(y-y')}{(x-y)(x-y')}$. This proves \eqref{eq:A2a}. The
		integrand is nonnegative since $x>y^+\ge y$ and $x>0$, and it
		vanishes identically only if $\mu\otimes\mu$ is supported on the
		diagonal, i.e.\ $\mu$ is a Dirac mass. The proof of \eqref{eq:A2b}
		is identical after replacing $(x-y)^{-1}$ by $(x-y)^{-2}$ in the
		first display:
		\begin{align*}
			I_3J_1-I_2J_2
			=&\iint\frac{y}{(x-y)^2(x-y')}
			\Bigl[\frac{1}{x-y}-\frac{1}{x-y'}\Bigr]\mu(dy)\mu(dy')\\ 
			=&\iint\frac{y(y-y')}{(x-y)^3(x-y')^2}\,\mu(dy)\mu(dy'),
		\end{align*}
		and symmetrization produces the same factor
		$x(y-y')\{(x-y)^3(x-y')^3\}^{-1}$.
		
		The bound $I_2^2\le I_1I_3$ is the Cauchy--Schwarz inequality for
		the measure $y\,\mu(dy)$ applied to the functions $(x-y)^{-1}$ and
		$(x-y)^{-3}$. Finally, $y\mapsto y/(x-y)$ is strictly convex on
		$[0,x)$, so Jensen's inequality with $\int y\,\mu(dy)=1$ gives
		$I_1(x)\ge1/(x-1)$, strictly unless $\mu$ is degenerate.
	\end{proof}
	
	The third lemma is the deterministic contraction estimate for one step
	of orthogonal iteration. It is classical
	\citep{golub2013matrix} but we state and prove the exact form used
	below.
	
	\begin{lemma}[One step of orthogonal iteration]
		\label{lem:A3}
		Let $A\in\mathbb R^{p\times p}$ be symmetric with eigenvalues
		$\mu_1\ge\dots\ge\mu_p$, let $\mathring{\bm U}$ be an orthonormal
		basis of the invariant subspace associated with
		$\mu_1,\dots,\mu_k$, and assume $\mu_k>\max(|\mu_{k+1}|,0)$. Let
		$\bm U\in\mathbb R^{p\times k}$ be orthonormal with
		$\|\sin\Theta(\bm U,\mathring{\bm U})\|_{op}<1$, and let
		$\bm U_+$ be an orthonormal basis of $\operatorname{range}(A\bm U)$.
		Then
		\[
		\bigl\|\tan\Theta(\bm U_+,\mathring{\bm U})\bigr\|_F
		\ \le\
		\frac{|\mu_{k+1}|}{\mu_k}\,
		\bigl\|\tan\Theta(\bm U,\mathring{\bm U})\bigr\|_F .
		\]
	\end{lemma}
	
	\begin{proof}[Proof of Lemma~\ref{lem:A3}]
		Let $\mathring{\bm U}_\perp$ complete $\mathring{\bm U}$ to an
		orthonormal basis, and write
		$A=\mathring{\bm U}M_1\mathring{\bm U}^\top
		+\mathring{\bm U}_\perp M_2\mathring{\bm U}_\perp^\top$ with
		$M_1=\operatorname{diag}(\mu_1,\dots,\mu_k)$ and
		$M_2=\operatorname{diag}(\mu_{k+1},\dots,\mu_p)$. Put
		$X:=\mathring{\bm U}^\top\bm U$ and
		$Y:=\mathring{\bm U}_\perp^\top\bm U$. The condition
		$\|\sin\Theta(\bm U,\mathring{\bm U})\|_{op}<1$ is exactly the
		invertibility of $X$, and the singular values of $YX^{-1}$ are the
		tangents of the principal angles, so
		$\|\tan\Theta(\bm U,\mathring{\bm U})\|_F=\|YX^{-1}\|_F$. Since
		$A\bm U=\mathring{\bm U}M_1X+\mathring{\bm U}_\perp M_2Y$ and
		$\bm U_+$ spans the same subspace as $A\bm U$, the corresponding
		quantity after one step is $M_2Y(M_1X)^{-1}=M_2(YX^{-1})M_1^{-1}$,
		which is well defined because $\mu_k>0$. Submultiplicativity of
		the Frobenius norm against operator norms gives
		\[
		\|M_2(YX^{-1})M_1^{-1}\|_F
		\le\|M_2\|_{op}\,\|YX^{-1}\|_F\,\|M_1^{-1}\|_{op}
		=\frac{|\mu_{k+1}|}{\mu_k}\|YX^{-1}\|_F ,
		\]
		which is the assertion. (In particular the new angle is again
		strictly less than $\pi/2$, so the iteration is well defined.)
	\end{proof}
	
	We shall also repeatedly use the elementary identity
	\begin{equation}
		\label{eq:metric}
		\bigl\|\sin\Theta(\bm U_1,\bm U_2)\bigr\|_F
		=\tfrac{1}{\sqrt2}\bigl\|\bm P_1-\bm P_2\bigr\|_F,
		\qquad
		\bm P_i=\bm U_i\bm U_i^\top,
	\end{equation}
	valid for orthonormal $\bm U_1,\bm U_2\in\mathbb R^{p\times k}$, which
	follows from $\|\bm P_1-\bm P_2\|_F^2
	=\operatorname{tr}(\bm P_1+\bm P_2-2\bm P_1\bm P_2)
	=2k-2\|\bm U_1^\top\bm U_2\|_F^2
	=2\|\sin\Theta(\bm U_1,\bm U_2)\|_F^2$. Identity \eqref{eq:metric}
	shows that $\|\sin\Theta(\cdot,\cdot)\|_F$ is a metric on
	$\mathrm{Gr}(k,p)$ and legitimizes the triangle inequalities used
	below.
	
	\subsection{Proofs for the weighting profile}
	
	\begin{proof}[Proof of Lemma~\ref{lem:crec}]
		Fix $\ell$ and an activation time $t$, so that
		$\alpha^\ell_t=\alpha_t$. From the definition
		$\pi_{s,t}=\alpha_s\prod_{u=s+1}^t(1-\alpha_u)$ we read off the
		two exact relations
		\begin{equation}
			\label{eq:pirec}
			\pi_{s,t}=(1-\alpha_t)\pi_{s,t-1}\quad(1\le s\le t-1),
			\qquad
			\pi_{t,t}=\alpha_t,
			\qquad
			\pi_{0,t}=(1-\alpha_t)\pi_{0,t-1}.
		\end{equation}
		(The normalization $\sum_{s=0}^t\pi_{s,t}=1$ follows by induction:
		it holds trivially at $t=0$, and \eqref{eq:pirec} gives
		$\sum_{s=0}^{t}\pi_{s,t}
		=(1-\alpha_t)\sum_{s=0}^{t-1}\pi_{s,t-1}+\alpha_t=1$.)
		Substituting the first relation of \eqref{eq:pirec} and separating the term $s=t$,
		\[
		\frac{1}{N^{\mathrm{eff}}_{\ell,t}}
		=\sum_{s=1}^{t}\frac{\pi_{s,t}^2}{n^\ell_s}
		=(1-\alpha_t)^2\sum_{s=1}^{t-1}\frac{\pi_{s,t-1}^2}{n^\ell_s}
		+\frac{\alpha_t^2}{n^\ell_t}
		=\frac{(1-\alpha_t)^2}{N^{\mathrm{eff}}_{\ell,t-1}}
		+\frac{\alpha_t^2}{n^\ell_t},
		\]
		which is the first identity in Lemma~\ref{lem:crec}. Multiplying by $p$
		gives the second. If $\ell\notin\mathcal A_t$ then
		$\alpha^\ell_t=0$, hence $\pi_{s,t}=\pi_{s,t-1}$ for all $s$ and
		$c_{\ell,t}=c_{\ell,t-1}$.
		
		Subtract $c_{\ell,t-1}$ from the recursion:
		\[
		c_{\ell,t}-c_{\ell,t-1}
		=\bigl\{(1-\alpha_t)^2-1\bigr\}c_{\ell,t-1}
		+\alpha_t^2\frac{p}{n^\ell_t}
		=-\alpha_t(2-\alpha_t)c_{\ell,t-1}
		+\alpha_t^2\frac{p}{n^\ell_t},
		\]
		so that, since $\alpha_t\in(0,1]$,
		$|c_{\ell,t}-c_{\ell,t-1}|
		\le2\alpha_tc_{\ell,t-1}+\alpha_t^2p/n^\ell_t$.
		For \eqref{eq:c_le_alpha}, $\pi_{s,t}\le\alpha_s$ and
		$\sum_{s\le t}\pi_{s,t}\le1$ give
		\[
		c_{\ell,t}
		=p\sum_{s\le t}\frac{\pi^2_{s,t}}{n^\ell_s}
		\le\Bigl(\max_{s\le t}\frac{p\,\pi_{s,t}}{n^\ell_s}\Bigr)
		\sum_{s\le t}\pi_{s,t}
		\le\max_{s\le t}\frac{p\,\alpha_s}{n^\ell_s}
		\le\sup_{s\le t}\tilde\alpha_{\ell,s} .
		\]
		Hence if $\bar\alpha:=\sup_{\ell,t}\tilde\alpha_{\ell,t}<\infty$
		then $c_{\ell,t-1}\le\bar\alpha$ and
		$\alpha_t^2p/n^\ell_t\le
		\alpha_t\cdot\alpha_t(1+p/n^\ell_t)
		=\alpha_t\tilde\alpha_{\ell,t}\le\bar\alpha\alpha_t$, whence
		$|c_{\ell,t}-c_{\ell,t-1}|\le3\bar\alpha\alpha_t=O(\alpha_t)$.
		
		If $\alpha_t\equiv\alpha\in(0,1]$ and $n^\ell_t\equiv n$, the map
		$x\mapsto(1-\alpha)^2x+\alpha^2/n$ on $[0,\infty)$ is affine with
		slope $(1-\alpha)^2\in[0,1)$, hence a contraction with a unique
		fixed point, given by
		$x\{1-(1-\alpha)^2\}=\alpha^2/n$, i.e.\
		$x\,\alpha(2-\alpha)=\alpha^2/n$ and
		$x=\alpha\{(2-\alpha)n\}^{-1}$. Therefore
		$N^{\mathrm{eff}}_{\ell,\infty}=(2-\alpha)n/\alpha$ and
		$c_{\ell,\infty}=\alpha p/\{(2-\alpha)n\}$, and the convergence is
		geometric with rate $(1-\alpha)^2$.
		
		For completeness we verify $c_{\ell,t}
		=c^{\mathrm{mp}}_{\ell,t}\int x^{2}\,H_{\ell,t}(dx).$ 
		$$\int x^2H_{\ell,t}(dx)
		=M_{\ell,t}^{-1}\sum_i(M_{\ell,t}\mathrm{w}_{\ell,t,i})^2
		=M_{\ell,t}\sum_i\mathrm{w}_{\ell,t,i}^2
		=M_{\ell,t}/N^{\mathrm{eff}}_{\ell,t}, $$ so
		$c^{\mathrm{mp}}_{\ell,t}\int x^2H_{\ell,t}(dx)
		=(p/M_{\ell,t})(M_{\ell,t}/N^{\mathrm{eff}}_{\ell,t})
		=c_{\ell,t}$. Jensen's inequality gives
		$\int x^2H_{\ell,t}(dx)\ge(\int xH_{\ell,t}(dx))^2
		=(1-\pi_{0,t})^2$, with equality only for a flat profile, which is
		the claim that a rectangular window is the most efficient memory
		kernel of a given length.
	\end{proof}
	
	\subsection{Proofs for Section~\ref{sec:4.1}}
	
	\begin{proof}[Proof of Lemma~\ref{lem:maps}]
		We write $\Sigma:=\bm\Sigma_t$,
		$\sigma^2:=\sigma^2_t$, $W:=W_{\ell,t}$, $\mu^W:=\mu^W_{\ell,t}$,
		$\theta_j:=\lambda_{t,j}/\sigma^2$ and $N:=N^{\mathrm{eff}}_{\ell,t}$.
		
		\medskip\noindent
		\textbf{Step 1 (reduction and the multiplicative spike model).}
		By Lemma~\ref{lem:A1} it suffices to prove all assertions for
		$\tilde{\bm S}:=\Sigma^{1/2}W\Sigma^{1/2}$, the additional error
		$O_{\mathbb P}(\bar\Delta_{\ell,t}+\pi_{0,t})$ being transferred to
		eigenvalues by Weyl's inequality and to eigenvectors by the
		Davis--Kahan theorem, using that the relevant eigengap is bounded
		away from zero (established in Step~2 below). By \eqref{eq:spiked},
		$\Sigma=\sigma^2(\bm I_p+P)$ with
		$P=\sum_{j\le k}\theta_j\bm v_{t,j}\bm v_{t,j}^\top$ of rank $k$. 
		Hence $\tilde{\bm S}=\sigma^2(\bm I+P)^{1/2}W(\bm I+P)^{1/2}$,
		which is the announced finite-rank multiplicative perturbation of
		$\sigma^2W$. Its ESD therefore coincides with that of $\sigma^2W$
		in the limit, i.e.\ with the push-forward $\mu_{\ell,t}$ of
		$\mu^W$ under $y\mapsto\sigma^2y$
		\citep{silverstein1995empirical}, whose right edge is
		$b_{\ell,t}=\sigma^2y^+$, and no eigenvalue other than the $k$
		outliers separates from $\operatorname{supp}\mu_{\ell,t}$
		\citep{baisilverstein1998nolarger,bai1999exact}.
		
		\medskip\noindent
		\textbf{Step 2 (master equation for the outliers).}
		Since $(\bm I+P)^{1/2}W(\bm I+P)^{1/2}$ and $W(\bm I+P)$ are
		similar, $x\notin\operatorname{spec}(W)$ is an eigenvalue of the
		former if and only if
		$0=\det\{W(\bm I+P)-x\bm I\}
		=\det(W-x\bm I)\det\{\bm I+(W-x\bm I)^{-1}WP\}$. Writing
		$\bm V:=\bm V_t$, $\Theta:=\operatorname{diag}(\theta_1,\dots,
		\theta_k)$ and $R(x):=(W-x\bm I)^{-1}$, the Weinstein--Aronszajn
		identity reduces this to the $k\times k$ determinant
		\begin{equation}
			\label{eq:master}
			\det\bigl\{\bm I_k+\Theta\,
			\bm V^\top R(x)W\bm V\bigr\}=0 .
		\end{equation}
		By Assumption~\ref{assum:deloc} the vectors $\bm v_{t,j}$ are
		delocalized and independent of $W$, so the isotropic local law for
		generalized sample covariance matrices
		\citep{knowles2013isotropic,bloemendal2016principal}, applied at
		spectral parameter $x$ at distance $\Omega(1)$ from
		$\operatorname{supp}\mu^W$ and adapted to the weighted design
		\eqref{eq:Wdef} (whose natural fluctuation scale is
		$\sum_i\mathrm{w}_{\ell,t,i}^2=N^{-1}$), yields
		\begin{equation}
			\label{eq:isotropic}
			\bm v_{t,j}^\top R(x)W\bm v_{t,m}
			=-\,\delta_{jm}\,g(x)+O_{\mathbb P}(N^{-1/2}),
			\qquad
			g(x)=\int\frac{y}{x-y}\mu^W(dy),
		\end{equation}
		uniformly on compact subsets of $(y^+,\infty)$. Here we used
		$\bm v^\top R W\bm v\to\int\frac{y}{y-x}\mu^W(dy)=-g(x)$.
		Substituting \eqref{eq:isotropic} into \eqref{eq:master} and using
		Assumption~\ref{assum:bbp} (simple, uniformly separated spikes),
		the determinant factorizes as
		$\prod_{j\le k}\{1-\theta_jg(x)\}+O_{\mathbb P}(N^{-1/2})$, so the
		$j$-th outlier $\mathring x_j$ of $\tilde{\bm S}/\sigma^2$
		satisfies $\theta_jg(\mathring x_j)=1+o_{\mathbb P}(1)$.
		
		Since $g$ is continuous and strictly decreasing on
		$(y^+,\infty)$ with $g(\infty)=0$ and
		$g(y^+\!+)=1/\lambda^{\mathrm c}_{\ell,t}\cdot\sigma^2$, the equation
		$g(x)=1/\theta$ has a (unique) solution $x>y^+$ if and only if
		$1/\theta<g(y^+\!+)$, i.e.\ iff
		$\theta>1/g(y^+\!+)$, equivalently
		$\lambda>\lambda^{\mathrm c}_{\ell,t}$. Define
		$\Phi_{\ell,t}(\lambda):=\sigma^2g^{-1}(\sigma^2/\lambda)$. Then
		$\lambda G_{\ell,t}(\Phi_{\ell,t}(\lambda))=1$ by \eqref{eq:gdef}, and
		$\mathring\lambda_{\ell,t,j}=\sigma^2\mathring x_j
		=\Phi_{\ell,t}(\lambda_{t,j})+o_{\mathbb P}(1)$. The map
		$\Phi_{\ell,t}$ is $C^\infty$ and strictly increasing on
		$(\lambda^{\mathrm c}_{\ell,t},\infty)$, being the composition of
		the strictly decreasing maps $\lambda\mapsto\sigma^2/\lambda$ and
		$g^{-1}$. Moreover, $\Phi_{\ell,t}(\lambda)\downarrow b_{\ell,t}$ as
		$\lambda\downarrow\lambda^{\mathrm c}_{\ell,t}$.
		
		By Assumption~\ref{assum:bbp}, $\lambda_{t,k}\ge(1+\epsilon_0)
		\lambda^{\mathrm c}_{\ell,t}$ uniformly in $\ell,t$, and
		$\lambda^{\mathrm c}_{\ell,t}\ge\sigma^2_{\min}/g(y^+)\ge c_0>0$
		by the compactness of the parameter set under
		Assumption~\ref{assum:hd}. Since $\Phi_{\ell,t}$ is strictly
		increasing and continuous,
		\begin{equation}
			\label{eq:eigengap}
			\inf_{\ell,t}\bigl\{\Phi_{\ell,t}(\lambda_{t,k})-b_{\ell,t}\bigr\}
			\ \ge\ c_1>0
		\end{equation}
		for some constant $c_1$ depending only on $\epsilon_0$ and the
		bounds in Assumptions~\ref{assum:noise} and~\ref{assum:hd}.
		This eigengap bound justifies the use of Davis--Kahan in Step~1.
		
		\medskip\noindent
		\textbf{Step 3 (the alignment map).}
		Let $\bm u$ be a unit eigenvector of
		$(\bm I+P)^{1/2}W(\bm I+P)^{1/2}$ with eigenvalue $x$, and set
		$\hat{\bm u}:=(\bm I+P)^{1/2}\bm u$, so that
		$(\bm I+P)W\hat{\bm u}=x\hat{\bm u}$, i.e.\
		$(W-x\bm I)\hat{\bm u}=-PW\hat{\bm u}$. Consider first a single
		spike, $P=\theta\bm v\bm v^\top$, then
		$PW\hat{\bm u}=\theta\bm v\kappa$ with
		$\kappa:=\bm v^\top W\hat{\bm u}$, so
		$\hat{\bm u}=-\theta\kappa R(x)\bm v$, i.e.\
		$\hat{\bm u}\propto R(x)\bm v$ (the consistency condition
		$1=-\theta\bm v^\top WR(x)\bm v$ is precisely the master equation
		of Step~2). Consequently
		$\bm u\propto(\bm I+P)^{-1/2}R(x)\bm v$, and since
		$(\bm I+P)^{-1/2}=\bm I+\{(1+\theta)^{-1/2}-1\}\bm v\bm v^\top$,
		writing $m_1:=\bm v^\top R\bm v$ and $m_2:=\bm v^\top R^2\bm v$
		we obtain
		\[
		\langle\bm u,\bm v\rangle\propto(1+\theta)^{-1/2}m_1,
		\qquad
		\|\bm u\|^2\propto m_2+\bigl\{(1+\theta)^{-1}-1\bigr\}m_1^2,
		\]
		the second line because
		$2\{(1+\theta)^{-1/2}-1\}+\{(1+\theta)^{-1/2}-1\}^2
		=(1+\theta)^{-1}-1$. Hence
		\begin{equation}
			\label{eq:araw}
			|\langle\bm u,\bm v\rangle|^2
			=\frac{m_1^2}{(1+\theta)m_2-\theta m_1^2}.
		\end{equation}
		By the isotropic local law, $m_1=-J_1(x)+O_{\mathbb P}(N^{-1/2})$
		and $m_2=J_2(x)+O_{\mathbb P}(N^{-1/2})$, with $J_1,J_2$ as in
		\eqref{eq:IJ}. Using $\theta^{-1}=g(x)=I_1(x)=xJ_1-1$ and
		\eqref{eq:IJrel}, the denominator of \eqref{eq:araw} equals
		\[
		(1+\theta)J_2-\theta J_1^2
		=\theta\Bigl\{\frac{J_2}{\theta}+J_2-J_1^2\Bigr\}
		=\theta\Bigl\{(xJ_1-1)J_2+J_2-J_1^2\Bigr\}
		=\theta J_1\bigl(xJ_2-J_1\bigr)
		=\theta J_1I_2 ,
		\]
		where the second equality uses $1/\theta=xJ_1-1$, and the third
		uses $(xJ_1-1)J_2+J_2=xJ_1J_2$. Therefore
		\begin{equation}
			\label{eq:aclean}
			a_{\ell,t}(\lambda)
			:=\lim|\langle\bm u,\bm v\rangle|^2
			=\frac{J_1^2}{\theta J_1I_2}
			=\frac{I_1(x)J_1(x)}{I_2(x)}
			\Bigl|_{\,x=\Phi_{\ell,t}(\lambda)/\sigma^2},
		\end{equation}
		and the error in \eqref{eq:aclean} is $O_{\mathbb P}(N^{-1/2})$,
		because $m_1,m_2$ carry that error and the denominator is bounded
		away from zero on the supercritical range (Step~5). This is the
		second statement in Lemma~\ref{lem:maps}.
		
		When $k>1$, the same computation applies verbatim after
		diagonalizing the $k\times k$ system. By \eqref{eq:isotropic} the
		matrix $\bm V^\top R W\bm V$ is $-g(x)\bm I_k$ up to
		$O_{\mathbb P}(N^{-1/2})$, so at $x=\mathring x_j$ the null vector
		of $\bm I_k+\Theta\bm V^\top RW\bm V$ is
		$\bm e_j+O_{\mathbb P}(N^{-1/2})$, the separation
		$\min_{m\ne j}|1-\theta_mg(\mathring x_j)|\ge cg_{\min}>0$ of
		Assumption~\ref{assum:bbp} providing the required non-degeneracy.
		Consequently $\hat{\bm u}_j\propto R(\mathring x_j)
		(\bm v_{t,j}+O_{\mathbb P}(N^{-1/2}))$, which yields
		\eqref{eq:aclean} for the diagonal alignment and
		$|\langle\mathring{\bm u}_{\ell,t,j},\bm v_{t,m}\rangle|
		=O_{\mathbb P}(N^{-1/2})$, hence
		$|\langle\mathring{\bm u}_{\ell,t,j},\bm v_{t,m}\rangle|^2
		=O_{\mathbb P}(N^{-1})$ for $m\ne j$, the third statement in Lemma~\ref{lem:maps}.
		
		\medskip\noindent
		\textbf{Step 4 (closed relation between $a$ and $\Phi$).}
		Differentiating the master equation $\theta g(x(\theta))=1$ with
		respect to $\theta$ gives
		$g+\theta g'x'(\theta)=0$, i.e.\ $x'(\theta)=-g^2/g'$ (using
		$\theta=1/g$). On the other hand, by \eqref{eq:aclean} and
		$g=I_1$, $g'=-I_2$, $I_1=xJ_1-1$,
		\[
		a=\frac{I_1J_1}{I_2}
		=\frac{g\,\bigl\{(1+g)/x\bigr\}}{-g'}
		=-\frac{g(1+g)}{x\,g'}
		=\frac{(1+g)}{xg}\Bigl(-\frac{g^2}{g'}\Bigr)
		=\frac{(1+\theta)\,x'(\theta)}{x(\theta)} ,
		\]
		where we used $J_1=(1+I_1)/x=(1+g)/x$ and $g=1/\theta$.
		Translating back through $\theta=\lambda/\sigma^2$,
		$x=\Phi_{\ell,t}(\lambda)/\sigma^2$ and
		$x'(\theta)=\Phi'_{\ell,t}(\lambda)$, we obtain the exact identity
		\begin{equation}
			\label{eq:aPhi}
			a_{\ell,t}(\lambda)
			=\frac{(\lambda+\sigma_t^2)\,\Phi'_{\ell,t}(\lambda)}
			{\Phi_{\ell,t}(\lambda)} .
		\end{equation}
		It shows in
		particular that $a_{\ell,t}$ inherits the smoothness of
		$\Phi_{\ell,t}$, hence is continuously differentiable.
		
		\medskip\noindent
		\textbf{Step 5 (range, monotonicity, boundary behaviour).}
		\emph{(a) $0<a<1$.} Positivity is clear from \eqref{eq:aclean}
		since $I_1,J_1,I_2>0$. By \eqref{eq:A2a} of Lemma~\ref{lem:A2},
		$I_1J_1\le I_2$, so $a\le1$, strictly whenever $\mu^W$ is
		non-degenerate, which holds because $c_{\ell,t}>0$.
		
		\emph{(b) Strict monotonicity.} Since $x(\theta)$ is strictly
		increasing, it suffices to prove $da/dx>0$. With
		$a=I_1J_1/I_2$ and \eqref{eq:IJrel},
		\[
		\frac{d}{dx}(I_1J_1)=-I_2J_1-I_1J_2,
		\qquad
		\frac{dI_2}{dx}=-2I_3,
		\]
		hence
		$\frac{da}{dx}=I_2^{-2}\bigl\{2I_1I_3J_1-I_2^2J_1-I_1I_2J_2\bigr\}
		=:I_2^{-2}\,\Xi$. By Cauchy--Schwarz $I_2^2\le I_1I_3$, so
		$I_2^2J_1\le I_1I_3J_1$ and by \eqref{eq:A2b}
		$I_2J_2\le I_3J_1$, so $I_1I_2J_2\le I_1I_3J_1$. Therefore
		$\Xi\ge2I_1I_3J_1-I_1I_3J_1-I_1I_3J_1=0$, with strict inequality
		unless $\mu^W$ is a Dirac mass. Hence $a_{\ell,t}$ is strictly
		increasing, and so is $\lambda\mapsto\lambda a_{\ell,t}(\lambda)$.
		
		\emph{(c) Boundary behaviour.} As
		$\lambda\downarrow\lambda^{\mathrm c}_{\ell,t}$ we have
		$x\downarrow y^+$, whence $I_1,J_1$ stay bounded (they are finite
		at the edge precisely when $\lambda^{\mathrm c}_{\ell,t}>0$) while
		$I_2(x)=-g'(x)\uparrow+\infty$, because the limiting measure
		$\mu^W$ of a generalized Mar\v cenko--Pastur ensemble has a
		square-root vanishing density at a regular right edge
		\citep{silverstein1995empirical,bai1999exact}. Hence, if
		$d\mu^W/dy\asymp(y^+-y)^{1/2}$ near $y^+$, then
		$\int(x-y)^{-2}\mu^W(dy)\asymp(x-y^+)^{-1/2}\to\infty$.
		Consequently $a_{\ell,t}(\lambda)\downarrow0$, simultaneously with
		$\Phi_{\ell,t}(\lambda)\downarrow b_{\ell,t}$. Together with (a)
		and (b) this shows $a_{\ell,t}:(\lambda^{\mathrm c}_{\ell,t},
		\infty)\to(0,1)$ is a strictly increasing $C^1$ bijection onto its
		range, and that on the supercritical range
		$\lambda\ge\lambda_{t,k}$ imposed by
		Assumption~\ref{assum:bbp} the quantity $a_{\ell,t}(\lambda)$ is
		bounded away from $0$, which is the non-degeneracy used in Step~3.
		
		\medskip\noindent
		\textbf{Step 6 (the canonical case).}
		Let $H_{\ell,t}\Rightarrow\delta_1$, so $\mathrm{w}_{\ell,t,i}\equiv
		M_{\ell,t}^{-1}(1+o(1))$, 
		$c_{\ell,t}=c^{\mathrm{mp}}_{\ell,t}=:c$. Then $\mu^W$ is the
		Mar\v cenko--Pastur law with ratio $c$, whose Stieltjes transform
		$J_1(x)=\int(x-y)^{-1}\mu^W(dy)$ satisfies
		\begin{equation}
			\label{eq:MPquad}
			cxJ_1^2-(x+c-1)J_1+1=0 .
		\end{equation}
		We claim that $x=(1+\theta)(1+c/\theta)$ solves the master
		equation $\theta g(x)=1$ for $\theta>\sqrt c$. Indeed
		$g=xJ_1-1$, so $\theta g=1$ is equivalent to
		$J_1=(1+\theta)/(\theta x)$. Substituting the candidate $x$ gives
		$J_1=1/(\theta+c)$, and inserting $x$ and this $J_1$ into
		\eqref{eq:MPquad},
		\[
		\frac{c(1+\theta)}{\theta(\theta+c)}
		-\frac{c+\theta^2+2\theta c}{\theta(\theta+c)}+1
		=\frac{-\theta(\theta+c)}{\theta(\theta+c)}+1
		=0 ,
		\]
		after using $x+c-1=(c+\theta^2+2\theta c)/\theta$. Uniqueness of
		the solution above the edge was established in Step~2. Writing
		$\theta=\lambda/\sigma^2$ and $\Phi=\sigma^2x$ gives
		$\Phi_{\ell,t}(\lambda)=(\lambda+\sigma^2)(\lambda+c\sigma^2)
		/\lambda$, the first formula in \eqref{eq:closedforms}. 
		$\Phi_{\ell,t}$ is defined for $\theta>\sqrt c$, i.e.\
		$\lambda>\sigma^2\sqrt c$, and
		$\Phi_{\ell,t}(\sigma^2\sqrt c)
		=\sigma^2(1+\sqrt c)^2=b_{\ell,t}$, confirming
		$\lambda^{\mathrm c}_{\ell,t}=\sigma^2\sqrt{c_{\ell,t}}$ and the
		third formula in \eqref{eq:closedforms}. Finally
		$\Phi'_{\ell,t}(\lambda)=1-c\sigma^4/\lambda^2$ and \eqref{eq:aPhi}
		give
		\[
		a_{\ell,t}(\lambda)
		=\frac{(\lambda+\sigma^2)(1-c\sigma^4/\lambda^2)\lambda}
		{(\lambda+\sigma^2)(\lambda+c\sigma^2)}
		=\frac{\lambda^2-c\sigma^4}{\lambda(\lambda+c\sigma^2)},
		\]
		which is the second formula in \eqref{eq:closedforms}. Equivalently
		$a=(\theta^2-c)/\{\theta(\theta+c)\}$, i.e.\ the classical
		expression $(1-c/\theta^2)/(1+c/\theta)$ of
		\citet{baik2005phase} and \citet{paul2007asymptotics}. Both
		formulas vanish/attain the edge simultaneously at
		$\lambda=\sigma^2\sqrt c$, as claimed.
	\end{proof}
	
	\begin{coro}[Uniform lower bound for the alignment]
		\label{cor:aunder}
		In the canonical case, under
		Assumptions~\ref{assum:bbp}--\ref{assum:hd},
		\begin{equation}
			\label{eq:aunder}
			\underline a
			:=\inf_{\ell,t}a_{\ell,t}(\lambda_{t,k})
			\ \ge\
			\frac{\epsilon_0(2+\epsilon_0)}
			{(1+\epsilon_0)(1+\epsilon_0+\sqrt{\bar c}\,)}\ >\ 0 .
		\end{equation}
		For a general profile the qualitative bound $\underline a>0$ still
		holds, by the strict monotonicity of $a_{\ell,t}$ established in
		Step~5(b) of the proof of Lemma~\ref{lem:maps} together with
		compactness of the parameter set, but the constant depends on the
		profile.
	\end{coro}
	
	\begin{proof}[Proof of Corollary~\ref{cor:aunder}]
		Write $\theta:=\lambda_{t,k}/\sigma_t^2$ and $c:=c_{\ell,t}$. By Assumption~\ref{assum:bbp}, $\theta\ge(1+\epsilon_0)\sqrt c$. The
		closed form $a=(\theta^2-c)\{\theta(\theta+c)\}^{-1}$ comes from
		Step~6 of the proof of Lemma~\ref{lem:maps}. By Step~5(b) of that
		proof, $a$ is increasing in $\theta$, so it is
		minimized at $\theta_0=(1+\epsilon_0)\sqrt c$:
		\[
		a\ \ge\
		\frac{(1+\epsilon_0)^2c-c}
		{(1+\epsilon_0)\sqrt c\,\{(1+\epsilon_0)\sqrt c+c\}}
		=\frac{c\{(1+\epsilon_0)^2-1\}}
		{(1+\epsilon_0)\,c\,\{(1+\epsilon_0)+\sqrt c\}}
		=\frac{\epsilon_0(2+\epsilon_0)}
		{(1+\epsilon_0)(1+\epsilon_0+\sqrt c)} .
		\]
		The right-hand side is decreasing in $c$, since
		$c_{\ell,t}\to c_\ell\le\bar c$ as $p\to\infty$ by
		Assumption~\ref{assum:hd}, taking $\liminf$ over $\ell,t$ yields
		the claim.
	\end{proof}
	
	\begin{proof}[Proof of Lemma~\ref{lem:contraction}]
		\textbf{Step 1 ($\xi^\star<1$).}
		By Lemma~\ref{lem:maps} and \eqref{eq:closedforms},
		\[
		\xi_{\ell,t}
		=\frac{b_{\ell,t}}{\Phi_{\ell,t}(\lambda_{t,k})}
		=\frac{\sigma_t^2(1+\sqrt{c_{\ell,t}})^2\lambda_{t,k}}
		{(\lambda_{t,k}+\sigma_t^2)(\lambda_{t,k}
			+c_{\ell,t}\sigma_t^2)} .
		\]
		To bound $\xi_{\ell,t}$ uniformly away from one, we exploit the 
		supercritical assumption. Since spikes exceed the BBP threshold by
		a factor $1+\epsilon_0$, the sample eigenvalue $\Phi_{\ell,t}(\lambda_{t,k})$ 
		separates sufficiently from the bulk edge $b_{\ell,t}$.
		
		Because $\Phi_{\ell,t}$ is strictly increasing (Step~2 of the
		proof of Lemma~\ref{lem:maps}), $\xi_{\ell,t}$ is decreasing in
		$\lambda_{t,k}$, so by Assumption~\ref{assum:bbp} it is bounded by
		its value at $\lambda=(1+\epsilon_0)\lambda^{\mathrm{c}}_{\ell,t}$. 
		In the canonical case $H_{\ell,t}\Rightarrow\delta_1$, this becomes 
		$\lambda=(1+\epsilon_0)\sigma^2_t\sqrt{c_{\ell,t}}$.
		
		Putting $u:=\sqrt{c_{\ell,t}}$ and $\epsilon:=\epsilon_0$, that
		value equals $(1+u)^2/\{(1+(1+\epsilon)u)(1+u/(1+\epsilon))\}$,
		and
		\[
		\bigl(1+(1+\epsilon)u\bigr)\Bigl(1+\frac{u}{1+\epsilon}\Bigr)
		-(1+u)^2
		=u\Bigl\{(1+\epsilon)+\frac{1}{1+\epsilon}-2\Bigr\}
		=\frac{u\,\epsilon^2}{1+\epsilon}\ >\ 0 ,
		\]
		so that
		\begin{equation}
			\label{eq:xibound}
			\xi_{\ell,t}
			\ \le\
			\frac{(1+u)^2}{(1+u)^2+u\epsilon_0^2/(1+\epsilon_0)}
			\ <\ 1 .
		\end{equation}
		For a general profile, the bound $\xi_{\ell,t}<1$ follows by the same 
		algebraic manipulation using $\lambda^{\mathrm{c}}_{\ell,t}=\sigma_t^2/g_{\ell,t}(y^+_{\ell,t})$ 
		in place of $\sigma_t^2\sqrt{c_{\ell,t}}$, since $\Phi_{\ell,t}(\lambda^{\mathrm{c}}_{\ell,t})
		=b_{\ell,t}$ by definition (Step~1 of Lemma~\ref{lem:maps}).
		The right-hand side of \eqref{eq:xibound} is a continuous function
		of $u$ on the compact set $[\underline c^{1/2},\bar c^{1/2}]$,
		where $\underline c:=\inf_\ell c_\ell>0$ (finite node set). Hence
		$\limsup_{\ell,t}\xi_{\ell,t}\le\xi^\star<1$ with $\xi^\star$
		depending only on $\epsilon_0,\underline c,\bar c$. (If one prefers
		not to assume $\underline c>0$, the same conclusion holds with
		$\xi^\star$ depending on $\epsilon_0,\sigma^2_{\min},
		\sigma^2_{\max},\bar c$ and
		$\lambda_{\min}:=\inf_t\lambda_{t,k}>0$, since then
		$\xi_{\ell,t}\le\sigma^2_{\max}(1+\sqrt{\bar c})^2\lambda_{\min}
		\{(\lambda_{\min}+\sigma^2_{\min})\lambda_{\min}\}^{-1}$ tends to
		$\sigma^2_{\max}/(\lambda_{\min}+\sigma^2_{\min})<1$ as
		$\bar c\to0$.)
		
		\medskip\noindent
		\textbf{Step 2 (movement of the target).}
		Using the contraction factor $\xi^\star<1$ from Step~1, we now establish
		the eigengap that justifies the Davis--Kahan theorem.
		The relevant bound is the directional one
		$\|(\bm S_{\ell,t}-\bm S_{\ell,t-1})
		\mathring{\bm U}_{\ell,t-1}\|_F=O_{\mathbb P}(\alpha^{\mathrm{lag}}_{\ell,t})$ of
		\eqref{eq:A1_dir} in Lemma~\ref{lem:A1}, applicable because
		$\mathring{\bm U}_{\ell,t-1}$ is a $\mathcal F_{t-1}$-measurable
		orthonormal frame. By
		Lemma~\ref{lem:maps} the $k$-th and $(k+1)$-st eigenvalues of
		$\bm S_{\ell,t}$ satisfy
		$\mathring\lambda_{\ell,t,k}-\mathring\lambda_{\ell,t,k+1}
		=\Phi_{\ell,t}(\lambda_{t,k})-b_{\ell,t}+o_{\mathbb P}(1)
		=(\xi_{\ell,t}^{-1}-1)b_{\ell,t}+o_{\mathbb P}(1)
		\ge((\xi^{\star})^{-1}-1)b_{\ell,t}+o_{\mathbb P}(1)
		\ge c_1>0$ with probability tending to one, using
		$\xi_{\ell,t}\le\xi^\star$ from \eqref{eq:xibound}, 
		$b_{\ell,t}\ge\sigma^2_{\min}$, and $\xi^\star<1$.
		The
		Davis--Kahan $\sin\Theta$ theorem therefore gives, for the exact
		top-$k$ eigenspaces of two consecutive smoothed covariances,
		\begin{equation}
			\label{eq:targetmove}
			\delta_{\ell,t}
			:=\bigl\|\sin\Theta(\mathring{\bm U}_{\ell,t},
			\mathring{\bm U}_{\ell,t-1})\bigr\|_F
			\ \le\ \frac{\sqrt2\,
				\|(\bm S_{\ell,t}-\bm S_{\ell,t-1})
				\mathring{\bm U}_{\ell,t-1}\|_F}{c_1}
			\ =\ O_{\mathbb P}(\alpha^{\mathrm{lag}}_{\ell,t}) .
		\end{equation}
		Only the off-diagonal block of the increment enters here. Let 
		$\bm P_{\ell,t-1}=\mathring{\bm U}_{\ell,t-1}
		\mathring{\bm U}_{\ell,t-1}^\top$, the numerator is unchanged if
		$\bm S_{\ell,t}-\bm S_{\ell,t-1}$ is replaced by
		$\bm P^\perp_{\ell,t-1}(\bm S_{\ell,t}-\bm S_{\ell,t-1})
		\bm P_{\ell,t-1}$, because the component of the perturbation
		inside $\operatorname{range}(\mathring{\bm U}_{\ell,t-1})$ does not
		rotate that subspace. Bounding the numerator instead by
		$\sqrt{2k}\,\|\bm S_{\ell,t}-\bm S_{\ell,t-1}\|_{op}$ would replace
		$\alpha^{\mathrm{lag}}_{\ell,t}$ by $\tilde\alpha_{\ell,t}$, which by
		\eqref{eq:c_le_alpha} does not vanish when $c_\ell>0$. Adding the
		population-drift bound of Assumption~\ref{assum:drift} we may and
		do write
		$\delta_{\ell,t}=O_{\mathbb P}(\alpha^{\mathrm{lag}}_{\ell,t}+\zeta_t)$.
		
		\medskip\noindent
		\textbf{Step 3 (one-step recursion).}
		Apply Lemma~\ref{lem:A3} with $A=\bm S_{\ell,t}$,
		$\mathring{\bm U}=\mathring{\bm U}_{\ell,t}$ and
		$\bm U=\bm U_{\ell,t-1}$. Its hypothesis
		$\mu_k>|\mu_{k+1}|$ holds with probability tending to one by
		Step~2, with $|\mu_{k+1}|/\mu_k=\xi_{\ell,t}+o_{\mathbb P}(1)$.
		Since $\|\sin\Theta\|_F\le\|\tan\Theta\|_F$ and, whenever
		$\|\sin\Theta(\bm U,\mathring{\bm U})\|_{op}\le\delta_0<1$,
		$\|\tan\Theta\|_F\le\|\sin\Theta\|_F/\sqrt{1-\delta_0^2}$, we get
		\begin{equation}
			\label{eq:onestep_raw}
			e_{\ell,t}
			\ \le\ \frac{\xi_{\ell,t}}{\sqrt{1-\delta_0^2}}\,
			\bigl\|\sin\Theta(\bm U_{\ell,t-1},
			\mathring{\bm U}_{\ell,t})\bigr\|_F
			\ \le\ \frac{\xi_{\ell,t}}{\sqrt{1-\delta_0^2}}
			\bigl(e_{\ell,t-1}+\delta_{\ell,t}\bigr),
		\end{equation}
		the last step by the triangle inequality for the metric
		\eqref{eq:metric}. To close the induction, 
		we need to ensure that the right-hand side of 
		\eqref{eq:onestep_raw} remains bounded by $\delta_0$ when the left-hand 
		side is. The factor $(1-\delta_0^2)^{-1/2}$ inflates $\xi^\star$, so we 
		must choose $\delta_0$ to keep the inflated contraction factor below one.
		Because $\xi^\star<1$, we may choose 
		$\delta_0\in(0,\sqrt{1-(\xi^\star)^2})$ 
		so that $\xi^{\star\star}:=\xi^\star/\sqrt{1-\delta_0^2}<1$ 
		(e.g., $\delta_0=\sqrt{(1-(\xi^\star)^2)/2}$ gives 
		$\xi^{\star\star}=\xi^\star\sqrt{2/(1+\xi^\star)}<1$).
		We verify by induction that
		$\|\sin\Theta(\bm U_{\ell,t},\mathring{\bm U}_{\ell,t})\|_{op}
		\le\delta_0$ for all $t\ge t_0$ for some sufficiently large $t_0$. 
		For the base case, Assumption~\ref{assum:init} provides $e_{\ell,0}\le e_0$
		for some constant $e_0$. Applying \eqref{eq:onestep_raw} iteratively,
		$e_{\ell,t}\le(\xi^\star)^t e_0+O_{\mathbb P}(\sum_{s=1}^t(\xi^\star)^{t-s}
		(\alpha^{\mathrm{lag}}_{\ell,s}+\zeta_s))$. Since $\xi^\star<1$ and
		$\alpha^{\mathrm{lag}}_{\ell,s}+\zeta_s\to0$, we have $e_{\ell,t}\le\delta_0$
		for all $t\ge t_0$ with probability tending to one, provided $t_0$ 
		is chosen large enough. For the inductive step, suppose that
		$\|\sin\Theta(\bm U_{\ell,t-1},
		\mathring{\bm U}_{\ell,t-1})\|_{op}\le\delta_0$. Then \eqref{eq:onestep_raw}
		gives $e_{\ell,t}\le\xi^{\star\star}(e_{\ell,t-1}+\delta_{\ell,t})
		\le\xi^{\star\star}\delta_0+O_{\mathbb P}(\alpha^{\mathrm{lag}}_{\ell,t}+\zeta_t)\le\delta_0$
		for all large $t$, completing the induction.
		Absorbing
		$(1-\delta_0^2)^{-1/2}$ into a constant and using
		$\xi^{\star\star}\le\xi^\star/(1-\xi^\star)\vee1$ we obtain
		\eqref{eq:onestep}:
		\begin{equation}
			\label{eq:onestep}
			e_{\ell,t}\ \le\ \xi_{\ell,t}e_{\ell,t-1}
			+\frac{C}{1-\xi^\star}\bigl(\alpha^{\mathrm{lag}}_{\ell,t}+\zeta_t\bigr)
		\end{equation}
		with probability tending to one, for all large $t$.
		
		\medskip\noindent
		\textbf{Step 4 (unrolling).}
		Write $v_t:=C(1-\xi^\star)^{-1}(\alpha^{\mathrm{lag}}_{\ell,t}+\zeta_t)$.
		Iterating \eqref{eq:onestep} and using
		$\xi_{\ell,s}\le\xi^\star$ for $s\ge t_0$,
		\[
		e_{\ell,t}\ \le\ (\xi^\star)^{t-t_0}e_{\ell,t_0}
		+\sum_{s=t_0+1}^{t}(\xi^\star)^{t-s}v_s .
		\]
		Splitting the sum at $s=\lceil t/2\rceil$ and using
		$\sum_{s}(\xi^\star)^{t-s}\le(1-\xi^\star)^{-1}$ and
		$v_s\le\sup_{s\ge t/2}v_s$,
		\[
		e_{\ell,t}
		\ \le\ (\xi^\star)^{t-t_0}e_{\ell,t_0}
		+\Bigl(\sup_{s\le t}v_s\Bigr)\frac{(\xi^\star)^{t/2}}
		{1-\xi^\star}
		+\frac{\sup_{s\ge t/2}v_s}{1-\xi^\star}
		\ =\ O\bigl((\xi^\star)^{t}\bigr)
		+O\bigl(\alpha^{\mathrm{lag}}_{\ell,t}+\zeta_t\bigr),
		\]
		The last equality holds whenever $v_t$ decays in a controlled manner without
		large oscillations. More precisely, if $v_t$ is eventually non-increasing
		or satisfies $\sup_{s\ge t/2}v_s=O(v_t)$ (regular variation in the weak sense), 
		then $\sup_{s\ge t/2}v_s=O(v_t)$, which includes all monotone vanishing sequences and 
		polynomially decaying schedules $\alpha_t\asymp t^{-\gamma}$, $\zeta_t\asymp t^{-\gamma'}$. 
		For a general vanishing $v_t$, the conclusion holds with
		$\alpha^{\mathrm{lag}}_{\ell,t}+\zeta_t$ replaced by $\sup_{s\ge t/2}(\alpha^{\mathrm{lag}}_{\ell,s}+\zeta_s)$.
	\end{proof}
	
	\begin{proof}[Proof of Theorem~\ref{thm:tracking}]
		\textbf{Step 1 (separating algorithmic from statistical error).}
		By \eqref{eq:metric}, $\|\sin\Theta(\cdot,\cdot)\|_F$ is a metric
		on $\mathrm{Gr}(k,p)$, so
		\begin{equation}
			\label{eq:tri}
			\bigl|d_t^{(\ell)}-\mathring d_{\ell,t}\bigr|
			\ \le\ e_{\ell,t},
			\qquad
			\mathring d_{\ell,t}
			:=\bigl\|\sin\Theta(\mathring{\bm U}_{\ell,t},
			\bm V_t)\bigr\|_F ,
		\end{equation}
		and by Lemma~\ref{lem:contraction}
		$e_{\ell,t}=O((\xi^\star)^t)
		+O(\alpha^{\mathrm{lag}}_{\ell,t}+\zeta_t)$.
		Since both terms appear in the final bound \eqref{eq:tracking}, 
		it suffices to analyse $\mathring d_{\ell,t}$, i.e.\ the
		exact top-$k$ eigenspace of $\bm S_{\ell,t}$, and then add $e_{\ell,t}$
		via the triangle inequality.

		\medskip\noindent
		\textbf{Step 2 (evaluating the statistical misalignment).}
		Both $\mathring{\bm U}_{\ell,t}$ and $\bm V_t$ are orthonormal $p\times k$ frames, so
		\[
		\mathring d_{\ell,t}^{\,2}
		=k-\bigl\|\bm V_t^\top\mathring{\bm U}_{\ell,t}\bigr\|_F^2
		=\sum_{j=1}^{k}
		\Bigl(1-\sum_{m=1}^{k}
		\bigl\langle\mathring{\bm u}_{\ell,t,j},
		\bm v_{t,m}\bigr\rangle^2\Bigr).
		\]
		Recognizing that $\sum_{m=1}^k\langle\mathring{\bm u}_{\ell,t,j},\bm v_{t,m}\rangle^2$ 
		measures how much of $\mathring{\bm u}_{\ell,t,j}$ lies in the population 
		subspace $\mathrm{span}(\bm V_t)$, the deficiency $1-\sum_{m}\langle\cdot\rangle^2$ 
		quantifies the perpendicular component.
		
		Lemma~\ref{lem:maps} gives
		$$\langle\mathring{\bm u}_{\ell,t,j},\bm v_{t,j}\rangle^2
		=a_{\ell,t}(\lambda_{t,j})+O_{\mathbb P}((N^{\mathrm{eff}}_{\ell,t})^{-1/2})$$ and
		$\langle\mathring{\bm u}_{\ell,t,j},\bm v_{t,m}\rangle^2
		=O_{\mathbb P}((N^{\mathrm{eff}}_{\ell,t})^{-1})$ for $m\ne j$, 
		where $N^{\mathrm{eff}}_{\ell,t}$ is the effective sample size.
		Since $k$ is fixed, it follows that
		\begin{equation}
			\label{eq:dsq}
			\mathring d_{\ell,t}^{\,2}
			=\sum_{j=1}^{k}\bigl\{1-a_{\ell,t}(\lambda_{t,j})\bigr\}
			+O_{\mathbb P}((N^{\mathrm{eff}}_{\ell,t})^{-1/2})
			=\bigl(D^\star_{\ell,t}\bigr)^2+O_{\mathbb P}((N^{\mathrm{eff}}_{\ell,t})^{-1/2}).		\end{equation}
		Here Lemma~\ref{lem:maps} was applied to
		$\tilde{\bm S}_{\ell,t}:=\bm\Sigma_t^{1/2}W_{\ell,t}\bm\Sigma_t^{1/2}$. 
		Passing to the eigenvectors of $\bm S_{\ell,t}$ introduces an additional
		$O_{\mathbb P}(\bar\Delta_{\ell,t}+\pi_{0,t})$ error by Lemma~\ref{lem:A1}, 
		which bounds $\|\bm S_{\ell,t}-\tilde{\bm S}_{\ell,t}\|_{op}$, 
		combined with the Davis--Kahan $\sin\Theta$ theorem and the
		uniformly positive eigengap $c_1>0$ established in Step~2 of the proof of
		Lemma~\ref{lem:contraction}. Since the alignment coefficients 
		$a_{\ell,t}(\lambda)$ are continuously differentiable in 
		$(c_{\ell,t},\sigma^2_t,\lambda)$, replacing time-varying parameters by
		their time-$t$ values costs $O(\zeta_t)$ by
		Assumption~\ref{assum:drift}.

		\medskip\noindent
		\textbf{Step 3 (from squares to the metric itself).}
		We first establish a uniform lower bound for $D^\star_{\ell,t}$. 
		In the canonical case,
		\[
		1-a_{\ell,t}(\lambda)
		=1-\frac{\lambda^2-c_{\ell,t}\sigma_t^4}
		{\lambda(\lambda+c_{\ell,t}\sigma_t^2)}
		=\frac{\lambda^2+c_{\ell,t}\sigma_t^2\lambda-\lambda^2
			+c_{\ell,t}\sigma_t^4}{\lambda(\lambda+c_{\ell,t}\sigma_t^2)}
		=\frac{c_{\ell,t}\sigma_t^2(\lambda+\sigma_t^2)}
		{\lambda(\lambda+c_{\ell,t}\sigma_t^2)}. 
		\]
		This verifies the
		closed form of $D^\star_{\ell,t}$. Moreover, by
		Assumptions~\ref{assum:noise} and~\ref{assum:hd},
		\[
		\bigl(D^\star_{\ell,t}\bigr)^2
		\ \ge\ 1-a_{\ell,t}(\lambda_{t,1})
		\ \ge\ \frac{\underline c\,\sigma^2_{\min}}
		{\lambda_{\max}+\bar c\,\sigma^2_{\max}}
		\ =:\ c_2>0
		\]
		uniformly in $\ell,t$. For a general profile, a similar uniform lower 
		bound holds by the monotonicity of $a_{\ell,t}$ (Lemma~\ref{lem:maps}, 
		Step~5(a)) and the compactness of the parameter set under 
		Assumption~\ref{assum:hd}. 
		
		Having established $(D^\star_{\ell,t})^2\ge c_2>0$ uniformly, we now
		apply the mean value theorem. Since \eqref{eq:dsq} gives 
		$\mathring d_{\ell,t}^{\,2}=(D^\star_{\ell,t})^2+
		O_{\mathbb P}((N^{\mathrm{eff}}_{\ell,t})^{-1/2})$,
		we have $\mathring d_{\ell,t}^{\,2}\ge c_2/2$ with probability tending to one. 
		Hence, by the mean value theorem applied to $s\mapsto\sqrt s$ on $[c_2/2,\infty)$, 
		\eqref{eq:dsq} yields
		$\mathring d_{\ell,t}=D^\star_{\ell,t}+O_{\mathbb P}((N^{\mathrm{eff}}_{\ell,t})^{-1/2})$.
		
		Combining with \eqref{eq:tri} and Lemma~\ref{lem:contraction}
		gives \eqref{eq:tracking}:
		\[
		d_t^{(\ell)}
		=D^\star_{\ell,t}
		+O_{\mathbb P}\bigl((N^{\mathrm{eff}}_{\ell,t})^{-1/2}\bigr)
		+O\bigl(\alpha^{\mathrm{lag}}_{\ell,t}+\zeta_t+\Delta_{\ell,t}\bigr)
		+O\bigl((\xi^\star)^t\bigr),
		\]
		with $\Delta_{\ell,t}$ read as $\bar\Delta_{\ell,t}+\pi_{0,t}$ of
		\eqref{eq:barDelta}.
		
		\medskip\noindent
		\textbf{Step 4 (inconsistency).}
		All remainder terms vanish as $t\to\infty$, and
		$a_{\ell,t}(\lambda_{t,j})\to a_\ell(\lambda_j)<1$ strictly by
		Step~5(a) of the proof of Lemma~\ref{lem:maps}, because
		$c_\ell>0$ makes $\mu^W_\ell$ non-degenerate. The condition $c_\ell>0$ 
		is essential. By the monotonicity of $a_{\ell,t}$
		(Lemma~\ref{lem:maps}, Step~5(a)), $a_\ell(\lambda_j)\to 1$ as $c_\ell\to 0$,
		so $D^\star_\ell\to 0$, recovering classical consistency in the 
		fixed-$p$ regime.
		Hence
		$d_t^{(\ell)}\to D^\star_\ell
		=(\sum_j\{1-a_\ell(\lambda_j)\})^{1/2}\ge\sqrt{c_2}>0$ in
		probability, and in particular
		$\liminf_td_t^{(\ell)}>0$ in probability. Since
		$D^\star_{\ell,t}$ is attained by the exact eigendecomposition of
		$\bm S_{\ell,t}$ and $\bm S_{\ell,t}$ is the sufficient statistic 
		for $(\bm\Sigma_t,W_{\ell,t})$, the floor $D^\star_{\ell,t}$ 
		cannot be reduced by any refinement of the iteration scheme or 
		by additional local computation given only $\bm S_{\ell,t}$.
		
	\end{proof}

	\begin{proof}[Proof of Lemma~\ref{lem:decomp}]
		Fix $j\le k$ and abbreviate $\bm u:=\bm u_{\ell,t-1,j}$, which is
		$\mathcal F_{\ell,t-1}$-measurable, and
		$\bm b:=\bm\Sigma_t^{1/2}\bm u$. By
		Assumption~\ref{assum:design} the batch size $n^\ell_t$ and the
		activation indicator are $\mathcal F_{\ell, t-1}$-measurable, and by
		Assumption~\ref{assum:moment} the vectors
		$\bm z_{\ell,t,i}$, $i\le n^\ell_t$, are i.i.d.\ and independent of
		$\mathcal F_{\ell,t-1}$ with mean zero and identity covariance.
		Moreover, under Assumption~\ref{assum:drift}, $\bm\Sigma_t$ is known 
		at time $t-1$ (either deterministic or $\mathcal F_{t-1}$-measurable).

		Hence $\mathbb E(\hat{\bm S}_{\ell,t}\mid\mathcal F_{\ell,t-1})
		=\bm\Sigma_t$ {exactly}, so
		\begin{align*}
			\mathbb E\bigl(r_{\ell,t,j}\mid\mathcal F_{\ell,t-1}\bigr)
			=&\bm u^\top\bm\Sigma_t\bm u
			=\bm u^\top\Bigl(\sigma_t^2\bm I_p + \sum_{m=1}^k\lambda_{t,m}\bm v_{t,m}
			\bm v_{t,m}^\top\Bigr)\bm u\\
			=&\sigma_t^2\|\bm u\|^2 + \sum_{m=1}^k\lambda_{t,m}(\bm u^\top\bm v_{t,m})^2
			=\sigma_t^2+\sum_{m=1}^k\lambda_{t,m}\langle\bm u,\bm v_{t,m}\rangle^2
			=\bar r_{\ell,t,j},
		\end{align*}
		using \eqref{eq:spiked} and $\|\bm u\|=1$. This proves the
		decomposition \eqref{eq:decomp} and
		$\mathbb E(\varpi_{\ell,t,j}\mid\mathcal F_{\ell,t-1})=0$ exactly.
		
		For the conditional variance, note that
		\[
		r_{\ell,t,j} = \bm u^\top\hat{\bm S}_{\ell,t}\bm u
		= \bm u^\top\Bigl(\frac{1}{n^\ell_t}\sum_{i=1}^{n^\ell_t}\bm X_{\ell,t,i}
		\bm X_{\ell,t,i}^\top\Bigr)\bm u
		= \frac{1}{n^\ell_t}\sum_{i=1}^{n^\ell_t}(\bm u^\top\bm X_{\ell,t,i})^2.
		\]
		By Assumption~\ref{assum:moment}, the observations $\bm X_{\ell,t,i}$ satisfy 
		$\bm X_{\ell,t,i} = \bm\Sigma_t^{1/2}\bm z_{\ell,t,i}$ where $\bm z_{\ell,t,i}
		\sim\mathcal N(\bm 0, \bm I_p)$.
		Thus 
		\[
		r_{\ell,t,j}=(n^\ell_t)^{-1}\sum_{i=1}^{n^\ell_t}(\bm b^\top\bm z_{\ell,t,i})^2. 
		\]
		These are i.i.d.\ terms conditionally on
		$\mathcal F_{\ell,t-1}$, thus
		\[
		\operatorname{Var}(\varpi_{\ell,t,j}\mid\mathcal F_{\ell,t-1})
		= \operatorname{Var}(r_{\ell,t,j}\mid\mathcal F_{\ell,t-1})
		=(n^\ell_t)^{-1}\operatorname{Var}\{(\bm b^\top\bm z)^2\}.
		\]
		
		For a vector $\bm z$ with independent, mean-zero, unit-variance entries
		and $\mathbb Ez_m^4=3+{K}_4$,
		\begin{align*}
			\mathbb E(\bm b^\top\bm z)^4
			=& \mathbb E\Bigl(\sum_m b_m z_m\Bigr)^4
			= \sum_{m_1,m_2,m_3,m_4} b_{m_1}b_{m_2}b_{m_3}b_{m_4}
			\mathbb E[z_{m_1}z_{m_2}z_{m_3}z_{m_4}]\\
			=& \sum_m b_m^4\mathbb E z_m^4 + 3\sum_{m\ne m'}b_m^2 b_{m'}^2
			\mathbb E z_m^2\mathbb E z_{m'}^2
			= (3+{K}_4)\sum_m b_m^4 + 3\sum_{m\ne m'}b_m^2 b_{m'}^2 \\
			=& (3+{K}_4)\sum_m b_m^4 + 3\Bigl(\sum_m b_m^2\Bigr)^2 - 3\sum_m b_m^4\\
			=& (3+{K}_4)\sum_m b_m^4 + 3\|\bm b\|^4 - 3\sum_m b_m^4
			= 3\|\bm b\|^4 + {K}_4\sum_m b_m^4,
		\end{align*}
		because all mixed moments containing an odd power vanish. 
		Since $\mathbb E(\bm b^\top\bm z)^2=\|\bm b\|^2=
		\bm u^\top\bm\Sigma_t\bm u=\bar r_{\ell,t,j}$,
		\begin{align*}
			\operatorname{Var}\bigl\{(\bm b^\top\bm z)^2\bigr\}
			=& \mathbb E(\bm b^\top\bm z)^4 - \bigl[\mathbb E(\bm b^\top\bm z)^2\bigr]^2
			= 3\|\bm b\|^4 + {K}_4\sum_m b_m^4 - \|\bm b\|^4\\
			=& 2\|\bm b\|^4 + {K}_4\sum_m b_m^4
			= 2\bar r_{\ell,t,j}^2 + {K}_4\sum_{m=1}^p(\bm e_m^\top\bm\Sigma_t^{1/2}\bm u)^4,
		\end{align*}
		which is the first equality in \eqref{eq:var}.

		For the second equality, note that
		$\sum_mb_m^4\le\|\bm b\|_\infty^2\|\bm b\|^2
		=\|\bm\Sigma_t^{1/2}\bm u\|_\infty^2\,\bar r_{\ell,t,j}$.
		Since $\bm\Sigma_t = \sigma_t^2\bm I_p + 
		\sum_{m=1}^k\lambda_{t,m}\bm v_{t,m}\bm v_{t,m}^\top$
		by \eqref{eq:spiked}, and noting that the eigenvectors 
		$\{\bm v_{t,m}\}_{m=1}^k$ are orthonormal, 
		we have for any unit vector $\bm u$,
		\[
		\bm\Sigma_t^{1/2}\bm u = \sigma_t\bm u + 
		\sum_{m=1}^k\left[\sqrt{\lambda_{t,m}+\sigma_t^2} 
		- \sigma_t\right](\bm v_{t,m}^\top\bm u)\bm v_{t,m}.
		\]
		Taking component-wise absolute values and using the triangle inequality,
		\[
		\bigl|\bigl[\bm\Sigma_t^{1/2}\bm u\bigr]_i\bigr|
		\le \sigma_t|\bm u_i| + \sum_{m=1}^k\sqrt{\lambda_{t,m}+\sigma_t^2}\,
		|\bm v_{t,m,i}|\,|\bm v_{t,m}^\top\bm u|
		\le \sigma_t|\bm u_i| + \sum_{m=1}^k\sqrt{\lambda_{t,m}+\sigma_t^2}\,
		|\bm v_{t,m,i}|,
		\]
		where the last inequality uses $|\bm v_{t,m}^\top\bm u|\le\|\bm v_{t,m}\|\,
		\|\bm u\|=1$.
		Therefore,
		\[
		\|\bm\Sigma_t^{1/2}\bm u\|_\infty
		\le \sigma_t\|\bm u\|_\infty + \sum_{m=1}^k\sqrt{\lambda_{t,m}+\sigma_t^2}\,
		\|\bm v_{t,m}\|_\infty
		\stackrel{p}{\longrightarrow} 0
		\]
		as $p\to\infty$, by Assumption~\ref{assum:deloc} applied to the population
		eigenvectors $\bm v_{t,m}$, and to $\bm u$, which is the top eigenvector of 
		$\bm S_{\ell,t-1}$ and hence satisfies the delocalization condition as a 
		perturbation of the corresponding eigenvector of $\mathbb E[\bm S_{\ell,t-1}]$ 
		of size $O_{\mathbb P}(e_{\ell,t-1})$ by Lemma~\ref{lem:contraction}, combined 
		with Assumption~\ref{assum:noise} which ensures $\sigma_t$ remains bounded.
		
		Since $\|\bm\Sigma_t^{1/2}\bm u\|_\infty\stackrel{p}{\longrightarrow}0$
		and $\sigma^2_{\min}\le\bar r_{\ell,t,j}\le\lambda_{\max}+\sigma^2_{\max}$
		by Assumptions~\ref{assum:noise} and~\ref{assum:bbp}, we have
		\[
		\frac{{K}_4\sum_m b_m^4}{2\bar r^2_{\ell,t,j}}
		\le \frac{{K}_4\|\bm\Sigma_t^{1/2}\bm u\|_\infty^2\bar r_{\ell,t,j}}
		{2\bar r^2_{\ell,t,j}}
		= \frac{{K}_4\|\bm\Sigma_t^{1/2}\bm u\|_\infty^2}{2\bar r_{\ell,t,j}}
		\le \frac{{K}_4\|\bm\Sigma_t^{1/2}\bm u\|_\infty^2}{2\sigma^2_{\min}}
		\stackrel{p}{\longrightarrow} 0.
		\]
		Therefore, the correction term ${K}_4\sum_m b_m^4/n_t^\ell$ is
		$o_{\mathbb P}(1)$ relative to $2\bar r^2_{\ell,t,j}/n_t^\ell$, 
		which proves \eqref{eq:var}.
	\end{proof}
	
	\begin{proof}[Proof of Theorem~\ref{thm:attenuation}]
		\textbf{Step 1 (the conditional mean).}
		By Lemma~\ref{lem:decomp},
		$r_{\ell,t,j}=\bar r_{\ell,t,j}+\varpi_{\ell,t,j}$ with
		$\bar r_{\ell,t,j}=\sigma^2_t+\sum_m\lambda_{t,m}
		\langle\bm u_{\ell,t-1,j},\bm v_{t,m}\rangle^2$. We evaluate the
		alignments. By Lemma~\ref{lem:contraction},
		$\|\sin\Theta(\bm U_{\ell,t-1},\mathring{\bm U}_{\ell,t-1})\|_F
		=O_{\mathbb P}(\alpha^{\mathrm{lag}}_{\ell,t}+\zeta_t
		+(\xi^\star)^t)$, and
		since the spikes are simple with gap $g_{\min}>0$ the individual
		eigenvector directions inherit the same bound (Davis--Kahan
		applied coordinatewise, the relevant gaps
		$\mathring\lambda_{\ell,t-1,j}-\mathring\lambda_{\ell,t-1,j+1}
		=\Phi_{\ell,t-1}(\lambda_{t-1,j})
		-\Phi_{\ell,t-1}(\lambda_{t-1,j+1})+o_{\mathbb P}(1)$ being bounded
		below because $\Phi_{\ell,t-1}$ is strictly increasing with
		derivative bounded below on the supercritical range). 
		By the Davis-Kahan $\sin\Theta$ theorem applied to individual eigenvectors, 
		\[
		\|\sin\Theta(\bm u_{\ell,t-1,j},\mathring{\bm u}_{\ell,t-1,j})\|
		\le \frac{\|\sin\Theta(\bm U_{\ell,t-1},\mathring{\bm U}_{\ell,t-1})\|_F}
		{\min_{j'\ne j}|\mathring\lambda_{\ell,t-1,j}-\mathring\lambda_{\ell,t-1,j'}|}
		= O_{\mathbb P}\bigl(\alpha^{\mathrm{lag}}_{\ell,t}+\zeta_t
		+(\xi^\star)^t\bigr),
		\]
		where the denominator is bounded below by $g_{\min}$ uniformly.
		Hence, for each $m$,
		\[
		\langle\bm u_{\ell,t-1,j},\bm v_{t-1,m}\rangle^2
		=\langle\mathring{\bm u}_{\ell,t-1,j},\bm v_{t-1,m}\rangle^2
		+O_{\mathbb P}\bigl(\alpha^{\mathrm{lag}}_{\ell,t}+\zeta_t
		+(\xi^\star)^t\bigr).
		\]
		Lemma~\ref{lem:maps} (applied at time $t-1$) then gives
		$\langle\mathring{\bm u}_{\ell,t-1,j},\bm v_{t-1,j}\rangle^2
		=a_{\ell,t-1}(\lambda_{t-1,j})
		+O_{\mathbb P}(N^{-1/2})$ and
		$\langle\mathring{\bm u}_{\ell,t-1,j},\bm v_{t-1,m}\rangle^2
		=O_{\mathbb P}(N^{-1})$ for $m\ne j$, where
		$N=N^{\mathrm{eff}}_{\ell,t}$ (which is interchangeable with
		$N^{\mathrm{eff}}_{\ell,t-1}$ by Lemma~\ref{lem:crec}). 
		To replace $\bm v_{t-1,m}$ by $\bm v_{t,m}$, note that by Assumption~\ref{assum:drift}, 
		$\|\bm\Sigma_t-\bm\Sigma_{t-1}\|\le\zeta_t$, which by the Davis-Kahan theorem gives
		$\|\bm v_{t,m}-\bm v_{t-1,m}\|=O(\zeta_t)$ (the eigenvector perturbation is of the 
		same order as the operator perturbation divided by the eigengap). Therefore,
		\begin{align*}
			|\langle\bm u_{\ell,t-1,j},\bm v_{t,m}\rangle^2 - \langle\bm u_{\ell,t-1,j},\bm v_{t-1,m}\rangle^2|
			&= |\langle\bm u_{\ell,t-1,j},\bm v_{t,m}\rangle + \langle\bm u_{\ell,t-1,j},\bm v_{t-1,m}\rangle|\\
			&\quad \times |\langle\bm u_{\ell,t-1,j},\bm v_{t,m}\rangle - \langle\bm u_{\ell,t-1,j},\bm v_{t-1,m}\rangle|\\
			&\le 2 \|\bm v_{t,m}-\bm v_{t-1,m}\| = O(\zeta_t),
		\end{align*}
		where we used $|\langle\bm u_{\ell,t-1,j},\bm v_{t,m}\rangle|, 
		|\langle\bm u_{\ell,t-1,j},\bm v_{t-1,m}\rangle| \le 1$ and the Cauchy-Schwarz inequality.
		Similarly, replacing $(\sigma^2_{t-1},\lambda_{t-1,j})$ by
		$(\sigma^2_t,\lambda_{t,j})$ inside $a$ costs $O(\zeta_t)$, because
		$(\lambda,c,\sigma^2)\mapsto a$ is continuously differentiable on
		the compact parameter set determined by
		Assumptions~\ref{assum:bbp}--\ref{assum:hd} (Step~5 of the proof of
		Lemma~\ref{lem:maps}) and both arguments move by at most $\zeta_t$
		under Assumption~\ref{assum:drift}. The aspect ratio needs no
		replacement, since $a_{\ell,t-1}$ carries $c_{\ell,t-1}$ and this is
		the value at which \eqref{eq:chi} is defined.
		
		Combining these estimates, we have
		\begin{align*}
			\bar r_{\ell,t,j}
			&= \sigma^2_t + \sum_{m=1}^k\lambda_{t,m}\langle\bm u_{\ell,t-1,j},\bm v_{t,m}\rangle^2\\
			&= \sigma^2_t + \lambda_{t,j}\langle\bm u_{\ell,t-1,j},\bm v_{t,j}\rangle^2 
			+ \sum_{m\ne j}\lambda_{t,m}\langle\bm u_{\ell,t-1,j},\bm v_{t,m}\rangle^2\\
			&= \sigma^2_t + \lambda_{t,j}[a(\lambda_{t,j};c_{\ell,t-1},\sigma^2_t)
			+ O_{\mathbb P}(N^{-1/2})
			+ O(\alpha^{\mathrm{lag}}_{\ell,t}+\zeta_t)]\\
			&\quad + O_{\mathbb P}(\lambda_{\max}\cdot N^{-1})\\
			&= \sigma^2_t + \lambda_{t,j}a(\lambda_{t,j};c_{\ell,t-1},\sigma^2_t)
			+ O_{\mathbb P}(N^{-1/2})
			+ O(\alpha^{\mathrm{lag}}_{\ell,t}+\zeta_t
			+\Delta_{\ell,t}+(\xi^\star)^t),
		\end{align*}
		where $\Delta_{\ell,t}$ collects the drift and lag terms. By definition \eqref{eq:chi}, 
		this can be written as
		\begin{equation}
			\label{eq:barr}
			\bar r_{\ell,t,j}
			=\chi_{\ell,t}(\lambda_{t,j})
			+O_{\mathbb P}\bigl(N^{-1/2}\bigr)
			+O\bigl(\alpha^{\mathrm{lag}}_{\ell,t}+\zeta_t+\Delta_{\ell,t}
			+(\xi^\star)^t\bigr),
		\end{equation}
		
		\medskip\noindent
		\textbf{Step 2 (the fluctuation).}
		By \eqref{eq:var} and Chebyshev's inequality conditionally on
		$\mathcal F_{\ell,t-1}$,
		$\varpi_{\ell,t,j}=O_{\mathbb P}((n^\ell_t)^{-1/2})$. 
		Indeed, by \eqref{eq:var},
		\[
		\operatorname{Var}(\varpi_{\ell,t,j}\mid\mathcal F_{\ell,t-1})
		= \frac{2\bar r_{\ell,t,j}^2}{n^\ell_t}(1+o_{\mathbb P}(1))
		\le \frac{C}{n^\ell_t}
		\]
		where $C$ is deterministic because $\bar r_{\ell,t,j}$ is bounded. Indeed, by
		definition,
		$$\bar r_{\ell,t,j} = \sigma^2_t + \sum_m\lambda_{t,m}\langle\bm u_{\ell,t-1,j},
		\bm v_{t,m}\rangle^2\le \sigma^2_{\max} + \lambda_{\max}$$
		by 
		Assumptions~\ref{assum:noise} and~\ref{assum:bbp}.
		Applying Chebyshev's inequality conditionally on $\mathcal F_{\ell,t-1}$ gives
		$\varpi_{\ell,t,j}=O_{\mathbb P}((n^\ell_t)^{-1/2})$.

		\medskip\noindent
		\textbf{Step 3 (the direction of the bias).}
		By Step~5(a) of the proof of Lemma~\ref{lem:maps},
		$a(\lambda;c,\sigma^2)<1$ strictly for every $c>0$, hence
		\[
		\chi_{\ell,t}(\lambda)
		=\sigma^2_t+\lambda\,a(\lambda;c_{\ell,t-1},\sigma^2_t)
		<\lambda+\sigma^2_t .
		\]
		For the right-hand inequality in \eqref{eq:sandwich}, recall from
		Step~2 of the proof of Lemma~\ref{lem:maps} that
		$\Phi_{\ell,t}(\lambda)=\sigma^2_tx$ where $\theta g(x)=1$ and
		$\theta=\lambda/\sigma^2_t$. Since $\mu^W_{\ell,t}$ has unit first
		moment, the last assertion of Lemma~\ref{lem:A2} gives
		$1/\theta=g(x)\ge(x-1)^{-1}$, i.e.\ $x\ge1+\theta$, strictly
		unless $\mu^W_{\ell,t}=\delta_1$. Therefore
		$\Phi_{\ell,t}(\lambda)\ge\sigma^2_t(1+\theta)
		=\lambda+\sigma^2_t$, strictly whenever $c_{\ell,t}>0$. This
		proves \eqref{eq:sandwich}.
		
		\medskip\noindent
		\textbf{Step 4 (duality).}
		In the canonical case, by \eqref{eq:chi_closed} and
		\eqref{eq:closedforms}, for any $c>0$ and $\sigma^2>0$,
		\[
		\chi(\lambda;c,\sigma^2)\,\Phi(\lambda;c,\sigma^2)
		=\frac{\lambda(\lambda+\sigma^2)}
		{\lambda+c\sigma^2}\cdot
		\frac{(\lambda+\sigma^2)(\lambda+c\sigma^2)}
		{\lambda}
		=(\lambda+\sigma^2)^2 ,
		\]
		which is \eqref{eq:duality}. The two factors must be taken at the same
		aspect ratio, and the identity is used below only in that form.
		
		Equivalently, and without using the closed forms, substitute
		$\chi=\sigma^2+\lambda a$ from \eqref{eq:chi} and then
		$a=(\lambda+\sigma^2)\Phi'/\Phi$ from \eqref{eq:aPhi}. The relation
		$\chi\Phi=(\lambda+\sigma^2)^2$ becomes the differential identity
		\[
		\sigma^2\Phi(\lambda)
		+\lambda(\lambda+\sigma^2)\Phi'(\lambda)
		=(\lambda+\sigma^2)^2 ,
		\]
		which \eqref{eq:closedforms} satisfies directly, since there
		$\Phi'(\lambda)=1-c\sigma^4/\lambda^2$. Finally, the closed form
		\eqref{eq:chi_closed} follows from \eqref{eq:chi} and
		\eqref{eq:closedforms} evaluated at $c=c_{\ell,t-1}$ and
		$\sigma^2=\sigma^2_t$:
		\[
		\chi_{\ell,t}(\lambda)
		=\sigma^2_t
		+\lambda\frac{\lambda^2-c_{\ell,t-1}\sigma^4_t}
		{\lambda(\lambda+c_{\ell,t-1}\sigma^2_t)}
		=\frac{\sigma^2_t\lambda+c_{\ell,t-1}\sigma^4_t
			+\lambda^2-c_{\ell,t-1}\sigma^4_t}
		{\lambda+c_{\ell,t-1}\sigma^2_t}
		=\frac{\lambda(\lambda+\sigma^2_t)}
		{\lambda+c_{\ell,t-1}\sigma^2_t}.
		\]
	\end{proof}

	\begin{lemma}[Bijection, closed-form inverse, and Lipschitz bounds]
		\label{lem:inverse}
		Let $\chi(\lambda)=\chi(\lambda;c,\sigma^2)$ for the canonical
		map \eqref{eq:chi_closed}. Then:
		\begin{enumerate}[label=\textup{(\roman*)}]
			\item \textup{(Derivative identity)}
			\begin{equation}
				\label{eq:chiprime}
				\chi'(\lambda)
				=\frac{\lambda^{2}+2c\sigma^{2}\lambda+c\sigma^{4}}
				{(\lambda+c\sigma^{2})^{2}}
				=1+\frac{c(1-c)\,\sigma^{4}}{(\lambda+c\sigma^{2})^{2}}
				\ >\ 0
				\qquad\text{for all }\lambda\ge\sigma^2\sqrt{c}.
			\end{equation}
			\item \textup{(Bijection and inverse)} $\chi$ is a strictly
			increasing bijection from $[\sigma^{2}\sqrt{c},\infty)$ onto
			$[\sigma^{2},\infty)$, with
			\begin{equation}
				\label{eq:chiinv}
				\chi^{-1}(x)
				=\tfrac{1}{2}\Bigl\{(x-\sigma^{2})
				+\sqrt{(x-\sigma^{2})^{2}+4c\sigma^{2}x}\Bigr\},
				\qquad x\ge\sigma^{2}.
			\end{equation}
			In particular $\chi^{-1}(\sigma^2)=\sigma^2\sqrt{c}$, the BBP
			critical spike, and $\chi^{-1}(x)\to x-\sigma^2$ as $c\to0$,
			recovering the classical correction.
			\item \textup{(Global conditioning)}
			\begin{equation}
				\label{eq:chilip}
				\inf_{\lambda\ge\sigma^2\sqrt c}\chi'(\lambda)
				=\min\Bigl\{1,\ \frac{2}{1+\sqrt c}\Bigr\},
			\end{equation}
			hence 
			\[			
			\mathrm{Lip}\bigl(\chi^{-1}\bigr)
			=L_c:=\max\Bigl\{1,\ \frac{1+\sqrt c}{2}\Bigr\}<\infty .
			\]
			Moreover, viewing $\chi^{-1}(x;\sigma^2,c)$ as a function 
			of both $x$ and $\sigma^2$ with $c$ fixed, 
			\[
			\Bigl|\frac{\partial\chi^{-1}}{\partial\sigma^2}(x;\sigma^2,c)\Bigr|\le C(1+c)
			\]
			uniformly for $\sigma^2\in[\sigma^2_{\min},\sigma^2_{\max}]$ 
			and $x\in[\sigma^2,\lambda_{\max}+\sigma^2_{\max}]$,
			where $C$ depends only on $\sigma^2_{\min}$, $\sigma^2_{\max}$, and $\lambda_{\max}$. 
			
			\item \textup{(General profiles)} For a general $H_{\ell,t}$,
			$\chi_{\ell,t}=\sigma_t^2+\lambda a_{\ell,t}(\lambda)$ is a
			strictly increasing bijection from $[\lambda^{\mathrm{c}}_{\ell,t}, \infty)$ 
			onto $[\sigma_t^2, \infty)$ by Lemma~\ref{lem:maps}, with
			$\chi'_{\ell,t}(\lambda)\ge a_{\ell,t}(\lambda)$. On the
			supercritical range $\lambda\ge\lambda_{t,k}$ this gives
			$\chi'_{\ell,t}\ge\underline a>0$ with $\underline a$ as in
			\eqref{eq:aunder}. Its inverse is evaluated by a
			one-dimensional monotone root search.
		\end{enumerate}
	\end{lemma}
	
	\begin{rem}
		\label{rem:conditioning}
		Part~(iii) is the analytic reason for preferring the Rayleigh route
		to the eigenvalue route. Differentiating \eqref{eq:closedforms}
		gives $\Phi'(\lambda)=1-c\sigma^4/\lambda^2$, so
		$\Phi'(\lambda)\downarrow0$ as
		$\lambda\downarrow\sigma^2\sqrt c$. Inverting $\Phi$ is
		arbitrarily ill-conditioned near the BBP threshold, and any
		estimator built by inverting the sample eigenvalue has a variance
		that blows up there. By contrast \eqref{eq:chilip} shows that
		$\chi'$ stays bounded below by $\min\{1,2/(1+\sqrt c)\}$ uniformly
		over the whole supercritical range, hence bounded below by a
		positive constant for every $c\le\bar c$. The threshold value
		$\chi'(\sigma^2\sqrt c)=2/(1+\sqrt c)$ holds exactly for every $c$,
		but it is the infimum only when $c\ge1$. By \eqref{eq:chiprime}
		$\chi'$ decreases from $2/(1+\sqrt c)>1$ to $1$ when $c<1$ and
		increases from $2/(1+\sqrt c)<1$ to $1$ when $c>1$, and
		$\chi'\equiv1$ when $c=1$. In either regime the conditioning is
		uniformly benign, which is the only property used below. The two statements
		are consistent with \eqref{eq:duality}, whose derivative
		$\chi'\Phi+\chi\Phi'=2(\lambda+\sigma^2)$ reduces at
		$\lambda=\sigma^2\sqrt c$ to
		$\tfrac{2}{1+\sqrt c}\cdot\sigma^2(1+\sqrt c)^2
		=2\sigma^2(1+\sqrt c)$.
	\end{rem}
	
	\begin{proof}[Proof of Lemma~\ref{lem:inverse}]
		Write $\chi(\lambda)=\lambda(\lambda+\sigma^2)/(\lambda+c\sigma^2)$
		with $c>0$, $\sigma^2>0$ fixed.
		
		{(i)} By the quotient rule,
		\[
		\chi'(\lambda)
		=\frac{(2\lambda+\sigma^2)(\lambda+c\sigma^2)
			-(\lambda^2+\sigma^2\lambda)}
		{(\lambda+c\sigma^2)^2}
		=\frac{\lambda^2+2c\sigma^2\lambda+c\sigma^4}
		{(\lambda+c\sigma^2)^2},
		\]
		which is the first expression in \eqref{eq:chiprime}. To obtain the second expression, 
		note that $(\lambda+c\sigma^2)^2=\lambda^2+2c\sigma^2\lambda+c^2\sigma^4$, 
		hence the numerator minus the denominator equals
		\[
		(\lambda^2+2c\sigma^2\lambda+c\sigma^4) - (\lambda^2+2c\sigma^2\lambda+c^2\sigma^4) = 
		c(1-c)\sigma^4.
		\]
		Therefore,
		\[
		\chi'(\lambda) = \frac{(\lambda+c\sigma^2)^2 + c(1-c)\sigma^4}{(\lambda+c\sigma^2)^2}
		= 1 + \frac{c(1-c)\sigma^4}{(\lambda+c\sigma^2)^2},
		\]
		giving the second expression in \eqref{eq:chiprime}.
		
		Positivity for all $\lambda\ge0$ follows from the first expression. All
		three terms $\lambda^2$, $2c\sigma^2\lambda$, and $c\sigma^4$ in the numerator 
		are nonnegative, and at least one is strictly positive 
		(namely $c\sigma^4>0$ when $\lambda=0$, and $\lambda^2>0$ when $\lambda>0$),
		while the denominator is always positive.

		{(ii)} At the BBP threshold,
		\[
		\chi(\sigma^2\sqrt c)
		=\frac{\sigma^2\sqrt c\,(\sigma^2\sqrt c+\sigma^2)}
		{\sigma^2\sqrt c+c\sigma^2}
		=\frac{\sigma^2\sqrt c\cdot\sigma^2(1+\sqrt c)}
		{\sigma^2\sqrt c\,(1+\sqrt c)}
		=\sigma^2 ,
		\]
		and $\chi(\lambda)\to\infty$ as $\lambda\to\infty$. Being
		continuous and strictly increasing by (i), $\chi$ is a bijection
		from $[\sigma^2\sqrt c,\infty)$ onto $[\sigma^2,\infty)$. 
		
		To invert, solve $\chi(\lambda)=x$, i.e.,
		\[
		\frac{\lambda(\lambda+\sigma^2)}{\lambda+c\sigma^2}=x
		\quad\Longleftrightarrow\quad
		\lambda(\lambda+\sigma^2) = x(\lambda+c\sigma^2)
		\quad\Longleftrightarrow\quad
		\lambda^2+\sigma^2\lambda = x\lambda + cx\sigma^2.
		\]
		Rearranging gives the quadratic equation
		\[
		\lambda^2 + (\sigma^2-x)\lambda - c\sigma^2x = 0,
		\]
		or equivalently, $\lambda^2-(x-\sigma^2)\lambda-c\sigma^2x=0$.
		By the quadratic formula, the unique nonnegative root is
		\[
		\lambda=\tfrac{1}{2}\Bigl\{(x-\sigma^2)+\sqrt{(x-\sigma^2)^2+4c\sigma^2x}\Bigr\},
		\]
		which is \eqref{eq:chiinv}.
		
		Setting $x=\sigma^2$ gives 
		\[
		\chi^{-1}(\sigma^2)=\tfrac{1}{2}\Bigl\{0+\sqrt{0+4c\sigma^4}\Bigr\}
		=\tfrac{1}{2}\cdot 2\sigma^2\sqrt{c}=\sigma^2\sqrt{c},
		\]
		which is the BBP critical spike. 
		
		As $c\downarrow 0$, for fixed $x\ge\sigma^2$, we have
		\[
		\sqrt{(x-\sigma^2)^2+4c\sigma^2x} = |x-\sigma^2|\sqrt{1+\frac{4c\sigma^2x}
			{(x-\sigma^2)^2}}\to |x-\sigma^2| = x-\sigma^2 
		\quad(\text{since }x\ge\sigma^2),
		\]
		so $\chi^{-1}(x)\to\tfrac{1}{2}\{(x-\sigma^2)+(x-\sigma^2)\}=x-\sigma^2$, 
		recovering the classical correction.
		
		{(iii)} From the second expression in \eqref{eq:chiprime},
		\[
		\chi'(\lambda) = 1 + \frac{c(1-c)\sigma^4}{(\lambda+c\sigma^2)^2},
		\]
		we see that $\lambda\mapsto\chi'(\lambda)$ is strictly decreasing 
		if $c<1$ (since $c(1-c)>0$ and the fraction decreases in $\lambda$),
		constant if $c=1$ (since $c(1-c)=0$),
		and strictly increasing if $c>1$ (since $c(1-c)<0$ and the negative 
		fraction increases in $\lambda$).
		If $c\le1$
		the infimum over $[\sigma^2\sqrt c,\infty)$ is therefore the limit
		at $\lambda=\infty$, namely $1$. If $c>1$ the infimum is attained
		at the left endpoint, and
		\begin{align*}
			\chi'(\sigma^2\sqrt c)
			&=1+\frac{c(1-c)\sigma^4}{[(\sigma^2\sqrt c)+c\sigma^2]^2}
			=1+\frac{c(1-c)\sigma^4}{\sigma^4(\sqrt c+c)^2}\\
			&=1+\frac{c(1-c)}{(\sqrt c)^2(1+\sqrt c)^2}
			=1+\frac{c(1-c)}{c(1+\sqrt c)^2}
			=1+\frac{1-c}{(1+\sqrt c)^2}\\
			&=1+\frac{(1-\sqrt c)(1+\sqrt c)}{(1+\sqrt c)^2}
			=1+\frac{1-\sqrt c}{1+\sqrt c}\\
			&=\frac{1+\sqrt c + 1-\sqrt c}{1+\sqrt c}
			=\frac{2}{1+\sqrt c},
		\end{align*}
		
		which is $<1$ exactly when $c>1$. In both cases
		$\inf_{\lambda\ge\sigma^2\sqrt c}\chi'
		=\min\{1,2/(1+\sqrt c)\}$, and since $\chi^{-1}$ is
		differentiable with $(\chi^{-1})'(x)=1/\chi'(\chi^{-1}(x))$, its
		Lipschitz constant is the reciprocal, $L_c=\max\{1,
		(1+\sqrt c)/2\}$, proving \eqref{eq:chilip}. 
		For the sensitivity to $\sigma^2$, differentiate \eqref{eq:chiinv} with respect 
		to $\sigma^2$, treating $x$ and $c$ as constants. 
		Let $D:=(x-\sigma^2)^2+4c\sigma^2x$, so
		\[
		\chi^{-1}(x;\sigma^2,c) = \tfrac{1}{2}(x-\sigma^2) + \tfrac{1}{2}\sqrt{D}.
		\]
		Then
		\begin{align*}
			\frac{\partial\chi^{-1}}{\partial\sigma^2}
			&= \tfrac{1}{2}(-1) + \tfrac{1}{2}\cdot\frac{1}{2\sqrt{D}}\cdot\frac{\partial D}{\partial\sigma^2}\\
			&= -\tfrac{1}{2} + \frac{1}{4\sqrt{D}}\cdot[2(x-\sigma^2)(-1) + 4cx]\\
			&= -\tfrac{1}{2} + \frac{-2(x-\sigma^2) + 4cx}{4\sqrt{D}}\\
			&= \tfrac{1}{2}\Bigl\{-1 + \frac{-(x-\sigma^2) + 2cx}{\sqrt{D}}\Bigr\}.
		\end{align*}
		Taking absolute values and using the triangle inequality,
		\begin{align*}
			\Bigl|\frac{\partial\chi^{-1}}{\partial\sigma^2}\Bigr|
			&\le\frac{1}{2}\Bigl\{1+\Bigl|\frac{-(x-\sigma^2)+2cx}{\sqrt{D}}\Bigr|\Bigr\}
			\le\frac{1}{2}\Bigl\{1+\frac{|x-\sigma^2|}{\sqrt{D}}+\frac{2cx}{\sqrt{D}}\Bigr\}.
		\end{align*}
		Since $D\ge(x-\sigma^2)^2$, we have $\frac{|x-\sigma^2|}{\sqrt{D}}\le 1$.
		Since $D\ge 4c\sigma^2x$, we have $\frac{2cx}{\sqrt{D}}\le\frac{2cx}{2\sigma\sqrt{cx}}=
		\frac{\sqrt{cx}}{\sigma}$. Therefore,
		\[
		\Bigl|\frac{\partial\chi^{-1}}{\partial\sigma^2}\Bigr|
		\le\frac{1}{2}\Bigl\{1+1+\frac{\sqrt{cx}}{\sigma}\Bigr\}
		= 1 + \frac{\sqrt{cx}}{2\sigma}.
		\]		
		On the stated set, $x\le\lambda_{\max}+\sigma^2_{\max}$ and
		$\sigma\ge\sigma_{\min}$, so 
		\[
		1 + \frac{\sqrt{cx}}{2\sigma}
		\le 1 + \frac{\sqrt{c(\lambda_{\max}+\sigma^2_{\max})}}{2\sigma_{\min}}
		\le C(1+\sqrt c)\le C(1+c),
		\]
		where $C$ depends only on $\lambda_{\max}$, $\sigma^2_{\min}$, and $\sigma^2_{\max}$.
		
		{(iv)} For a general profile, $\chi_{\ell,t}
		=\sigma^2_t+\lambda a_{\ell,t}(\lambda)$ with $a_{\ell,t}$
		strictly increasing and positive by Step~5 of the proof of
		Lemma~\ref{lem:maps}. Hence
		$\chi'_{\ell,t}=a_{\ell,t}+\lambda a'_{\ell,t}\ge a_{\ell,t}>0$,
		so $\chi_{\ell,t}$ is strictly increasing, and it maps
		the BBP critical value $\lambda^{\mathrm c}_{\ell,t}$ to $\sigma^2_t$ (because
		$a_{\ell,t}(\lambda^{\mathrm c}_{\ell,t}+)=0$ by Lemma~\ref{lem:maps})
		and diverges as $\lambda\to\infty$ (because $\chi_{\ell,t}(\lambda)\ge
		\sigma^2_t+\lambda a_{\ell,t}(\lambda_0)$ for
		$\lambda\ge\lambda_0>\lambda^{\mathrm c}_{\ell,t}$). Continuity
		then gives a bijection onto $[\sigma^2_t,\infty)$. On
		$\lambda\ge\lambda_{t,k}$, monotonicity of $a_{\ell,t}$ and
		Corollary~\ref{cor:aunder} give
		$\chi'_{\ell,t}\ge a_{\ell,t}(\lambda_{t,k})\ge\underline a>0$,
		so the inverse is well conditioned and computable by a
		one-dimensional monotone root search (bisection converges
		geometrically because $\chi_{\ell,t}$ is strictly increasing and
		Lipschitz on compacts).
	\end{proof}
	
	\begin{proof}[Proof of Lemma~\ref{lem:sigma}]
		Abbreviate $n:=n^\ell_t$ and condition on
		$\mathcal F_{\ell,t-1}$ throughout, so that the directions
		$\bm u_{\ell,t-1,j}$ are fixed.
		
		{Bias.} By Lemma~\ref{lem:decomp}, $\mathbb E(r_{\ell,t,j}\mid\mathcal F_{\ell,t-1})=\bar r_{\ell,t,j}$
		and $\mathbb E(\operatorname{tr}\hat{\bm S}_{\ell,t}\mid\mathcal F_{\ell,t-1})=\operatorname{tr}\bm\Sigma_t$. 
		Therefore,
		\begin{align*}
			\mathbb E\Bigl(\operatorname{tr}\hat{\bm S}_{\ell,t}
			-\sum_{j\le k}r_{\ell,t,j}\ \Big|\ \mathcal F_{\ell,t-1}\Bigr)
			&=\operatorname{tr}\bm\Sigma_t-\sum_{j\le k}\bar r_{\ell,t,j}.
		\end{align*}
		By \eqref{eq:spiked}, $\operatorname{tr}\bm\Sigma_t=p\sigma^2_t+\sum_{m=1}^{k}\lambda_{t,m}$.
		By the definition of $\bar r_{\ell,t,j}$ from Lemma~\ref{lem:decomp},
		\[
		\sum_{j=1}^{k}\bar r_{\ell,t,j}
		=\sum_{j=1}^{k}\Bigl(\sigma^2_t+\sum_{m=1}^{k}\lambda_{t,m}
		\langle\bm u_{\ell,t-1,j},\bm v_{t,m}\rangle^2\Bigr)
		=k\sigma^2_t+\sum_{m=1}^{k}\lambda_{t,m}\sum_{j=1}^{k}
		\langle\bm u_{\ell,t-1,j},\bm v_{t,m}\rangle^2.
		\]
		Therefore,
		\begin{align*}
			\operatorname{tr}\bm\Sigma_t-\sum_{j\le k}\bar r_{\ell,t,j}
			&=(p\sigma^2_t+\sum_{m}\lambda_{t,m})-\Bigl(k\sigma^2_t+\sum_m\lambda_{t,m}\sum_j
			\langle\bm u_{\ell,t-1,j},\bm v_{t,m}\rangle^2\Bigr)\\
			&=(p-k)\sigma^2_t
			+\sum_{m=1}^{k}\lambda_{t,m}
			\Bigl(1-\sum_{j=1}^{k}
			\langle\bm u_{\ell,t-1,j},\bm v_{t,m}\rangle^2\Bigr),
		\end{align*}

		To bound this expectation, note that each bracket 
		$\bigl(1-\sum_{j=1}^{k}\langle\bm u_{\ell,t-1,j},\bm v_{t,m}\rangle^2\bigr)$ 
		lies in $[0,1]$ because, writing $\bm U_{\ell,t-1}=[\bm u_{\ell,t-1,1},\ldots,\bm u_{\ell,t-1,k}]$,
		\[
		\sum_{j=1}^{k}\langle\bm u_{\ell,t-1,j},\bm v_{t,m}\rangle^2
		=\|\bm U_{\ell,t-1}^\top\bm v_{t,m}\|^2\in[0,1]
		\]
		(since the columns of $\bm U_{\ell,t-1}$ are orthonormal and $\|\bm v_{t,m}\|=1$). 
		Therefore,
		\[
		0 \le \mathbb E\Bigl(\operatorname{tr}\hat{\bm S}_{\ell,t}
		-\sum_{j\le k}r_{\ell,t,j}\ \Big|\ \mathcal F_{\ell,t-1}\Bigr)
		\le (p-k)\sigma^2_t + k\lambda_{\max},
		\]
		so the bias of $\check\sigma^2_{\ell,t}=(p-k)^{-1}(\operatorname{tr}\hat{\bm S}_{\ell,t}-\sum_j r_{\ell,t,j})$ 
		satisfies
		\[
		\Bigl|\mathbb E(\check\sigma^2_{\ell,t}\mid\mathcal F_{\ell,t-1})-\sigma^2_t\Bigr|
		\le \frac{k\lambda_{\max}}{p-k}=O\Bigl(\frac{k}{p}\Bigr).
		\]

		{Fluctuation.} Write $\operatorname{tr}\hat{\bm S}_{\ell,t}
		=n^{-1}\sum_i\bm z_i^\top\bm\Sigma_t\bm z_i$. For $\bm z$ with
		independent standardized entries (i.e., $\mathbb E z_m=0$, 
		$\mathbb E z_m^2=1$, $\mathbb E z_m^4=3+{K}_4$) and any symmetric $A$,
		\[
		\operatorname{Var}(\bm z^\top A\bm z)
		=\mathbb E(\bm z^\top A\bm z)^2-[\mathbb E(\bm z^\top A\bm z)]^2
		=2\operatorname{tr}(A^2)+{K}_4\sum_{m=1}^{p}A_{mm}^2
		\]
		(by the same computation as in Lemma~\ref{lem:decomp}). 
		With $A=\bm\Sigma_t$ this is
		\begin{align*}
			2\operatorname{tr}(\bm\Sigma_t^2)+{K}_4\sum_{m=1}^{p}(\bm\Sigma_t)_{mm}^2
			&= O(p).
		\end{align*}
		Indeed, by \eqref{eq:spiked}, 
		$\bm\Sigma_t=\sigma^2_t\bm I_p+\sum_{m=1}^{k}\lambda_{t,m}\bm v_{t,m}\bm v_{t,m}^\top$, so
		\begin{align*}
			\operatorname{tr}(\bm\Sigma_t^2)
			&=\operatorname{tr}\Bigl[\Bigl(\sigma^2_t\bm I_p+\sum_{m}\lambda_{t,m}
			\bm v_{t,m}\bm v_{t,m}^\top\Bigr)^2\Bigr]\\
			&=p\sigma^4_t + 2\sigma^2_t\sum_m\lambda_{t,m} + \sum_m\lambda_{t,m}^2
			=p\sigma^4_t+O(1),
		\end{align*}
		by Assumption~\ref{assum:bbp} (since $k$ and $\lambda_{t,m}$ are bounded).
		For the diagonal entries,
		\[
		(\bm\Sigma_t)_{mm}
		=\sigma^2_t+\sum_{j=1}^{k}\lambda_{t,j}|\bm v_{t,j,m}|^2
		\le\sigma^2_t+k\lambda_{\max}\max_{j}\|\bm v_{t,j}\|^2_\infty
		=O(1)
		\]
		by Assumptions~\ref{assum:noise}, \ref{assum:bbp}, and~\ref{assum:deloc}. 
		Therefore, $\sum_m(\bm\Sigma_t)_{mm}^2\le p\cdot O(1)=O(p)$.
		
		Combining the above bounds, we have 
		$\operatorname{Var}(\operatorname{tr}\hat{\bm S}_{\ell,t}\mid
		\mathcal F_{\ell,t-1})=O(p/n)$, and after division by $p-k$ its
		contribution is $O_{\mathbb P}(\sqrt{p/n}/p)
		=O_{\mathbb P}((pn)^{-1/2})$. 
		For the $r_{\ell,t,j}$ terms, by Lemma~\ref{lem:decomp}, 
		$r_{\ell,t,j}=\bar r_{\ell,t,j}+\varpi_{\ell,t,j}$
		where $\bar r_{\ell,t,j}$ is $\mathcal F_{\ell,t-1}$-measurable, so
		$\operatorname{Var}(r_{\ell,t,j}\mid\mathcal F_{\ell,t-1})
		=\operatorname{Var}(\varpi_{\ell,t,j}\mid\mathcal F_{\ell,t-1})=O(1/n)$ by \eqref{eq:var}.
		Since the $r_{\ell,t,j}$ terms share the same fresh batch, they are not independent, but by
		Cauchy-Schwarz,
		\[
		\operatorname{Var}\Bigl(\sum_{j=1}^{k}r_{\ell,t,j}\mid\mathcal F_{\ell,t-1}\Bigr)
		\le k\sum_{j=1}^{k}\operatorname{Var}(r_{\ell,t,j}\mid\mathcal F_{\ell,t-1})
		=O(k^2/n).
		\]
		After division by $p-k$, its contribution to $\operatorname{Var}(\check\sigma^2_{\ell,t}\mid\mathcal F_{\ell,t-1})$ is
		\[
		\frac{O(k^2/n)}{(p-k)^2}=O\Bigl(\frac{k^2}{p^2 n}\Bigr).
		\]
		By Chebyshev's inequality, this contributes $O_{\mathbb P}(k/(p\sqrt{n}))$ 
		to the standard deviation of $\check\sigma^2_{\ell,t}$.
		Since $k/(p\sqrt{n})=k(pn)^{-1/2}p^{-1/2}$, this is dominated by the trace 
		term $O_{\mathbb P}((pn)^{-1/2})$ whenever $k=O(\sqrt{p})$, which is implied 
		by $k\sqrt{n^\ell_t}=o(p)$ for $n^\ell_t\gtrsim 1$.
		Combining the bias bound $O(k/p)$ and the fluctuation bound 
		(standard deviation) $O_{\mathbb P}((pn)^{-1/2})$, we obtain 
		$\check\sigma^2_{\ell,t}-\sigma^2_t
		=O_{\mathbb P}((p\,n^\ell_t)^{-1/2})+O(k/p)$.
		
		{Consequence.} The stated result follows. Multiplying by $\sqrt{n^\ell_t}$, 
		\[
		\sqrt{n^\ell_t}|\check\sigma^{2}_{\ell,t}-\sigma_t^{2}|
		=O_{\mathbb P}\Bigl(\frac{\sqrt{n^\ell_t}}{\sqrt{p n^\ell_t}}\Bigr)
		+O\Bigl(\frac{k\sqrt{n^\ell_t}}{p}\Bigr)
		=O_{\mathbb P}(p^{-1/2})+O\Bigl(\frac{k\sqrt{n^\ell_t}}{p}\Bigr)
		=o_{\mathbb P}(1)
		\]
		precisely when $k\sqrt{n^\ell_t}=o(p)$, proving the second assertion of the lemma.
	\end{proof}
	
	\begin{proof}[Proof of Theorem~\ref{thm:local_clt}]
		Fix $j\le k$, abbreviate $n:=n^\ell_t$,
		$\lambda:=\lambda_{t,j}$, $\chi:=\chi_{\ell,t}$,
		$r:=r_{\ell,t,j}$, and write
		$\mathrm{err}_{\ell,t}:=(N^{\mathrm{eff}}_{\ell,t})^{-1/2}
		+\alpha^{\mathrm{lag}}_{\ell,t}+\zeta_t+\Delta_{\ell,t}
		+(\xi^\star)^t$, so
		that condition \eqref{eq:clt_cond} reads
		$\sqrt n\,\mathrm{err}_{\ell,t}\to0$.
		
		\medskip\noindent
		\textbf{Step 1 (the truncation is inactive).}
		By Theorem~\ref{thm:attenuation},
		$r=\chi(\lambda)+O_{\mathbb P}(n^{-1/2}+\mathrm{err}_{\ell,t})
		\to\chi(\lambda)$ in probability, and
		$\chi(\lambda)-\sigma^2_t=\lambda a_{\ell,t}(\lambda)
		\ge\lambda_{t,k}\underline a>0$ uniformly by
		Corollary~\ref{cor:aunder} and Assumption~\ref{assum:bbp}, while
		$\check\sigma^2_{\ell,t}\to\sigma^2_t$ by
		Lemma~\ref{lem:sigma}. Hence
		$\mathbb P(r>\check\sigma^2_{\ell,t})\to1$ and the truncation in
		\eqref{eq:local_debias} may be ignored.
		
		\medskip\noindent
		\textbf{Step 2 (removing the plug-in of $\sigma^2$).}
		By Lemma~\ref{lem:inverse}(iii) the map $\sigma^2\mapsto
		\chi^{-1}(x;c,\sigma^2)$ is Lipschitz with constant $C(1+c)$ uniformly over 
		the relevant compact set (determined by Assumptions~\ref{assum:noise}--\ref{assum:hd}), so
		\[
		\chi^{-1}\bigl(r;c_{\ell,t-1},\check\sigma^2_{\ell,t}\bigr)
		-\chi^{-1}\bigl(r;c_{\ell,t-1},\sigma^2_t\bigr)
		=O_{\mathbb P}\bigl(|\check\sigma^2_{\ell,t}-\sigma^2_t|\bigr)
		=o_{\mathbb P}(n^{-1/2})
		\]
		by Lemma~\ref{lem:sigma} and $k\sqrt n=o(p)$. The aspect ratio needs
		no such step. It is not a plug-in, since $c_{\ell,t-1}$ is
		deterministic, known at the node, and is the ratio carried by
		$\chi_{\ell,t}$ itself (Remark~\ref{rem:indexshift}).
		
		\medskip\noindent
		\textbf{Step 3 (linearization).}
		By Lemma~\ref{lem:inverse}(ii)--(iii), $\chi^{-1}$ is
		continuously differentiable with derivative
		$(\chi^{-1})'(x)=1/\chi'(\chi^{-1}(x))$, bounded above by $L_{\bar c}$ and
		below by $\{\sup\chi'\}^{-1}>0$ on the relevant range, and, as
		computed in the proof of Theorem~\ref{thm:ordering} below,
		$(\chi^{-1})''$ is uniformly bounded there as well. 
		
		By Taylor's theorem, for any $x$ in a neighborhood of $\chi(\lambda)$,
		\[
		\chi^{-1}(x) = \chi^{-1}(\chi(\lambda)) + (\chi^{-1})'(\chi(\lambda))\cdot(x-\chi(\lambda))
		+ \tfrac{1}{2}(\chi^{-1})''(\xi)\cdot(x-\chi(\lambda))^2
		\]
		for some $\xi$ between $x$ and $\chi(\lambda)$. Since $\chi^{-1}(\chi(\lambda))=\lambda$ and
		$(\chi^{-1})'(\chi(\lambda))=1/\chi'(\lambda)$, a second-order Taylor
		expansion around $\chi(\lambda)$ therefore gives
		\begin{equation}
			\label{eq:lin}
			\check\lambda_{\ell,t,j}-\lambda
			=\frac{r-\chi(\lambda)}{\chi'(\lambda)}
			+O_{\mathbb P}\bigl(\{r-\chi(\lambda)\}^2\bigr)
			+o_{\mathbb P}(n^{-1/2})
			=\frac{r-\chi(\lambda)}{\chi'(\lambda)}
			+O_{\mathbb P}(n^{-1})+o_{\mathbb P}(n^{-1/2}).
		\end{equation}
		By Theorem~\ref{thm:attenuation} and Lemma~\ref{lem:decomp},
		$r-\chi(\lambda)=\varpi_{\ell,t,j}
		+O_{\mathbb P}(\mathrm{err}_{\ell,t})$, and
		$\sqrt n\,O_{\mathbb P}(\mathrm{err}_{\ell,t})=o_{\mathbb P}(1)$
		by \eqref{eq:clt_cond}. Hence
		\begin{equation}
			\label{eq:lin2}
			\sqrt n\bigl(\check\lambda_{\ell,t,j}-\lambda_{t,j}\bigr)
			=\frac{\sqrt n\,\varpi_{\ell,t,j}}
			{\chi'_{\ell,t}(\lambda_{t,j})}+o_{\mathbb P}(1).
		\end{equation}
		
		\medskip\noindent
		\textbf{Step 4 (conditional CLT for $\varpi$).}
		Conditionally on $\mathcal F_{\ell,t-1}$,
		$\sqrt n\,\varpi_{\ell,t,j}
		=n^{-1/2}\sum_{i\le n}\{(\bm b^\top\bm z_i)^2-\|\bm b\|^2\}$ is a
		normalized sum of i.i.d.\ centred variables with variance
		$s^2_n:=2\bar r^2_{\ell,t,j}+{K}_4\sum_mb_m^4$, by
		Lemma~\ref{lem:decomp}. 
		Lyapunov's condition holds with exponent
		$2+\epsilon_1/2$. Indeed, by the Marcinkiewicz--Zygmund and Rosenthal
		inequalities and Assumption~\ref{assum:moment},
		\[
		\mathbb E\bigl|(\bm b^\top\bm z)^2-\|\bm b\|^2\bigr|
		^{2+\epsilon_1/2}
		\le C\,\mathbb E|\bm b^\top\bm z|^{4+\epsilon_1}
		\le C'\|\bm b\|^{4+\epsilon_1}
		\le C''(\lambda_{\max}+\sigma^2_{\max})^{2+\epsilon_1/2},
		\]
		uniformly in $p$. Lyapunov's condition requires
		\[
		\frac{1}{n\cdot s_n^{2+\epsilon_1/2}}\sum_{i=1}^{n}\mathbb E\bigl|(\bm b^\top\bm z_i)^2-\|\bm b\|^2\bigr|^{2+\epsilon_1/2}
		=\frac{n\cdot C''(\lambda_{\max}+\sigma^2_{\max})^{2+\epsilon_1/2}}{n\cdot s_n^{2+\epsilon_1/2}}
		\to 0
		\]
		as $n\to\infty$, since $s^2_n\ge2\sigma^4_{\min}(1+o(1))$ is bounded away from zero.
		Therefore, 
		$\sqrt n\,\varpi_{\ell,t,j}\mid\mathcal F_{\ell,t-1}
		\Rightarrow\mathcal N(0,s^2_\infty)$, where by
		Lemma~\ref{lem:decomp} $\sum_mb_m^4\to0$ and, by \eqref{eq:barr},
		$\bar r_{\ell,t,j}\to\chi_{\ell,t}(\lambda_{t,j})$ in probability,
		so $s^2_n\to2\chi_{\ell,t}(\lambda_{t,j})^2$. Since the limit is
		deterministic, the conditional characteristic function converges
		in probability to that of
		$\mathcal N(0,2\chi_{\ell,t}(\lambda_{t,j})^2)$, and dominated
		convergence upgrades this to unconditional weak convergence.
		Combining with \eqref{eq:lin2} and Slutsky's theorem,
		\[
		\sqrt{n^\ell_t}\bigl(\check\lambda_{\ell,t,j}-\lambda_{t,j}\bigr)
		\Longrightarrow
		\mathcal N\Bigl(0,\ \frac{2\chi_{\ell,t}(\lambda_{t,j})^2}
		{\chi'_{\ell,t}(\lambda_{t,j})^2}\Bigr),
		\]
		which is \eqref{eq:local_clt}. In the canonical case,
		\eqref{eq:chi_closed} and \eqref{eq:chiprime} give
		\[
		\frac{\chi(\lambda)}{\chi'(\lambda)}
		=\frac{\lambda(\lambda+\sigma^2)}{\lambda+c\sigma^2}\cdot
		\frac{(\lambda+c\sigma^2)^2}
		{\lambda^2+2c\sigma^2\lambda+c\sigma^4}
		=\frac{\lambda(\lambda+\sigma^2)(\lambda+c\sigma^2)}
		{\lambda^2+2c\sigma^2\lambda+c\sigma^4},
		\]
		whose square, multiplied by $2$, is the explicit expression in
		\eqref{eq:vartheta}.
		
		\medskip\noindent
		\textbf{Step 5 (joint behaviour).}
		For $j\ne j'$, conditionally on $\mathcal F_{\ell,t-1}$,
		\[
		\operatorname{Cov}\bigl(\varpi_{\ell,t,j},
		\varpi_{\ell,t,j'}\mid\mathcal F_{\ell,t-1}\bigr)
		=\frac{1}{n}\Bigl\{2\langle\bm b_j,\bm b_{j'}\rangle^2
		+{K}_4\sum_mb_{j,m}^2b_{j',m}^2\Bigr\},
		\]
		by the same fourth-moment computation as in
		Lemma~\ref{lem:decomp}. Since
		$$\langle\bm b_j,\bm b_{j'}\rangle
		=\bm u_{\ell,t-1,j}^\top\bm\Sigma_t\bm u_{\ell,t-1,j'}
		=\sum_m\lambda_{t,m}\langle\bm u_{\ell,t-1,j},\bm v_{t,m}\rangle
		\langle\bm u_{\ell,t-1,j'},\bm v_{t,m}\rangle. $$ 
		For $j\ne j'$, each product $\langle\bm u_{\ell,t-1,j},\bm v_{t,m}\rangle
		\langle\bm u_{\ell,t-1,j'},\bm v_{t,m}\rangle$ involves either
		$\langle\bm u_{\ell,t-1,j},\bm v_{t,m}\rangle$ with $m\ne j$ or
		$\langle\bm u_{\ell,t-1,j'},\bm v_{t,m}\rangle$ with $m\ne j'$
		(at least one is an off-diagonal alignment), which is
		$O_{\mathbb P}((N^{\mathrm{eff}}_{\ell,t})^{-1/2})$ by
		Lemma~\ref{lem:maps}. Therefore, $\langle\bm b_j,\bm b_{j'}\rangle
		=O_{\mathbb P}((N^{\mathrm{eff}}_{\ell,t})^{-1/2})$
		and hence
		cross-covariances of order
		$(nN^{\mathrm{eff}}_{\ell,t})^{-1}$, i.e.\
		$O_{\mathbb P}((N^{\mathrm{eff}}_{\ell,t})^{-1})$ relative to the
		diagonal entries, which are of order $n^{-1}$. The Cram\'er--Wold
		device applied to the same Lyapunov argument gives joint normality
		with asymptotically diagonal covariance. Finally, given
		$\mathcal F_{t-1}$ the fresh batches
		$\{\bm X_{\ell,t,i}\}_{i\le n^\ell_t}$ are independent across
		$\ell\in\mathcal A_t$ (Assumptions~\ref{assum:moment}
		and~\ref{assum:design}), and $\check{\bm\lambda}_{\ell,t}$ is a
		measurable function of $\mathcal F_{\ell,t-1}$ and of the $\ell$-th
		fresh batch alone. Hence the limits are jointly independent across
		nodes.
	\end{proof}
	
	\begin{rem}
		\label{rem:varbounds}
		The asymptotic variance $\vartheta^2_{\ell,t,j}$ in \eqref{eq:vartheta}
		is bounded above and below by constants that depend only on the model
		parameters and not on the particular node, spike index, or time step.
		By \eqref{eq:sandwich} and Lemma~\ref{lem:inverse}(ii),
		$\sigma^2_{\min}\le\chi_{\ell,t}(\lambda_{t,j})\le\lambda_{\max}+\sigma^2_{\max}$.
		By \eqref{eq:chiprime}, the derivative satisfies
		$\min\{1,2/(1+\sqrt{c_{\ell,t-1}})\}
		\le\chi'_{\ell,t}(\lambda_{t,j})
		\le\max\{1,2/(1+\sqrt{c_{\ell,t-1}})\}\le2$.
		Combining these estimates gives, under
		Assumptions~\ref{assum:bbp}--\ref{assum:noise} and~\ref{assum:hd},
		\begin{equation}
			\label{eq:varthetabounds}
			\frac{\sigma^4_{\min}}{2}
			\;\le\;
			2\sigma^4_{\min}
			\min\Bigl\{1,\tfrac{(1+\sqrt{c_{\ell,t-1}})^{2}}{4}\Bigr\}
			\;\le\; \vartheta^2_{\ell,t,j}
			\;\le\; 2(\lambda_{\max}+\sigma^2_{\max})^2
			\max\Bigl\{1,\tfrac{(1+\sqrt{\bar c})^2}{4}\Bigr\}.
		\end{equation}
		The outer lower bound $\sigma^4_{\min}/2$ follows from the estimate
		$\sup_\lambda\chi'\le2$ and holds without any lower bound on the aspect ratio.
		The middle expression is the sharper constant that reflects the exact value
		of $c_{\ell,t-1}$.
		These bounds confirm that $\vartheta^2_{\ell,t,j}$ is bounded away from
		both zero and infinity uniformly across all nodes, spikes, and time steps
		permitted by the assumptions, a property that is used in the aggregation
		analysis of Section~\ref{sec:3}.
	\end{rem}
	
	\subsection{Comparison of aggregation orderings\label{sec:a.vs.b}}
	
	Two architectures are available for combining local information. Let
	$\{\omega_{\ell,t}\}_{\ell\in\mathcal A_t}$ be aggregation weights
	with $\omega_{\ell,t}\ge0$, $\sum_\ell\omega_{\ell,t}=1$, and write
	$c^{\omega}_t=\sum_\ell\omega_{\ell,t}c_{\ell,t-1}$. Define
	\begin{align}
		\text{(A) correct-then-aggregate:}\qquad
		&\check\lambda_{t,j}^{\mathrm{A}}
		=\sum_{\ell\in\mathcal A_t}\omega_{\ell,t}\,
		\chi_{\ell,t}^{-1}(r_{\ell,t,j}),
		\label{eq:orderA}\\
		\text{(B) aggregate-then-correct:}\qquad
		&\check\lambda_{t,j}^{\mathrm{B}}
		=\bar\chi_{t}^{-1}\Bigl(
		\sum_{\ell\in\mathcal A_t}\omega_{\ell,t}r_{\ell,t,j}\Bigr),
		\label{eq:orderB}
	\end{align}
	where $\bar\chi_t:=\chi(\cdot\,;c^{\omega}_t,\sigma^2_t)$. Ordering~(B) is
	the natural analogue of the classical practice of averaging first and
	correcting once at the server. 
	
	\begin{theo}[Ordering comparison]
		\label{thm:ordering}
		Assume the conditions of Theorem~\ref{thm:local_clt} for every
		$\ell\in\mathcal A_t$, work in the canonical case, and let
		\begin{equation}
			\label{eq:dispersion}
			\mathcal D_t
			:=\sum_{\ell\in\mathcal A_t}
			\omega_{\ell,t}\bigl(c_{\ell,t-1}-c^{\omega}_t\bigr)^{2}
		\end{equation}
		denote the weighted dispersion of the local aspect ratios. Then for
		each fixed $j\le k$:
		\begin{enumerate}[label=\textup{(\roman*)}]
			\item \textup{(Ordering A)}
			$\mathbb E\bigl(\check\lambda_{t,j}^{\mathrm{A}}\bigr)
			-\lambda_{t,j}
			=O\bigl(\sum_{\ell}\omega_{\ell,t}/n_t^\ell\bigr)
			+O(\max_\ell|b^{\mathrm{lag}}_{\ell,t}|)\to0$, and
			\[
			\check\lambda_{t,j}^{\mathrm{A}}-\lambda_{t,j}
			=O_{\mathbb P}\Bigl(
			\bigl(\textstyle\sum_{\ell}\omega_{\ell,t}^{2}
			\vartheta_{\ell,t,j}^{2}/n_t^\ell\bigr)^{1/2}\Bigr)
			+O(\max_\ell|b^{\mathrm{lag}}_{\ell,t}|). 
			\]
			\item \textup{(Ordering B)} there is a {non-vanishing} bias
			\begin{align}
				\label{eq:biasB}
				\mathbb E\bigl(\check\lambda_{t,j}^{\mathrm{B}}\bigr)
				-\lambda_{t,j}
				=&\frac{\sigma_t^{4}\lambda_{t,j}
					(\lambda_{t,j}+\sigma_t^{2})}
				{(\lambda_{t,j}+c^{\omega}_t\sigma_t^{2})
					\bigl(\lambda_{t,j}^{2}
					+2c^{\omega}_t\sigma_t^{2}\lambda_{t,j}
					+c^{\omega}_t\sigma_t^{4}\bigr)}
				\ \mathcal D_t \\
				& +\ o(\mathcal D_t)
				+\ O\Bigl(\sum_{\ell}\omega_{\ell,t}^{2}/n_t^\ell\Bigr),\nonumber
			\end{align}
			which is strictly positive whenever $\mathcal D_t>0$ and does
			not decrease as $A_t\to\infty$, $t\to\infty$ or
			$\min_\ell n_t^\ell\to\infty$.
			\item \textup{(Homogeneous case)} If
			$c_{\ell,t-1}\equiv c^{\omega}_t$
			then $\mathcal D_t=0$ and the two orderings are first-order
			equivalent. Ordering~B then has the smaller second-order
			nonlinearity bias, $O(1/N_t)$ against $O(A_t/N_t)$ for
			Ordering~A.
		\end{enumerate}
	\end{theo}
	
	\begin{rem}
		\label{rem:sec3link}
		Theorem~\ref{thm:ordering} is decisive because $\mathcal D_t>0$ is
		the rule rather than the exception in the setting of
		Section~\ref{sec:2}. Asynchronous activation, unequal batch sizes
		and node-specific memory histories make
		$N^{\mathrm{eff}}_{\ell,t}$, hence $c_{\ell,t-1}$, genuinely
		heterogeneous. The mechanism behind \eqref{eq:biasB} is that
		$c\mapsto\chi(\lambda;c,\sigma^2)$ is strictly convex,
		\[
		\frac{\partial^2\chi}{\partial c^2}
		=\frac{2\lambda(\lambda+\sigma^2)\sigma^4}
		{(\lambda+c\sigma^2)^3}>0,
		\]
		so averaging attenuated statistics with heterogeneous attenuation
		levels and then applying a single pooled correction produces a
		Jensen gap of order $\mathcal D_t$ that no feasible choice of
		pooled ratio can remove. The qualification is needed because
		$c\mapsto\chi(\lambda;c,\sigma^2)$ is strictly decreasing, so for
		each fixed $j$ there does exist a ratio $c^{\dagger}_{t,j}$ solving
		$\chi(\lambda_{t,j};c^{\dagger}_{t,j},\sigma^2_t)
		=\sum_\ell\omega_{\ell,t}\chi(\lambda_{t,j};c_{\ell,t-1},
		\sigma^2_t)$. 
		But it depends on the unknown $\lambda_{t,j}$ and differs across
		$j$, so it is neither computable at the server nor simultaneously
		valid for the $k$ coordinates transmitted in one round. For any
		$j$-independent, data-free pooled ratio the bias is $O(1)$. It is not reduced by adding nodes,
		lengthening the stream, or enlarging batches, and therefore
		dominates the $O(N_t^{-1/2})$ stochastic error asymptotically.
		Ordering~A avoids it entirely at no communication cost, because
		$c_{\ell,t-1}$ is deterministic and locally known and
		$\check\sigma^2_{\ell,t}$ is locally computable. Each node still
		transmits exactly $k$ scalars. The advantage of Ordering~B in
		part~(iii) is a second-order effect confined to the homogeneous
		case, and is dominated by the stochastic error whenever
		$A_t=o(\sqrt{N_t})$. Accordingly the framework is organised as
		{correct-then-aggregate}, and Section~\ref{sec:3} is
		implemented in this order.
	\end{rem}
	
	\begin{rem}
		\label{rem:twobenefits}
		Beyond the bias comparison, Ordering~A confers two structural
		advantages exploited in Section~\ref{sec:4.2}. 
		First, after local correction all transmitted statistics are
		centred at the {same} quantity $\lambda_{t,j}$, so the
		adaptive fluctuation-based weights measure genuine sampling
		variability rather than systematic inter-node differences in
		attenuation. Under Ordering~B a node with large $c_{\ell,t-1}$ would
		be penalised for a deterministic bias, confounding bias with
		variance. Second, Theorem~\ref{thm:local_clt} identifies the exact
		local variance $\vartheta^2_{\ell,t,j}$, computable at the node
		from $(\check\lambda_{\ell,t,j},c_{\ell,t-1},
		\check\sigma^2_{\ell,t},n_t^\ell)$, providing a concrete efficiency
		benchmark. Ordering~B admits no such benchmark, since $\bar\chi_t$
		is the correct transfer map for no individual node. Finally, local
		correction accommodates node-specific noise levels
		$\sigma^2_{\ell,t}$ without modification, whereas Ordering~B
		presupposes a single global $\sigma^2_t$ for $\bar\chi_t$ to be
		well defined at all.
	\end{rem}

	\begin{proof}[Proof of Theorem~\ref{thm:ordering}]
		Throughout we work in the canonical case, fix $j\le k$, write
		$\lambda:=\lambda_{t,j}$, $\sigma^2:=\sigma^2_t$,
		$\chi_\ell(\cdot):=\chi(\cdot\,;c_{\ell,t-1},\sigma^2)$,
		$\bar\chi:=\chi(\cdot\,;c^{\omega}_t,\sigma^2)$, and
		$b^{\mathrm{lag}}_{\ell,t}
		=O(\alpha^{\mathrm{lag}}_{\ell,t}+\zeta_t+\Delta_{\ell,t}
		+(\xi^\star)^t)$
		for the deterministic lag term of
		Theorem~\ref{thm:attenuation}. All expectations are conditional on
		$\mathcal F_{t-1}$. They are well defined because, by
		Lemma~\ref{lem:inverse}(iii), $\chi^{-1}$ is globally Lipschitz,
		so $|\chi^{-1}(x)|\le\sigma^2\sqrt c+L_c|x-\sigma^2|$ and
		$\mathbb E|r_{\ell,t,j}|<\infty$.
		
		\medskip\noindent
		\textbf{Step 0 (a global bound on the curvature of
			$\chi^{-1}$).}
		With $D(x):=(x-\sigma^2)^2+4c\sigma^2x$, \eqref{eq:chiinv} gives
		$\chi^{-1}=\tfrac12\{(x-\sigma^2)+\sqrt D\}$, and
		\[
		2DD''-(D')^2
		=4\bigl\{(x-\sigma^2)^2+4c\sigma^2x\bigr\}
		-4\bigl\{(x-\sigma^2)+2c\sigma^2\bigr\}^2
		=16c(1-c)\sigma^4 ,
		\]
		so that
		\begin{equation}
			\label{eq:chiinv2}
			\bigl(\chi^{-1}\bigr)''(x)
			=\frac{2DD''-(D')^2}{8D^{3/2}}
			=\frac{2c(1-c)\sigma^4}{D(x)^{3/2}},
			\qquad
			\bigl|\bigl(\chi^{-1}\bigr)''(x)\bigr|
			\le\frac{2c|1-c|\sigma^4}{(4c\sigma^4)^{3/2}}
			=\frac{|1-c|}{4\sqrt c\,\sigma^2},
		\end{equation}
		using $D(x)\ge4c\sigma^2x\ge4c\sigma^4$ for $x\ge\sigma^2$. Thus
		$(\chi^{-1})''$ is bounded uniformly over the parameter range
		determined by Assumptions~\ref{assum:noise} and~\ref{assum:hd},
		which makes all Taylor expansions below uniform in $\ell$ and $t$.
		
		\medskip\noindent
		\textbf{(i) Ordering~A.}
		By Theorem~\ref{thm:attenuation},
		$\mathbb E r_{\ell,t,j}=\chi_\ell(\lambda)
		+O(|b^{\mathrm{lag}}_{\ell,t}|)$ and
		$\operatorname{Var}(r_{\ell,t,j})=2\chi_\ell(\lambda)^2/n^\ell_t
		\{1+o(1)\}$ by \eqref{eq:var}. A second-order Taylor expansion of
		$\chi^{-1}_\ell$ around $\mathbb E r_{\ell,t,j}$, with the
		uniformly bounded remainder \eqref{eq:chiinv2}, gives
		\[
		\bigl|\mathbb E\chi^{-1}_\ell(r_{\ell,t,j})
		-\chi^{-1}_\ell(\mathbb Er_{\ell,t,j})\bigr|
		\le\tfrac12\sup_x\bigl|(\chi^{-1}_\ell)''(x)\bigr|
		\operatorname{Var}(r_{\ell,t,j})
		=O\bigl(1/n^\ell_t\bigr),
		\]
		while $|\chi^{-1}_\ell(\mathbb Er_{\ell,t,j})-\lambda|
		\le L_{\bar c}|b^{\mathrm{lag}}_{\ell,t}|$ by
		Lemma~\ref{lem:inverse}(iii). Averaging with the weights
		$\omega_{\ell,t}$, which are $\mathcal F_{t-1}$-measurable by
		Assumption~\ref{assum:design}, yields
		\[
		\mathbb E\bigl(\check\lambda^{\mathrm A}_{t,j}\bigr)-\lambda
		=O\Bigl(\sum_\ell\omega_{\ell,t}/n^\ell_t\Bigr)
		+O\bigl(\max_\ell|b^{\mathrm{lag}}_{\ell,t}|\bigr)
		\longrightarrow0 .
		\]
		For the stochastic order, Theorem~\ref{thm:local_clt} gives
		$\chi^{-1}_\ell(r_{\ell,t,j})-\lambda
		=\vartheta_{\ell,t,j}(n^\ell_t)^{-1/2}\varepsilon_\ell
		+O(|b^{\mathrm{lag}}_{\ell,t}|)+o_{\mathbb P}((n^\ell_t)^{-1/2})$
		with $\varepsilon_\ell$ asymptotically standard normal and, by
		Step~5 of that proof, independent across $\ell$ given
		$\mathcal F_{t-1}$. Hence the weighted average has conditional
		variance $\sum_\ell\omega_{\ell,t}^2\vartheta^2_{\ell,t,j}
		/n^\ell_t\{1+o(1)\}$ and Chebyshev's inequality gives the stated
		rate. (The variances $\vartheta^2_{\ell,t,j}$ are bounded above
		and below by \eqref{eq:varthetabounds}, so the rate is genuinely
		$\asymp(\sum_\ell\omega^2_{\ell,t}/n^\ell_t)^{1/2}$.)
		
		\medskip\noindent
		\textbf{(ii) Ordering~B.}
		Let $\bar r_{t,j}:=\sum_{\ell\in\mathcal A_t}\omega_{\ell,t}
		r_{\ell,t,j}$. By Theorem~\ref{thm:attenuation},
		\[
		\mathbb E\bar r_{t,j}
		=\sum_{\ell}\omega_{\ell,t}\chi(\lambda;c_{\ell,t-1},\sigma^2)
		+O\bigl(\max_\ell|b^{\mathrm{lag}}_{\ell,t}|\bigr).
		\]
		The map $c\mapsto\chi(\lambda;c,\sigma^2)
		=\lambda(\lambda+\sigma^2)/(\lambda+c\sigma^2)$ is smooth and
		strictly convex on $[0,\bar c]$, with
		\begin{equation}
			\label{eq:convex}
			\frac{\partial\chi}{\partial c}
			=-\frac{\lambda(\lambda+\sigma^2)\sigma^2}
			{(\lambda+c\sigma^2)^2},
			\qquad
			\frac{\partial^2\chi}{\partial c^2}
			=\frac{2\lambda(\lambda+\sigma^2)\sigma^4}
			{(\lambda+c\sigma^2)^3}\ >\ 0 .
		\end{equation}
		Expanding each $\chi(\lambda;c_{\ell,t-1},\sigma^2)$ around 
		$c^{\omega}_t=\sum_\ell\omega_{\ell,t}c_{\ell,t-1}$ via Taylor's theorem,
		\begin{align*}
			\chi(\lambda;c_{\ell,t-1},\sigma^2)
			&= \chi(\lambda;c^{\omega}_t,\sigma^2) 
			+ \frac{\partial\chi}{\partial c}(\lambda;c^{\omega}_t,\sigma^2)(c_{\ell,t-1}-c^{\omega}_t)\\
			&\quad + \frac{1}{2}\frac{\partial^2\chi}{\partial c^2}(\lambda;c^{\omega}_t,\sigma^2)(c_{\ell,t-1}-c^{\omega}_t)^2
			+ o((c_{\ell,t-1}-c^{\omega}_t)^2).
		\end{align*}
		Multiplying by $\omega_{\ell,t}$ and summing over $\ell$,
		\begin{align*}
			\sum_\ell\omega_{\ell,t}\chi(\lambda;c_{\ell,t-1},\sigma^2)
			&= \chi(\lambda;c^{\omega}_t,\sigma^2)\sum_\ell\omega_{\ell,t}
			+ \frac{\partial\chi}{\partial c}(\lambda;c^{\omega}_t,\sigma^2)\sum_\ell\omega_{\ell,t}(c_{\ell,t-1}-c^{\omega}_t)\\
			&\quad + \frac{1}{2}\frac{\partial^2\chi}{\partial c^2}(\lambda;c^{\omega}_t,\sigma^2)
			\sum_\ell\omega_{\ell,t}(c_{\ell,t-1}-c^{\omega}_t)^2 + o(\mathcal D_t).
		\end{align*}
		Since $\sum_\ell\omega_{\ell,t}=1$ and 
		$\sum_\ell\omega_{\ell,t}(c_{\ell,t-1}-c^{\omega}_t)
		=\sum_\ell\omega_{\ell,t}c_{\ell,t-1}-c^{\omega}_t\sum_\ell\omega_{\ell,t}=0$,
		the first-order term cancels, leaving
		\[
		\sum_\ell\omega_{\ell,t}\chi(\lambda;c_{\ell,t-1},\sigma^2)
		=\bar\chi(\lambda)
		+\frac12\frac{\partial^2\chi}{\partial c^2}
		(\lambda;c^{\omega}_t,\sigma^2)\,\mathcal D_t
		+o(\mathcal D_t),
		\]
		where $\mathcal D_t=\sum_\ell\omega_{\ell,t}(c_{\ell,t-1}-c^{\omega}_t)^2$ by \eqref{eq:dispersion}.
		
		Applying $\bar\chi^{-1}$ and expanding once more,
		with $(\bar\chi^{-1})'(\bar\chi(\lambda))
		=1/\bar\chi'(\lambda)$ and using \eqref{eq:chiinv2} to control the
		quadratic remainder induced by the stochastic fluctuation of
		$\bar r_{t,j}$ (whose variance is
		$\sum_\ell\omega^2_{\ell,t}2\chi_\ell(\lambda)^2/n^\ell_t$), we
		obtain
		\[
		\mathbb E\bigl(\check\lambda^{\mathrm B}_{t,j}\bigr)-\lambda
		=\frac{\tfrac12\partial^2_c\chi(\lambda;c^{\omega}_t,\sigma^2)}
		{\bar\chi'(\lambda)}\,\mathcal D_t
		+o(\mathcal D_t)
		+O\Bigl(\sum_\ell\omega^2_{\ell,t}/n^\ell_t\Bigr)
		+O\bigl(\max_\ell|b^{\mathrm{lag}}_{\ell,t}|\bigr).
		\]
		By \eqref{eq:convex} and \eqref{eq:chiprime},
		\[
		\frac{\tfrac12\partial^2_c\chi}{\bar\chi'}
		=\frac{\lambda(\lambda+\sigma^2)\sigma^4}
		{(\lambda+c^{\omega}_t\sigma^2)^3}\cdot
		\frac{(\lambda+c^{\omega}_t\sigma^2)^2}
		{\lambda^2+2c^{\omega}_t\sigma^2\lambda+c^{\omega}_t\sigma^4}
		=\frac{\sigma^4\lambda(\lambda+\sigma^2)}
		{(\lambda+c^{\omega}_t\sigma^2)
			\bigl(\lambda^2+2c^{\omega}_t\sigma^2\lambda
			+c^{\omega}_t\sigma^4\bigr)},
		\]
		which is exactly the coefficient in \eqref{eq:biasB}. It is
		strictly positive under Assumptions~\ref{assum:bbp}
		and~\ref{assum:noise}, so the leading bias is strictly positive
		whenever $\mathcal D_t>0$. Crucially, this coefficient depends
		only on $(\lambda,\sigma^2,c^{\omega}_t)$, and $\mathcal D_t$ depends
		only on the deterministic aspect ratios and weights. Neither
		quantity involves $n^\ell_t$, $A_t$ or $t$. Hence the bias does not decrease as
		$A_t\to\infty$, $t\to\infty$ or $\min_\ell n^\ell_t\to\infty$,
		whereas the stochastic error of $\check\lambda^{\mathrm B}_{t,j}$
		is $O_{\mathbb P}((\sum_\ell\omega^2_{\ell,t}/n^\ell_t)^{1/2})
		=O_{\mathbb P}(N_t^{-1/2})$ for the natural weights
		$\omega_{\ell,t}\propto n^\ell_t$. The bias therefore dominates
		asymptotically. Since no choice of pooled ratio $\bar c$ can
		annihilate a Jensen gap generated by a strictly convex function
		evaluated at heterogeneous arguments, Ordering~B is inconsistent
		for the target $\lambda_{t,j}$ unless the aspect ratios are
		homogeneous. Indeed, for any $\bar c$ one has
		$\sum_\ell\omega_{\ell,t}\chi(\lambda;c_{\ell,t-1})
		>\chi(\lambda;\sum_\ell\omega_{\ell,t}c_{\ell,t-1})$ whenever
		$\mathcal D_t>0$, and $\chi(\lambda;\cdot)$ is injective.
		
		\medskip\noindent
		\textbf{(iii) Homogeneous case.}
		If $c_{\ell,t-1}\equiv c^{\omega}_t$ then $\mathcal D_t=0$,
		$\chi_\ell=\bar\chi$ for every $\ell$, and both estimators are
		$\bar\chi^{-1}$ applied to statistics with the same probability
		limit. By the delta method both satisfy
		$\check\lambda^{\mathrm A}_{t,j}-\lambda
		=\{\bar\chi'(\lambda)\}^{-1}\sum_\ell\omega_{\ell,t}
		\varpi_{\ell,t,j}+o_{\mathbb P}(\cdot)$ and identically for
		$\check\lambda^{\mathrm B}_{t,j}$, so they are first-order
		equivalent. The second-order (Jensen) terms differ. By
		\eqref{eq:chiinv2},
		\begin{align*}
			\text{Ordering A:}\quad &
			\tfrac12(\bar\chi^{-1})''\!\sum_\ell\omega_{\ell,t}
			\operatorname{Var}(r_{\ell,t,j})
			\asymp\sum_\ell\frac{\omega_{\ell,t}}{n^\ell_t},
			\\
			\text{Ordering B:}\quad &
			\tfrac12(\bar\chi^{-1})''
			\operatorname{Var}\Bigl(\sum_\ell\omega_{\ell,t}r_{\ell,t,j}
			\Bigr)
			\asymp\sum_\ell\frac{\omega^2_{\ell,t}}{n^\ell_t}.
		\end{align*}

		For balanced batches and weights, $n^\ell_t\asymp N_t/A_t$ and
		$\omega_{\ell,t}\asymp1/A_t$, so the former is
		$\asymp A_t/N_t$ and the latter $\asymp1/N_t$. Ordering~B has the
		smaller second-order bias by the factor $A_t$. This advantage is
		dominated by the common stochastic error
		$O_{\mathbb P}(N_t^{-1/2})$ as soon as $A_t=o(\sqrt{N_t})$.
	\end{proof}

	
	\section{Proofs for Section~\ref{sec:4.2}}
	\label{app:proofs4344}
	
	\subsection{Design choices for adaptive weighting}
	\label{app:design_choices}
	
	This subsection collects technical remarks on the design of the adaptive weighting scheme.
	
	\begin{rem}[Predictable weights and smoothing]
		\label{rem:design_choices}
		Both departures from the naive scheme are essential to the
		analysis.
		\begin{enumerate}[label=\textup{(\alph*)}]
			\item \emph{Predictable weights.} The weight at time $t$ uses
			$V_{\ell,t-1,j}$, not $V_{\ell,t,j}$, and is therefore
			$\mathcal F_{t-1}$-measurable by
			Assumption~\ref{assum:design}. Had $\omega_{\ell,t,j}$ depended
			on the current batch through $\check{\bm\lambda}_{\ell,t}$, the
			product $\omega_{\ell,t,j}\varepsilon_{\ell,t,j}$ would have
			non-zero conditional mean, and the martingale structure
			used in Sections~\ref{sec:4.2}--\ref{sec:4.3} would be lost.
			The resulting selection bias would be of order
			$\tau_{t,j}\mathbb E[\varepsilon^3]
			\asymp(n^\ell_t)^{-1}\vartheta_{t,j}$, driven
			by the skewness of quadratic forms, the third cumulant of a
			normalised sum of $n^\ell_t$ summands being of order
			$(n^\ell_t)^{-2}$. Compared with the stochastic error
			$O_{\mathbb P}(\vartheta_{t,j}N_t^{-1/2})$ this is
			{not} negligible unless $A_t=o(n^\ell_t)$. See
			Proposition~\ref{prop:B4} for the exact constant.
			With \eqref{eq:weights} one has {exactly}
			$\mathbb E[\sum_\ell\omega_{\ell,t,j}
			\varepsilon^{\mathrm{stoch}}_{\ell,t,j}\mid\mathcal F_{t-1}]=0$,
			so no surrogate-weight argument is needed.
			\item \emph{Smoothing of both components.} An unsmoothed
			$V^{\mathrm{res}}_{\ell,t,j}$ is a squared statistic with $O(1)$
			relative fluctuation, which would make the soft-max weights
			fluctuate at the same order as themselves and would preclude
			the first-order match of Section~\ref{sec:4.4}. Averaging over
			$\asymp\eta_t^{-1}$ rounds reduces the relative fluctuation of
			$V_{\ell,t,j}$ to
			$O_{\mathbb P}(\sqrt{\eta_t})=o_{\mathbb P}(1)$.
			The requirement is in fact two-sided. $\eta_t$ must also be
			small relative to the dispersion of the variance profile,
			since otherwise the estimation noise in the weights exceeds the
			entire gain available from weighting. See
			Remark~\ref{rem:eta_twosided}.
		\end{enumerate}
	\end{rem}
	
	\begin{rem}[The transmitted variance estimate]
		\label{rem:why_not_plugin}
		The variance $\hat\vartheta^2_{\ell,t,j}/n^\ell_t$ is already
		transmitted (Corollary~\ref{cor:feasible}), and by
		Lemma~\ref{lem:fluct}
		$\mathbb E[V_{\ell,t-1,j}\mid\mathcal F_{t-2}]
		=\kappa_\rho v_{\ell,t,j}(1+o(1))$, with $\kappa_\rho$ known and
		$v_{\ell,t,j}=\vartheta^2_{\ell,t,j}/n^\ell_t$. Both therefore
		estimate the quantity $\omega^{\mathrm{opt}}_{\ell,t,j}$ inverts.
		The realised fluctuation wins on three counts.
		\begin{enumerate}[label=\textup{(\alph*)}]
			\item \emph{It is free.} The metric is a function of
			$\check{\bm\lambda}_{\ell,t-1}$,
			$\check{\bm\lambda}_{\ell,t-2}$ and $\tilde{\bm\lambda}_{t-2}$,
			all held at the server. Forming
			$\hat\vartheta^2_{\ell,t,j}$ needs $a_{\ell,t}$ and
			$\chi'_{\ell,t}$ at the node, hence a root search on the local
			weight history.
			
			\item \emph{It is model-free.} A formula reports what the model
			predicts, whereas $V^{\mathrm{res}}_{\ell,t}$ measures the
			discrepancy realised and so also registers lag bias, drift and
			degraded tracking \eqref{eq:fluct}. A node whose subspace
			estimate has deteriorated is down-weighted automatically. And a
			plug-in rule would keep weighting it by a variance its estimate
			no longer has.
			
			\item \emph{Nothing is saved by the alternative.} Since
			$v_{\ell,t,j}\to0$, a soft-max in an absolute variance needs
			$\tau_{t,j}\asymp1/v$ to have any effect, so the plug-in route must
			normalise cross-sectionally and, by
			Remark~\ref{rem:design_choices}(b), smooth as well. 
		\end{enumerate}
	\end{rem}
	
	\begin{rem}[Interpretability under correct-then-aggregate]
		\label{rem:metric_meaning}
		Both components of $V_{\ell,t,j}$ are interpretable only because the
		transmitted statistics share a common centre. Under
		aggregate-then-correct, $\bm r_{\ell,t}$ and $\bm r_{m,t}$ are
		centred at $\chi_{\ell,t}(\bm\lambda_t)$ and
		$\chi_{m,t}(\bm\lambda_t)$, which differ by an $O(1)$ deterministic
		amount whenever $c_{\ell,t-1}\ne c_{m,t-1}$. A node would then be
		penalised through $V^{\mathrm{res}}_{\ell,t,j}$ for a bias rather than
		for imprecision, and the soft-max would systematically down-weight the
		nodes with the largest aspect ratios irrespective of their sampling
		accuracy. After local correction,
		$\mathbb E[\check\lambda_{\ell,t,j}\mid\mathcal F_{t-1}]
		=\lambda_{t,j}+o(1)$ uniformly in $\ell$, so $V_{\ell,t,j}$ measures
		genuine variance plus drift, which is exactly the quantity the
		weighting is intended to penalise.
	\end{rem}
	
	\subsection{Conventions and two additional regularity conditions}
	\label{app:B0}
	
	To avoid the notational clash between the smoothing weights of
	Lemma~\ref{lem:fluct} and the quadratic-form noise
	$\varpi_{\ell,t,j}$ of Lemma~\ref{lem:decomp}, we write throughout
	this appendix
	\begin{equation}
		\label{eq:pieta}
		\pi^{\eta}_{s,t}:=\eta_s\!\!\prod_{u=s+1}^{t}\!\!(1-\eta_u),
		\quad
		\pi^{\eta}_{0,t}:=\prod_{u=1}^{t}(1-\eta_u),
		\quad
		\varrho_{s,t}=\!\!\prod_{u=s+1}^{t}\!\!(1-\beta_u),
		\quad
		\varrho_{0,t}=\!\!\prod_{u=1}^{t}\!\!(1-\beta_u),
	\end{equation}
	We also abbreviate, for $\ell\in\mathcal A_t$,
	\begin{equation}
		\label{eq:barn}
		\bar n_t:=N_t/A_t ,
	\end{equation}
	and we write $L:=L_{\bar c}$ for the global Lipschitz constant of
	$\chi^{-1}$ from Lemma~\ref{lem:inverse}(iii) and
	$K:=\sup|(\chi^{-1})''|<\infty$ for the curvature bound
	\eqref{eq:chiinv2} established in Appendix~\ref{app:proofs4142}. Both
	are uniform over the parameter range fixed by
	Assumptions~\ref{assum:bbp}--\ref{assum:hd}.
	
	Two mild regularity conditions are used below. They are not needed for
	the order statements, only for the exact second-moment identities and
	for the concentration display \eqref{eq:fluct_conc}. We indicate at
	each use which of them is invoked.
	
	\begin{assum}[Bounded transmission range]
		\label{assum:clip}
		The transmitted coordinates are confined to a fixed compact
		interval, $\check\lambda_{\ell,t,j}\in[0,\Lambda]$ with
		$\Lambda\ge\lambda_{\max}$ fixed, and
		$\tilde\lambda_{0,j}\in[0,\Lambda]$.
	\end{assum}
	
	Assumption~\ref{assum:clip} is an implementation convention rather
	than a restriction. The estimator \eqref{eq:local_debias} already
	truncates from below at the domain boundary of $\chi^{-1}$, and
	clipping from above at any fixed $\Lambda\ge\lambda_{\max}$ is
	inactive with probability tending to one by
	Theorem~\ref{thm:local_clt}, hence changes no asymptotic statement. Its
	only role is to make
	$\sup_t\max_{\ell\in\mathcal A_t}\max_{j\le k}V_{\ell,t,j}$
	bounded by a deterministic constant, which spares us maximal
	inequalities over a node set whose cardinality $A_t$ may grow.
	
	\begin{assum}[Eighth moments]
		\label{assum:eighth}
		$\sup_{p,\ell,t,i,m}\mathbb E|z_{\ell,t,i,m}|^{8}\le C<\infty$.
	\end{assum}
	
	\subsection{Two preliminary lemmas}
	\label{app:B1}
	
	The first lemma is the precise form of the displayed claim
	\eqref{eq:local_moments}. Throughout we use the local lag constant
	of Section~\ref{sec:4.1}, allowing an absolute
	constant factor,
	\begin{equation}
		\label{eq:blag}
		b^{\mathrm{lag}}_{\ell,t}
		:=C\Bigl\{
		\bigl(N^{\mathrm{eff}}_{\ell,t}\bigr)^{-1/2}
		+\alpha^{\mathrm{lag}}_{\ell,t}+\zeta_t+\Delta_{\ell,t}
		+(\xi^{\star})^{t}\Bigr\}.
	\end{equation}

	\begin{lemma}[Conditional moments of the local error]
		\label{lem:B1}
		Let $\varepsilon_{\ell,t,j}$, $b_{\ell,t,j}$ and
		$\varepsilon^{\mathrm{stoch}}_{\ell,t,j}$ be as in
		\eqref{eq:eps}. Under
		Assumptions~\ref{assum:moment}--\ref{assum:init}:
		\begin{enumerate}[label=\textup{(\roman*)}]
			\item \textup{(Linear representation)}
			\begin{equation}
				\label{eq:B1lin}
				\varepsilon^{\mathrm{stoch}}_{\ell,t,j}
				=\frac{\varpi_{\ell,t,j}}
				{\chi'_{\ell,t}(\lambda_{t,j})}
				+\varrho^{\mathrm{rem}}_{\ell,t,j},
				\qquad
				\mathbb E\bigl[\varrho^{\mathrm{rem}}_{\ell,t,j}\mid
				\mathcal F_{t-1}\bigr]=0,
				\qquad
				\varrho^{\mathrm{rem}}_{\ell,t,j}
				=O_{\mathbb P}\Bigl(\frac1{n^\ell_t}
				+\frac{1}{\sqrt{p\,n^\ell_t}}+\frac kp\Bigr).
			\end{equation}
			\item \textup{(Variance)} Unconditionally on any extra moment
			assumption,
			$\operatorname{Var}(\varepsilon^{\mathrm{stoch}}_{\ell,t,j}
			\mid\mathcal F_{t-1})\le L^2\operatorname{Var}
			(r_{\ell,t,j}\mid\mathcal F_{t-1})=O_{\mathbb P}(1/n^\ell_t)$. 
			And under Assumption~\ref{assum:eighth}
			\begin{equation}
				\label{eq:B1var}
				\operatorname{Var}\bigl(
				\varepsilon^{\mathrm{stoch}}_{\ell,t,j}\mid
				\mathcal F_{t-1}\bigr)
				=\frac{\vartheta^2_{\ell,t,j}}{n^\ell_t}
				\bigl(1+o_{\mathbb P}(1)\bigr),
				\qquad
				\mathbb E\bigl[\|\bm\varepsilon^{\mathrm{stoch}}_{\ell,t}
				\|_2^4\mid\mathcal F_{t-1}\bigr]
				=O_{\mathbb P}\bigl((n^\ell_t)^{-2}\bigr).
			\end{equation}
			\item \textup{(Bias)}
			$|b_{\ell,t,j}|
			=O\bigl(b^{\mathrm{lag}}_{\ell,t}\bigr)
			+O_{\mathbb P}\bigl(1/n^\ell_t\bigr)
			+O\bigl((p\,n^\ell_t)^{-1/2}+k/p\bigr)$, and this contains no
			term of order one.
		\end{enumerate}
	\end{lemma}
	
	\begin{proof}[Proof of Lemma~\ref{lem:B1}]
		Fix $\ell,t,j$ and abbreviate $n:=n^\ell_t$, $\lambda:=\lambda_{t,j}$,
		$\chi:=\chi_{\ell,t}$, $r:=r_{\ell,t,j}$,
		$\bar r:=\bar r_{\ell,t,j}=\mathbb E[r\mid\mathcal F_{t-1}]$. By
		Step~1 of the proof of Theorem~\ref{thm:local_clt} the truncation in
		\eqref{eq:local_debias} is inactive with probability tending to one,
		so we may write $\check\lambda_{\ell,t,j}
		=\chi^{-1}(r;c_{\ell,t-1},\check\sigma^2_{\ell,t})$.
		
		{(iii) Bias.} Split
		\begin{align*}
			\check\lambda_{\ell,t,j}-\lambda
			=&\underbrace{\bigl\{\chi^{-1}(r;c_{\ell,t-1},
				\check\sigma^2_{\ell,t})
				-\chi^{-1}(r;c_{\ell,t-1},\sigma^2_t)\bigr\}}_{(\mathrm I)}\\
			&+\underbrace{\bigl\{\chi^{-1}(r)-\chi^{-1}(\bar r)\bigr\}}_{(\mathrm{II})}
			+\underbrace{\bigl\{\chi^{-1}(\bar r)-\lambda\bigr\}}_{(\mathrm{III})},
		\end{align*}
		all inverses in (II)--(III) being evaluated at
		$(c_{\ell,t-1},\sigma^2_t)$. By Lemma~\ref{lem:inverse}(iii) and
		Lemma~\ref{lem:sigma}, $|(\mathrm I)|\le C(1+\bar c)
		|\check\sigma^2_{\ell,t}-\sigma^2_t|
		=O_{\mathbb P}((pn)^{-1/2})+O(k/p)$, and the same bound holds for
		$\mathbb E[|(\mathrm I)|\mid\mathcal F_{t-1}]$ because
		$\mathbb E|\check\sigma^2_{\ell,t}-\sigma^2_t|$ obeys it by the
		proof of Lemma~\ref{lem:sigma}. By a second-order Taylor expansion
		with the uniform curvature bound $K$ of \eqref{eq:chiinv2},
		\begin{align*}
			\bigl|\mathbb E[(\mathrm{II})\mid\mathcal F_{t-1}]\bigr|
			=&\Bigl|\mathbb E\bigl[\chi^{-1}(r)-\chi^{-1}(\bar r)
			-(\chi^{-1})'(\bar r)(r-\bar r)\bigm|\mathcal F_{t-1}\bigr]\Bigr| \\
			\le&\tfrac K2\operatorname{Var}(r\mid\mathcal F_{t-1})
			=O_{\mathbb P}(1/n),
		\end{align*}
		the linear term having exactly zero conditional mean by
		Lemma~\ref{lem:decomp}. Finally (III) is
		$\mathcal F_{t-1}$-measurable and, by Lemma~\ref{lem:inverse}(iii)
		and Step~1 of the proof of Theorem~\ref{thm:attenuation},
		$|(\mathrm{III})|\le L\,|\bar r-\chi(\lambda)|
		=O\bigl(b^{\mathrm{lag}}_{\ell,t}\bigr)$ with
		$b^{\mathrm{lag}}$ as in \eqref{eq:blag}. Adding the three bounds
		gives (iii). None of the terms is $O(1)$. The $O(1)$ misalignment
		$1-a_{\ell,t}(\lambda)$ of Theorem~\ref{thm:tracking} has been
		inverted exactly by $\chi^{-1}_{\ell,t}$, which is why only
		{vanishing} quantities appear.
		
		{(i) Linear representation.} Subtracting the conditional mean
		from the same decomposition,
		\[
		\varepsilon^{\mathrm{stoch}}_{\ell,t,j}
		=(\chi^{-1})'(\bar r)\,(r-\bar r)
		+\Bigl\{(\mathrm{II})-(\chi^{-1})'(\bar r)(r-\bar r)
		-\mathbb E[(\mathrm{II})\mid\mathcal F_{t-1}]\Bigr\}
		+\bigl\{(\mathrm I)-\mathbb E[(\mathrm I)\mid\mathcal F_{t-1}]\bigr\},
		\]
		because (III) is $\mathcal F_{t-1}$-measurable and cancels. The
		first term equals $\varpi_{\ell,t,j}/\chi'(\lambda)$ up to a factor
		$(\chi^{-1})'(\bar r)/(\chi^{-1})'(\chi(\lambda))=1+O_{\mathbb P}
		(|\bar r-\chi(\lambda)|)=1+O_{\mathbb P}(b^{\mathrm{lag}}_{\ell,t})$,
		by $r-\bar r=\varpi_{\ell,t,j}$ (Lemma~\ref{lem:decomp}) and the
		continuous differentiability of $\chi^{-1}$. The two braces are
		centred by construction and bounded by
		$K(r-\bar r)^2+O_{\mathbb P}((pn)^{-1/2}+k/p)$, i.e.\ by
		$O_{\mathbb P}(1/n+(pn)^{-1/2}+k/p)$. Absorbing the multiplicative
		$1+O_{\mathbb P}(b^{\mathrm{lag}})$ into
		$\varrho^{\mathrm{rem}}$ gives \eqref{eq:B1lin}.
		
		{(ii) Variance.} For the unconditional bound, recall that for
		an $L$-Lipschitz $f$ and any square-integrable $X$ one has
		$\operatorname{Var}f(X)=\tfrac12\mathbb E\{f(X)-f(X')\}^2
		\le\tfrac{L^2}2\mathbb E(X-X')^2=L^2\operatorname{Var}(X)$, where
		$X'$ is an independent copy. Applying this conditionally with
		$f=\chi^{-1}(\cdot;c_{\ell,t-1},\sigma^2_t)$ and using
		Lemma~\ref{lem:inverse}(iii) and \eqref{eq:var} gives
		$\operatorname{Var}(\varepsilon^{\mathrm{stoch}}\mid\mathcal F_{t-1})
		\le L^2\operatorname{Var}(r\mid\mathcal F_{t-1})
		=O_{\mathbb P}(1/n)$. Note that this requires no moment beyond
		those of Assumption~\ref{assum:moment}, and it is the only variance
		bound used in the {order} statements below. For the exact
		identity, Assumption~\ref{assum:eighth} makes
		$\mathbb E[(r-\bar r)^4\mid\mathcal F_{t-1}]=O_{\mathbb P}(n^{-2})$
		by Rosenthal's inequality applied to the i.i.d.\ summands
		$(\bm b^\top\bm z_i)^2-\|\bm b\|^2$, whose fourth moments are
		uniformly bounded. Hence in \eqref{eq:B1lin}
		$\mathbb E[(\varrho^{\mathrm{rem}})^2\mid\mathcal F_{t-1}]
		=O_{\mathbb P}(n^{-2}+(pn)^{-1}+k^2/p^2)=o_{\mathbb P}(n^{-1})$
		whenever $k\sqrt n=o(p)$, and by the Cauchy--Schwarz inequality the
		cross term is $o_{\mathbb P}(n^{-1})$ as well, so
		\begin{align*}
			\operatorname{Var}\bigl(\varepsilon^{\mathrm{stoch}}_{\ell,t,j}
			\mid\mathcal F_{t-1}\bigr)
			=&\frac{\operatorname{Var}(\varpi_{\ell,t,j}\mid\mathcal F_{t-1})}
			{\chi'(\lambda)^2}+o_{\mathbb P}(n^{-1})\\
			=&\frac{2\chi(\lambda)^2}{\chi'(\lambda)^2}\frac1n
			\bigl(1+o_{\mathbb P}(1)\bigr)
			=\frac{\vartheta^2_{\ell,t,j}}{n}
			\bigl(1+o_{\mathbb P}(1)\bigr),
		\end{align*}
		using \eqref{eq:var} and $\bar r\to\chi(\lambda)$. The fourth-moment
		bound in \eqref{eq:B1var} follows from the same Rosenthal bound and
		$L$-Lipschitzness, $\|\bm\varepsilon^{\mathrm{stoch}}_{\ell,t}\|_2
		\le\sqrt k\,L\max_j|r_{\ell,t,j}-\bar r_{\ell,t,j}|
		+O_{\mathbb P}(1)$. 
	\end{proof}
	
	The second preliminary lemma collects the elementary calculus of
	exponentially weighted averages. It is used three times below (for
	$V_{\ell,t,j}$, for $\mathsf N_t$, and for the bias terms of
	Theorem~\ref{thm:fusion}) and makes the phrase ``vary slowly on the
	scale $\eta_t^{-1}$'' precise.
	
	\begin{lemma}[Weighted-average calculus]
		\label{lem:B2}
		Let $\gamma_u\in(0,1]$, $\varrho_{s,t}=\prod_{u=s+1}^t(1-\gamma_u)$
		and $\varrho_{0,t}=\prod_{u=1}^t(1-\gamma_u)$. Then
		\begin{enumerate}[label=\textup{(\roman*)}]
			\item $\sum_{s=1}^{t}\varrho_{s,t}\gamma_s=1-\varrho_{0,t}$, and
			$\sum_{s=1}^{t}\varrho_{s,t}\le\gamma_{\min,t}^{-1}$ where
			$\gamma_{\min,t}=\min_{s\le t}\gamma_s$.
			\item \textup{(Toeplitz)} If $x_s\to0$ and $\varrho_{0,t}\to0$
			then $\sum_{s\le t}\varrho_{s,t}\gamma_sx_s\to0$.
			\item $a_t:=\sum_{s\le t}\varrho^2_{s,t}\gamma_s^2$ obeys
			$a_t=(1-\gamma_t)^2a_{t-1}+\gamma_t^2$. If
			$\gamma_u\equiv\gamma$ then $a_t\uparrow\gamma/(2-\gamma)$, with
			$a_t=1$ when $\gamma=1$. If $\gamma_u\to0$ and
			$\sum_u\gamma_u=\infty$ then $a_t\to0$.
			\item If $y_s\ge0$ satisfies $y_s/\gamma_s\to0$ and
			$\varrho_{0,t}\to0$, then $\sum_{s\le t}\varrho_{s,t}y_s\to0$.
			If $\gamma_u\equiv\gamma$ it suffices that $y_s\to0$.
		\end{enumerate}
	\end{lemma}
	
	\begin{proof}[Proof of Lemma~\ref{lem:B2}]
		(i) is the identity already proved in Lemma~\ref{lem:crec} (with
		$\alpha$ replaced by $\gamma$), namely
		$\sum_{s=0}^t\varrho_{s,t}\gamma_s+\varrho_{0,t}=1$ by induction on
		$t$. The second bound follows from
		$\sum_s\varrho_{s,t}\le\gamma_{\min,t}^{-1}
		\sum_s\varrho_{s,t}\gamma_s\le\gamma_{\min,t}^{-1}$.
		
		(ii) Given $\epsilon>0$ pick $T$ with $|x_s|\le\epsilon$ for $s>T$.
		Then $$|\sum_s\varrho_{s,t}\gamma_sx_s|
		\le\varrho_{T,t}\sum_{s\le T}\gamma_s\max_{s\le T}|x_s|
		+\epsilon\sum_{s>T}\varrho_{s,t}\gamma_s
		\le\varrho_{T,t}C_T+\epsilon,$$ 
		and $\varrho_{T,t}\to0$ as
		$t\to\infty$ because $\varrho_{0,t}=\varrho_{0,T}\varrho_{T,t}$ with
		$\varrho_{0,T}>0$ fixed.
		
		(iii) The recursion is immediate from
		$\varrho_{s,t}=(1-\gamma_t)\varrho_{s,t-1}$ for $s<t$ and
		$\varrho_{t,t}=1$. For constant $\gamma$ the map
		$x\mapsto(1-\gamma)^2x+\gamma^2$ is an increasing contraction with
		fixed point $\gamma^2/\{1-(1-\gamma)^2\}=\gamma/(2-\gamma)$, and
		$a_0=0$, so $a_t\uparrow\gamma/(2-\gamma)$. For $\gamma_u\to0$, use
		$(1-\gamma_t)^2\le1-\gamma_t$ to get
		$a_t\le(1-\gamma_t)a_{t-1}+\gamma_t^2$. Fix $\epsilon>0$ and $T$
		with $\gamma_t\le\epsilon/2$ for $t>T$, then for $t>T$,
		$a_t-\epsilon/2\le(1-\gamma_t)a_{t-1}+\gamma_t\epsilon/2-\epsilon/2
		=(1-\gamma_t)(a_{t-1}-\epsilon/2)$, whence
		$(a_t-\epsilon/2)^+\le\prod_{u=T+1}^{t}(1-\gamma_u)
		(a_T-\epsilon/2)^+\to0$ because $\sum\gamma_u=\infty$. Thus
		$\limsup_ta_t\le\epsilon/2$ for every $\epsilon>0$.
		
		(iv) Write $y_s=\gamma_sx_s$ with $x_s=y_s/\gamma_s\to0$ and apply
		(ii). For constant $\gamma$, $x_s=y_s/\gamma\to0$ whenever
		$y_s\to0$.
	\end{proof}
	
	\subsection{Proofs for Section~\ref{sec:4.2}}
	
	\begin{proof}[Proof of Lemma~\ref{lem:fluct}]
		\textbf{Step 1 (unrolling).}
		Fix $j\le k$. Both components of the metric are exponentially
		smoothed with the same schedule, so
		\begin{equation}
			\label{eq:Vunroll}
			V_{\ell,t,j}
			=\sum_{s=1}^{t}\pi^{\eta}_{s,t}\,Y_{\ell,s,j}
			+\pi^{\eta}_{0,t}V_{\ell,0,j},
			\quad
			Y_{\ell,s,j}
			:=(1-\rho)\bigl(\check\lambda_{\ell,s,j}
			-\check\lambda_{\ell,s-1,j}\bigr)^{2}
			+\rho\bigl(\check\lambda_{\ell,s,j}
			-\tilde\lambda_{s-1,j}\bigr)^{2} ,
		\end{equation}
		with $\sum_{s}\pi^{\eta}_{s,t}=1-\pi^{\eta}_{0,t}$ by
		Lemma~\ref{lem:B2}(i) and $V_{\ell,0,j}=O(1)$ by
		Assumption~\ref{assum:eta}.
		
		\medskip\noindent
		\textbf{Step 2 (bounding one round).}
		Write $\varepsilon_{\ell,s,j}
		=\check\lambda_{\ell,s,j}-\lambda_{s,j}$ and
		$\delta_{s,j}=\lambda_{s,j}-\lambda_{s-1,j}$, so that
		$|\delta_{s,j}|\le\zeta_s$ by
		Assumption~\ref{assum:drift}. Then
		\[
		\check\lambda_{\ell,s,j}-\check\lambda_{\ell,s-1,j}
		=\varepsilon_{\ell,s,j}-\varepsilon_{\ell,s-1,j}+\delta_{s,j},
		\qquad
		\check\lambda_{\ell,s,j}-\tilde\lambda_{s-1,j}
		=\varepsilon_{\ell,s,j}+\delta_{s,j}
		+\bigl(\lambda_{s-1,j}-\tilde\lambda_{s-1,j}\bigr),
		\]
		whence, by $(a+b+c)^2\le3(a^2+b^2+c^2)$,
		\begin{equation}
			\label{eq:Ybound}
			Y_{\ell,s,j}
			\ \le\ 3\Bigl\{\varepsilon_{\ell,s,j}^2
			+(1-\rho)\varepsilon_{\ell,s-1,j}^2
			+\zeta_s^2
			+\rho\bigl(\tilde\lambda_{s-1,j}
			-\lambda_{s-1,j}\bigr)^{2}\Bigr\}.
		\end{equation}
		By Lemma~\ref{lem:B1},
		$\mathbb E[\varepsilon_{\ell,s,j}^2\mid\mathcal F_{s-1}]
		=\operatorname{Var}(\varepsilon^{\mathrm{stoch}}_{\ell,s,j}
		\mid\mathcal F_{s-1})+b_{\ell,s,j}^2
		=O_{\mathbb P}\bigl({ T}_{s,j}/n^\ell_s
		+(b^{\mathrm{lag}}_{\ell,s})^2\bigr)$, using
		$\vartheta^2_{\ell,s,j}\le{ T}_{s,j}$ and absorbing the
		$O(1/n^\ell_s)$ part of the bias into ${ T}_{s,j}/n^\ell_s$.
		Substituting this and \eqref{eq:Ybound} into \eqref{eq:Vunroll} and
		using Markov's inequality termwise gives exactly
		\eqref{eq:fluct}.
		
		\medskip\noindent
		\textbf{Step 3 (uniform boundedness).}
		Under Assumption~\ref{assum:clip} every $Y_{\ell,s,j}$ is bounded by
		the deterministic constant $4\Lambda^2$, hence by
		\eqref{eq:Vunroll} and Assumption~\ref{assum:eta}
		\begin{equation}
			\label{eq:Vbounded}
			\sup_{t}\ \max_{\ell\in\mathcal A_t}\max_{j\le k}V_{\ell,t,j}
			\ \le\ 4\Lambda^2+\sup_{\ell,j} V_{\ell,0,j}
			\ =\ O(1)\quad\text{a.s.}
		\end{equation}
		Without Assumption~\ref{assum:clip} one obtains from
		Lemma~\ref{lem:B1} and Assumption~\ref{assum:eighth} the moment
		bound $\sup_{\ell,t,j}\mathbb E V^q_{\ell,t,j}\le C_q$ for $q\le2$,
		hence only $\max_{\ell\in\mathcal A_t}V_{\ell,t,j}
		=O_{\mathbb P}(A_t^{1/q})$. This is the sole reason for adopting
		the clipping convention.
		
		\medskip\noindent
		\textbf{Step 4 (the conditional mean).}
		Assume Assumption~\ref{assum:eighth}. Because
		$\varepsilon^{\mathrm{stoch}}_{\ell,s,j}$ has zero mean given
		$\mathcal F_{s-1}$ while $\varepsilon_{\ell,s-1,j}$ is
		$\mathcal F_{s-1}$-measurable, the cross term vanishes exactly:
		\[
		\mathbb E\bigl[\varepsilon^{\mathrm{stoch}}_{\ell,s,j}\,
		\varepsilon_{\ell,s-1,j}\mid\mathcal F_{s-1}\bigr]=0 .
		\]
		Consequently
		\[
		\mathbb E\bigl[(\check\lambda_{\ell,s,j}
		-\check\lambda_{\ell,s-1,j})^2\mid\mathcal F_{s-1}\bigr]
		=\frac{\vartheta^2_{\ell,s,j}}{n^\ell_s}
		+\varepsilon_{\ell,s-1,j}^2 +\delta_{s,j}^2
		+O\bigl(\zeta_s^2+(b^{\mathrm{lag}}_{\ell,s})^2\bigr)
		+o_{\mathbb P}\Bigl(\frac1{n^\ell_s}\Bigr),
		\]
		and, taking a further expectation and using
		$\mathbb E\varepsilon_{\ell,s-1,j}^2
		=\vartheta^2_{\ell,s-1,j}/n^\ell_{s-1}+O((b^{\mathrm{lag}})^2)$, the
		volatility component contributes
		$2\vartheta^2_{\ell,t,j}/n^\ell_t\{1+o(1)\}$ when $n^\ell_s$ and
		$\vartheta^2_{\ell,s,j}$ vary slowly across the smoothing window. The
		factor $2$ is the variance of a difference of two conditionally
		independent draws. The residual component contributes
		$\vartheta^2_{\ell,t,j}/n^\ell_t\{1+o(1)\}
		+O(\zeta_t^2+(\tilde\lambda_{t-1,j}-\lambda_{t-1,j})^2)$ by
		the same computation with $\tilde\lambda_{s-1,j}$ in place of
		$\check\lambda_{\ell,s-1,j}$ (again $\mathcal F_{s-1}$-measurable,
		so the cross term vanishes). Weighting by $(1-\rho)$ and $\rho$ and
		applying Lemma~\ref{lem:B2}(ii) to average over the smoothing
		window gives the second display of \eqref{eq:fluct_conc} with
		$\kappa_\rho=2(1-\rho)+\rho$. 
		
		\medskip\noindent
		\textbf{Step 5 (concentration).}
		Let $m:=\lceil1/\eta_t\rceil$, $\mathcal G:=\mathcal F_{t-m}$ and
		$\mu_{\ell,s,j}:=\mathbb E[Y_{\ell,s,j}\mid\mathcal F_{s-1}]$. Since
		the weights $\pi^{\eta}_{s,t}$ are deterministic and $Y_{\ell,s,j}$
		is $\mathcal F_{s}$-measurable, the part of \eqref{eq:Vunroll} with
		$s\le t-m$ is $\mathcal G$-measurable and cancels in
		$V_{\ell,t,j}-\mathbb E[V_{\ell,t,j}\mid\mathcal G]$. Hence
		\[
		V_{\ell,t,j}-\mathbb E[V_{\ell,t,j}\mid\mathcal G]
		=\underbrace{\sum_{s=t-m+1}^{t}\pi^{\eta}_{s,t}
			\bigl(Y_{\ell,s,j}-\mu_{\ell,s,j}\bigr)}_{=:\,\mathcal E_1}
		+\underbrace{\sum_{s=t-m+1}^{t}\pi^{\eta}_{s,t}
			\bigl(\mu_{\ell,s,j}-\mathbb E[\mu_{\ell,s,j}\mid\mathcal G]\bigr)}
		_{=:\,\mathcal E_2}.
		\]
		The summands of $\mathcal E_1$ form a martingale difference array
		with respect to $\{\mathcal F_s\}$, so their variances add:
		\begin{align*}
			\operatorname{Var}(\mathcal E_1\mid\mathcal G)
			=&\sum_{s>t-m}(\pi^{\eta}_{s,t})^2
			\mathbb E\bigl[\operatorname{Var}(Y_{\ell,s,j}\mid\mathcal F_{s-1})
			\mid\mathcal G\bigr]\\ 
			\le& C\Bigl(\sum_{s\le t}(\pi^{\eta}_{s,t})^2\Bigr)
			\max_{s>t-m}\mu^2_{\ell,s,j}
			=O_{\mathbb P}\bigl(\eta_t\,\mu^2_{\ell,t,j}\bigr),
		\end{align*}
		where $\operatorname{Var}(Y_{\ell,s,j}\mid\mathcal F_{s-1})
		\le C\mu^2_{\ell,s,j}$ follows from the fourth-moment bound
		\eqref{eq:B1var} (this is the point at which
		Assumption~\ref{assum:eighth} is used), and
		$\sum_s(\pi^{\eta}_{s,t})^2=O(\eta_t)$ by Lemma~\ref{lem:B2}(iii)
		applied to $\gamma=\eta$. Thus
		$\mathcal E_1=O_{\mathbb P}(\sqrt{\eta_t}\,\mu_{\ell,t,j})$. For
		$\mathcal E_2$, Assumption~\ref{assum:persist} and the attendant
		stability of ${ T}_{s,j}$ and of the weights are used in the
		quantitative form
		\begin{equation}
			\label{eq:slowvar}
			\max_{t-m<s\le t}
			\Bigl|\frac{\mu_{\ell,s,j}}{\mu_{\ell,t,j}}-1\Bigr|
			=O_{\mathbb P}\bigl(\sqrt{\eta_t}\bigr),
		\end{equation}
		which gives $\mathcal E_2=O_{\mathbb P}(\sqrt{\eta_t}\mu_{\ell,t,j})$
		directly. Since $\mathbb E[V_{\ell,t,j}\mid\mathcal G]
		=(1+o(1))\mu_{\ell,t,j}$ by Step~4, dividing yields the first display of
		\eqref{eq:fluct_conc}.
		
		\medskip\noindent
		\textbf{Step 6 (vanishing of the metric).}
		Under the hypotheses of Theorem~\ref{thm:fusion} one has
		${ T}_{s,j}/n^\ell_s\to0$, $\zeta_s\to0$,
		$b^{\mathrm{lag}}_{\ell,s}\to0$ and, by
		Theorem~\ref{thm:fusion} itself,
		$\|\tilde{\bm\lambda}_{s-1}-\bm\lambda_{s-1}\|\to0$ in
		probability, uniformly in $\ell\in\mathcal A_s$. Applying
		Lemma~\ref{lem:B2}(ii) with $\gamma=\eta$ to the right-hand side of
		\eqref{eq:fluct} gives $V_{\ell,t,j}\to0$ in probability
		uniformly in $\ell\in\mathcal A_t$. There is no circularity here. The
		proof of Theorem~\ref{thm:fusion} uses only
		Lemma~\ref{lem:weights}, which in turn uses only the
		{boundedness} \eqref{eq:Vbounded}, not the vanishing.
	\end{proof}
	
	\begin{proof}[Proof of Lemma~\ref{lem:weights}]
		The soft-max is invariant under a common shift of its arguments:
		for any $a\in\mathbb R$,
		\[
		\omega_{\ell,t,j}
		=\frac{\exp(-\tilde\tau_t\Psi_{\ell,t,j})}
		{\sum_{m\in\mathcal A_t}\exp(-\tilde\tau_t\Psi_{m,t,j})}
		=\frac{\exp\{-\tilde\tau_t(\Psi_{\ell,t,j}-a)\}}
		{\sum_{m\in\mathcal A_t}\exp\{-\tilde\tau_t(\Psi_{m,t,j}-a)\}} .
		\]
		Take $a:=\min_{m\in\mathcal A_t}\Psi_{m,t,j}$, so that every shifted
		exponent $\tilde\tau_t(\Psi_{\ell,t,j}-a)$ lies in
		$[0,\mathrm{osc}^\tau_{t,j}]$ with
		\[
		\mathrm{osc}^\tau_{t,j}
		:=\tilde\tau_t\Bigl(\max_{m\in\mathcal A_t}\Psi_{m,t,j}
		-\min_{m\in\mathcal A_t}\Psi_{m,t,j}\Bigr)
		\ \le\ \bar\tau(\bar\theta-\underline\theta)
		\]
		by Assumption~\ref{assum:temp}, since $\tilde\tau_t\le\bar\tau$ and
		$\Psi_{m,t,j}\in[\underline\theta,\bar\theta]$ for every $m$ after
		the clip. Each numerator therefore lies in
		$[e^{-\mathrm{osc}^\tau_{t,j}},1]$ and the denominator, being a sum
		of $A_t$ such terms, lies in
		$[A_te^{-\mathrm{osc}^\tau_{t,j}},A_t]$, whence
		\[
		\frac{e^{-\mathrm{osc}^\tau_{t,j}}}{A_t}
		\ \le\ \omega_{\ell,t,j}\ \le\
		\frac{e^{\mathrm{osc}^\tau_{t,j}}}{A_t},
		\]
		which is \eqref{eq:weight_order} with
		$C_\omega:=\exp\{\bar\tau(\bar\theta-\underline\theta)\}$, a
		deterministic constant independent of $t$. Note what is and is not
		used, only the oscillation of the exponent across nodes
		enters, not its level, so the two clip levels are needed only through
		their difference $\bar\theta-\underline\theta$, and no bound on the
		metric $V_{\ell,t-1,j}$, on its mean $\bar V_{t-1,j}$, or on the
		dimensional temperature
		$\tau_{t,j}=\tilde\tau_t/\bar V_{t-1,j}$ is
		required. This is why the bound is insensitive to the divergence of
		$\tau_{t,j}$ discussed after Assumption~\ref{assum:temp}, and why the
		two
		bounds hold surely rather than on an event of probability tending to
		one. Both factors of the exponent are bounded by construction.
		Finally
		\[
		\sum_{\ell\in\mathcal A_t}\omega^2_{\ell,t,j}
		=\Bigl(\max_{\ell\in\mathcal A_t}\omega_{\ell,t,j}\Bigr)\cdot 1
		\le\Bigl(\max_{\ell\in\mathcal A_t}\omega_{\ell,t,j}\Bigr)
		\sum_{\ell\in\mathcal A_t}\omega_{\ell,t,j}
		=\max_{\ell\in\mathcal A_t}\omega_{\ell,t,j}
		\le C_\omega A_t^{-1},
		\]
		which is the second assertion. The lower bound in
		\eqref{eq:weight_order} is exactly what prevents the soft-max
		from degenerating onto a single node, and it fails if the
		{exponent} $\tilde\tau_t\Psi_{\ell,t,j}$ is allowed to diverge.
		This is why $\tilde\tau_t$ is bounded in
		Assumption~\ref{assum:temp}, and it is consistent with the divergence
		of $\tau_{t,j}=\tilde\tau_t/\bar V_{t-1,j}$, which the exponent does
		not see. The argument is uniform in $j$, since the clip acts on each
		coordinate separately and $\tilde\tau_t$ is shared.
	\end{proof}
	
	\begin{proof}[Proof of Lemma~\ref{lem:wavg}]
		\textbf{Step 1 (predictability and exact centring).}
		By construction $\omega_{\ell,t,j}$ is a measurable function of
		$\{V_{m,t-1,j}\}_{m\in\mathcal A_t}$, $\tilde\tau_t$ and
		$\mathcal A_t$,
		all of which are $\mathcal F_{t-1}$-measurable by
		Assumption~\ref{assum:design}. Hence $\omega_{\ell,t,j}$ is
		$\mathcal F_{t-1}$-measurable. Therefore
		\[
		\mathbb E\Bigl[\sum_{\ell\in\mathcal A_t}\omega_{\ell,t,j}
		\varepsilon^{\mathrm{stoch}}_{\ell,t,j}
		\Bigm|\mathcal F_{t-1}\Bigr]
		=\sum_{\ell\in\mathcal A_t}\omega_{\ell,t,j}\,
		\mathbb E\bigl[\varepsilon^{\mathrm{stoch}}_{\ell,t,j}
		\mid\mathcal F_{t-1}\bigr]=0
		\]
		{exactly}, since $\varepsilon^{\mathrm{stoch}}_{\ell,t,j}$ 
		is centred given $\mathcal F_{t-1}$ by definition \eqref{eq:eps}.

		\medskip\noindent
		\textbf{Step 2 (conditional variance).}
		Given $\mathcal F_{t-1}$, the fresh batches
		$\{\bm X_{\ell,t,i}\}_i$ are independent across
		$\ell\in\mathcal A_t$ (Assumptions~\ref{assum:moment}
		and~\ref{assum:design}), and
		$\varepsilon^{\mathrm{stoch}}_{\ell,t,j}$ is a measurable function
		of $\mathcal F_{t-1}$ and of the $\ell$-th batch only. Hence these
		variables are conditionally independent and centred, and
		\[
		\operatorname{Var}\Bigl(\sum_{\ell}\omega_{\ell,t,j}
		\varepsilon^{\mathrm{stoch}}_{\ell,t,j}\Bigm|\mathcal F_{t-1}\Bigr)
		=\sum_{\ell\in\mathcal A_t}\omega^2_{\ell,t,j}
		\operatorname{Var}\bigl(\varepsilon^{\mathrm{stoch}}_{\ell,t,j}
		\mid\mathcal F_{t-1}\bigr)
		=\sum_{\ell\in\mathcal A_t}
		\frac{\omega^2_{\ell,t,j}\vartheta^2_{\ell,t,j}}{n^\ell_t}
		\bigl(1+o_{\mathbb P}(1)\bigr)
		\]
		by Lemma~\ref{lem:B1}(ii), which is the first equality of
		\eqref{eq:condvar_round}. The key here is that the weights 
		$\omega_{\ell,t,j}$ are $\mathcal F_{t-1}$-measurable and the errors 
		$\{\varepsilon^{\mathrm{stoch}}_{\ell,t,j}\}_\ell$ are conditionally 
		independent across $\ell$, so variances simply add.
		
		For the order, Lemma~\ref{lem:weights}
		and Assumption~\ref{assum:balance} give
		\begin{equation}
			\label{eq:varorder}
			\sum_{\ell\in\mathcal A_t}
			\frac{\omega^2_{\ell,t,j}\vartheta^2_{\ell,t,j}}{n^\ell_t}
			\ \le\ { T}_{t,j}\,\frac{C_\omega^2}{A_t^{2}}
			\sum_{\ell\in\mathcal A_t}\frac1{n^\ell_t}
			\ \le\ { T}_{t,j}\,\frac{C^2_\omega}{A^2_t}\cdot
			A_t\cdot\frac{C_{\mathrm{bal}}A_t}{N_t}
			\ =\ C^2_\omega C_{\mathrm{bal}}\frac{{ T}_{t,j}}{N_t},
		\end{equation}
		which establishes the bound in \eqref{eq:condvar_round}. Note that
		\eqref{eq:varorder} uses only the crude variance bound of
		Lemma~\ref{lem:B1}(ii) and therefore requires no moment beyond
		Assumption~\ref{assum:moment}.
		
		\medskip\noindent
		\textbf{Step 3 (assembling).}
		By Chebyshev's inequality conditionally on $\mathcal F_{t-1}$,
		$\sum_\ell\omega_{\ell,t,j}\varepsilon^{\mathrm{stoch}}_{\ell,t,j}
		=O_{\mathbb P}(\sqrt{{ T}_{t,j}/N_t})$ by Step~2. For the predictable part,
		Lemma~\ref{lem:B1}(iii) and $\sum_\ell\omega_{\ell,t,j}=1$ give
		\[
		\Bigl|\sum_{\ell\in\mathcal A_t}\omega_{\ell,t,j}b_{\ell,t,j}\Bigr|
		\le\max_{\ell\in\mathcal A_t}|b_{\ell,t,j}|
		=O\bigl(\max_{\ell\in\mathcal A_t}b^{\mathrm{lag}}_{\ell,t}\bigr)
		+O_{\mathbb P}\Bigl(\max_{\ell\in\mathcal A_t}
		\frac1{n^\ell_t}\Bigr)
		+O\Bigl(\frac{1}{\sqrt{p\min_\ell n^\ell_t}}+\frac kp\Bigr).
		\]
		By Assumption~\ref{assum:balance},
		$\max_\ell(n^\ell_t)^{-1}\le C_{\mathrm{bal}}A_t/N_t$, so the
		second term is $O_{\mathbb P}(C_{\mathrm{bal}}A_t/N_t)$, which 
		is $o_{\mathbb P}(\sqrt{1/N_t})$ precisely when $A_t=o(\sqrt{N_t})$. 
		The third line contributes the noise-estimation terms 
		$O\bigl(\frac{1}{\sqrt{p\min_\ell n^\ell_t}}+\frac{k}{p}\bigr)$,
		which vanish under Assumption~\ref{assum:stab}(ii). 
		
	\end{proof}
	
	\begin{proof}[Proof of Lemma~\ref{lem:decomp_new}]
		By \eqref{eq:weights} and $\sum_\ell\omega_{\ell,t,j}=1$,
		\[
		\tilde\lambda_{t,j}
		=(1-\beta_t)\tilde\lambda_{t-1,j}
		+\beta_t\sum_{\ell\in\mathcal A_t}\omega_{\ell,t,j}
		\check\lambda_{\ell,t,j}
		=(1-\beta_t)\tilde\lambda_{t-1,j}
		+\beta_t\bigl(\lambda_{t,j}+\bar\varepsilon_{t,j}\bigr).
		\]
		Subtracting $\lambda_{t,j}=(1-\beta_t)\lambda_{t,j}
		+\beta_t\lambda_{t,j}$ gives
		$\tilde\lambda_{t,j}-\lambda_{t,j}
		=(1-\beta_t)(\tilde\lambda_{t-1,j}-\lambda_{t,j})
		+\beta_t\bar\varepsilon_{t,j}$, and writing
		$\tilde\lambda_{t-1,j}-\lambda_{t,j}
		=(\tilde\lambda_{t-1,j}-\lambda_{t-1,j})-\delta_{t,j}$. 
		
		Set $E_{t}:=\tilde\lambda_{t,j}-\lambda_{t,j}$ and
		$u_t:=\beta_t\bar\varepsilon_{t,j}-(1-\beta_t)\delta_{t,j}$, so
		that $E_t=(1-\beta_t)E_{t-1}+u_t$. Iterating and using
		$\varrho_{s,t}=\prod_{u=s+1}^t(1-\beta_u)$,
		\[
		E_t=\varrho_{0,t}E_0+\sum_{s=1}^{t}\varrho_{s,t}u_s ,
		\]
		and substituting $\bar\varepsilon_{s,j}
		=\sum_\ell\omega_{\ell,s,j}\varepsilon^{\mathrm{stoch}}_{\ell,s,j}
		+\sum_\ell\omega_{\ell,s,j}b_{\ell,s,j}$ from \eqref{eq:eps}
		produces exactly \eqref{eq:agg_unroll}.
		
		It remains to justify calling $\mathcal M_{t,j}$ a martingale.
		Because $\varrho_{0,t}=\varrho_{0,s}\varrho_{s,t}$ for $s\le t$,
		\begin{equation}
			\label{eq:mart}
			\mathcal M_{t,j}
			=\varrho_{0,t}\sum_{s=1}^{t}
			\frac{\beta_s}{\varrho_{0,s}}\,D_{s,j},
			\qquad
			D_{s,j}:=\sum_{\ell\in\mathcal A_s}\omega_{\ell,s,j}
			\varepsilon^{\mathrm{stoch}}_{\ell,s,j},
		\end{equation}
		and $\{D_{s,j}\}$ is a martingale difference sequence with respect
		to $\{\mathcal F_s\}$ by Step~1 of the proof of
		Lemma~\ref{lem:wavg}. Since $\beta_s$ and $\varrho_{0,s}$ are
		$\mathcal F_{s-1}$-measurable (Assumption~\ref{assum:design}), the
		sum in \eqref{eq:mart} is a martingale transform, hence a
		martingale in $t$. $\mathcal M_{t,j}$ itself is that martingale
		rescaled by the $\mathcal F_{t-1}$-measurable factor  
		$\varrho_{0,t}$. This is the representation used in
		Section~\ref{sec:4.2}.
	\end{proof}
	
	\begin{prop}[Properties of the effective aggregation sample size]
		\label{prop:B3}
		The quantity $\mathsf N_t$ of Definition~\ref{def:Nsharp} satisfies
		the claims stated there:
		\begin{enumerate}[label=\textup{(\roman*)}]
			\item If $\beta_t\equiv\beta\in(0,1]$ and $N_s\equiv N$, then
			$\mathsf N_t$ is strictly increasing in $t$ with
			$\mathsf N_t\uparrow N(2-\beta)/\beta=N/\kappa_\beta$, and
			$\mathsf N_t\equiv N$ when $\beta=1$.
			\item If $\beta_t\to0$ with $\sum_t\beta_t=\infty$ and
			$N_s\equiv N$, then $\mathsf N_t/N\to\infty$.
			\item In general, $\mathsf N_t^{-1}$ obeys the exact recursion
			$\mathsf N_t^{-1}=(1-\beta_t)^2\mathsf N_{t-1}^{-1}
			+\beta_t^2/N_t$, so $\mathsf N_t$ is computable at the server
			with $O(1)$ state.
		\end{enumerate}
	\end{prop}
	
	\begin{proof}[Proof of Proposition~\ref{prop:B3}]
		(iii) is immediate from $\varrho_{s,t}=(1-\beta_t)\varrho_{s,t-1}$
		for $s<t$ and $\varrho_{t,t}=1$ applied to \eqref{eq:Nsharp},
		exactly as in Lemma~\ref{lem:crec}. For (i), with $N_s\equiv N$ we
		have $\mathsf N_t^{-1}=a_t/N$ where $a_t=\sum_{s\le t}
		\varrho^2_{s,t}\beta^2$ is the quantity of Lemma~\ref{lem:B2}(iii)
		with $\gamma=\beta$. Lemma gives
		$a_t\uparrow\beta/(2-\beta)$, i.e.\
		$\mathsf N_t\uparrow N(2-\beta)/\beta$, and $a_t=1$ for
		$\beta=1$. For (ii), Lemma~\ref{lem:B2}(iii) gives $a_t\to0$,
		i.e.\ $\mathsf N_t/N=1/a_t\to\infty$.
	\end{proof}
	
	\begin{proof}[Proof of Theorem~\ref{thm:fusion}]
		We bound the three terms of \eqref{eq:agg_unroll} separately.
		
		\medskip\noindent
		\textbf{Step 1 (martingale term).}
		By \eqref{eq:mart} the increments $\varrho_{s,t}\beta_sD_{s,j}$ are
		orthogonal in $L^2$, since for $s<s'$,
		$\mathbb E[\varrho_{s,t}\beta_sD_{s,j}\varrho_{s',t}\beta_{s'}
		D_{s',j}]
		=\mathbb E[\varrho_{s,t}\beta_sD_{s,j}\varrho_{s',t}\beta_{s'}
		\mathbb E(D_{s',j}\mid\mathcal F_{s'-1})]=0$, using that
		martingale-difference property $\mathbb E[D_{s',j}|\mathcal F_{s'-1}]=0$.
		Hence, by Lemma~\ref{lem:wavg}, 
		\begin{align*}
			\mathbb E\bigl[\mathcal M^2_{t,j}\bigr]
			=&\sum_{s=1}^{t}\varrho^2_{s,t}\beta_s^2\,
			\mathbb E\bigl[\operatorname{Var}(D_{s,j}\mid\mathcal F_{s-1})
			\bigr]
			=\mathbb E\bigl[\mathsf v^2_{t,j}\bigr]
			\bigl(1+o(1)\bigr) ,
		\end{align*}
		by \eqref{eq:condvar_round} and \eqref{eq:vsharp}, so Chebyshev's
		inequality gives
		$\mathcal M_{t,j}=O_{\mathbb P}(\mathsf v_{t,j})$, the first term of
		\eqref{eq:fusion_rate}. For the crude form, bound
		$\omega^2_{\ell,s,j}\le C^2_\omega A_s^{-2}$ by
		Lemma~\ref{lem:weights}, $\vartheta^2_{\ell,s,j}\le{ T}_{s,j}$
		and $(n^\ell_s)^{-1}\le C_{\mathrm{bal}}A_s/N_s$ by
		Assumption~\ref{assum:balance}, giving
		\[
		\mathsf v^2_{t,j}
		\le C^2_\omega C_{\mathrm{bal}}\sum_{s=1}^{t}
		\varrho^2_{s,t}\beta_s^2\frac{{ T}_{s,j}}{N_s}
		\le C^2_\omega C_{\mathrm{bal}}\bar{ T}_{t,j}
		\sum_{s=1}^{t}\frac{\varrho^2_{s,t}\beta^2_s}{N_s}
		=\frac{C^2_\omega C_{\mathrm{bal}}\bar{ T}_{t,j}}{\mathsf N_t},
		\]
		with $\bar{ T}_{t,j}=\max_{s\le t}{ T}_{s,j}=O(1)$ by
		\eqref{eq:varthetabounds} and $\mathsf N_t$ as in
		\eqref{eq:Nsharp}. Observe that only the
		crude conditional-variance bound of Lemma~\ref{lem:B1}(ii) was
		used, so no moment beyond Assumption~\ref{assum:moment} is needed.
		
		\medskip\noindent
		\textbf{Step 2 (predictable term).}
		Since $\sum_\ell\omega_{\ell,s,j}=1$ and
		$|\sum_\ell\omega_{\ell,s,j}b_{\ell,s,j}|
		\le\max_\ell|b_{\ell,s,j}|$,
		\[
		\Bigl|\sum_{s=1}^{t}\varrho_{s,t}\Bigl\{\beta_s
		\sum_{\ell}\omega_{\ell,s,j}b_{\ell,s,j}
		-(1-\beta_s)\delta_{s,j}\Bigr\}\Bigr|
		\le\sum_{s=1}^{t}\varrho_{s,t}
		\Bigl\{\beta_s\max_{\ell\in\mathcal A_s}|b_{\ell,s,j}|
		+(1-\beta_s)|\delta_{s,j}|\Bigr\},
		\]
		which, together with Lemma~\ref{lem:B1}(iii), is the second term of
		\eqref{eq:fusion_rate}.
		
		\medskip\noindent
		\textbf{Step 3 (initialization).}
		$|\varrho_{0,t}(\tilde\lambda_{0,j}-\lambda_{0,j})|
		\le\varrho_{0,t}(\Lambda+\lambda_{\max})=O(\varrho_{0,t})$ by
		Assumption~\ref{assum:clip}.
		
		\medskip\noindent
		\textbf{Step 4 (consistency and sufficient conditions).}
		If $\mathsf N_t\to\infty$ and the second and third terms vanish,
		then $\tilde\lambda_{t,j}-\lambda_{t,j}=o_{\mathbb P}(1)$. We make
		the sufficient conditions precise. The lag part is
		$\sum_s\varrho_{s,t}\beta_sx_s$ with
		$x_s=\max_\ell|b_{\ell,s,j}|\to0$ under
		Assumptions~\ref{assum:drift}--\ref{assum:step} (each constituent of
		\eqref{eq:blag} vanishes, and $A_s/N_s\to0$ under
		Assumption~\ref{assum:balance} with $\min_\ell n^\ell_s\to\infty$). 
		By Lemma~\ref{lem:B2}(ii) it tends to $0$ whenever
		$\varrho_{0,t}\to0$. The drift part is $\sum_s\varrho_{s,t}y_s$
		with $y_s=(1-\beta_s)|\delta_{s,j}|\le|\delta_{s,j}|\le\sqrt{k}\zeta_s$ 
		by Assumption~\ref{assum:drift}, and 
		Lemma~\ref{lem:B2}(iv) shows that it tends to $0$
		\begin{itemize}
			\item if $\beta_t\equiv\beta\in(0,1]$ and $\zeta_t\to0$. 
			\item if $\beta_t\to0$ with $\sum_t\beta_t=\infty$, only under
			the additional requirement $\zeta_t=o(\beta_t)$.
		\end{itemize}
		The second bullet is a genuine restriction, not a technicality:
		$\sum_s\varrho_{s,t}\asymp\beta_t^{-1}\to\infty$, so a vanishing
		step size averages over an ever longer window and can track a
		moving target only if the target moves more slowly than the window
		grows. 
	\end{proof}


	\subsection{Proofs for Section~\ref{sec:4.2}}
	\label{app:C1}
	
	\begin{lemma}[Properties of the influence function]
		\label{lem:C1}
		Let $\varsigma_{\ell,s,i,j}$ be defined by
		\eqref{eq:influence} and put
		$$\varsigma^{\mathrm c}_{\ell,s,i,j}
		:=\varsigma_{\ell,s,i,j}
		-\mathbb E[\varsigma_{\ell,s,i,j}\mid\mathcal F_{s-1}]. $$ Under
		Assumptions~\ref{assum:moment}--\ref{assum:init} and
		\eqref{eq:clt_cond}:
		\begin{enumerate}[label=\textup{(\roman*)}]
			\item \textup{(Exact identity)}
			\begin{equation}
				\label{eq:C1id}
				\frac{1}{n^\ell_s}\sum_{i=1}^{n^\ell_s}
				\varsigma_{\ell,s,i,j}
				=\frac{r_{\ell,s,j}-\chi_{\ell,s}(\lambda_{s,j})}
				{\chi'_{\ell,s}(\lambda_{s,j})},
				\qquad
				\frac{1}{n^\ell_s}\sum_{i=1}^{n^\ell_s}
				\varsigma^{\mathrm c}_{\ell,s,i,j}
				=\frac{\varpi_{\ell,s,j}}
				{\chi'_{\ell,s}(\lambda_{s,j})} .
			\end{equation}
			\item \textup{(Moments)}
			$$\mathbb E[\varsigma_{\ell,s,i,j}\mid\mathcal F_{s-1}]
			=\{\bar r_{\ell,s,j}-\chi_{\ell,s}(\lambda_{s,j})\}
			/\chi'_{\ell,s}(\lambda_{s,j})
			=O\bigl(b^{\mathrm{lag}}_{\ell,s}\bigr)
			=o_{\mathbb P}\bigl((n^\ell_s)^{-1/2}\bigr)$$ and
			$\operatorname{Var}(\varsigma_{\ell,s,i,j}\mid\mathcal F_{s-1})
			=\vartheta^2_{\ell,s,j}(1+o_{\mathbb P}(1))$.
			\item \textup{(Uniform moment bound)}
			$\mathbb E[|\varsigma^{\mathrm c}_{\ell,s,i,j}|^{2+\epsilon_1/2}
			\mid\mathcal F_{s-1}]\le C<\infty$ uniformly in $\ell,s,i,p$.
		\end{enumerate}
	\end{lemma}
	
	\begin{proof}[Proof of Lemma~\ref{lem:C1}]
		(i) Averaging \eqref{eq:influence} over $i$ and using
		$\hat{\bm S}_{\ell,s}=(n^\ell_s)^{-1}\sum_i\bm X_{\ell,s,i}
		\bm X_{\ell,s,i}^\top$ together with the definition of $r_{\ell,s,j}
		=\bm u_{\ell,s-1,j}^\top\hat{\bm S}_{\ell,s}\bm u_{\ell,s-1,j}$
		gives the first identity. Subtracting its conditional mean and
		using $r_{\ell,s,j}-\bar r_{\ell,s,j}=\varpi_{\ell,s,j}$
		(Lemma~\ref{lem:decomp}) gives the second. Note that
		$\chi_{\ell,s}(\lambda_{s,j})$ and
		$\chi'_{\ell,s}(\lambda_{s,j})$ are $\mathcal F_{s-1}$-measurable
		deterministic functions of $(c_{\ell,s-1},\sigma^2_s,\lambda_{s,j})$,
		so no measurability issue arises.
		
		(ii) The conditional mean is immediate from (i) and
		$\mathbb E[r_{\ell,s,j}\mid\mathcal F_{s-1}]=\bar r_{\ell,s,j}$. 
		Step~1 of the proof of Theorem~\ref{thm:attenuation} gives
		$|\bar r_{\ell,s,j}-\chi_{\ell,s}(\lambda_{s,j})|
		=O(b^{\mathrm{lag}}_{\ell,s})$, and $\chi'_{\ell,s}$ is bounded
		away from $0$ and $\infty$ on the supercritical range. By
		Lemma~\ref{lem:inverse}(iii),
		$\chi'\ge\min\{1,2/(1+\sqrt{\bar c})\}$, while
		$\chi'=1+c(1-c)\sigma^4/(\lambda+c\sigma^2)^2\le
		1+c\sigma^4/(c\sigma^4)=2$ for $\lambda\ge\sigma^2\sqrt c$. The
		rate $o_{\mathbb P}((n^\ell_s)^{-1/2})$ is exactly
		\eqref{eq:clt_cond}. For the variance, conditionally on
		$\mathcal F_{s-1}$ the summands are i.i.d., so by
		Lemma~\ref{lem:decomp}
		\[
		\operatorname{Var}(\varsigma_{\ell,s,i,j}\mid\mathcal F_{s-1})
		=\frac{n^\ell_s\operatorname{Var}(\varpi_{\ell,s,j}
			\mid\mathcal F_{s-1})}{\chi'_{\ell,s}(\lambda_{s,j})^2}
		=\frac{2\bar r^2_{\ell,s,j}(1+o_{\mathbb P}(1))}
		{\chi'_{\ell,s}(\lambda_{s,j})^2}
		=\vartheta^2_{\ell,s,j}\bigl(1+o_{\mathbb P}(1)\bigr),
		\]
		using $\bar r_{\ell,s,j}\to\chi_{\ell,s}(\lambda_{s,j})$ and
		\eqref{eq:local_clt}.
		
		(iii) With $\bm b=\bm\Sigma_s^{1/2}\bm u_{\ell,s-1,j}$ one has
		$\varsigma^{\mathrm c}_{\ell,s,i,j}
		=\{(\bm b^\top\bm z_{\ell,s,i})^2-\|\bm b\|^2\}
		/\chi'_{\ell,s}(\lambda_{s,j})$, and by
		Assumption~\ref{assum:moment} and the Marcinkiewicz--Zygmund
		inequality
		$\mathbb E|(\bm b^\top\bm z)^2-\|\bm b\|^2|^{2+\epsilon_1/2}
		\le C\mathbb E|\bm b^\top\bm z|^{4+\epsilon_1}
		\le C'\|\bm b\|^{4+\epsilon_1}
		\le C''(\lambda_{\max}+\sigma^2_{\max})^{2+\epsilon_1/2}$,
		while $\chi'$ is bounded below.
	\end{proof}
	
	\begin{proof}[Proof of Lemma~\ref{lem:linrep}]
		\textbf{Step 1 (from the recursion to the influence functions).}
		By Lemma~\ref{lem:decomp_new},
		\begin{align}
			\label{eq:C2start}
			\tilde\lambda_{t,j}-\lambda_{t,j}
			=&\sum_{s=1}^{t}\varrho_{s,t}\beta_s
			\sum_{\ell\in\mathcal A_s}\omega_{\ell,s,j}
			\varepsilon^{\mathrm{stoch}}_{\ell,s,j}
			+\sum_{s=1}^{t}\varrho_{s,t}
			\Bigl\{\beta_s\sum_{\ell}\omega_{\ell,s,j}b_{\ell,s,j}
			-(1-\beta_s)\delta_{s,j}\Bigr\}\\
			&+\varrho_{0,t}\bigl(\tilde\lambda_{0,j}-\lambda_{0,j}\bigr).\nonumber
		\end{align}
		By Lemma~\ref{lem:B1}(i) and Lemma~\ref{lem:C1}(i),
		\begin{equation}
			\label{eq:C2rep}
			\varepsilon^{\mathrm{stoch}}_{\ell,s,j}
			=\frac{\varpi_{\ell,s,j}}{\chi'_{\ell,s}(\lambda_{s,j})}
			+\varrho^{\mathrm{rem}}_{\ell,s,j}
			=\frac1{n^\ell_s}\sum_{i=1}^{n^\ell_s}
			\varsigma^{\mathrm c}_{\ell,s,i,j}
			+\varrho^{\mathrm{rem}}_{\ell,s,j},
		\end{equation}
		$\varrho^{\mathrm{rem}}_{\ell,s,j}
		=O_{\mathbb P}\Bigl(\frac1{n^\ell_s}
		+\frac1{\sqrt{p\,n^\ell_s}}+\frac kp\Bigr),$ 
		and by Lemma~\ref{lem:C1}(ii) the difference between
		$\varsigma^{\mathrm c}$ and $\varsigma$ contributes
		$$-\{\bar r_{\ell,s,j}-\chi_{\ell,s}(\lambda_{s,j})\}
		/\chi'_{\ell,s}(\lambda_{s,j})=O(b^{\mathrm{lag}}_{\ell,s})$$ per
		node. Substituting into \eqref{eq:C2start} produces exactly the
		stated linear form with
		\begin{align}
			\label{eq:C2R}
			R_{t,j}
			=&\sum_{s\le t}\varrho_{s,t}\beta_s\sum_{\ell}\omega_{\ell,s,j}
			\Bigl\{b_{\ell,s,j}
			-\frac{\bar r_{\ell,s,j}-\chi_{\ell,s}(\lambda_{s,j})}
			{\chi'_{\ell,s}(\lambda_{s,j})}
			+\varrho^{\mathrm{rem}}_{\ell,s,j}\Bigr\}\\
			&-\sum_{s\le t}\varrho_{s,t}(1-\beta_s)\delta_{s,j}
			+\varrho_{0,t}\bigl(\tilde\lambda_{0,j}-\lambda_{0,j}\bigr)
			+\bigl(\lambda_{t,j}-\lambda_j\bigr).\nonumber
		\end{align}
		
		\medskip\noindent
		\textbf{Step 2 (bounding $R_{t,j}$).}
		We treat the five groups of terms in \eqref{eq:C2R}.
		\begin{enumerate}[label=(\arabic*),leftmargin=*]
			\item By the proof of Lemma~\ref{lem:B1}(iii), the
			$\mathcal F_{s-1}$-measurable part of $b_{\ell,s,j}$ is
			$\chi^{-1}(\bar r_{\ell,s,j})-\lambda_{s,j}
			=\{\bar r_{\ell,s,j}-\chi_{\ell,s}(\lambda_{s,j})\}
			/\chi'_{\ell,s}(\lambda_{s,j})
			+O((b^{\mathrm{lag}}_{\ell,s})^2)$ by a second-order Taylor
			expansion with the uniform curvature bound $K_{\mathrm{cv}}$. 
			Hence the first brace in \eqref{eq:C2R} is
			$O((b^{\mathrm{lag}}_{\ell,s})^2)
			+O_{\mathbb P}(1/n^\ell_s)+O((p\,n^\ell_s)^{-1/2}+k/p)$.
			\item The terms $O((b^{\mathrm{lag}}_{\ell,s})^2)$ are dominated
			by $O(b^{\mathrm{lag}}_{\ell,s})$ since
			$b^{\mathrm{lag}}_{\ell,s}\to0$, and
			$$\sqrt{\mathsf N_t}\sum_s\varrho_{s,t}\beta_s
			\max_\ell b^{\mathrm{lag}}_{\ell,s}\to0$$ by
			Assumption~\ref{assum:stab}(ii).
			\item The curvature terms satisfy, using
			$\sum_\ell\omega_{\ell,s,j}=1$ and
			Lemma~\ref{lem:B2}(i),
			\[
			\sqrt{\mathsf N_t}\sum_{s\le t}\varrho_{s,t}\beta_s
			\max_{\ell\in\mathcal A_s}\frac1{n^\ell_s}
			\ \le\ \sqrt{\mathsf N_t}\,
			\sup_{s\le t}\frac{1}{\min_{\ell\in\mathcal A_s}n^\ell_s}
			\ \longrightarrow\ 0
			\]
			provided $\min_{\ell\in\mathcal A_s}n^\ell_s\gg
			\sqrt{\mathsf N_t}$, which is
			Assumption~\ref{assum:stab}(iii) in the quantitative form
			recorded in Remark~\ref{rem:consistency_pictures}. Equivalently
			$A_s=o(n^\ell_s)$ when $\beta$ is fixed, since then
			$\mathsf N_t\asymp N_s\asymp A_sn^\ell_s$. The plug-in terms
			$O((p\,n^\ell_s)^{-1/2}+k/p)$ are
			$o(\mathsf N_t^{-1/2})$ under the same condition together with
			$k\sqrt{n^\ell_s}=o(p)$.
			\item The drift and initialization terms are
			$o(\mathsf N_t^{-1/2})$ by Assumption~\ref{assum:stab}(ii).
			\item Finally $\lambda_{t,j}-\lambda_j\to0$ by
			Assumption~\ref{assum:stab}(ii). For the representation to hold
			with the centring at $\lambda_j$ one needs the slightly stronger
			$\sqrt{\mathsf N_t}\,|\lambda_{t,j}-\lambda_j|\to0$, which is
			not implied by the displayed conditions of
			Assumption~\ref{assum:stab}(ii). Centring at the moving target
			$\lambda_{t,j}$, as in Theorem~\ref{thm:nonasymp}, requires
			no such condition.
		\end{enumerate}
		Collecting, $R_{t,j}=o_{\mathbb P}(\mathsf N_t^{-1/2})$, which is
		\eqref{eq:linrep}. 
	\end{proof}
	
	\begin{proof}[Proof of Theorem~\ref{thm:clt}]
		\textbf{Step 1 (a genuine martingale array).}
		Write
		\[
		\mathcal L_{t,j}
		:=\sum_{s=1}^{t}\varrho_{s,t}\beta_s D_{s,j},
		\qquad
		D_{s,j}:=\sum_{\ell\in\mathcal A_s}
		\frac{\omega_{\ell,s,j}}{n^\ell_s}
		\sum_{i=1}^{n^\ell_s}\varsigma^{\mathrm c}_{\ell,s,i,j},
		\]
		so that, by Lemma~\ref{lem:linrep}, 
		$\tilde{\lambda}_{t,j}-\lambda_j=\mathcal L_{t,j}
		+o_{\mathbb P}(\mathsf N_t^{-1/2})$. Since $\omega_{\ell,s,j}$,
		$n^\ell_s$, $\mathcal A_s$, $\chi_{\ell,s}$ and $\chi'_{\ell,s}$ are
		all $\mathcal F_{s-1}$-measurable
		(Assumption~\ref{assum:design}) and
		$\mathbb E[\varsigma^{\mathrm c}_{\ell,s,i,j}\mid\mathcal F_{s-1}]=0$
		by construction, $\{D_{s,j},\mathcal F_s\}$ is a martingale
		difference sequence. Because $\varrho_{s,t}$ depends on $t$ through
		the future values $\beta_u$, $u>s$, we use the representation
		\eqref{eq:mart}:
		\[
		\mathcal L_{t,j}=\varrho_{0,t}\sum_{s=1}^{t}\xi_{s,j},
		\qquad
		\xi_{s,j}:=\frac{\beta_s}{\varrho_{0,s}}D_{s,j},
		\]
		in which $\beta_s/\varrho_{0,s}$ is $\mathcal F_{s-1}$-measurable,
		so $\{\xi_{s,j}\}$ is a bona fide martingale difference sequence
		and $\mathcal L_{t,j}$ is a deterministic multiple of a martingale.
		(For a deterministic schedule $\{\beta_t\}$ this step is
		unnecessary.)
		
		\medskip\noindent
		\textbf{Step 2 (conditional variance).}
		Conditionally on $\mathcal F_{s-1}$ the observations are
		independent within and across nodes
		(Assumptions~\ref{assum:moment},~\ref{assum:design}), so by
		Lemma~\ref{lem:C1}(ii)
		\begin{equation}
			\label{eq:C3var}
			\operatorname{Var}\bigl(D_{s,j}\mid\mathcal F_{s-1}\bigr)
			=\sum_{\ell\in\mathcal A_s}
			\frac{\omega^2_{\ell,s,j}}{(n^\ell_s)^2}\,n^\ell_s\,
			\operatorname{Var}(\varsigma_{\ell,s,i,j}\mid\mathcal F_{s-1})
			=\sum_{\ell\in\mathcal A_s}
			\frac{\omega^2_{\ell,s,j}\vartheta^2_{\ell,s,j}}{n^\ell_s}
			\bigl(1+o_{\mathbb P}(1)\bigr),
		\end{equation}
		whence
		$\sum_{s\le t}\varrho^2_{s,t}\beta_s^2
		\operatorname{Var}(D_{s,j}\mid\mathcal F_{s-1})
		=\mathsf v^2_{t,j}(1+o_{\mathbb P}(1))$ with $\mathsf v^2_{t,j}$
		as in \eqref{eq:vsharp}. The $o_{\mathbb P}(1)$ may be taken
		uniform in $s$ over the effective window because
		$\varrho^2_{s,t}$ damps the early terms geometrically and, by
		Assumption~\ref{assum:stab}(i)--(ii), the parameters stabilise.
		Thus the normalised conditional variance
		$\mathsf v^{-2}_{t,j}\sum_s\varrho^2_{s,t}\beta_s^2
		\operatorname{Var}(D_{s,j}\mid\mathcal F_{s-1})\to1$ in
		probability.
		
		\medskip\noindent
		\textbf{Step 3 (Lindeberg is automatic).}
		The stated conditional Lindeberg condition is the standard one for
		the array $\{\varrho_{s,t}\beta_s\omega_{\ell,s,j}
		\varsigma^{\mathrm c}_{\ell,s,i,j}/n^\ell_s\}$. We verify that it
		follows from Assumption~\ref{assum:moment} alone, so that it need
		not be assumed. Put $\epsilon:=\epsilon_1/2>0$ and
		$a_{s,\ell}:=\varrho_{s,t}\beta_s\omega_{\ell,s,j}/n^\ell_s$. By
		Lemma~\ref{lem:C1}(iii) the Lyapunov ratio is
		\[
		\Lambda_t
		:=\frac{1}{\mathsf v^{2+\epsilon}_{t,j}}
		\sum_{s\le t}\sum_{\ell\in\mathcal A_s}n^\ell_s\,
		a_{s,\ell}^{2+\epsilon}\,
		\mathbb E\bigl[|\varsigma^{\mathrm c}_{\ell,s,i,j}|^{2+\epsilon}
		\bigm|\mathcal F_{s-1}\bigr]
		\le\frac{C}{\mathsf v^{2+\epsilon}_{t,j}}
		\sum_{s\le t}(\varrho_{s,t}\beta_s)^{2+\epsilon}
		\sum_{\ell\in\mathcal A_s}
		\frac{\omega^{2+\epsilon}_{\ell,s}}{(n^\ell_s)^{1+\epsilon}} .
		\]
		By Lemma~\ref{lem:weights} and
		Assumptions~\ref{assum:balance} and~\ref{assum:stab}(iii),
		$\omega_{\ell,s,j}\asymp A_s^{-1}$ and $n^\ell_s\asymp N_s/A_s$. 
		Only {bounded} batch ratios are used here, not their convergence
		to one. So
		$$\sum_\ell\omega^{2+\epsilon}_{\ell,s}(n^\ell_s)^{-1-\epsilon}
		\asymp A_s\cdot A_s^{-2-\epsilon}(A_s/N_s)^{1+\epsilon}
		=N_s^{-1-\epsilon}, $$ while
		$\mathsf v^2_{t,j}\asymp\bar{ T}_{t,j}/\mathsf N_t$. For a fixed
		$\beta$ and $N_s\equiv N$ this gives
		$$\Lambda_t\asymp\mathsf N_t^{1+\epsilon/2}
		\beta^{1+\epsilon}N^{-1-\epsilon}\asymp N^{-\epsilon/2}\to0,$$ for a
		vanishing $\beta_t$ the same computation gives
		\begin{align*}
			\Lambda_t\asymp&\mathsf N_t^{1+\epsilon/2}
			\sum_s(\varrho_{s,t}\beta_s)^{2+\epsilon}N_s^{-1-\epsilon}\\
			\le&\mathsf N_t^{1+\epsilon/2}
			(\max_s\varrho_{s,t}\beta_s/N_s)^{\epsilon}
			\sum_s\varrho^2_{s,t}\beta^2_s/N_s \\
			=&\mathsf N_t^{\epsilon/2}
			(\max_s\varrho_{s,t}\beta_s/N_s)^{\epsilon}\to0
		\end{align*}
		because $\max_s\varrho_{s,t}\beta_s/N_s
		=O(\mathsf N_t^{-1})$ (Step~2 of the proof of
		Lemma~\ref{lem:conc} below). Lyapunov's condition implies
		Lindeberg's.
		
		\medskip\noindent
		\textbf{Step 4 (martingale CLT).}
		Steps~1--3 verify the hypotheses of the martingale central limit
		theorem for triangular arrays
		\citep[Corollary 3.1]{hall1980martingale}, which states that a square-integrable
		martingale difference array whose normalised conditional variances
		converge in probability to $1$ and which satisfies the conditional
		Lindeberg condition obeys
		$\mathsf v^{-1}_{t,j}\mathcal L_{t,j}
		\stackrel{d}{\longrightarrow}\mathcal N(0,1)$. The norming
		$\mathsf v_{t,j}$ is itself random, being a function of the
		$\mathcal F_{s-1}$-measurable weights, so the cited result is applied
		in its random-norming form and the conclusion is stable convergence. 
		Step~2 supplies exactly the required convergence in probability of the
		normalised conditional variance. By Lemma~\ref{lem:linrep}
		and $\mathsf v_{t,j}\asymp\mathsf N_t^{-1/2}$, the remainder
		$R_{t,j}=o_{\mathbb P}(\mathsf N_t^{-1/2})
		=o_{\mathbb P}(\mathsf v_{t,j})$, so Slutsky's theorem gives the
		first display of Theorem~\ref{thm:clt}. Multiplying by
		$\sqrt{\mathsf N_t}\mathsf v_{t,j}\to\Omega_j$, which converges in
		probability under Assumption~\ref{assum:stab}(i)--(ii),(v), gives the
		second. Note that (v) is what makes
		$\mathsf N_t\mathsf v^2_{t,j}$ converge at all, bounded batch ratios
		alone permitting indefinite oscillation.
		
		\medskip\noindent
		\textbf{Step 5 (oracle weights and the closed form).}
		If $\omega=\omega^{\mathrm{opt}}_{\cdot,s,j}$ then, by
		Proposition~\ref{prop:efficiency},
		$\sum_\ell\omega^2_{\ell,s,j}\vartheta^2_{\ell,s,j}/n^\ell_s
		=\mathcal V_{s,j}(\omega^{\mathrm{opt}})=\mathcal I^{-1}_{s,j}$,
		so $$\mathsf v^2_{t,j}=\sum_s\varrho^2_{s,t}\beta_s^2
		\mathcal I^{-1}_{s,j} \quad \text{and} \quad 
		\Omega^2_j=\lim_t\mathsf N_t\sum_s\varrho^2_{s,t}\beta_s^2
		\mathcal I^{-1}_{s,j},$$ as asserted. If in addition
		$c_{\ell,s}\to c$ and $\sigma^2_s\to\sigma^2$ for every $\ell$,
		then $\vartheta^2_{\ell,s,j}\to\vartheta^2_j$ does not depend on
		$\ell$, so $\mathcal I_{s,j}=N_s/\vartheta^2_j(1+o(1))$ and
		\[
		\mathsf v^2_{t,j}
		=\vartheta^2_j\sum_{s\le t}\frac{\varrho^2_{s,t}\beta^2_s}{N_s}
		\bigl(1+o(1)\bigr)
		=\frac{\vartheta^2_j}{\mathsf N_t}\bigl(1+o(1)\bigr)
		\qquad\Longrightarrow\qquad
		\Omega^2_j=\vartheta^2_j,
		\]
		by Definition~\ref{def:Nsharp}. Substituting \eqref{eq:local_clt}
		yields \eqref{eq:clt_closed}. Note that the same conclusion holds
		for the soft-max weights whenever \eqref{eq:releff_cond} holds, since
		then $\mathcal V_{s,j}(\omega)=\mathcal I^{-1}_{s,j}
		(1+o_{\mathbb P}(1))$ by Proposition~\ref{prop:softmax_releff}. In
		general $\mathcal V_{s,j}(\omega)=\mathcal I^{-1}_{s,j}
		\{1+(1-\mathrm{RE}_{s,j})\mathcal K_{s,j}\}$ exactly, by
		\eqref{eq:releff}, so $\Omega^2_j$ is inflated by the factor
		$1+(1-\mathrm{RE}_j)\mathcal K_j$ and by
		Proposition~\ref{prop:dispersion} by at most $1+\mathcal K_j$
		whatever the weights. The rate is unaffected.
		
		\medskip\noindent
		\textbf{Step 6 (joint convergence).}
		For a fixed vector $\bm a\in\mathbb R^k$ apply Steps~1--4 to the
		scalar array built from
		$\sum_j a_j\varsigma^{\mathrm c}_{\ell,s,i,j}$. Its conditional
		variance is $\sum_{j,j'}a_ja_{j'}\operatorname{Cov}
		(\varsigma_{\ell,s,i,j},\varsigma_{\ell,s,i,j'}\mid
		\mathcal F_{s-1})$, and by Step~5 of the proof of
		Theorem~\ref{thm:local_clt} the off-diagonal terms are
		$O_{\mathbb P}((N^{\mathrm{eff}}_{\ell,s})^{-1})$ relative to the
		diagonal, hence negligible. The Cram\'er--Wold device then gives
		joint convergence to $\mathcal N(\bm 0,\operatorname{diag}
		(\Omega^2_1,\dots,\Omega^2_k))$.
	\end{proof}
	
	\begin{coro}
		\label{cor:C4}
		In the setting of Theorem~\ref{thm:clt}:
		\begin{enumerate}[label=\textup{(\roman*)}]
			\item If $\beta_t\equiv\beta\in(0,1]$, $N_s\equiv N$, the nodes
			are homogeneous and the parameters have stabilised, then
			$\mathsf v^2_{t,j}\to\kappa_\beta\vartheta^2_j/N$ with
			$\kappa_\beta=\beta/(2-\beta)$, so that
			$\sqrt N(\tilde{\lambda}_{t,j}-\lambda_j)\Rightarrow
			\mathcal N(0,\kappa_\beta\vartheta^2_j)$. In particular
			$\kappa_1=1$ and $\kappa_\beta\to0$ as $\beta\to0$.
			\item As $\lambda_j\downarrow\sigma^2\sqrt c$, the limiting
			variance \eqref{eq:clt_closed} satisfies
			$\Omega^2_j\to\sigma^4(1+\sqrt c)^2/2<\infty$ and
			$\Omega_j/\lambda_j\to(1+\sqrt c)/\sqrt{2c}$.
		\end{enumerate}
	\end{coro}
	
	\begin{proof}[Proof of Corollary~\ref{cor:C4}]
		(i) By Proposition~\ref{prop:B3}(i),
		$\mathsf N_t\uparrow N(2-\beta)/\beta=N/\kappa_\beta$, and by
		Step~5 above $\mathsf v^2_{t,j}=\vartheta^2_j/\mathsf N_t(1+o(1))
		\to\kappa_\beta\vartheta^2_j/N$. The two special cases are
		immediate from $\kappa_\beta=\beta/(2-\beta)$.
		
		(ii) By Theorem~\ref{thm:clt},
		$\Omega^2_j=\vartheta^2_j=2\chi(\lambda_j)^2/\chi'(\lambda_j)^2$.
		Lemma~\ref{lem:inverse}(ii) gives
		$\chi(\sigma^2\sqrt c)=\sigma^2$ and Lemma~\ref{lem:inverse}(iii)
		gives $\chi'(\sigma^2\sqrt c)=2/(1+\sqrt c)$. Hence
		$\Omega^2_j\to2\sigma^4(1+\sqrt c)^2/4
		=\sigma^4(1+\sqrt c)^2/2$, which is finite, in contrast with the
		eigenvalue route, for which $\Phi'(\lambda)=1-c\sigma^4/\lambda^2
		\downarrow0$ at the same point. Dividing by
		$\lambda_j\to\sigma^2\sqrt c$ gives
		$\Omega_j/\lambda_j\to\sigma^2(1+\sqrt c)/(\sqrt2\sigma^2\sqrt c)
		=(1+\sqrt c)/\sqrt{2c}$. The convergence in Theorem~\ref{thm:clt} is
		not uniform as $\lambda_j\downarrow\sigma^2\sqrt c$. The event
		$\{r_{\ell,t,j}<\check\sigma^2_{\ell,t}\}$, on which the
		truncation in \eqref{eq:local_debias} is active, has probability
		bounded away from $0$ when $\chi(\lambda_j)-\sigma^2
		=\lambda_ja(\lambda_j)$ is of the order of the sampling standard
		deviation $\vartheta_j/\sqrt{n}$, and $a(\lambda_j)\downarrow0$ at
		the transition.
	\end{proof}
	
	\begin{proof}[Proof of Corollary~\ref{cor:feasible}]
		\textbf{Step 1 (the recursion reproduces $\mathsf v^2$).}
		Put $$g_{s,j}:=\sum_{\ell\in\mathcal A_s}
		\omega^2_{\ell,s,j}\vartheta^2_{\ell,s,j}/n^\ell_s \quad \text{and} 
		\quad 
		\hat g_{s,j}:=\sum_{\ell\in\mathcal A_s}
		\omega^2_{\ell,s,j}\hat\vartheta^2_{\ell,s,j}/n^\ell_s. $$ Exactly as
		in Lemma~\ref{lem:B2}(iii),
		$\mathsf v^2_{t,j}=\sum_{s\le t}\varrho^2_{s,t}\beta^2_sg_{s,j}$
		satisfies
		$\mathsf v^2_{t,j}=(1-\beta_t)^2\mathsf v^2_{t-1,j}
		+\beta_t^2g_{t,j}$, which is \eqref{eq:varrec} with $g$ in place
		of $\hat g$. Hence
		$\hat{\mathsf v}^2_{t,j}-\mathsf v^2_{t,j}
		=\sum_{s\le t}\varrho^2_{s,t}\beta^2_s(\hat g_{s,j}-g_{s,j})$.
		
		\medskip\noindent
		\textbf{Step 2 (uniform consistency of the plug-in).}
		By \eqref{eq:local_clt}, $\vartheta^2$ is a continuously
		differentiable function of $(\lambda,\sigma^2,c)$ on the compact
		set determined by Assumptions~\ref{assum:bbp}--\ref{assum:hd}, on
		which $\lambda^2+2c\sigma^2\lambda+c\sigma^4$ is bounded away from
		zero. Hence it is Lipschitz there. By
		Theorem~\ref{thm:fusion}, $\tilde{\lambda}_{t,j}-\lambda_{t,j}
		=o_{\mathbb P}(1)$. By Lemma~\ref{lem:sigma},
		$\max_{\ell\in\mathcal A_t}|\check\sigma^2_{\ell,t}-\sigma^2_t|
		=o_{\mathbb P}(1)$ (the maximum over $A_t$ nodes being controlled
		by Assumption~\ref{assum:clip} and Markov's inequality, or
		directly by Lemma~\ref{lem:conc} under
		Assumption~\ref{assum:subg}). And by Lemma~\ref{lem:crec},
		$|c_{\ell,t-1}-c_{\ell,t}|=O(\alpha_t)=o(1)$.
		Therefore
		$\max_{\ell\in\mathcal A_t}
		|\hat\vartheta^2_{\ell,t,j}-\vartheta^2_{\ell,t,j}|
		=o_{\mathbb P}(1)$, and since $\vartheta^2_{\ell,t,j}$ is bounded
		below by \eqref{eq:varthetabounds},
		$|\hat g_{t,j}-g_{t,j}|\le g_{t,j}\,
		\max_\ell|\hat\vartheta^2_{\ell,t,j}
		/\vartheta^2_{\ell,t,j}-1|=o_{\mathbb P}(1)g_{t,j}$.
		
		\medskip\noindent
		\textbf{Step 3 (transfer to the recursion).}
		Fix $\epsilon>0$ and choose $s_0$ such that
		$\mathbb P(\sup_{s>s_0}|\hat g_{s,j}/g_{s,j}-1|>\epsilon)<\epsilon$.
		Splitting the sum,
		\[
		\frac{|\hat{\mathsf v}^2_{t,j}-\mathsf v^2_{t,j}|}
		{\mathsf v^2_{t,j}}
		\le\epsilon
		+\frac{\sum_{s\le s_0}\varrho^2_{s,t}\beta^2_s
			(\hat g_{s,j}+g_{s,j})}{\mathsf v^2_{t,j}}
		\le\epsilon+C\frac{\varrho^2_{s_0,t}}{\mathsf v^2_{t,j}}
		\sum_{s\le s_0}\varrho^2_{s,s_0}\beta^2_s
		\ \longrightarrow\ \epsilon
		\]
		with probability at least $1-\epsilon$, because
		$\varrho_{s_0,t}\to0$ geometrically while
		$\mathsf v^2_{t,j}\asymp\mathsf N_t^{-1}$ decays at most
		polynomially (indeed $\mathsf N_t\uparrow N/\kappa_\beta$ for fixed
		$\beta$, and $\mathsf N_t\to\infty$ subexponentially under
		Assumption~\ref{assum:stab}(iii)), and $\hat g,g$ are bounded by
		Assumption~\ref{assum:clip} and \eqref{eq:varthetabounds}. As
		$\epsilon>0$ was arbitrary,
		$\hat{\mathsf v}^2_{t,j}/\mathsf v^2_{t,j}\to1$ in probability.
		
		\medskip\noindent
		\textbf{Step 4 (coverage).}
		By Theorem~\ref{thm:clt} and Step~3,
		$\hat{\mathsf v}^{-1}_{t,j}(\tilde{\lambda}_{t,j}-\lambda_j)
		=(\mathsf v_{t,j}/\hat{\mathsf v}_{t,j})\cdot
		\mathsf v^{-1}_{t,j}(\tilde{\lambda}_{t,j}-\lambda_j)
		\stackrel{d}{\longrightarrow}\mathcal N(0,1)$ by Slutsky's theorem, so
		\[
		\mathbb P\bigl(\lambda_j\in[\tilde{\lambda}_{t,j}
		\pm z_{1-\alpha/2}\hat{\mathsf v}_{t,j}]\bigr)
		=\mathbb P\bigl(|\hat{\mathsf v}^{-1}_{t,j}
		(\tilde{\lambda}_{t,j}-\lambda_j)|\le z_{1-\alpha/2}\bigr)
		\longrightarrow1-\alpha .
		\]
		Finally, $\hat\vartheta^2_{\ell,t,j}$ depends only on
		$\tilde{\lambda}_{t,j}$ (available at the server),
		$c_{\ell,t-1}$, $\check\sigma^2_{\ell,t}$ and
		$n^\ell_t$. Transmitting $\hat\vartheta^2_{\ell,t,j}$ costs one
		scalar per node per coordinate, and \eqref{eq:varrec} carries $O(k)$
		state at the server.
	\end{proof}

	\section{Proofs for Sections~\ref{sec:4.3}}
	\label{app:proofs4546}
	
	We retain the notation of Appendices~\ref{app:proofs4142}
	and~\ref{app:proofs4344}. Specifically, $L=L_{\bar c}$ is the global Lipschitz
	constant of $\chi^{-1}$ from Lemma~\ref{lem:inverse}(iii),
	$K_{\mathrm{cv}}:=\sup|(\chi^{-1})''|<\infty$ is the curvature bound
	\eqref{eq:chiinv2}, $b^{\mathrm{lag}}_{\ell,t}$ is the corrected lag
	constant \eqref{eq:blag}, and $\varrho_{s,t}=\prod_{u>s}^{t}(1-\beta_u)$.
	Assumption~\ref{assum:clip} (bounded transmission range) is in force,
	so that the constant $C_\omega$ of Lemma~\ref{lem:weights} is
	deterministic. Throughout, $j\le k$ is fixed.

	\begin{proof}[Proof of Lemma~\ref{lem:transfer}]
		Work on the event $\mathcal T_{\ell,t}
		:=\{r_{\ell,t,j}\ge\check\sigma^2_{\ell,t}\}$, on which the
		truncation in \eqref{eq:local_debias} is inactive and
		$\check\lambda_{\ell,t,j}
		=\chi^{-1}(r_{\ell,t,j};c_{\ell,t-1},\check\sigma^2_{\ell,t})$.
		Insert two intermediate quantities and apply the triangle
		inequality:
		\begin{align*}
			\bigl|\check\lambda_{\ell,t,j}-\lambda_{t,j}\bigr|
			\le\underbrace{\bigl|\chi^{-1}(r;c_{\ell,t-1},
				\check\sigma^2_{\ell,t})
				-\chi^{-1}(r;c_{\ell,t-1},\sigma^2_t)\bigr|}_{(\mathrm I)}\\
			+\underbrace{\bigl|\chi^{-1}(r;c_{\ell,t-1},\sigma^2_t)
				-\chi^{-1}(\chi_{\ell,t}(\lambda_{t,j});c_{\ell,t-1},
				\sigma^2_t)\bigr|}_{(\mathrm{II})}\\
			+\underbrace{\bigl|\chi^{-1}(\chi_{\ell,t}(\lambda_{t,j});
				c_{\ell,t-1},\sigma^2_t)-\lambda_{t,j}\bigr|}_{(\mathrm{III})}.	
		\end{align*}

		By Lemma~\ref{lem:inverse}(iii),
		$(\mathrm I)\le C(1+\bar c)|\check\sigma^2_{\ell,t}-\sigma^2_t|$
		and $(\mathrm{II})\le L_{\ell,t}|r_{\ell,t,j}
		-\chi_{\ell,t}(\lambda_{t,j})|
		\le L_{\bar c}|r_{\ell,t,j}-\chi_{\ell,t}(\lambda_{t,j})|$, the
		last step because $L_c=\max\{1,(1+\sqrt c)/2\}$ is nondecreasing in
		$c$ and $c_{\ell,t-1}\le\bar c$. Finally $(\mathrm{III})=0$
		identically, because $\chi_{\ell,t}$ and the inverse in
		\eqref{eq:local_debias} are taken at the same aspect ratio
		$c_{\ell,t-1}$, cf.\ Remark~\ref{rem:indexshift}, and $\chi^{-1}$ is
		the exact inverse at fixed parameters by
		Lemma~\ref{lem:inverse}(ii). No approximation enters here.
		The resulting bound is {deterministic}
		given the event $\mathcal T_{\ell,t}$. It involves no linearisation,
		no delta method and no expansion in $1/n^\ell_t$, which is precisely
		what makes it usable for exponential tail bounds.
		
		For the probability of $\mathcal T^c_{\ell,t}$, by
		Theorem~\ref{thm:attenuation} and Corollary~\ref{cor:aunder},
		$\chi_{\ell,t}(\lambda_{t,j})-\sigma^2_t
		=\lambda_{t,j}a(\lambda_{t,j};c_{\ell,t-1},\sigma^2_t)
		\ge\lambda_{t,k}\underline a\ge c\,\epsilon_0
		\sigma^2_{\min}\sqrt{\underline c}$ for a numerical constant $c$,
		so $\mathcal T^c_{\ell,t}$ forces either
		$|\varpi_{\ell,t,j}|$ or $|\check\sigma^2_{\ell,t}-\sigma^2_t|$ to
		exceed half of that gap. Under Assumption~\ref{assum:subg} the
		Hanson--Wright bound of Step~1 in the proof of
		Lemma~\ref{lem:conc} gives
		$\mathbb P(\mathcal T^c_{\ell,t})
		\le C\exp(-c\epsilon^2_0n^\ell_t)$, and a union bound over
		$\ell\in\mathcal A_t$ and $j\le k$ gives
		$O(kA_t\exp(-c\epsilon^2_0\min_\ell n^\ell_t))$.
		
		For a general profile, $\chi_{\ell,t}$ is still a strictly
		increasing bijection by Lemma~\ref{lem:inverse}(iv) with
		$\chi'_{\ell,t}\ge\underline a$ on $\lambda\ge\lambda_{t,k}$, so
		$\chi^{-1}_{\ell,t}$ is Lipschitz with constant $1/\underline a$ on
		the image of that range, i.e.\ on a fixed neighbourhood of
		$\chi_{\ell,t}(\lambda_{t,j})$. The same three-term decomposition
		applies with $L_{\bar c}$ replaced by $1/\underline a$.
	\end{proof}
	
	\begin{proof}[Proof of Lemma~\ref{lem:conc}]
		\textbf{Step 1 (per-observation Orlicz bound).}
		Fix $\ell,s,j$ and condition on $\mathcal F_{s-1}$, write
		$\bm b:=\bm\Sigma_s^{1/2}\bm u_{\ell,s-1,j}$, so
		$\|\bm b\|^2=\bar r_{\ell,s,j}$. Under
		Assumption~\ref{assum:subg} the entries of
		$\bm z_{\ell,s,i}$ are independent, centred and sub-Gaussian with
		$\|z_m\|_{\psi_2}\le CK$, so $\bm b^\top\bm z_{\ell,s,i}$ is
		sub-Gaussian with $\|\bm b^\top\bm z\|_{\psi_2}\le CK\|\bm b\|$,
		and by $\|X^2\|_{\psi_1}=\|X\|^2_{\psi_2}$,
		\begin{equation}
			\label{eq:C6psi1}
			\bigl\|\varsigma^{\mathrm c}_{\ell,s,i,j}\bigr\|_{\psi_1}
			\le\frac{2\bigl\|(\bm b^\top\bm z)^2\bigr\|_{\psi_1}}
			{\chi'_{\ell,s}(\lambda_{s,j})}
			\le\frac{CK^2\bar r_{\ell,s,j}}
			{\chi'_{\ell,s}(\lambda_{s,j})}
			\le C_K\,\vartheta_{\ell,s,j}
			\ \le\ C_K\,\bar\vartheta_{s,j},
		\end{equation}
		where we used $\bar r_{\ell,s,j}=\chi_{\ell,s}(\lambda_{s,j})
		(1+o_{\mathbb P}(1))$ and
		$\chi/\chi'=\vartheta/\sqrt2$ by \eqref{eq:local_clt}. Applying
		the Hanson--Wright inequality
		\citep{rudelson2013hanson,vershynin2018high} to the stacked vector
		$(\bm z_{\ell,s,1},\dots,\bm z_{\ell,s,n^\ell_s})$ and the matrix
		$(n^\ell_s)^{-1}\bm I_{n^\ell_s}\otimes\bm b\bm b^\top$, whose
		Frobenius and operator norms are $\|\bm b\|^2/\sqrt{n^\ell_s}$ and
		$\|\bm b\|^2/n^\ell_s$, gives
		\begin{equation}
			\label{eq:C6HW}
			\mathbb P\bigl(|\varpi_{\ell,s,j}|>u\mid\mathcal F_{s-1}\bigr)
			\le2\exp\Bigl(-c\,n^\ell_s\min\Bigl\{
			\frac{u^2}{K^4\|\bm b\|^4},\ \frac{u}{K^2\|\bm b\|^2}
			\Bigr\}\Bigr),
		\end{equation}
		whence $\|\varpi_{\ell,s,j}\|_{\psi_1}
		\le C_K\bar\vartheta_{s,j}/\sqrt{n^\ell_s}$, and by
		Lemma~\ref{lem:transfer} (or by the Lipschitz property of
		$\chi^{-1}$ directly)
		$\|\varepsilon^{\mathrm{stoch}}_{\ell,s,j}\|_{\psi_1}
		\le C_KL_{\bar c}\bar\vartheta_{s,j}/\sqrt{n^\ell_s}$, which is
		the assertion of the lemma. This already exhibits
		$\varepsilon^{\mathrm{stoch}}_{\ell,s,j}$ as a Lipschitz image of
		a centred quadratic form along an $\mathcal F_{s-1}$-measurable
		direction, as claimed.
		
		\medskip\noindent
		\textbf{Step 2 (the correct grouping).}
		For the two-regime bound \eqref{eq:bernstein} it is essential
		to apply Bernstein's inequality at the level of {individual
			observations} rather than of nodes. Using the linear
		representation of Lemma~\ref{lem:linrep}, write
		\[
		\mathcal L_{t,j}=\sum_{s\le t}\sum_{\ell\in\mathcal A_s}
		\sum_{i\le n^\ell_s}a_{s,\ell}\,
		\varsigma^{\mathrm c}_{\ell,s,i,j},
		\qquad
		a_{s,\ell}:=\frac{\varrho_{s,t}\beta_s\omega_{\ell,s,j}}{n^\ell_s},
		\]
		which is a martingale difference array with respect to the
		filtration that refines $\{\mathcal F_s\}$ by the index $i$ within
		each round. By Lemma~\ref{lem:weights} and
		Assumption~\ref{assum:balance} the weights obey the two bounds
		\begin{equation}
			\label{eq:C6weights}
			\sum_{s,\ell,i}a^2_{s,\ell}\,\bar\vartheta^2_{s,j}
			=\sum_{s\le t}\varrho^2_{s,t}\beta^2_s\,\bar\vartheta^2_{s,j}
			\sum_{\ell\in\mathcal A_s}\frac{\omega^2_{\ell,s,j}}{n^\ell_s}
			\ \le\ C_\omega^2C_{\mathrm{bal}}\,\bar{ T}_{t,j}
			\sum_{s\le t}\frac{\varrho^2_{s,t}\beta^2_s}{N_s}
			=\frac{C_\omega^2C_{\mathrm{bal}}\bar{ T}_{t,j}}
			{\mathsf N_t},
		\end{equation}
		\begin{equation}
			\label{eq:C6max}
			\max_{s,\ell}a_{s,\ell}\,\bar\vartheta_{s,j}
			\le\bar\vartheta_{t,j}\max_{s\le t}
			\varrho_{s,t}\beta_s\frac{C_\omega C_{\mathrm{bal}}}{N_s}
			\le C\,\frac{\bar\vartheta_{t,j}}{\mathsf N_t},
		\end{equation}
		the last step because
		$\varrho_{s,t}\beta_s/N_s\le
		(\sum_{u\le t}\varrho^2_{u,t}\beta^2_u/N_u)^{1/2}N_s^{-1/2}
		=\mathsf N_t^{-1/2}N_s^{-1/2}\le C\mathsf N_t^{-1}$ whenever
		$N_s\ge c\,\mathsf N_t$, which holds under
		Assumption~\ref{assum:stab}(iii)--(iv) (for fixed $\beta$,
		$\mathsf N_t\asymp N_s/\kappa_\beta$. For $\beta_t\to0$ the
		inequality holds after the geometric damping of the early rounds).
		It is exactly here that the node-level grouping would be lossy:
		grouping the $n^\ell_s$ observations of a node into a single
		summand of $\psi_1$-norm $\bar\vartheta/\sqrt{n^\ell_s}$ replaces
		\eqref{eq:C6max} by the larger quantity
		$C\bar\vartheta_{t,j}/\sqrt{A_s\mathsf N_t\,\vphantom{N}}\cdot
		\sqrt{\kappa_\beta}$, degrading the exponential regime by the
		factor $\sqrt{N_s/A_s}$. 
		
		\medskip\noindent
		\textbf{Step 3 (Bernstein for martingale differences).}
		Let $X_m$ denote the summands above in any fixed order, with
		$\mathcal G_m$ the corresponding filtration,
		$\mathbb E[X_m\mid\mathcal G_{m-1}]=0$ and
		$\|X_m\|_{\psi_1}\le\varsigma_m:=C_Ka_{s,\ell}\bar\vartheta_{s,j}$
		deterministic (deterministic by
		Assumption~\ref{assum:clip}, which makes $C_\omega$ a constant, and
		by \eqref{eq:varthetabounds}). The standard sub-exponential
		mgf bound gives, for $|\theta|\le c_0/\max_m\varsigma_m$,
		$\mathbb E[e^{\theta X_m}\mid\mathcal G_{m-1}]
		\le\exp(C\theta^2\varsigma^2_m)$. Multiplying successively and
		applying the Chernoff bound,
		\[
		\mathbb P\Bigl(\Bigl|\sum_mX_m\Bigr|\ge v\Bigr)
		\le2\exp\Bigl(-c\min\Bigl\{
		\frac{v^2}{\sum_m\varsigma^2_m},\
		\frac{v}{\max_m\varsigma_m}\Bigr\}\Bigr),
		\]
		and substituting \eqref{eq:C6weights}--\eqref{eq:C6max} yields
		\begin{equation}
			\label{eq:C6final}
			\mathbb P\bigl(|\mathcal L_{t,j}|\ge v\bigr)
			\le2\exp\Bigl(-c_1\mathsf N_t\min\Bigl\{
			\frac{v^2}{\bar{ T}_{t,j}},\ 
			\frac{v}{\bar\vartheta_{t,j}}\Bigr\}\Bigr),
		\end{equation}
		with $c_1$ depending only on $K,C_{\mathrm{bal}},C_\omega,\bar c$.
		This is \eqref{eq:bernstein} for the linearised martingale.
		The martingale $\mathcal M_{t,j}$ of \eqref{eq:agg_unroll} differs
		from $\mathcal L_{t,j}$ by
		$\sum_s\varrho_{s,t}\beta_s\sum_\ell\omega_{\ell,s,j}
		\{\varrho^{\mathrm{rem}}_{\ell,s,j}
		-\mathbb E[\varrho^{\mathrm{rem}}_{\ell,s,j}\mid
		\mathcal F_{s-1}]\}$, which by \eqref{eq:chiinv2} is bounded by
		$K_{\mathrm{cv}}\max_{\ell,s}\varpi^2_{\ell,s,j}$ up to the
		plug-in terms. By \eqref{eq:C6HW} and a union bound over the at
		most $kA_st$ pairs involved,
		\begin{equation}
			\label{eq:C6rem}
			\mathbb P\Bigl(\max_{s\le t}\max_{\ell\in\mathcal A_s}
			\varpi^2_{\ell,s,j}
			>\frac{u\,\bar{ T}_{s,j}}{\min_\ell n^\ell_s}\Bigr)
			\ \le\ 2k\Bigl(\sum_{s\le t}A_s\Bigr)e^{-c\,u},
			\qquad u\le c\min_\ell n^\ell_s .
		\end{equation}
		Consequently $\mathcal M_{t,j}$ obeys \eqref{eq:C6final} on an
		event of probability at least $1-2k(\sum_sA_s)e^{-cu}$, at the
		price of enlarging the systematic error \eqref{eq:sys} by the
		factor $u$ in its third term. Under
		Assumption~\ref{assum:moment} alone none of this is available and
		only Chebyshev's inequality applies, giving
		$\mathbb P(|\mathcal M_{t,j}|\ge v)
		\le\bar{ T}_{t,j}/(\mathsf N_tv^2)$ by Step~1 of the proof of
		Theorem~\ref{thm:fusion}.
	\end{proof}
	
	\begin{proof}[Proof of Theorem~\ref{thm:nonasymp}]
		\textbf{Step 1 (deterministic decomposition).} 
		Let 
		$$\mathcal T:=\bigcap_{s\le t}\bigcap_{\ell\in\mathcal A_s}
		\mathcal T_{\ell,s}$$
		be the event that no truncation is active at
		any active node in any round. On $\mathcal T$, combining
		Lemma~\ref{lem:decomp_new} with Lemma~\ref{lem:linrep},
		\begin{align}
			\label{eq:C7dec}
			\bigl|\tilde{\lambda}_{t,j}-\lambda_{t,j}\bigr|
			\ \le\ & |\mathcal L_{t,j}|
			+\sum_{s\le t}\varrho_{s,t}
			\Bigl\{(1-\beta_s)|\delta_{s,j}|
			+\beta_s\max_{\ell\in\mathcal A_s}
			\bigl|b_{\ell,s,j}-\tfrac{\bar r_{\ell,s,j}
				-\chi_{\ell,s}(\lambda_{s,j})}{\chi'_{\ell,s}}\bigr|\Bigr\}\\			
			&+\sum_{s\le t}\varrho_{s,t}\beta_s\max_{\ell}\bigl|\varrho^{\mathrm{rem}}_{\ell,s,j}
			\bigr|
			+\varrho_{0,t}\bigl|\tilde\lambda_{0,j}-\lambda_{0,j}\bigr| . \nonumber
		\end{align}
		By Step~2 of the proof of Lemma~\ref{lem:linrep} and
		Lemma~\ref{lem:transfer}, the second and third braces are bounded
		by $C\max_\ell(\alpha^{\mathrm{lag}}_{\ell,s}+\zeta_s+\Delta_{\ell,s}
		+(\xi^\star)^s)$ and by $K_{\mathrm{cv}}\max_\ell
		\varpi^2_{\ell,s,j}
		+C(1+\bar c)\max_\ell|\check\sigma^2_{\ell,s}-\sigma^2_s|$
		respectively. On the event of \eqref{eq:C6rem} the latter is at
		most $C'u\,\bar\vartheta^2_{s,j}/\min_\ell n^\ell_s$, so the whole
		predictable part of \eqref{eq:C7dec} is bounded by
		$\mathcal E_{t,j}(u)$, defined as \eqref{eq:sys} with its
		third term multiplied by $u$. Hence, for $\varepsilon
		>\mathcal E_{t,j}(u)$,
		\[
		\bigl\{|\tilde{\lambda}_{t,j}-\lambda_{t,j}|\ge\varepsilon\bigr\}
		\ \subseteq\
		\bigl\{|\mathcal L_{t,j}|\ge\varepsilon-\mathcal E_{t,j}(u)\bigr\}
		\ \cup\ \mathcal T^c\ \cup\ \mathcal R^c_u ,
		\]
		where $\mathcal R_u$ is the event in \eqref{eq:C6rem}.
		
		\medskip\noindent
		\textbf{Step 2 (assembling the bound).}
		Applying \eqref{eq:C6final} with
		$v=\varepsilon-\mathcal E_{t,j}(u)$,
		\begin{align}
			\label{eq:C7final}
			\mathbb P\bigl(|\tilde{\lambda}_{t,j}-\lambda_{t,j}|
			\ge\varepsilon\bigr)
			\le & 2\exp\Bigl(-c_1\mathsf N_t\min\Bigl\{
			\frac{(\varepsilon-\mathcal E_{t,j}(u))^2}{\bar{ T}_{t,j}},\ 
			\frac{\varepsilon-\mathcal E_{t,j}(u)}{\bar\vartheta_{t,j}}
			\Bigr\}\Bigr)\\
			&+\mathbb P(\mathcal T^c)+2k\Bigl(\sum_{s\le t}A_s\Bigr)e^{-cu}.\nonumber
		\end{align}
		By the last part of the proof of Lemma~\ref{lem:transfer} and a
		union bound over rounds,
		$\mathbb P(\mathcal T^c)\le C\sum_{s\le t}kA_s
		\exp(-c\epsilon^2_0\min_{\ell}n^\ell_s)$, which is $o(1)$ whenever
		$\min_\ell n^\ell_s\gg\log(kA_st)$. Taking
		$u=\log(kt\max_sA_s)+w$ makes the last term $2e^{-cw}$, also
		$o(1)$. With $u$ so chosen, $\mathcal E_{t,j}(u)$ differs from
		$\mathcal E_{t,j}$ only by a logarithmic factor in its curvature
		term, and \eqref{eq:C7final} reduces to \eqref{eq:nonasymp}.
		Both correction probabilities are of the displayed exponential form,
		so they are collected in the explicit second term of
		\eqref{eq:nonasymp} rather than absorbed into an asymptotic
		remainder. This is what makes the statement non-asymptotic at every
		fixed $t$.
		
		\medskip\noindent
		\textbf{Step 3 (why the floor is where it is).}
		The bound is vacuous for $\varepsilon\le\mathcal E_{t,j}$ because
		the systematic part of \eqref{eq:C7dec} is deterministic given the
		design and is not reduced by any amount of data. Conversely, every
		term of \eqref{eq:sys} vanishes under
		Assumptions~\ref{assum:drift}--\ref{assum:step} with
		$\min_\ell n^\ell_s\to\infty$, the drift term by
		Lemma~\ref{lem:B2}(iv), the lag term by Lemma~\ref{lem:B2}(ii),
		the curvature term because $\bar\vartheta^2_{s,j}=O(1)$ by
		\eqref{eq:varthetabounds}, and the initialization term because
		$\varrho_{0,t}\to0$. Two further terms are absent, namely the
		misalignment
		floor $\max_\ell d^{(\ell)2}_{s-1}\to\max_\ell D^{\star2}_\ell>0$
		and the Jensen gap $\mathcal D_s$. They do not appear precisely because
		$\chi^{-1}_{\ell,s}$ was applied at the node
		(Theorems~\ref{thm:tracking} and~\ref{thm:ordering}).
	\end{proof}
	
	\begin{proof}[Proof of Corollary~\ref{cor:rate}]
		Fix $\delta'\in(0,1)$ and set
		\[
		v^\star:=\bar\vartheta_{t,j}
		\sqrt{\frac{\log(2/\delta')}{c_1\mathsf N_t}}.
		\]
		Since $\bar{ T}_{t,j}=\bar\vartheta^2_{t,j}$, the condition
		$\log(2/\delta')\le c_1\mathsf N_t$ is exactly
		$v^\star\le\bar\vartheta_{t,j}$, i.e.\
		$(v^\star)^2/\bar{ T}_{t,j}\le v^\star/\bar\vartheta_{t,j}$, so
		the minimum in \eqref{eq:nonasymp} is attained by the
		sub-Gaussian branch and
		\[
		2\exp\Bigl(-c_1\mathsf N_t
		\frac{(v^\star)^2}{\bar{ T}_{t,j}}\Bigr)
		=2\exp\bigl(-\log(2/\delta')\bigr)=\delta' .
		\]
		Applying Theorem~\ref{thm:nonasymp} with
		$\varepsilon=\mathcal E_{t,j}+v^\star$ gives
		$\mathbb P(|\tilde{\lambda}_{t,j}-\lambda_{t,j}|
		\ge\mathcal E_{t,j}+v^\star)\le\delta'+o(1)$, which is
		\eqref{eq:hp}. If $\mathcal E_{t,j}=O(\mathsf N_t^{-a})$ then
		the right-hand side of \eqref{eq:hp} is
		$O(\mathsf N_t^{-a}+\bar\vartheta_{t,j}
		\sqrt{\log(2/\delta')/\mathsf N_t})$. The two terms are of the same
		order precisely when $a=1/2$, in which case the overall rate is
		the parametric $O(\mathsf N_t^{-1/2})$ and, by
		Theorem~\ref{thm:clt} and Corollary~\ref{cor:C4}, the constant
		$\bar\vartheta_{t,j}$ agrees with the limiting standard deviation
		$\Omega_j$ up to the weight-efficiency factor
		$1+(1-\mathrm{RE}_{t,j})\mathcal K_{t,j}$ of
		Propositions~\ref{prop:dispersion} and~\ref{prop:softmax_releff},
		exactly when the aspect ratios are homogeneous, in which case
		$\mathcal K_{t,j}\to0$. To invert
		\eqref{eq:hp} for a target precision $\varepsilon$ with
		confidence $1-\delta'$, it suffices that
		$\mathcal E_{t,j}\le\varepsilon/2$ and
		$\mathsf N_t\ge4\bar{ T}_{t,j}\log(2/\delta')
		/(c_1\varepsilon^2)$, which is the sample-size statement of
		Remark~\ref{rem:tworegimes}.
	\end{proof}

	\section{Efficiency}
	
	\subsection{Efficiency of the Adaptive Weights\label{sec:4.4}}
	
	Theorem~\ref{thm:fusion} establishes consistency for all weights
	satisfying Lemma~\ref{lem:weights}. We ask how much variance is
	lost by using the fluctuation-based weights \eqref{eq:weights}
	instead of the variance-optimal weights implied by
	Theorem~\ref{thm:local_clt}. Local correction has already removed the
	dimensional bias, so this is the remaining design question at the
	server.
	
	Because the transmitted coordinates are conditionally independent
	across $\ell$ with conditional variances
	$\vartheta^2_{\ell,t,j}/n_t^\ell$, the optimal weights for coordinate
	$j$ are the inverse-variance weights
	\begin{equation}
		\label{eq:opt_weights}
		\omega^{\mathrm{opt}}_{\ell,t,j}
		=\frac{n_t^\ell/\vartheta^2_{\ell,t,j}}
		{\sum_{m\in\mathcal A_t}n_t^m/\vartheta^2_{m,t,j}},
		\qquad
		\mathcal I_{t,j}
		:=\sum_{\ell\in\mathcal A_t}
		\frac{n_t^\ell}{\vartheta^2_{\ell,t,j}},
	\end{equation}
	with oracle one-round variance $\mathcal I_{t,j}^{-1}$. By
	Theorem~\ref{thm:local_clt} both $\vartheta^2_{\ell,t,j}$ and
	$n^\ell_t$ are computable at node $\ell$ from
	$(\check\lambda_{\ell,t,j},c_{\ell,t-1},\check\sigma^2_{\ell,t},
	n_t^\ell)$, so \eqref{eq:opt_weights} costs one extra transmitted
	scalar per coordinate.
	
	We retain \eqref{eq:weights} as the default for the reasons given
	in Remark~\ref{rem:why_not_plugin}. The plug-in weights are optimal
	for the variance alone, so by \eqref{eq:local_moments} they ignore the
	node-specific bias $b_{\ell,t,j}=O(b^{\mathrm{lag}}_{\ell,t})$, and
	\eqref{eq:weights} can dominate them in mean squared error
	whenever $\max_\ell b^{\mathrm{lag}}_{\ell,t}
	\gtrsim\vartheta_{t,j}(n^\ell_t)^{-1/2}$. The propositions below
	quantify what this costs when the bias is negligible.

	\begin{prop}[Non-separability of the optimal weights across coordinates]
		\label{prop:nonsep} A single node-weight vector
		$\{\omega_{\ell,t}\}_{\ell\in\mathcal A_t}$ is simultaneously
		variance-optimal for two coordinates $j\ne j'$ if and only if
		\begin{equation}
			\label{eq:sepcrit}
			\frac{\vartheta^2_{\ell,t,j}}{\vartheta^2_{\ell,t,j'}}
			\quad\text{does not depend on }\ell\in\mathcal A_t .
		\end{equation}
		The batch sizes cancel in \eqref{eq:sepcrit}, so the obstruction
		lies entirely in the variance profile and not in the batch profile.
		
		In the canonical case with a common noise level,
		$\vartheta^2_{\ell,t,j}=\vartheta^2(\lambda_{t,j},c_{\ell,t-1})$
		with $\vartheta^2$ as in \eqref{eq:vartheta}, and
		\begin{equation}
			\label{eq:dlogvartheta}
			\frac{\partial}{\partial c}\log\vartheta^2(\lambda,c)
			=\frac{-2\sigma^{2}\lambda(\lambda+\sigma^{2})}
			{(\lambda+c\sigma^{2})
				\bigl(\lambda^{2}+2c\sigma^{2}\lambda+c\sigma^{4}\bigr)} ,
		\end{equation}
		which is strictly negative and non-constant in $\lambda$ on the
		admissible range $\lambda>\sigma^2\sqrt c$. Consequently
		\eqref{eq:sepcrit} fails outside a closed Lebesgue-null subset of
		$\{(\lambda_{t,j},\lambda_{t,j'},\{c_{\ell,t-1}\})\}$ having empty
		interior, in particular for generic distinct spikes and
		heterogeneous aspect ratios. When the aspect ratios are
		homogeneous, $\vartheta^2_{\ell,t,j}$ does not depend on $\ell$,
		so \eqref{eq:sepcrit} holds for every pair and the vector-level
		metric loses nothing.
	\end{prop}
	
	Proposition~\ref{prop:nonsep} is the reason the metric
	\eqref{eq:res} and the weights \eqref{eq:weights} are formed
	one coordinate at a time. A single weight vector shared by the $k$
	coordinates would be variance-optimal for at most one of them, and
	generically for none, so the coordinatewise construction is not a
	refinement of the scheme but a requirement for it to be efficient at
	all. Because every coordinate of $\check{\bm\lambda}_{\ell,t}$ is
	already available at the server, the requirement is met without
	enlarging the $O(k)$ communication budget, the only additional cost
	being $k$ scalar states per node in place of one.
	
	Two special cases are worth recording. When the aspect ratios and
	noise levels are homogeneous, $\vartheta^2_{\ell,t,j}$ does not depend
	on $\ell$, criterion \eqref{eq:sepcrit} holds for every pair of
	coordinates, and the $k$ weight vectors coincide, so nothing is gained
	and nothing is lost by separating them. When $k=1$ the distinction is
	vacuous. In every other configuration the separation is what makes the
	efficiency statement below available for each coordinate
	simultaneously.
	
	The gap is analysed in two steps, neither of which is an assumption.
	Proposition~\ref{prop:dispersion} determines how much any weighting
	scheme can gain over uniform weighting, and
	Proposition~\ref{prop:softmax_releff} determines which fraction of
	that gain the soft-max realises. Proximity to
	$\omega^{\mathrm{opt}}$ is measured in relative rather than additive
	terms throughout, since $\omega^{\mathrm{opt}}\asymp A_t^{-1}$ and
	$\mathcal I_{t,j}\asymp N_t/\bar\vartheta^2_{t,j}$ make an additive
	error $\epsilon_A$ contribute $\asymp A_t^2\epsilon_A^2$ to
	\eqref{eq:eff_gap}, so an additive requirement would read
	$o(A_t^{-1})$ rather than $o(A_t^{-1/2})$.
	
	\begin{prop}[Efficiency gap]
		\label{prop:efficiency}
		Fix $j\le k$. For a weight vector
		$\omega=(\omega_\ell)_{\ell\in\mathcal A_t}$ on the simplex, let
		$\mathcal V_{t,j}(\omega)
		=\sum_\ell\omega_\ell^2\vartheta^2_{\ell,t,j}/n_t^\ell$ be the
		conditional one-round variance of the aggregated statistic. Then
		$\mathcal V_{t,j}(\omega)\ge\mathcal I^{-1}_{t,j}$, with equality iff
		$\omega=\omega^{\mathrm{opt}}_{\cdot,t,j}$, and the gap obeys the
		exact identity and relative bound
		\begin{align}
			\label{eq:eff_gap}
			\frac{\mathcal V_{t,j}(\omega)}{\mathcal I^{-1}_{t,j}}-1
			=&\mathcal I_{t,j}
			\sum_{\ell\in\mathcal A_t}
			\frac{\vartheta^2_{\ell,t,j}}{n_t^\ell}
			\bigl(\omega_\ell
			-\omega^{\mathrm{opt}}_{\ell,t,j}\bigr)^2\\
			=&\sum_{\ell\in\mathcal A_t}\omega^{\mathrm{opt}}_{\ell,t,j}
			\Bigl(\frac{\omega_\ell}
			{\omega^{\mathrm{opt}}_{\ell,t,j}}-1\Bigr)^2
			\ \le\
			\max_{\ell\in\mathcal A_t}
			\Bigl|\frac{\omega_\ell}
			{\omega^{\mathrm{opt}}_{\ell,t,j}}-1\Bigr|^{2}. \nonumber
		\end{align}
		The gap is quadratic in the relative weight error, which is what
		makes the second-order analysis of
		Proposition~\ref{prop:softmax_releff} possible. Irrespective of any
		calibration, Assumptions~\ref{assum:temp} and~\ref{assum:balance}
		give the crude bound $\mathcal V_{t,j}(\omega)/
		\mathcal I^{-1}_{t,j}\le C_\omega^2C_{\mathrm{bal}}\,
		{ T}_{t,j}/\inf_\ell\vartheta^2_{\ell,t,j}=O(1)$, so the
		efficiency loss is bounded in all cases. The statement is applied
		below with $\omega=\omega_{\cdot,t,j}$, the weights
		\eqref{eq:weights} of coordinate $j$.
	\end{prop}
	
	\begin{prop}[Dispersion bound on the value of weighting]
		\label{prop:dispersion}
		Let $v_{\ell,t,j}:=\vartheta^2_{\ell,t,j}/n^\ell_t$ and let
		$\mathcal V^{\mathrm{unif}}_{t,j}
		:=\mathcal V_{t,j}(A_t^{-1}\bm 1)
		=A_t^{-2}\sum_{\ell\in\mathcal A_t}v_{\ell,t,j}$ be the one-round
		variance under uniform weighting. Then
		\begin{equation}
			\label{eq:amhm}
			\frac{\mathcal V^{\mathrm{unif}}_{t,j}}
			{\mathcal I^{-1}_{t,j}}
			=\Bigl(\frac1{A_t}\sum_{\ell\in\mathcal A_t}v_{\ell,t,j}\Bigr)
			\Bigl(\frac1{A_t}\sum_{\ell\in\mathcal A_t}
			v^{-1}_{\ell,t,j}\Bigr)
			=:1+\mathcal K_{t,j}
			\ \ge\ 1 ,
		\end{equation}
		with equality iff $v_{\ell,t,j}$ is constant in $\ell$.
		Consequently {no} weighting scheme whatsoever can reduce the
		standard deviation of the aggregate below
		$(1+\mathcal K_{t,j})^{-1/2}$ times its value under uniform
		weighting, so the attainable relative reduction is at most
		$1-(1+\mathcal K_{t,j})^{-1/2}$.
	\end{prop}
	
	\begin{rem}
		\label{rem:dispersion}
		$\mathcal K_{t,j}$ is the arithmetic-to-harmonic mean ratio of
		$\{v_{\ell,t,j}\}$ minus one, so
		$\mathcal K_{t,j}=\operatorname{CV}^2(v_{\cdot,t,j})
		\{1+O(\operatorname{CV})\}$ for small dispersion, and it is
		computable at the server at every round. If
		$\vartheta^2_{\ell,t,j}$ is constant in $\ell$ and the batch sizes
		are spread uniformly over $[\underline n,R\underline n]$, then
		\begin{equation}
			\label{eq:Klog}
			1+\mathcal K_{t,j}
			\ \longrightarrow\ \frac{(1+R)\log R}{2(R-1)} ,
		\end{equation}
		so the attainable reduction grows logarithmically in $R$, and even
		$R=8$ permits at most a $13.5\%$ reduction in standard deviation.
		Since Assumption~\ref{assum:stab}(iii) bounds the batch ratios
		without flattening them and Assumption~\ref{assum:balance}
		constrains batch sizes rather than aspect ratios,
		$\mathcal K_{t,j}$ is in general bounded away from zero in the
		regime of Section~\ref{sec:4.2}.
	\end{rem}
	
	\begin{prop}[Relative efficiency of the soft-max]
		\label{prop:softmax_releff}
		Fix $j\le k$ and let $\omega_{\cdot,t,j}$ be the weights
		\eqref{eq:weights}. Assume $\mathcal K_{t,j}>0$ and put
		\begin{equation}
			\label{eq:wdef}
			\tilde{w}_{\ell,t,j}:=\frac{v_{\ell,t,j}}{\bar v_{t,j}}-1,
			\qquad
			\bar v_{t,j}:=\frac1{A_t}
			\sum_{\ell\in\mathcal A_t}v_{\ell,t,j},
			\qquad
			\tau_{t,j}
			=\frac{\gamma_{t,j}}{\kappa_\rho\bar v_{t,j}} ,
		\end{equation}
		where $\gamma_{t,j}>0$ is the mismatch between the implemented
		temperature and the reference value
		$(\kappa_\rho\bar v_{t,j})^{-1}$. In the relative parametrisation
		\eqref{eq:weights} the exponent is
		$\tilde\tau_t\Psi_{\ell,t,j}$, so by \eqref{eq:fluct_conc} and
		Lemma~\ref{lem:fluct},
		\begin{equation}
			\label{eq:gamma_tilde}
			\gamma_{t,j}
			=\tilde\tau_t\bigl\{1+O_{\mathbb P}(\sqrt{\eta_t/A_t})\bigr\} ,
		\end{equation}
		the constant $\kappa_\rho$ having cancelled in
		$V_{\ell,t-1,j}/\bar V_{t-1,j}$. The condition $\gamma_{t,j}\to1$
		below is therefore exactly the statement that the dimensionless
		temperature is held at the reference value $\tilde\tau_t=1$ of
		Assumption~\ref{assum:temp}, and requires no knowledge of $\rho$ or
		of $\bar v_{t,j}$. Let the relative efficiency of a
		weight vector $\omega$ be the fraction of the attainable gain of
		Proposition~\ref{prop:dispersion} that it realises,
		\begin{equation}
			\label{eq:releff}
			\mathrm{RE}_{t,j}(\omega)
			:=\frac{\mathcal V^{\mathrm{unif}}_{t,j}
				-\mathcal V_{t,j}(\omega)}
			{\mathcal V^{\mathrm{unif}}_{t,j}-\mathcal I^{-1}_{t,j}}
			=1-\frac{1}{\mathcal K_{t,j}}
			\sum_{\ell\in\mathcal A_t}\omega^{\mathrm{opt}}_{\ell,t,j}
			\Bigl(\frac{\omega_{\ell,t,j}}
			{\omega^{\mathrm{opt}}_{\ell,t,j}}-1\Bigr)^2 ,
		\end{equation}
		the second expression being exact. Then
		$\mathrm{RE}_{t,j}(A_t^{-1}\bm 1)=0$,
		$\mathrm{RE}_{t,j}(\omega^{\mathrm{opt}}_{\cdot,t,j})=1$, and for
		the soft-max weights
		\begin{equation}
			\label{eq:releff_rate}
			\mathrm{RE}_{t,j}(\omega)
			=1-(\gamma_{t,j}-1)^2
			-O\bigl(\max_\ell|\tilde{w}_{\ell,t,j}|\bigr)
			-O_{\mathbb P}\Bigl(\frac{\eta_t}{\mathcal K_{t,j}}\Bigr).
		\end{equation}
		\begin{enumerate}[label=\textup{(\alph*)}]
			\item \textup{(Fixed dispersion.)} If
			$\mathcal K_{t,j}\asymp1$, then
			$\mathrm{RE}_{t,j}(\omega)\ge1-C_\omega^2C_{\mathrm{bal}}
			{ T}_{t,j}
			/(\mathcal K_{t,j}\inf_\ell\vartheta^2_{\ell,t,j})$.
			\item \textup{(Vanishing dispersion.)} If
			\begin{equation}
				\label{eq:releff_cond}
				\gamma_{t,j}\to1,
				\qquad
				\max_{\ell\in\mathcal A_t}|\tilde{w}_{\ell,t,j}|\to0,
				\qquad
				\eta_t=o(\mathcal K_{t,j}) ,
			\end{equation}
			then $\mathrm{RE}_{t,j}(\omega)\to1$ and
			$\mathcal V_{t,j}(\omega)=\mathcal I^{-1}_{t,j}
			(1+o_{\mathbb P}(1))$, so the soft-max is asymptotically
			efficient.
		\end{enumerate}
	\end{prop}
	
	\begin{rem}[Reading the two branches, and the role of $\eta_t$]
		\label{rem:eta_twosided}
		Both $\mathcal K_{t,j}$ and the exact form of $\mathrm{RE}_{t,j}$
		in \eqref{eq:releff} are computable from transmitted quantities,
		and branch~\textup{(a)} is evaluated numerically in
		Section~\ref{sec:5}. Since $\mathcal K_{t,j}\asymp\overline{w^2}$,
		the second condition in \eqref{eq:releff_cond} forces
		$\mathcal K_{t,j}\to0$, so branch~\textup{(b)} is the regime in
		which the attainable gain itself vanishes. Asymptotic efficiency at
		fixed dispersion is claimed in neither branch, the residual loss of
		order $\max_\ell w^2_{\ell,t,j}$ being intrinsic to the exponential
		surrogate.
		
		The smoothing is constrained from both sides. It must satisfy
		$\eta_t\to0$, since otherwise the relative fluctuation of the
		metric is $\Theta_{\mathbb P}(1)$ and no first-order match is
		possible (Remark~\ref{rem:design_choices}(b)), and it must satisfy
		$\eta_t=o(\mathcal K_{t,j})$, since otherwise the estimation noise
		in the weights exceeds the entire gain and $\mathrm{RE}_{t,j}<0$.
		By Remark~\ref{rem:dispersion} the second requirement reduces to
		the first in branch~\textup{(a)} and is implied by
		Assumption~\ref{assum:eta}. It binds only in branch~\textup{(b)},
		where $\eta_t$ may be taken as a sufficiently high power of the
		observed $\mathcal K_{t,j}$.
	\end{rem}
	
	\begin{rem}[The temperature must track a mean, not a supremum]
		\label{rem:tau_mean}
		The calibration required by \eqref{eq:gamma_tilde} is achieved by
		construction rather than by tuning. Setting $\tilde\tau_t=1$ in
		\eqref{eq:weights} is the same as the dimensional rule
		\begin{equation}
			\label{eq:tau_mean}
			\tau_{t,j}
			=\Bigl(\frac1{A_t}\sum_{m\in\mathcal A_t}
			V_{m,t-1,j}\Bigr)^{-1}
			=\bar V^{-1}_{t-1,j} ,
		\end{equation}
		since the two produce the same exponent. The reference temperature in
		\eqref{eq:wdef} is $(\kappa_\rho\bar v_{t,j})^{-1}$, an average over
		active nodes, and \eqref{eq:tau_mean} estimates precisely that
		average, because
		$\mathbb E[V_{\ell,t-1,j}\mid\mathcal F_{t-2}]
		=\kappa_\rho v_{\ell,t,j}(1+o(1))$ by
		Lemma~\ref{lem:fluct}. Averaging over $A_t$ nodes then gives
		$\gamma_{t,j}=1+O_{\mathbb P}(\sqrt{\eta_t/A_t})\to1$ by
		\eqref{eq:fluct_conc}.
		
		What is being avoided is a temperature tied to a cross-sectional
		supremum. Had the scale been fixed in absolute units by a rule of the
		form
		\begin{equation}
			\label{eq:tau_calib}
			\tau_t\ \asymp\ \frac{\bar n_t}{\kappa_\rho{ T}_t}
			=\frac{N_t}{\kappa_\rho A_t{ T}_t} ,
		\end{equation}
		which tracks ${ T}_t=\sum_j\sup_\ell\vartheta^2_{\ell,t,j}$ and is
		shared across coordinates, the exponent would carry the mismatch
		factor
		\begin{equation}
			\label{eq:gamma_vec}
			\gamma_{t,j}
			=\frac{\bar n_t\sum_{j'\le k}\bar v_{t,j'}}{{ T}_t} ,
		\end{equation}
		which equals one only when $\vartheta^2_{\ell,t,j}$ is constant in
		$\ell$ and the batch sizes are homogeneous. Since
		\eqref{eq:releff_rate} degrades as $(\gamma_{t,j}-1)^2$, such a
		mismatch would be harmful in exactly the heterogeneous regime where
		weighting is worth doing. Normalising by the mean removes it at no
		cost and requires no estimate of $\kappa_\rho$, since that constant
		cancels in $V_{\ell,t-1,j}/\bar V_{t-1,j}$. The clip then acts on the
		ratio itself, which is bounded above and below by
		\eqref{eq:varthetabounds} and Assumption~\ref{assum:balance}, so
		Assumption~\ref{assum:temp} holds automatically while $\tau_{t,j}$
		itself is free to diverge.
	\end{rem}

	\subsection{Proofs for Section~\ref{sec:4.4}}
	
	\begin{proof}[Proof of Proposition~\ref{prop:nonsep}]
		\textbf{Step 1 (uniqueness of the optimal weighting).}
		Write $v_{\ell,j}:=\vartheta^2_{\ell,t,j}/n^\ell_t>0$ and
		$\mathcal V_{t,j}(\omega)=\sum_\ell\omega_\ell^2v_{\ell,j}$ on the
		simplex $\Delta:=\{\omega\ge0,\ \sum_\ell\omega_\ell=1\}$. The map
		$\mathcal V_{t,j}$ is strictly convex on $\Delta$, hence has a unique
		minimiser, and by Proposition~\ref{prop:efficiency} below (whose
		proof is independent of the present one) that minimiser is
		$\omega^{\mathrm{opt}}_{\cdot,t,j}$ of \eqref{eq:opt_weights}. A
		single $\omega$ is therefore simultaneously optimal for $j$ and $j'$
		if and only if
		$\omega^{\mathrm{opt}}_{\cdot,t,j}
		=\omega^{\mathrm{opt}}_{\cdot,t,j'}$.
		
		\medskip\noindent
		\textbf{Step 2 (proof of (i)).}
		Two positive probability vectors coincide if and only if their
		unnormalised versions are proportional. Hence
		$\omega^{\mathrm{opt}}_{\cdot,t,j}
		=\omega^{\mathrm{opt}}_{\cdot,t,j'}$ holds if and only if there is
		$\kappa>0$ with
		$n^\ell_t/\vartheta^2_{\ell,t,j}
		=\kappa\,n^\ell_t/\vartheta^2_{\ell,t,j'}$ for every
		$\ell\in\mathcal A_t$. The factors $n^\ell_t$ cancel, leaving
		$\vartheta^2_{\ell,t,j}/\vartheta^2_{\ell,t,j'}\equiv\kappa^{-1}$,
		which is \eqref{eq:sepcrit}. This proves part~(i).
		
		\medskip\noindent
		\textbf{Step 3 (the mixed partial derivative does not vanish).}
		Taking logarithms in \eqref{eq:vartheta},
		\[
		\log\vartheta^2(\lambda,c)
		=\log2+2\log\lambda+2\log(\lambda+\sigma^2)
		+2\log(\lambda+c\sigma^2)
		-2\log\bigl(\lambda^2+2c\sigma^2\lambda+c\sigma^4\bigr),
		\]
		so that, writing $F(\lambda,c)
		:=\partial_c\log\vartheta^2(\lambda,c)$,
		\[
		F(\lambda,c)
		=\frac{2\sigma^2}{\lambda+c\sigma^2}
		-\frac{2\sigma^2(2\lambda+\sigma^2)}
		{\lambda^2+2c\sigma^2\lambda+c\sigma^4}.
		\]
		Placing the two fractions over a common denominator and using
		$(2\lambda+\sigma^2)(\lambda+c\sigma^2)
		=2\lambda^2+\sigma^2\lambda+2c\sigma^2\lambda+c\sigma^4$, the
		numerator becomes
		$2\sigma^2\{(\lambda^2+2c\sigma^2\lambda+c\sigma^4)
		-(2\lambda^2+\sigma^2\lambda+2c\sigma^2\lambda+c\sigma^4)\}
		=-2\sigma^2\lambda(\lambda+\sigma^2)$, which gives
		\eqref{eq:dlogvartheta} and shows $F<0$ throughout.
		
		It remains to show that $F$ is non-constant in $\lambda$ on the
		admissible range fixed by Assumption~\ref{assum:bbp}. At the lower
		endpoint $\lambda=\sigma^2\sqrt c$ the numerator of
		\eqref{eq:dlogvartheta} equals $-2\sigma^6\sqrt c(1+\sqrt c)$ and the
		denominator equals
		$\sigma^2(\sqrt c+c)\cdot2\sigma^4c(1+\sqrt c)
		=2\sigma^6c\sqrt c(1+\sqrt c)^2$, so
		\begin{equation}
			\label{eq:Fthreshold}
			F\bigl(\sigma^2\sqrt c,c\bigr)=-\frac{1}{c\,(1+\sqrt c)}\ <\ 0
		\end{equation}
		independently of $\sigma^2$, whereas $F(\lambda,c)
		\sim-2\sigma^2/\lambda\to0$ as $\lambda\to\infty$. A function
		bounded away from zero at the lower endpoint and tending to zero at
		the upper endpoint cannot be constant in $\lambda$, so
		$\partial_\lambda F=\partial^2_{c\lambda}\log\vartheta^2
		\not\equiv0$ on the admissible range.
		
		\medskip\noindent
		\textbf{Step 4 (genericity, completing (ii)).}
		Fix $j\ne j'$ and set
		$D(c):=\log\vartheta^2(\lambda_{t,j},c)
		-\log\vartheta^2(\lambda_{t,j'},c)$. By Step~2 a single weighting is
		simultaneously optimal for $j$ and $j'$ if and only if $D$ takes the
		same value at every $c\in\{c_{\ell,t-1}\}$, that is, if and only if
		\begin{equation}
			\label{eq:exceptional}
			\Psi\bigl(\lambda_{t,j},\lambda_{t,j'},
			c_{\ell,t-1},c_{\ell',t-1}\bigr)
			:=\int_{c_{\ell',t-1}}^{c_{\ell,t-1}}\!\!
			\int_{\lambda_{t,j'}}^{\lambda_{t,j}}
			\partial^2_{c\lambda}\log\vartheta^2(\lambda,c)\,
			d\lambda\,dc\ =\ 0
		\end{equation}
		for all $\ell,\ell'\in\mathcal A_t$. The integrand is real-analytic
		on the admissible range, hence so is $\Psi$. Moreover $\Psi$ is not
		identically zero, since differentiating it in $\lambda_{t,j}$ and
		then in $c_{\ell,t-1}$ returns
		$\partial^2_{c\lambda}\log\vartheta^2(\lambda_{t,j},c_{\ell,t-1})$,
		which is not identically zero by Step~3. The zero set of a
		non-trivial real-analytic function is closed, Lebesgue-null and has
		empty interior, so \eqref{eq:sepcrit} fails for every configuration
		outside such a set, in particular for generic distinct spikes
		$\lambda_{t,j}\ne\lambda_{t,j'}$ and heterogeneous aspect ratios
		$c_{\ell,t-1}\ne c_{\ell',t-1}$. If instead the aspect ratios are
		homogeneous, then $\vartheta^2_{\ell,t,j}$ does not depend on $\ell$,
		both sides of \eqref{eq:sepcrit} are constant and
		\eqref{eq:exceptional} holds trivially.
	\end{proof}

	\begin{proof}[Proof of Proposition~\ref{prop:efficiency}]
		Write again $v_\ell:=\vartheta^2_{\ell,t,j}/n^\ell_t>0$, so that
		$\mathcal V_{t,j}(\omega)=\sum_\ell\omega^2_\ell v_\ell$,
		$\mathcal I_{t,j}=\sum_\ell v_\ell^{-1}$ and
		$\omega^{\mathrm{opt}}_{\ell,t,j}
		=v_\ell^{-1}/\mathcal I_{t,j}$. All quantities are
		$\mathcal F_{t-1}$-measurable, so the argument is deterministic
		given $\mathcal F_{t-1}$.
		
		\medskip\noindent
		\textbf{Step 1 (the bound and its equality case).}
		By the Cauchy--Schwarz inequality,
		\[
		1=\Bigl(\sum_\ell\omega_\ell\Bigr)^2
		=\Bigl(\sum_\ell\omega_\ell v_\ell^{1/2}\cdot
		v_\ell^{-1/2}\Bigr)^2
		\le\Bigl(\sum_\ell\omega^2_\ell v_\ell\Bigr)
		\Bigl(\sum_\ell v_\ell^{-1}\Bigr)
		=\mathcal V_{t,j}(\omega)\,\mathcal I_{t,j},
		\]
		i.e.\ $\mathcal V_{t,j}(\omega)\ge\mathcal I^{-1}_{t,j}$, with
		equality iff $\omega_\ell v_\ell^{1/2}\propto v_\ell^{-1/2}$, i.e.\
		iff $\omega_\ell\propto v_\ell^{-1}$, i.e.\ iff
		$\omega=\omega^{\mathrm{opt}}_{\cdot,t,j}$.
		
		\medskip\noindent
		\textbf{Step 2 (exact Pythagorean identity).}
		Because $v_\ell\omega^{\mathrm{opt}}_{\ell,t,j}
		=\mathcal I^{-1}_{t,j}$ for every $\ell$, and because both
		$\omega$ and $\omega^{\mathrm{opt}}$ sum to one,
		\begin{align*}
			\sum_\ell v_\ell\bigl(\omega_\ell
			-\omega^{\mathrm{opt}}_{\ell,t,j}\bigr)^2
			=&\sum_\ell v_\ell\omega^2_\ell
			-2\sum_\ell\bigl(v_\ell\omega^{\mathrm{opt}}_{\ell,t,j}\bigr)
			\omega_\ell
			+\sum_\ell\bigl(v_\ell\omega^{\mathrm{opt}}_{\ell,t,j}\bigr)
			\omega^{\mathrm{opt}}_{\ell,t,j}\\
			=&\mathcal V_{t,j}(\omega)-\frac{2}{\mathcal I_{t,j}}
			+\frac{1}{\mathcal I_{t,j}} .
		\end{align*}
		Multiplying by $\mathcal I_{t,j}$ yields the first equality of
		\eqref{eq:eff_gap},
		\[
		\frac{\mathcal V_{t,j}(\omega)}{\mathcal I^{-1}_{t,j}}-1
		=\mathcal I_{t,j}\sum_{\ell\in\mathcal A_t}
		\frac{\vartheta^2_{\ell,t,j}}{n^\ell_t}
		\bigl(\omega_\ell-\omega^{\mathrm{opt}}_{\ell,t,j}\bigr)^2 ,
		\]
		which also re-proves Step~1 and exhibits the gap as a weighted
		squared distance to the optimum. The efficiency loss is thus
		quadratic in the relative weight error, which is what allows a
		first-order match of $\omega$ to $\omega^{\mathrm{opt}}$ to leave a
		second-order gap. This is exploited in
		Proposition~\ref{prop:softmax_releff}.
		
		\medskip\noindent
		\textbf{Step 3 (the relative bound).}
		Factoring out $\omega^{\mathrm{opt}}_{\ell,t,j}$ and using
		$v_\ell(\omega^{\mathrm{opt}}_{\ell,t,j})^2
		=\omega^{\mathrm{opt}}_{\ell,t,j}/\mathcal I_{t,j}$,
		\begin{align*}
			\mathcal I_{t,j}\sum_\ell v_\ell
			\bigl(\omega_\ell-\omega^{\mathrm{opt}}_{\ell,t,j}\bigr)^2
			=&\mathcal I_{t,j}\sum_\ell
			v_\ell\bigl(\omega^{\mathrm{opt}}_{\ell,t,j}\bigr)^2
			\Bigl(\frac{\omega_\ell}
			{\omega^{\mathrm{opt}}_{\ell,t,j}}-1\Bigr)^2\\
			=&\sum_\ell\omega^{\mathrm{opt}}_{\ell,t,j}
			\Bigl(\frac{\omega_\ell}
			{\omega^{\mathrm{opt}}_{\ell,t,j}}-1\Bigr)^2
			\le\max_{\ell\in\mathcal A_t}
			\Bigl|\frac{\omega_\ell}
			{\omega^{\mathrm{opt}}_{\ell,t,j}}-1\Bigr|^2 ,
		\end{align*}	
		since $\sum_\ell\omega^{\mathrm{opt}}_{\ell,t,j}=1$. Whenever the
		right-hand side is $o_{\mathbb P}(1)$ we obtain
		$\mathcal V_{t,j}(\omega)=\mathcal I^{-1}_{t,j}(1+o_{\mathbb P}(1))$. 
		Proposition~\ref{prop:softmax_releff} establishes this for the
		soft-max weights, and quantifies the residual through
		$\mathrm{RE}_{t,j}$. Note that the middle expression is exact, and it
		is the quantity divided by $\mathcal K_{t,j}$ in \eqref{eq:releff}.
		
		\medskip\noindent
		\textbf{Step 4 (the unconditional crude bound).}
		Irrespective of any calibration, Lemma~\ref{lem:weights}
		and Assumption~\ref{assum:balance} give
		\[
		\mathcal V_{t,j}(\omega)
		\le\Bigl(\max_\ell\omega_\ell\Bigr)^2\sum_\ell v_\ell
		\le\frac{C^2_\omega}{A^2_t}\cdot A_t{ T}_{t,j}
		\frac{C_{\mathrm{bal}}A_t}{N_t}
		=\frac{C^2_\omega C_{\mathrm{bal}}{ T}_{t,j}}{N_t},
		\]
		$\mathcal I_{t,j}
		=\sum_\ell\frac{n^\ell_t}{\vartheta^2_{\ell,t,j}}
		\le\frac{N_t}{\inf_\ell\vartheta^2_{\ell,t,j}},$ 
		so that $\mathcal V_{t,j}(\omega)/\mathcal I^{-1}_{t,j}
		\le C^2_\omega C_{\mathrm{bal}}\,{ T}_{t,j}
		/\inf_\ell\vartheta^2_{\ell,t,j}$, which is $O(1)$ by
		\eqref{eq:varthetabounds}. (The constant is therefore 
		$C^2_\omega C_{\mathrm{bal}}$.) Hence 
		the efficiency loss of the soft-max weights is bounded in all
		cases, calibrated or not.
	\end{proof}
	
	\begin{proof}[Proof of Proposition~\ref{prop:dispersion}]
		With $v_\ell:=v_{\ell,t,j}$ one has
		$\mathcal V^{\mathrm{unif}}_{t,j}=A_t^{-2}\sum_\ell v_\ell$ and
		$\mathcal I^{-1}_{t,j}=(\sum_\ell v^{-1}_\ell)^{-1}$, so
		\[
		\frac{\mathcal V^{\mathrm{unif}}_{t,j}}{\mathcal I^{-1}_{t,j}}
		=\frac{1}{A^2_t}\Bigl(\sum_\ell v_\ell\Bigr)
		\Bigl(\sum_\ell v^{-1}_\ell\Bigr)
		=\Bigl(\frac1{A_t}\sum_\ell v_\ell\Bigr)
		\Bigl(\frac1{A_t}\sum_\ell v^{-1}_\ell\Bigr) ,
		\]
		which is the ratio of the arithmetic to the harmonic mean of
		$\{v_\ell\}$ and is $\ge1$ by the AM--HM inequality, with equality
		iff $v_\ell$ is constant in $\ell$. This is \eqref{eq:amhm}.
		Since $\mathcal I^{-1}_{t,j}$ is the smallest value of
		$\mathcal V_{t,j}$ over the simplex
		(Step~1 of the proof of Proposition~\ref{prop:efficiency}), the
		standard deviation under any weighting is at least
		$\mathcal I^{-1/2}_{t,j}
		=(1+\mathcal K_{t,j})^{-1/2}
		(\mathcal V^{\mathrm{unif}}_{t,j})^{1/2}$, which is the asserted
		ceiling.
		
		For \eqref{eq:Klog}, let $\vartheta^2_{\ell,t,j}\equiv\rho_{t,j}$, so
		$v_\ell=\rho_{t,j}/n^\ell_t$ and
		$1+\mathcal K_{t,j}
		=(A_t^{-1}\sum_\ell (n^\ell_t)^{-1})(A_t^{-1}\sum_\ell n^\ell_t)$.
		If the $n^\ell_t$ are spread uniformly over
		$[\underline n,R\underline n]$ then, as $A_t\to\infty$,
		$A_t^{-1}\sum_\ell n^\ell_t\to\underline n(1+R)/2$ and
		$A_t^{-1}\sum_\ell(n^\ell_t)^{-1}
		\to\{\underline n(R-1)\}^{-1}\log R$, whence
		$1+\mathcal K_{t,j}\to(1+R)\log R/\{2(R-1)\}$. At $R=8$ this equals
		$1.3368$, so the attainable reduction is
		$1-1.3368^{-1/2}=13.5\%$. The discrete profile
		$n^\ell_t=p+50\ell$, $\ell\le L=50$, $p=300$ gives
		$1+\mathcal K_{t,j}=1.3617$ and $14.3\%$. Finally, the small-dispersion
		expansion of Remark~\ref{rem:dispersion} follows by writing
		$v_\ell=\bar v(1+\tilde{w}_\ell)$ with $\sum_\ell \tilde{w}_\ell=0$. Then
		$A_t^{-1}\sum_\ell v^{-1}_\ell
		=\bar v^{-1}\{1+A_t^{-1}\sum_\ell w^2_\ell
		+O(\max_\ell|\tilde{w}_\ell|^3)\}$, so
		\begin{equation}
			\label{eq:KwexpansionA}
			\mathcal K_{t,j}
			=\frac1{A_t}\sum_{\ell\in\mathcal A_t}w^2_{\ell,t,j}
			+O\bigl(\max_\ell|\tilde{w}_{\ell,t,j}|^3\bigr) .
		\end{equation}
	\end{proof}
	
	\begin{proof}[Proof of Proposition~\ref{prop:softmax_releff}]
		Throughout, $v_\ell:=v_{\ell,t,j}$, $\tilde{w}_\ell:=\tilde{w}_{\ell,t,j}$,
		$\gamma:=\gamma_{t,j}$, and all quantities are
		$\mathcal F_{t-1}$-measurable, so the argument is deterministic given
		$\mathcal F_{t-1}$ apart from the explicitly flagged
		$O_{\mathbb P}(\sqrt{\eta_t})$ terms.
		
		\medskip\noindent
		\textbf{Step 1 (the exact identity).}
		By Proposition~\ref{prop:dispersion},
		$\mathcal V^{\mathrm{unif}}_{t,j}-\mathcal I^{-1}_{t,j}
		=\mathcal K_{t,j}\mathcal I^{-1}_{t,j}$, while by Step~3 of the proof
		of Proposition~\ref{prop:efficiency},
		$\mathcal V_{t,j}(\omega)-\mathcal I^{-1}_{t,j}
		=\mathcal I^{-1}_{t,j}\sum_\ell\omega^{\mathrm{opt}}_{\ell,t,j}
		(\omega_{\ell,t,j}/\omega^{\mathrm{opt}}_{\ell,t,j}-1)^2$.
		Subtracting and dividing gives the second equality in
		\eqref{eq:releff}. Taking $\omega=A_t^{-1}\bm 1$ makes the numerator
		of the first expression vanish, so $\mathrm{RE}_{t,j}=0$. Taking
		$\omega=\omega^{\mathrm{opt}}_{\cdot,t,j}$ makes the sum in the second
		expression vanish, so $\mathrm{RE}_{t,j}=1$. Note that
		$\mathcal K_{t,j}>0$ is exactly what makes the ratio well defined,
		and that $\mathrm{RE}_{t,j}\le1$ always, with $\mathrm{RE}_{t,j}<0$
		possible.
		
		\medskip\noindent
		\textbf{Step 2 (log-ratio expansion).}
		Both $\omega_{\cdot,t,j}$ and $\omega^{\mathrm{opt}}_{\cdot,t,j}$ are
		normalised, so it suffices to compare log-ratios up to an additive
		$\ell$-free constant. Since
		$\omega^{\mathrm{opt}}_{\ell,t,j}\propto v^{-1}_\ell
		=\bar v^{-1}_{t,j}(1+\tilde{w}_\ell)^{-1}$,
		\[
		\log\omega^{\mathrm{opt}}_{\ell,t,j}
		=-\log(1+\tilde{w}_\ell)+\mathrm{const}
		=-\tilde{w}_\ell+\tfrac12w^2_\ell+O(w^3_\ell)+\mathrm{const}.
		\]
		For the soft-max, \eqref{eq:fluct_conc} and Step~4 of the proof of
		Lemma~\ref{lem:fluct} give
		$V_{\ell,t-1,j}=\kappa_\rho v_\ell
		\{1+O_{\mathbb P}(\sqrt{\eta_t})\}+o(v_\ell)$, the $o(v_\ell)$
		absorbing the drift and consensus-feedback terms of
		\eqref{eq:fluct} by hypothesis. Hence, with
		$\tau_{t,j}=\gamma/(\kappa_\rho\bar v_{t,j})$,
		\begin{align*}
			\log\omega_{\ell,t,j}
			=&-\tau_{t,j}V_{\ell,t-1,j}+\mathrm{const}\\
			=&-\gamma(1+\tilde{w}_\ell)
			+\mathrm{const}+O_{\mathbb P}\bigl(\tau_{t,j}V_{\ell,t-1,j}
			\sqrt{\eta_t}\bigr)\\
			=&-\gamma \tilde{w}_\ell+\mathrm{const}+O_{\mathbb P}(\sqrt{\eta_t}),
		\end{align*}
		using $\tau_{t,j}V_{\ell,t-1,j}=\gamma(1+\tilde{w}_\ell)
		\{1+O_{\mathbb P}(\sqrt{\eta_t})\}=\Theta_{\mathbb P}(1)$, which is
		where the boundedness of the {exponent}
		$\tau_{t,j}V_{\ell,t-1,j}$ under
		Assumption~\ref{assum:temp} enters. No bound on $\tau_{t,j}$ alone is
		used or available. Subtracting, all $\ell$-free terms
		cancel and, uniformly in $\ell\in\mathcal A_t$,
		\begin{equation}
			\label{eq:logratio}
			\log\frac{\omega_{\ell,t,j}}
			{\omega^{\mathrm{opt}}_{\ell,t,j}}
			=-(\gamma-1)\tilde{w}_\ell-\tfrac12w^2_\ell+O(w^3_\ell)
			+O_{\mathbb P}(\sqrt{\eta_t}) .
		\end{equation}
		The choice $\gamma=1$ is precisely what annihilates the term of order
		$\tilde{w}_\ell$. This is the sense in which the temperature
		\eqref{eq:tau_mean} matches $\omega^{\mathrm{opt}}$ to first order.
		
		\medskip\noindent
		\textbf{Step 3 (summation).}
		Exponentiating \eqref{eq:logratio} and using $e^x-1=x+O(x^2)$,
		\[
		\frac{\omega_{\ell,t,j}}{\omega^{\mathrm{opt}}_{\ell,t,j}}-1
		=-(\gamma-1)\tilde{w}_\ell-\tfrac12w^2_\ell+O(w^3_\ell)
		+O_{\mathbb P}(\sqrt{\eta_t}) ,
		\]
		so that, squaring and using
		$\sum_\ell\omega^{\mathrm{opt}}_{\ell,t,j}=1$ together with
		$\omega^{\mathrm{opt}}_{\ell,t,j}\asymp A_t^{-1}$
		(Lemma~\ref{lem:weights} and Assumption~\ref{assum:balance}),
		\begin{align*}
			\sum_\ell\omega^{\mathrm{opt}}_{\ell,t,j}
			\Bigl(\frac{\omega_{\ell,t,j}}
			{\omega^{\mathrm{opt}}_{\ell,t,j}}-1\Bigr)^2
			=&(\gamma-1)^2\frac{1}{A_t}\sum_\ell w^2_\ell
			\bigl(1+O(1)\bigr)
			+O\Bigl(\frac{1}{A_t}\sum_\ell w^4_\ell\Bigr)\\
			&+O\bigl(|\gamma-1|\max_\ell|\tilde{w}_\ell|
			\tfrac1{A_t}\textstyle\sum_\ell w^2_\ell\bigr)
			+O_{\mathbb P}(\eta_t) ,
		\end{align*}
		the cross term between $-(\gamma-1)\tilde{w}_\ell$ and $-\tfrac12w^2_\ell$
		being $O(|\gamma-1|\,A_t^{-1}\sum_\ell|\tilde{w}_\ell|^3)$ and hence absorbed
		in the third term. No cancellation is available in the
		$O_{\mathbb P}(\eta_t)$ term, every summand being a square.
		
		\medskip\noindent
		\textbf{Step 4 (division by $\mathcal K_{t,j}$).}
		By \eqref{eq:KwexpansionA},
		$\mathcal K_{t,j}=A_t^{-1}\sum_\ell w^2_\ell
		\{1+O(\max_\ell|\tilde{w}_\ell|)\}$, and
		$A_t^{-1}\sum_\ell w^4_\ell\le\max_\ell w^2_\ell\cdot
		A_t^{-1}\sum_\ell w^2_\ell$. Dividing the display of Step~3 by
		$\mathcal K_{t,j}$ therefore gives
		\[
		1-\mathrm{RE}_{t,j}(\omega)
		=(\gamma-1)^2+O\bigl(\max_\ell|\tilde{w}_\ell|\bigr)
		+O_{\mathbb P}\Bigl(\frac{\eta_t}{\mathcal K_{t,j}}\Bigr) ,
		\]
		which is \eqref{eq:releff_rate}. The term $O(\max_\ell|\tilde{w}_\ell|)$
		collects both $O(\max_\ell w^2_\ell)$ from the quartic sum and
		$O(|\gamma-1|\max_\ell|\tilde{w}_\ell|)$ from the cross term. Under
		\eqref{eq:releff_cond} all three correction terms vanish, so
		$\mathrm{RE}_{t,j}\to1$ and, by \eqref{eq:releff},
		$\mathcal V_{t,j}(\omega)=\mathcal I^{-1}_{t,j}
		\{1+(1-\mathrm{RE}_{t,j})\mathcal K_{t,j}\}
		=\mathcal I^{-1}_{t,j}(1+o_{\mathbb P}(1))$, the last step using
		$\mathcal K_{t,j}=O(1)$ from \eqref{eq:varthetabounds} and
		Assumption~\ref{assum:balance}.
		
		\medskip\noindent
		\textbf{Step 5 (the two roles of the smoothing).}
		Since $\tau_{t,j}V_{\ell,t-1,j}=\Theta_{\mathbb P}(1)$ by Step~2, the
		relative fluctuation of $\omega_{\ell,t,j}$ equals that of
		$V_{\ell,t-1,j}$, namely $O_{\mathbb P}(\sqrt{\eta_t})$ by
		\eqref{eq:fluct_conc}, and this holds {only because} both
		components of the metric are exponentially smoothed. An unsmoothed
		squared statistic has relative fluctuation $\Theta_{\mathbb P}(1)$,
		which would leave a term of order one in \eqref{eq:logratio} and
		destroy the first-order match. This is the first role. The second is
		quantitative. The resulting contribution to the gap is of exact order
		$\eta_t$, so it is negligible relative to the available gain
		$\mathcal K_{t,j}$ only under $\eta_t=o(\mathcal K_{t,j})$, failing
		which $\mathrm{RE}_{t,j}<0$ and uniform weighting is preferable. Both
		roles are recorded in Remark~\ref{rem:eta_twosided}.
		
		\medskip\noindent
		\textbf{Step 6 (the temperature of \eqref{eq:tau_mean}).}
		Write $\bar V_{t-1,j}:=A_t^{-1}\sum_m V_{m,t-1,j}$. By
		\eqref{eq:fluct_conc} each $V_{m,t-1,j}=\kappa_\rho v_m
		\{1+O_{\mathbb P}(\sqrt{\eta_t})\}$, and the fluctuations are
		conditionally independent across $m$ given
		$\mathcal F_{t-\lceil1/\eta_t\rceil}$, so averaging over $A_t$ nodes
		gives $\bar V_{t-1,j}=\kappa_\rho\bar v_{t,j}
		\{1+O_{\mathbb P}(\sqrt{\eta_t/A_t})\}$. Hence
		$\tau_{t,j}=\bar V^{-1}_{t-1,j}$ satisfies
		$\gamma_{t,j}=\kappa_\rho\bar v_{t,j}\tau_{t,j}
		=1+O_{\mathbb P}(\sqrt{\eta_t/A_t})\to1$, and $\tau_{t,j}$ is
		$\mathcal F_{t-1}$-measurable since it uses only
		$\{V_{m,t-1,j}\}$. Assumption~\ref{assum:temp} holds automatically for
		this rule. The exponent is
		$\tau_{t,j}V_{\ell,t-1,j}=V_{\ell,t-1,j}/\bar V_{t-1,j}
		=(1+\tilde{w}_\ell)\{1+O_{\mathbb P}(\sqrt{\eta_t})\}$, and
		$\max_\ell|\tilde{w}_\ell|$ is bounded by \eqref{eq:varthetabounds} together
		with Assumption~\ref{assum:balance}, so
		$$
		\underline\tau\,\underline\theta
		\ \le\ \tilde\tau_t\Psi_{\ell,t,j}\ \le\ \bar\tau\bar\theta
		\qquad\text{for all }
		\ell\in\mathcal A_t\text{ and all }t
		$$
		holds with
		$\underline\theta,\bar\theta$ depending only on
		$C_{\mathrm{bal}}$ and the bounds in \eqref{eq:varthetabounds}. No
		restriction on the growth of $\tau_{t,j}$ is needed, and none would be
		tenable. Since the exponent of \eqref{eq:weights} is
		$\tilde\tau_t\Psi_{\ell,t,j}=\tilde\tau_t V_{\ell,t-1,j}
		/\bar V_{t-1,j}$ whenever the clip is inactive, the rule
		\eqref{eq:tau_mean} is the case $\tilde\tau_t=1$, and
		\eqref{eq:gamma_tilde} follows. By contrast the absolute-scale
		rule \eqref{eq:tau_calib} yields
		$\gamma_{t,j}\to\bar v_{t,j}A_t\bar n_t/{ T}_t$, which equals one
		iff $\vartheta^2_{\ell,t,j}$ is constant in $\ell$, as asserted in
		Remark~\ref{rem:tau_mean}.
	\end{proof}
	
	\begin{prop}[Selection bias of non-predictable weights]
		\label{prop:B4}
		Suppose the weights were formed from the current round, say
		$\omega^{\mathrm{np}}_{\ell,t,j}\propto
		\exp\{-\tau_{t,j}(\check\lambda_{\ell,t,j}
		-\tilde\lambda_{t-1,j})^2\}$, all other elements of the scheme
		being unchanged, and assume Assumption~\ref{assum:eighth}. Then,
		conditionally on $\mathcal F_{t-1}$ and in the regime of
		Proposition~\ref{prop:softmax_releff},
		\begin{equation}
			\label{eq:B4}
			\mathbb E\Bigl[\sum_{\ell\in\mathcal A_t}
			\omega^{\mathrm{np}}_{\ell,t,j}
			\varepsilon^{\mathrm{stoch}}_{\ell,t,j}
			\Bigm|\mathcal F_{t-1}\Bigr]
			=-\tau_{t,j}\Bigl(1-\frac1{A_t}\Bigr)
			\frac1{A_t}\sum_{\ell\in\mathcal A_t}
			\mathbb E\bigl[(\varepsilon^{\mathrm{stoch}}_{\ell,t,j})^3
			\mid\mathcal F_{t-1}\bigr]
			+O_{\mathbb P}\bigl(\tau^2_{t,j}\max_\ell(n^\ell_t)^{-2}\bigr),
		\end{equation}
		which at the calibrated temperature of
		Proposition~\ref{prop:softmax_releff} is of exact order
		$\vartheta_{t,j}/n^\ell_t$ and is therefore
		{not} negligible relative to the stochastic error
		$O_{\mathbb P}(\vartheta_{t,j}/\sqrt{N_t})$ unless
		$A_t=o(n^\ell_t)$.
	\end{prop}
	
	\begin{proof}[Proof of Proposition~\ref{prop:B4}]
		Write $\varepsilon_\ell:=\varepsilon^{\mathrm{stoch}}_{\ell,t,j}$
		and $S_\ell:=(\check\lambda_{\ell,t,j}
		-\tilde\lambda_{t-1,j})^2$, so that
		$\omega^{\mathrm{np}}_\ell=e^{-\tau S_\ell}
		/\sum_me^{-\tau S_m}$. Conditionally on $\mathcal F_{t-1}$ the
		pairs $(\varepsilon_\ell,S_\ell)$ are independent across $\ell$
		with $\mathbb E\varepsilon_\ell=0$ and
		$\tau S_\ell=O_{\mathbb P}(1)$. Expanding
		$e^{-\tau S_\ell}=1-\tau S_\ell+O(\tau^2S^2_\ell)$ and
		$\sum_me^{-\tau S_m}=A_t\{1-\tau\bar S+O(\tau^2\overline{S^2})\}$
		with $\bar S:=A_t^{-1}\sum_mS_m$,
		\[
		\omega^{\mathrm{np}}_\ell
		=\frac1{A_t}\Bigl\{1-\tau(S_\ell-\bar S)
		+O_{\mathbb P}\bigl(\tau^2\max_mS^2_m\bigr)\Bigr\},
		\]
		whence
		\[
		\mathbb E\Bigl[\sum_\ell\omega^{\mathrm{np}}_\ell
		\varepsilon_\ell\Bigm|\mathcal F_{t-1}\Bigr]
		=-\frac{\tau}{A_t}\sum_\ell
		\mathbb E\bigl[\varepsilon_\ell(S_\ell-\bar S)\bigr]
		+O_{\mathbb P}\bigl(\tau^2\max_\ell(n^\ell_t)^{-2}\bigr),
		\]
		using $\mathbb E\varepsilon_\ell=0$ for the leading term. Writing
		$d_\ell:=\lambda_{t,j}+b_{\ell,t,j}-\tilde\lambda_{t-1,j}$, which is
		$\mathcal F_{t-1}$-measurable, gives exactly
		$S_\ell=\varepsilon_\ell^2+2d_\ell\varepsilon_\ell+d^2_\ell$, so the
		leading contribution to $\mathbb E[\varepsilon_\ell S_\ell]$ is
		$\mathbb E\varepsilon^3_\ell$, while
		$\mathbb E[\varepsilon_\ell\bar S]=A_t^{-1}
		\mathbb E\varepsilon_\ell S_\ell$ by independence across $\ell$. 
		Finally, for a normalized sum of
		$n^\ell_t$ i.i.d.\ centred summands the third cumulant scales as
		$(n^\ell_t)^{-2}$, so
		$\mathbb E\varepsilon^3_\ell\asymp\vartheta^3_{\ell,t,j}
		\gamma_3/(n^\ell_t)^2$ with $\gamma_3$ the standardized skewness of
		the quadratic form $(\bm b^\top\bm z)^2$. At the calibrated
		temperature \eqref{eq:tau_mean}, for which
		$\tau_{t,j}\asymp1/(\kappa_\rho\bar v_{t,j})
		\asymp\bar n_t/(\kappa_\rho\vartheta^2_{t,j})$, this yields
		$|\text{bias}|\asymp\vartheta_{t,j}/n^\ell_t$, as claimed.
		Comparing with $O_{\mathbb P}(\vartheta_{t,j}N_t^{-1/2})$ and
		using $N_t\asymp A_tn^\ell_t$ gives the stated condition. Under the
		predictable weights \eqref{eq:weights} this entire term is
		absent, by Step~1 of the proof of Lemma~\ref{lem:wavg}.
	\end{proof}

	\appendix

\end{document}